\documentclass[utopia
              ,singlespaced
              ,fullpage
              ]{ucalgmthesis}
\usepackage{pgfplots}
\usepgfplotslibrary{patchplots}
\usetikzlibrary{patterns, positioning, arrows}

\usepackage{tikz}
\usepackage{tikz-cd}

\usetikzlibrary{backgrounds}
\usetikzlibrary{arrows}
\usetikzlibrary{shapes,shapes.geometric,shapes.misc}

\tikzstyle{tikzfig}=[baseline=-0.25em,scale=0.5]

\pgfkeys{/tikz/tikzit fill/.initial=0}
\pgfkeys{/tikz/tikzit draw/.initial=0}
\pgfkeys{/tikz/tikzit shape/.initial=0}
\pgfkeys{/tikz/tikzit category/.initial=0}

\pgfdeclarelayer{edgelayer}
\pgfdeclarelayer{nodelayer}
\pgfsetlayers{background,edgelayer,nodelayer,main}

\tikzstyle{none}=[inner sep=0mm]

\newcommand{\tikzfig}[1]{%
{\tikzstyle{every picture}=[tikzfig]
\IfFileExists{#1.tikz}
  {\input{#1.tikz}}
  {%
      \IfFileExists{./figures/#1.tikz}
        {\input{./figures/#1.tikz}}
        {\tikz[baseline=-0.5em]{\node[draw=red,font=\color{red},fill=red!10!white] {\textit{#1}};}}%
    }
  }%
}

\tikzstyle{every loop}=[]

\tikzstyle{new style 0}=[fill=white, draw=black, shape=circle]

\usetikzlibrary{calc}
\usetikzlibrary{decorations.pathmorphing}

\tikzset{curve/.style={settings={#1},to path={(\tikztostart)
    .. controls ($(\tikztostart)!\pv{pos}!(\tikztotarget)!\pv{height}!270:(\tikztotarget)$)
    and ($(\tikztostart)!1-\pv{pos}!(\tikztotarget)!\pv{height}!270:(\tikztotarget)$)
    .. (\tikztotarget)\tikztonodes}},
    settings/.code={\tikzset{quiver/.cd,#1}
        \def\pv##1{\pgfkeysvalueof{/tikz/quiver/##1}}},
    quiver/.cd,pos/.initial=0.35,height/.initial=0}

\tikzset{tail reversed/.code={\pgfsetarrowsstart{tikzcd to}}}
\tikzset{2tail/.code={\pgfsetarrowsstart{Implies[reversed]}}}
\tikzset{2tail reversed/.code={\pgfsetarrowsstart{Implies}}}
\tikzset{no body/.style={/tikz/dash pattern=on 0 off 1mm}}

\usepackage{amsmath,amsthm}

\usepackage{thmtools} 
\usepackage{thm-restate}

\theoremstyle{plain}
\newtheorem{theorem}{Theorem}[section]
\newtheorem{definition}[theorem]{Definition}
\newtheorem{observation}[theorem]{Observation}
\newtheorem{proposition}[theorem]{Proposition}
\newtheorem{corollary}[theorem]{Corollary}
\newtheorem{remark}[theorem]{Remark}
\newtheorem{lemma}[theorem]{Lemma}

\newtheorem{example}[theorem]{Example}
\newtheorem{notation}[theorem]{Notation}

\usepackage{microtype}

\usepackage[colorlinks]{hyperref} 

\usepackage[numbered]{bookmark}

\usepackage[round]{natbib}
\usepackage{blindtext}

\author{Benjamin MacAdam}

\title{The functorial semantics of Lie theory}

\degree{Doctor of Philosophy}

\prog{Graduate Program in Computer Science}

\monthname{JUNE}

\thesisyear{2022}

\hypersetup{pdfinfo={Title={\thetitle},Author={\theauthor}}}

\renewcommand{\o}{\circ}
\newcommand{\anc}{\varrho}

\newcommand{\prolong}{A \ts{\anc}{T\pi} TA}
\newcommand{\prol}{\mathcal{L}}
\renewcommand{\d}[0]{\mathcal{D}}
\newcommand{\s}[0]{\mathsf{Set}}
\newcommand{\wone}[0]{\mathsf{Weil}_1}
\newcommand{\weil}[0]{\mathsf{Weil}}
\newcommand{\w}[0]{{\mathcal{W}}}
\newcommand{\W}[0]{{\mathcal{W}}}

\renewcommand{\c}[0]{\mathcal{C}}
\renewcommand{\a}[0]{\mathcal{A}}
\newcommand{\e}[0]{\mathcal{E}}
\newcommand{\m}[0]{\mathcal{M}}

\renewcommand{\j}[0]{\mathcal{J}}

\newcommand{\vv}[0]{\mathcal{V}}

\newcommand{\C}[0]{\mathbb{C}}
\newcommand{\D}[0]{\mathbb{D}}
\newcommand{\N}[0]{\mathbb{N}}
\newcommand{\R}[0]{\mathbb{R}}

\newcommand{\Z}[0]{\mathbb{Z}}
\newcommand{\sh}[0]{\mathcal{S}}
\renewcommand{\th}[0]{\mathcal{T}}

\newcommand{\tnat}[0]{\alpha} 
\newcommand{\x}[0]{\times}
\newcommand{\ox}[0]{\otimes}
\newcommand{\phat}[0]{\hat{p}}
\newcommand{\yon}[0]{\mathcal{Y}}
\newcommand{\cat}[0]{\mathsf{Cat}}

\newcommand{\<}{\langle}
\renewcommand{\>}{\rangle}
\newcommand{\ts}[2]{{_{#1}}\!\times_{#2}\!}

\usepackage{todonotes}
\usepackage{proof}
\usepackage{cleveref}
\newcommand{\Diff}[0]{\mathsf{Diff}}
\newcommand{\Dlin}[0]{\mathsf{DLin}}

\pgfplotsset{compat=1.16}

\begin{document}

\frontmatter



\makethesistitle



\chapter{Abstract}


Ehresmann's introduction of differentiable groupoids in the 1950s may be seen as a starting point for two diverging lines of research, many-object Lie theory (the study of Lie algebroids and Lie groupoids) and sketch theory. This thesis uses tangent categories to build a bridge between these two lines of research, providing a structural account of Lie algebroids and the Lie functor. 

To accomplish this, we develop the theory of involution algebroids, which are a tangent-categorical sketch of Lie algebroids. We show that the category of Lie algebroids is precisely the category of involution algebroids in smooth manifolds, and that the category of Weil algebras is precisely the classifying category of an involution algebroid. This exhibits the category of Lie algebroids as a tangent-categorical functor category, and the Lie functor via precomposition with a functor
\[
    \partial: \wone \to \th_{\mathsf{Gpd}},
\]bringing Lie algebroids and the Lie functor into the realm of \emph{functorial semantics}.


\chapter{Preface}

This thesis is the original work of the author.

This thesis project arose from attempts to broadly understand differential geometry, and in particular the differential geometry of mechanics. First steps were taken with Jonathan Gallagher in understanding how the enriched perspective on category theory introduced in \cite{Garner2018} might relate to differential geometric structures. The basic structures in this thesis, however, came out of discussions and collaboration with Matthew Burke in the Lie theory context.

\Cref{ch:tangent_categories} is an introduction to the theory of tangent categories, and contains no new results.  
\Cref{ch:differential_bundles} began as a joint project with Matthew, who ultimately had to leave the project due to time constraints, although by that time we had already found the basic structure of the proof that differential bundles are vector bundles in the category of smooth manifolds (proved in Theorem \ref{iso-of-cats-dbun-sman}).
The current structure of the chapter, particularly with the emphasis on associative coalgebras of the weak tangent comonad $(T,\ell)$ and the tight connection to Grabowski's previous work on the Euler Vector Field construction (\cite{Grabowski2009}), is original work.
The basic results in that chapter appear in \cite{MacAdam2021}; however, the results have been streamlined (there was originally a notion of a \emph{strong} differential bundle, in Section \ref{sec:lifts-pdbs-dbs} it is observed that all differential bundles are strong), and there are some new observations about linear connections (Theorem \ref{thm:linear-connection-from-total-space}).

\Cref{ch:involution-algebroids} is a rewrite of a preprint written with Matthew on involution algebroids (\cite{Burke2019}), while \Cref{sec:connections_on_an_involution_algebroid} is due to conversations with Richard Garner (we expect to release a new paper on involution algebroids based on these results as a joint work). The original idea of augmenting an anchored bundle with an involution map is entirely due to Matthew, and the isomorphism on objects between involution and Lie algebroids follows calculations shared by Richard Garner (Proposition \ref{prop:inv-algd-mor-conn}). My own contribution in this section comes from connecting this to the work of \cite{Martinez2001}, as well as giving the bijection on morphisms for involution and Lie algebroids (Theorem \ref{thm:iso-of-cats-Lie}).

The first three chapters provide the background required for the deeper results in \Cref{chap:weil-nerve}. 
In the work of \cite{Weinstein1996,Martinez2001,Leon2005} on Lie algebroids, it was observed that the ``prolongation'' of a Lie algebroid acted like a tangent bundle. Proposition \ref{prop:second-tangent-structure-inv-algds-2} makes this intuition precise by showing the prolongation is a second tangent structure on the category of Lie algebroids. The main theorem in this chapter (Theorem \ref{thm:iso-of-cats-inv-emcs}) shows that involution algebroids in a tangent category $\C$ are equivalent to tangent functors $(A,\alpha)$ from $\wone$ to $\C$ so that the functor $A$ preserves transverse limits and the natural transformation $\alpha$ is $T$-cartesian (Definition \ref{def:cart-nat}). These are entirely new observations about Lie algebroids and are original work of the author. 


Finally, \Cref{ch:inf-nerve-and-realization} puts the first four chapters into the language of enriched category theory, using Garner's enriched perspective on tangent categories (\cite{Garner2018}). The first original results in this chapter demonstrate that differential bundles and anchored bundles are models of $\w$-sketches (Propositions \ref{prop:Lambda-is-refl-subcat} and \ref{prop:nerve-anc-work}), where $\w$ is the site of enrichment for tangent categories.
Next, the enriched theories framework of \cite{Bourke2019} is used to prove Theorem \ref{thm:pullback-in-cat-of-cats-inv-algd}, that involution algebroids are models of a nervous theory, which is the enriched version of Theorem \ref{thm:iso-of-cats-inv-emcs}. 
The thesis concludes with the \emph{Lie Realization}, Theorem \ref{thm:lie-realization}, which is a new characterization of Lie differentiation and introduces an entirely new way to construct adjunctions between categories of ``smooth groupoids'' and categories of ``Lie algebroids'' using purely enriched-categorical methods (this is in contrast to the geometric approach used in \cite{Crainic2003} and the homotopy theoretic approach in \cite{Sullivan1977}).

This thesis touches on relatively advanced topics in two areas of math: differential geometry (Lie algebroids) and enriched category theory (enriched nerve constructions).
I have striven to keep it as self-contained as possible, introducing the category of smooth manifolds and tangent categories and including an appendix with the basics of enriched category theory and locally presentable category theory. Material on foundational category theory (which is to say, anything that can be found in \cite{MacLane1988}) and basic differential calculus (see any calculus textbook) is used without citation or introduction; this includes limit, adjunction, monads and monadicity theorems, and the calculus of Kan extensions and coends. Some facts about horizontal composition spans are used in \Cref{chap:weil-nerve}, but nothing that goes beyond the basic definition.

  
\chapter{Acknowledgments}

I would first like to thank my supervisor, Robin Cockett, for guiding this project and for his investment of time in reading various preprints and walking through proofs with me. I also owe a special debt of thanks to Matthew Burke, who was a close collaborator on this project and spearheaded the application of tangent categories to Lie theory, and to Jonathan Gallagher for extensive discussions about tangent categories. I am also grateful to Rory Lucyshyn-Wright for spending a summer teaching me enriched category theory, and to Kristine Bauer for helpful conversations over the years. I would also like to thank all of the members of the Peripatetic Seminar and others over the years for friendly discussion and feedback: Chad Nester, Prashant Kumar, J.S. Lemay, Priyaa Srinivasan, Cole Comfort, Daniel Satanove, Rachel Hardeman, and Geoff Vooys. 

Finally, I am grateful to my parents and family, and to my wife Niloofar for her support as I wrote this thesis.

\dedication{For my wife, Niloofar.}

\chapter{Notation}\label{ch:notation}

We start with a table of symbols:
\begin{center}
    \begin{tabular}{|p{2cm}|p{13cm}|}
    \hline
    \textbf{Notation}  &
       \\\hline
    $\C, \D, \dots$  &
      A (usually tangent) category, treated as a general context for mathematics, denoted using mathbb \\ \hline
    $\a, \c, \dots$  &
     A small category treated as a mathematical object, denoted by mathcal \\ \hline
    $A, B, \dots$ &
     Objects in a category and also functors. written using capital letters \\ \hline
    $f, \phi$ &
     Morphisms in a category and natural transformations: lower-case Roman and Greek letters \\ \hline
    $f \o g$ &
     The composition of two maps $g:A \to B, f:B \to C$ (applicative notation) \\ \hline
    $\pi_i$ &
     The projection from the $i^{th}$ component of an $n$-fold pullback $A_0 \ts{q_0}{q_1} \dots \ts{q_{n-1}}{q_n}A_n$ or a product $\prod^n A_i$ \\ \hline
    $F.G$ &
     Composition of two functors, $F:\mathbb{D} \to \mathbb{E}, G: \mathbb{C} \to \mathbb{D}$ (applicative notation) \\ \hline
    $\phi.G$ &
     Whiskering of a natural transformation \\ \hline
    $\ox$ &
     Tensor product in a monoidal category \\ \hline
    $\boxtimes$ &
     A restricted notion of span composition, introduced in Definition \ref{def:boxtimes-span} \\ \hline
     $T, p, 0, +, \ell, c$ &
    The data for an arbitrary tangent category, introduced in Definition \ref{def:tangent-cat} \\ \hline
    $D, \odot, 0, !, \delta$ &
     The data for an infinitesimal object, introduced in Definition \ref{def:inf-object} \\ \hline
    $T^n$ &
     Iterated application of an endofunctor $T$ \\ \hline
    $T_n$ &
     The $n$-fold pullback power of $p:T \to id$ for a tangent category, $T \ts{p}{p} \dots \ts{p}{p} T$ \\ \hline
    \end{tabular}
\end{center}

\newpage 

We also provide a table of categories:
\begin{center}
    \begin{tabular}{|p{2cm}|p{13cm}|}
    \hline
    \textbf{Notation}  &
       \\ \hline
    $\mathsf{SMan}$ &
     The category of smooth manifolds \\ \hline
    $\C^\ell$ &
     The category of lifts in a tangent category $\C$, introduced in \Cref{def:lift} \\ \hline
    $\mathsf{NonSing}(\C)$ &
     The category of non-singular lifts in a tangent category $\C$, introduced in \Cref{def:non-singular-lift} \\ \hline
    $\mathsf{PDiff}(\C)$ &
     The category of pre-differential bundles in a tangent category $\C$, introduced in \Cref{def:pdb}(i) \\ \hline
    $\mathsf{Diff}(\C)$ &
     The category of differential bundles in a tangent category $\C$, introduced in \Cref{def:pdb}(ii) \\ \hline
    $\mathsf{LieAlgd}$ &
     The category of Lie algebroids, introduced in \Cref{sec:Lie_algebroids} \\ \hline
    $\mathsf{InvAlgd}(\C)$ &
    The category of anchored bundles in a tangent category $\C$, introduced in \Cref{def:anchored_bundles} \\ \hline
    $\mathsf{Anc}(A)$ &     
    The category of anchored bundles in a tangent category $\C$, introduced in \Cref{def:anchored_bundles} \\ \hline
    $\mathsf{Anc}^\prol(A)$ &
     The category of involution algebroids with chosen prolongations in a tangent category $\C$ \\ \hline
    $\wone$ &
     The category of Weil algebras, introduced in \Cref{sec:weil-algebras-tangent-structure} \\ \hline
    $W$ &
     Notation for the Weil algebra $\N[x]/x^2$ \\ \hline
    $\w$ &
     The category of transverse-limit-preserving functors $\wone \to \s$, introduced in \Cref{def:weil-space} \\ \hline
    $\wone^n$ & The category of Weil algebras with width $n$, used in \Cref{def:truncated-wone} \\ \hline
    $\wone^*$ & The full subcategory of $\wone$ spanned by $\{ \N, W\}$ \\ \hline
    $\prol$ &
      The classifying $\w$-category of an anchored bundles and all of its prolongations, introduced in Definition \ref{def:prol} \\ \hline
    \end{tabular}
\end{center}
\tableofcontents







\mainmatter

\chapter*{Introduction}


The study of Lie groupoids and Lie algebroids goes back to Charles Ehresmann and his student Jean Pradines in the late 1950s, building upon Sophus Lie's original research into the application of groups of smooth symmetries to solving ordinary differential equations (\cite{lie1893theorie}). Motivated by partial differential equations,  \cite{ehresmann1959categories} introduced the notion of a \emph{differentiable groupoid}, which models the \emph{internal} symmetries of a smooth manifold, in contrast to the \emph{external} symmetries given by a Lie group (a group object in the category of smooth manifolds) or Lie group action.  \cite{Pradines1967} extended the Lie functor (which sends Lie groups and Lie group actions to Lie algebras and Lie algebra actions) from external to internal symmetries and introduced the notion of a \emph{Lie algebroid}. In doing so, he identified some shortcomings in Ehresmann's original definition of differentiable groupoids, introducing the modern notion of a \emph{Lie groupoid}. 

Ehresmann's investigations into differentiable groupoids initiated one of the major differential geometry research programmes of the second half of the twentieth century, the study of Lie groupoids and Lie algebroids (which one may refer to as \emph{many-object} Lie theory to distinguish it from the ``single-object'' Lie groups and Lie algebras that had classically been studied). Research into many-object Lie theory in the 80s and 90s focused on extending Lie's second theorem and the Cartan--Lie theorem, which in modern terms state that Lie algebras form a \emph{coreflective subcategory} of Lie groups; that is, the Lie functor has a fully faithful left adjoint (\cite{MACKENZIE2000445,moerdijk2002integrability,nistor2000groupoids}). The left adjoint is often called \emph{Lie integration}, and in their famous paper \cite{Crainic2003} found the exact conditions governing whether a Lie algebroid integrates to a Lie groupoid. \cite{Weinstein1996} initiated a line of research into classical mechanics on Lie algebroids and groupoids, extending Poincar\`{e}'s development of mechanics on a space with a Lie group action and the Euler--Poincar\`{e} equations (\cite{poincare1901}; see \cite{marle2013} for a modern treatment);  this has been further developed by Eduardo Martinez and his collaborators (\cite{de2005lagrangian,Martinez2001,Martinez2018,fusca2018}). 

However, Ehresmann's work in differentiable groupoids signalled a change in focus for his own research, as he increasingly focused within the then-new area of category theory. Over the course of the 60s and 70s Ehresmann published a string of influential papers in the nascent area of \emph{functorial semantics}, developing the formalism of sketch theory. 
Sketch theory has proven to be highly influential in mathematical logic, and has been active for some 40 years, being extended to syntax/semantics adjunctions \cite{gabriel2006lokal},  enriched category theory \cite{Kelly1982} and generalized limit doctrines \cite{ADAMEK20027}.
The influence of sketch theory can still be seen on Lie theory in the work of Kirill Mackenzie and his collaborators to develop the theory of double Lie algebroids, Lie-algebroid groupoids, double-vector bundles, and other tensor-product theories (\cite{Mackenzie1992, Mackenzie2011}).

This thesis aims to provide a structural account of the Lie functor from Lie groupoids to Lie algebroids, using tangent categories \cite{Cockett2014} to unify Ehresmann's many-object Lie theory and sketch theory. Tangent categories provide a syntactic description of tangent structure based on Kock and Lawvere's synthetic differential geometry (\cite{Kock2006}, \cite{Lawvere1979}) and the Weil functor formalism (a comprehensive account may be found in \cite{Kolar1993}, while the explicit link to abstract tangent structure is found in \cite{Leung2017}, another line of research in differential geometry that has run parallel to modern Lie theory (albeit with some exchange, e.g. \cite{kolar2007functorial}). Recent work has recast tangent categories as a class of \emph{enriched categories} (\cite{Garner2018}), making it possible for modern techniques from sketch theory and functorial semantics to be applied to differential geometry. In doing so, we demonstrate that the language of tangent categories sheds light on the study of classical mechanics on Lie algebroids and groupoids, as well as Mackenzie's investigations into ``Ehresmann doubles'' of vector bundles and Lie algebroids.

\subsection*{Overview}
The first three chapters of this thesis build on previous work in the tangent category literature (for example, \cite{Cockett2017,Cockett2018,LucyshynWright2018}), providing tangent-categorical sketches of differential-geometric structures.  These structures follow Ehresmann's original notion of a sketch quite closely; they are specified as graphs (a collection of objects and arrows in the category) with a set of diagrams that must commute and cones that must be universal, except that the data may now include the tangent functor $T$ and the tangent natural transformations $p$, 0, +, $\ell$, and $c$. Each sketch is accompanied by a proof that its category of models in smooth manifolds (which we shall often write $\mathsf{SMan}$) is precisely the category of geometric structures it seeks to model. 

The first chapter reviews the basic theory of tangent categories, paying particular attention to the category of smooth manifolds. The first examples of ``sketches'' from the tangent categories literature are covered, namely differential objects and affine connections, which model vector spaces and connections respectively (\cite{Cockett2017,Cockett2018}). The chapter concludes with a study of \emph{tangent submersions}, which model submersions from classical differential geometry and are a useful example of the sort of work that occurs in Chapters \ref{ch:differential_bundles} and \ref{ch:involution-algebroids}.

The second and third chapters develop tangent categorical sketches for vector bundles and Lie algebroids respectively. Chapter \ref{ch:differential_bundles} extends the observation due to \cite{Grabowski2009} that the category of vector bundles is a full subcategory of multiplicative monoid actions by the non-negative reals $\R^+$, and then applies the \emph{Euler vector field} construction (Definition \ref{def:evf}). The tangent categorical sketch for a vector bundle, called a \emph{differential bundle}, is then developed based on a morphism $\lambda:E \to TE$,  and an isomorphism of categories between differential bundles in $\mathsf{SMan}$ and the category of smooth vector bundles is proved. Chapter \ref{ch:involution-algebroids} introduces \emph{involution algebroids}, which replace the bracket of a Lie algebroid with an involution map
\[
   \sigma: \prolong \to \prolong
\]
(where $\anc:A \to TM$ is the \emph{anchor} of the Lie algebroid). Using Martinez's presentation of the \emph{structure equations} for a Lie algebroid (\cite{Martinez2001}), we are once again able to prove an isomorphism of categories, this time that the category of Lie algebroids is isomorphic to that of involution algebroids in $\mathsf{SMan}$. This provides the initial bridge between differential geometric structures and tangent categorical sketches, making it possible to apply more sophisticated techniques in Chapters \ref{chap:weil-nerve} and \ref{ch:inf-nerve-and-realization}.



The fourth chapter constructs a syntactic tangent category for Lie algebroids, and demonstrates that Lie algebroids are precisely generalized tangent bundles. We call this result the \emph{Weil nerve}, as it follows the same structure as Grothendieck's original nerve theorem \cite{Segal1974}, in this case using the \emph{categories} presentation of tangent categories due to \cite{Leung2017}. This result has useful implications for the study of generalized mechanics and geometric structures on Lie algebroids, as it introduces a novel tangent structure on the category of Lie algebroids. This novel tangent structure corresponds to Poincar\`{e}/Weinstein/Martinez's characterization of classical mechanics on a Lie algebroid.

The fifth and final chapter introduces the enriched categories perspective on tangent categories from \cite{Garner2018}, so that the work in Chapters \ref{ch:differential_bundles} and \ref{ch:involution-algebroids} may be rephrased using the enriched sketches of \cite{Kelly2005}.  We construct a functor from the syntactic category of involution algebroids to the syntactic category of a groupoid-in-a-tangent-category, thus giving a presentation of the Lie functor in the spirit of Ehresmann's sketch theory. The syntactic version of the Lie functor is built by constructing another novel tangent structure on the category of Lie groupoids (or more generally, groupoids in a tangent category), which also agrees with previous investigations into classical mechanics on a Lie groupoid. As a final result, we demonstrate that in a \emph{locally presentable tangent category} we may use a left Kan extension to construct a left adjoint to the Lie functor, which we call the the \emph{Lie realization}.

\chapter{Tangent categories}%
\label{ch:tangent_categories}

Tangent categories are an example of convergent evolution in mathematics, in which two unrelated lines of research with very different aims have arrived at a common endpoint, in this case the same formal setting for abstract differential geometry. The older line of research has its roots in differential geometry proper and, in particular, Weil's algebraic characterization of the tangent bundle of a smooth manifold (\cite{Weil1953}). Weil's work motivated Kock and Lawvere's development of synthetic differential geometry, presented in the book of the same name \cite{Lawvere1979,Kock2006} as well as the Weil functor formalism of \cite{Kolar1993}. The second, more recent line of research has its foundations in theoretical computer science following the publication of \emph{Linear Logic} by \cite{girard1987linear}. Ehrhard and Regnier noticed that some models of linear logic have a notion of the ``Taylor series approximation'' of a proof; this led to the development of differential linear logic by \cite{ehrhard2003differential}. Blute, Cockett, and Seely studied the categorical semantics of models of linear logic equipped with the derivative operation - that is, they identified those categories whose internal language are were models of differential linear logic - developing a categorical theory of differentiation in \cite{blute2006differential,Blute2009}.

Tangent categories arise naturally in each line of research: on the first path with the distillation of synthetic differential geometry into abstract tangent functors in \cite{Rosicky1984}, and more recently when Cockett and Cruttwell refined abstract tangent functors following their investigations into the manifold categories of cartesian differential restriction categories in \cite{MR2861119,Cockett2014}. In a sense, they are categories that axiomatize Weil's characterization of the tangent bundle as an endofunctor in a way that captures the combinatorics of higher-order derivatives when looking at a certain class of internal commutative monoids (\cite{cockett2011faa}), as will be made precise in Chapter \ref{chap:weil-nerve}. Tangent categories also pull Kock and Lawvere's synthetic differential geometry into the framework of enriched category theory, which is explored in Chapter \ref{ch:inf-nerve-and-realization}.

Instances of tangent structure abound throughout mathematics and computer science. For example,  many categories of geometric spaces have natural tangent structure, such as the category of \emph{convenient manifolds} (the category of manifolds modelled locally by \emph{convenient vector spaces} \cite{Kriegl1997}) and the category of schemes (see point (ii) in Example 2 of \cite{Garner2018}). An example from mathematical logic is the category of K\"{o}the sequence spaces( \cite{MR1934421}), and categorical models of the differential lambda-calculus (\cite{MR4037417}). More recently, tangent and differential categories have found applications in differentiable programming and machine learning (\cite{wilson2021reverse}), and to understanding Johnson and McCarthy's \emph{functor calculus} \cite{Bauer2016}.

This thesis studies differential geometric structures using the language of tangent categories, following the tradition of synthetic differential geometry. As such, this chapter will develop tangent categories with a focus on the category of smooth manifolds. Extending the study of these formal structures in the context of novel tangent categories is a significant endeavour and should be treated as a direction for future research. The first section introduces Cartesian differential categories as the categorical axiomatization of multivariable calculus. The second section introduces the category of smooth manifolds and two characterizations of its tangent bundle (kinematic versus operational),  while the third section identifies the structure of the kinematic tangent bundle that characterizes abstract tangent structures. The fourth section presents a pair of structures that allow for ``local-coordinate calculations'' in the tangent category of differential objects and connections. The final section introduces \emph{tangent submersions}. A submersion is a differentiable map between differentiable manifolds whose differential is everywhere surjective; as a preview of the work in Chapters \ref{ch:differential_bundles} and \ref{ch:involution-algebroids}, this section shows that in the category of smooth manifolds, a tangent submersion is precisely a submersion. Section \ref{sec:submersions} first appeared in \cite{MacAdam2021}, and is the only original work in this chapter.

\section{Differential calculus}%
\label{sec:differential-calculus}

As with most treatments of synthetic differential geometry, e.g. \cite{Kock2006}, it makes sense to begin with the differential calculus - in this case, an introduction to the categorical theory of differentiation. Categorical differentiation has recently gained quite a bit of attention due to its relationship with machine learning \cite{Cockett2019}, and applications to homotopy theory \cite{Bauer2016}. This section will just consider the basic structures introduced in \cite{Blute2009}, and the canonical example of a cartesian differential category (the category of finite-dimensional real vector spaces and smooth maps between them).

\begin{definition}\label{def:clac}[Definition 1.2.1 \cite{Blute2009}]
    A cartesian left additive category is a cartesian category\footnote{We use the standard notation where $1$ is the terminal object, $\x$ is product, and $\pi_i$ is the $i^{th}$ projection.} $\C$ so that:
    \begin{enumerate}[(i)]
        \item Each hom-set $\C(A,B)$ is a commutative monoid with addition $+_{AB}$ and zero map $0_{AB}:A \to B$ (the subscript $AB$ will be suppressed when the context is clear).
        \item The composition operation $\o$ preserves addition \textit{on the left}: \[(g+h) \o f= g \o f + h \o f\]
        \item Projection is an additive map (preserves addition): \[\pi_i \o (f + g) = (\pi_i \o g \o f) + (\pi_i \o g)\]
        Where $\pi_i$ denotes the projection from the $i^{th}$ component of a product or pullback.
j    \end{enumerate}
\end{definition}
There are various examples of cartesian left additive categories - they all fit the same pattern of a category where each object is equipped with a non-natural, but coherent, choice of linear structure:
\begin{example}\label{ex:clacs}
    ~\begin{enumerate}[(i)]
        \item Any category with biproducts is a cartesian left additive category where every map is additive. 
        \item The category of cartesian spaces $\mathsf{CartSp}$, whose objects are finite-dimensional real vector spaces and morphisms are smooth maps between them, is a cartesian left additive category."
        Clearly smooth maps from $A \to B$ are closed under addition, projection is an additive map, and $(g+h) \o f= g \o f + h \o f$.
        \item   The category of topological vector spaces and continuous morphisms is a cartesian left additive category. 
    \end{enumerate}
\end{example}
In fact, cartesian left additive categories may equivalently be described as cartesian categories where each object has a coherent choice of commutative monoid structure.
\begin{proposition}[\cite{Blute2009}]%
    \label{prop:clac-defs}
    The following are equivalent:
    \begin{enumerate}[(i)]
        \item $\C$ is a cartesian left additive category
        \item $\C$ is a cartesian category so that each object has a chosen commutative monoid structure $(A,+_A,0_A)$ where the following coherence holds: \[
    \begin{tikzcd}
        (A \x B)^2 \rar{+_{A\x B}} \dar[swap]{\tau} & A \x B \\
        A^2 \x B^2 \urar[swap]{+_A \x +_B}
    \end{tikzcd}
\] (where $\tau = ((\pi_0\o \pi_0,\pi_0\o \pi_1),(\pi_1\o\pi_0, \pi_1\o\pi_1))$).
        \item There is a category with biproducts $\C^+$ and a bijective-on-objects subcategory inclusion $i: \C^+ \hookrightarrow\C$ that creates products.
    \end{enumerate}
\end{proposition}

Cartesian left additive categories provide an appropriate to define a differentiation operation. Recall that the usual derivative of a map $f: \R \to \R$ from elementary calculus can be written
\[
    \frac{\partial f}{\partial x}: \R \to \R. 
\]
More generally, for a map $f: \R^n \to \R$, one writes the \emph{Jacobian} of $f$ at $x$:
\[
    J[f]: \R^n \to (\R^n \multimap \R^m) :=
    \begin{bmatrix}
        \frac{\partial f_1}{\partial x_1} & \dots & \frac{\partial f_1}{\partial x_n} \\
        \vdots & \ddots & \vdots \\
        \frac{\partial f_1}{\partial x_1} & \dots & \frac{\partial f_1}{\partial x_n} 
    \end{bmatrix}
\]\pagenote{tidied notation, removed jargon}
The Jacobian, however, requires some notion of a ``matrix''\footnote{This would be called an \emph{internal hom} in the 
 categorical logic literature.} representing a linear map from $\mathbb{R}^n$ to $\mathbb{R}^m$---not every category that has a notion of differentiation supports that operation. Instead, the \emph{directional derivative}:
\[
    D[f](x,v) := \lim_{d \to 0} \frac{f(x + t\cdot v)}{t}
\]
gives an appropriately general notion of differentiation that extends to categories where the space of linear maps $A \to B$ is not representable by an object in the category \footnote{In the case of automatic differentiation, it is also worth noticing that computing the directional derivative of a map $\mathbb{R}^n \to \mathbb{R}^m$ has complexity $2\mathcal{O}(f)$, while forming the Jacobian has complexity $n\mathcal{O}(f)$, so the directional derivative is a more appropriate primitive for purely practical computational reasons (see Section 5 of \cite{hoffmann2016hitchhiker} for a discussion of the computational complexity of forward-mode automatic differentiation).}.\pagenote{Spelled out AD.} 
A cartesian differential category axiomatizes the directional derivative as a combinator on a cartesian left-additive category.

\begin{definition}\label{def:cdc}[Definition 2.1.1 in \cite{Blute2009}]
    A cartesian differential category is a cartesian left additive category equipped with a combinator (e.g. a function on hom-sets)\pagenote{
    I have added a precise reference to definition of a CDC, and included a brief definition for what a combinator is.}
    \[
        \infer{A \x A \xrightarrow[{D[f]}]{} B}{A \xrightarrow{f} B}
    \]
    satisfying the following axioms:
  \begin{enumerate}[{[CD.1]}]
      \item Additive:
        \[D[f+g] = D[f] + D[g] \hspace{1cm} D[0] = 0\]
      \item Additive in the second variable:
        \[D[f] \o (g, h+k) = D[f]\o (g,h) + D[f]\o (g,k) \hspace{1cm} D[f] \o (g,0) = 0\]
      \item Projection is linear:
        \[D[\pi_i] =  \pi_i \o \pi_1 \hspace{1cm} D[id] = \pi_1\]
      \item Pairing:
        \[D[(f,g)] = (D[f],D[g])\]
      \item Chain rule:
        \[D[g\o f] = D[g] \o (\pi_0, D[f])\]
      \item Linear in the second variable:
        \[D[D[f]] \o ((a,0),(0,d)) = D[f] \o (a,d)\]
      \item Symmetry of partial differentiation:
        \[D[D[f]] \o ((a,b),(c,d)) = D[D[f]] \o ((a,c),(b,d))\]
  \end{enumerate}
\end{definition}
\begin{example}
    The category of cartesian spaces (Example \ref{ex:clacs}(ii)), $\mathsf{CartSp}$, is the canonical cartesian differential category. Let $f: \mathbb{R}^n \to \mathbb{R}^m $, and consider its Jacobian at $x \in \mathbb{R}^n , J[f](x) \in \mathbb{R}^{n \x m}$.  Define the differential combinator:
	    \[
	        D[f]\o (u,v) = J[f](v) \cdot u = \lim_{t \to 0} \frac{f(x + t\cdot v)}{t}.
	    \]\pagenote{Removed extra examples, since they distracted from the original point.}
	\item In \cite{Bauer2016}, the authors construct a cartesian differential category based on the Abelian functor calculus of \cite{MR1451606}.
	\item In \cite{wilson2021reverse}, the authors consider a cartesian differential category whose objects are $\mathbb{Z}_2$-modules to apply gradient-based methods to learn the parameters of of Boolean circuits. 
	\pagenote{
	We have added to 
	}
\end{example}
Every cartesian differential category comes with a notion of linearity.
This notion of linearity is strictly stronger than additivity - there do exist examples of non-linear additive maps.
\begin{definition}[Definition 2.2.1 of \cite{Blute2009}]\label{def:linear-map-in-cdc}
    A map $f:A \to B$ is \textit{linear} whenever $D[f] \o (0_{AB},id) = f$. 
    \pagenote{
        I have added a precise reference to the definition of a linear map, and the $0$ has now been defined in the definition of a cartesian left additive category.
    }
\end{definition}
We denote the category of linear maps in a cartesian differential category $\C$ as $\mathsf{Lin}(\C)$.
The category $\mathsf{Lin}(\C)$ will have biproducts, and will be a cartesian differential subcategory of $\C$.
\begin{lemma}[Corollary 2.2.3 in \cite{Blute2009}]\label{lem:collected-remarks-lin}\pagenote{
    I have added the reference to this lemma, this was originally missing.
}
    Let $\C$ be a cartesian differential category, and denote its category of linear maps as $\mathsf{Lin}(\C)$.
    \begin{enumerate}[(i)]
        \item Linear maps preserve addition.
        \item The category $\mathsf{Lin}(\C)$ is a bijective-on-objects subcategory of $\C$ with biproducts, and the inclusion $\mathsf{Lin}(\C) \hookrightarrow \C$ creates products. 
        \item Every category with biproducts is a cartesian differential category, where
        \[
            D[f] = f \o \pi_1,
        \]
        and this differential structure makes the inclusion $\mathsf{Lin}(\C) \hookrightarrow \C$ preserve the left additive structure and differential combinator (that is, it is a cartesian differential functor). 
    \end{enumerate}
\end{lemma}

\section{The category of smooth manifolds}%
\label{sec:smooth-manifolds}


Tangent categories axiomatize a more general structure than differential calculus, one in which spaces are only ``locally linear.''
The category of smooth manifolds gives the historically canonical example, and a good portion of this thesis relates to structures internal to that category,
so it seems worthwhile to set a working definition for that context.
We follow \cite{Tu2011}, and allow for disconnected components of a manifold to have different dimensions.

\begin{definition}%
\label{def:smooth-manifold}[Definitions 5.5--5.7 in \cite{Tu2011}]
    A \emph{chart} on a topological space $M$ is pair $(U_i, \phi_i:U_i \hookrightarrow \R^n)$, where $U_i$ is an open subset $U_i \subseteq M$ and $\phi_i:U_i \to \R^n$ is a local homeomorphism.
    An \emph{atlas} is a collection of charts $\{ (U_i,\phi_i:U_i \to \R^n) | i \in I \}$ (where $n$ is fixed for each connected component of $M$) so that for each $i,j \in I$, the transition function $\psi_{i,j}$ that completes the diagram
    \[\begin{tikzcd}
        {\phi_{i}^{-1}(U_{i} \cap U_j)} & {\R^n} \\
        {U_{i} \cap U_j} & {U_j}
        \arrow["{\phi_i}", from=1-1, to=2-1]
        \arrow["\subseteq", from=2-1, to=2-2]
        \arrow["{\phi_j}", from=2-2, to=1-2]
        \arrow["{\psi_{i,j}}", dashed, from=1-1, to=1-2]
    \end{tikzcd}\]
    is a smooth map.
    \begin{figure}
        \centering
        \begin{tikzpicture}

    \path[->] (0.8, 0) edge [bend right] node[left, xshift=-2mm] {$\phi_i$} (-1, -2.9);
    \draw[white,fill=white] (0.06,-0.57) circle (.15cm);
    \path[->] (-0.7, -3.05) edge [bend right] node [right, yshift=-3mm] {$\phi^{-1}_i$} (1.093, -0.11);
    \draw[white, fill=white] (0.95,-1.2) circle (.15cm);

    \path[->] (5.8, -2.8) edge [bend left] node[midway, xshift=-5mm, yshift=-3mm] {$\phi^{-1}_j$} (3.8, -0.35);
    \draw[white, fill=white] (4,-1.1) circle (.15cm);
    \path[->] (4.2, 0) edge [bend left] node[right, xshift=2mm] {$\phi_j$} (6.2, -2.8);
    \draw[white, fill=white] (4.54,-0.12) circle (.15cm);

    \draw[smooth cycle, tension=0.4, fill=white, pattern color=brown, pattern=north west lines, opacity=0.7] plot coordinates{(2,2) (-0.5,0) (3,-2) (5,1)} node at (3,2.3) {$M$};


    \draw[smooth cycle, pattern color=orange, pattern=crosshatch dots] 
        plot coordinates {(1,0) (1.5, 1.2) (2.5,1.3) (2.6, 0.4)} 
        node [label={[label distance=-0.3cm, xshift=-2cm, fill=white]:$U_i$}] {};
    \draw[smooth cycle, pattern color=blue, pattern=crosshatch dots] 
        plot coordinates {(4, 0) (3.7, 0.8) (3.0, 1.2) (2.5, 1.2) (2.2, 0.8) (2.3, 0.5) (2.6, 0.3) (3.5, 0.0)} 
        node [label={[label distance=-0.8cm, xshift=.75cm, yshift=1cm, fill=white]:$U_j$}] {};

    \draw[thick, ->] (-3,-5) -- (0, -5) node [label=above:$\phi_i(U_i)$] {};
    \draw[thick, ->] (-3,-5) -- (-3, -2) node [label=right:$\mathbb{R}^m$] {};

    \draw[->] (0, -3.85) -- node[midway, above]{$\psi_{ij}$} (4.5, -3.85);

    \draw[thick, ->] (5, -5) -- (8, -5) node [label=above:$\phi_j(U_j)$] {};
    \draw[thick, ->] (5, -5) -- (5, -2) node [label=right:$\mathbb{R}^m$] {};

    \draw[white, pattern color=orange, pattern=crosshatch dots] (-0.67, -3.06) -- +(180:0.8) arc (180:270:0.8);
    \fill[even odd rule, white] [smooth cycle] plot coordinates{(-2, -4.5) (-2, -3.2) (-0.8, -3.2) (-0.8, -4.5)} (-0.67, -3.06) -- +(180:0.8) arc (180:270:0.8);
    \draw[smooth cycle] plot coordinates{(-2, -4.5) (-2, -3.2) (-0.8, -3.2) (-0.8, -4.5)};
    \draw (-1.45, -3.06) arc (180:270:0.8);

    \draw[white, pattern color=blue, pattern=crosshatch dots] (5.7, -3.06) -- +(-90:0.8) arc (-90:0:0.8);
    \fill[even odd rule, white] [smooth cycle] plot coordinates{(7, -4.5) (7, -3.2) (5.8, -3.2) (5.8, -4.5)} (5.7, -3.06) -- +(-90:0.8) arc (-90:0:0.8);
    \draw[smooth cycle] plot coordinates{(7, -4.5) (7, -3.2) (5.8, -3.2) (5.8, -4.5)};
    \draw (5.69, -3.85) arc (-90:0:0.8);

\end{tikzpicture}     
        \caption{Overlapping charts in an atlas (Credit for this tex code belongs to user Cragfelt \url{https://tex.stackexchange.com/a/388493/101171}.)}
        \label{fig:overlapping-charts}
    \end{figure}
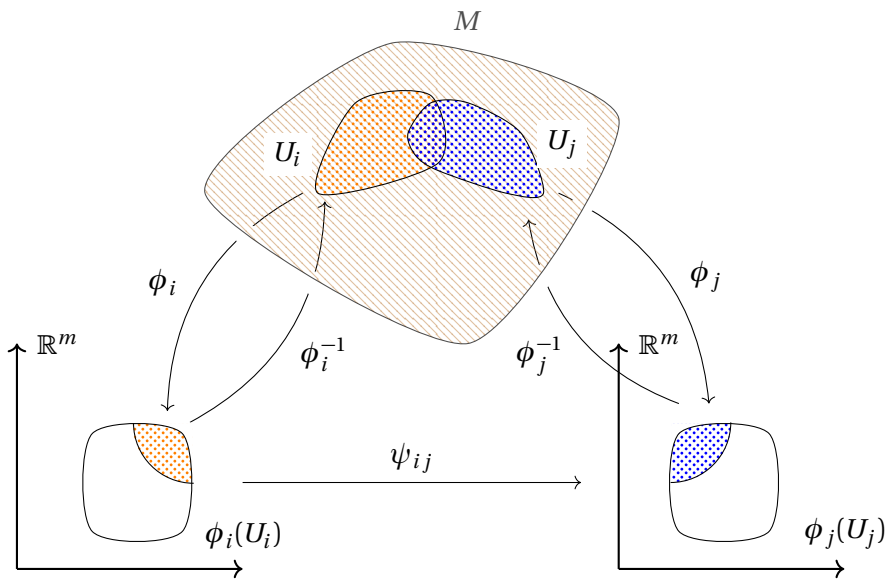
    A \emph{smooth manifold} is a topological space equipped with a (maximal\footnote{With respect to subset inclusion.}) atlas. \pagenote{This definition was originally incomplete, and it lacked a reference. I have added the missing details and included a reference.
    } A morphism of smooth manifolds is a topological map $f: M \to N$ that is locally smooth - for each chart pair of charts $(U_i, \phi_i)$ on $M$ and $(V_j, \theta_j)$, see Figure \ref{fig:overlapping-charts} for an illustration. The map $f$ is smooth whenever each map $f_{i,j}$ that completes the diagram is smooth 
    \[\begin{tikzcd}
        {U_i} & {V_j} \\
        M & N
        \arrow[from=1-1, to=2-1]
        \arrow[from=1-2, to=2-2]
        \arrow["f", from=2-1, to=2-2]
        \arrow["{f_{i,j}}", dashed, from=1-1, to=1-2]
    \end{tikzcd}\]

    The \emph{category of smooth manifolds}, $\mathsf{SMan}$, is the category of smooth manifolds and their morphisms.
\end{definition}
\begin{remark}
    In \Cref{ch:differential_bundles}, some results will implicitly use partition of unity arguments, which require that the underlying topological space for a manifold is \emph{Hausdorff} and has a countable basis (i.e. it is \emph{second-countable}). We will avoid any direct reference to these properties, and so we omit them from the definition of a smooth manifold.
\end{remark}

\begin{example}
    ~\begin{enumerate}[(i)]
        \item Each vector space $\R^n$, for $n\in \N$, has a canonical smooth manifold structure whose atlas is a single chart (the identity map $\mathbb{R}^n \to \mathbb{R}^n$).
        \item Most geometric shapes that do not have any singularities or sharp edges can be equipped with an atlas without any issue.
        For example, consider the circle:  $\{ (\cos(x), \sin(x)) : x \in [-\pi, \pi) \}$
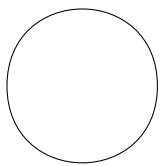
        \[\begin{tikzpicture}
	\begin{pgfonlayer}{nodelayer}
		\node [style=none] (0) at (-1, 0) {};
		\node [style=none] (1) at (1, 0) {};
	\end{pgfonlayer}
	\begin{pgfonlayer}{edgelayer}
		\draw [bend left=90, looseness=1.75] (0.center) to (1.center);
		\draw [bend right=90, looseness=1.75] (0.center) to (1.center);
	\end{pgfonlayer}
\end{tikzpicture}
\]
        For any appropriately small $\epsilon>0$, there are two charts from $I_\epsilon = (-\epsilon, \pi+\epsilon)$; 
        \[
            \phi_0(t) = (\cos(t), \sin(t)), \hspace{0.25cm} \phi_1(t) = (\cos(t-\pi), \sin(t-\pi))
        \]
        making the circle a smooth manifold. \pagenote{
           I have added more details to the chart maps here.
        }
    \end{enumerate}
\end{example}
Category theory has not seen as many applications in differential geometry as it has in topology or algebra, likely because the category of smooth manifolds (Definition \ref{def:smooth-manifold}) is somewhat poorly behaved.
The two main reasons that the category of smooth manifolds is ``inconvenient'' are as follows:
\begin{itemize}
    \item The set of smooth maps between two manifolds $M, N$ fails, in general, to form a smooth manifold; thus, the category is not cartesian closed.
    \item The category of manifolds does not have quotients or arbitrary fibre products.
    \pagenote{
       As per Kristine's recommendation, we give more specific limits and colimits that fail to exist in the category of smooth manifolds.
    }
\end{itemize}
However, the category of smooth manifolds admits some limits, for example finite products.
\begin{proposition}[1.12 \cite{Kolar1993}]
    The category \emph{SMan} of smooth manifolds has finite products.
\end{proposition}
\begin{proof}
    Given two manifolds $M, N$, take the product of their underlying topological spaces together with the product charts
    \[
        (\phi_i \x \psi_j): U_i \x V_j \to \R^n \x \R^m.
    \]
\end{proof}
The category of smooth manifolds---using our definition where disconnected components of a manifold may have different dimensions---does have a class of (co)limits used throughout this paper, namely idempotent splittings.\pagenote{
   I have added the definition of an idempotent splitting.
}
\begin{definition}[\cite{MR850528}]
    An \emph{idempotent} is an endomorphism $e:E \to E$ so that $e \o e = e$. The \emph{splitting} of an idempotent is given by a pair of maps $e = s\o r$ so that $r \o s = id$. The existence of a splitting of an idempotent $e$ is equivalent to asking that the following pair of parallel arrows has a (co)equalizer:
    \[\begin{tikzcd}
        E & E
        \arrow["e", shift left=1, from=1-1, to=1-2]
        \arrow[shift right=1, Rightarrow, no head, from=1-1, to=1-2]
    \end{tikzcd}\]
    The idempotent splitting $\C$ of a category (also known as the \emph{Cauchy completion} of the category), is the full subcategory of presheaves $[\C^{op}, \s]$ that are retracts of representable functors.
    Any functor into a category with idempotent splittings will factor through this inclusion of categories, so it is the \emph{free cocompletion} of $\C$ under idempotent splittings.
\end{definition}
\begin{proposition}[\cite{MR1003203}]
    The category of smooth manifolds is the idempotent splitting of the category whose objects are open subsets of Cartesian spaces and whose morphisms are smooth maps $f: U \to V$.
\end{proposition}
An idempotent in the idempotent splitting of a category is also an idempotent in the base category, and thus admits a splitting.
\begin{corollary}
    The category of smooth manifolds is closed to idempotent splittings: for every map $e:M \to M$ so that $e = e \o e$ there exists a pair of maps $r:Q \to M, s:M \to Q$ so that $e = s\o r, r \o s = id$.\pagenote{
       I have added the coherences on $r,s$, so that $r \o s = id, s \o r = e$.    }
\end{corollary}

The main construction of interest on the category of smooth manifolds, for tangent categories at least, is the \emph{tangent bundle}. Given  a smooth map $f:M \to N$, restrict it to a morphism between coordinate patches, so it may be regarded as a map $f|_U: U \to V$, where $U \subseteq \R^m, V \subseteq \R^n$. This gives a local derivative operation (remembering that $\pi_i$ denotes the $i^{th}$ projection from a product $\prod^n A_i$)\footnote{
The wording here was originally muddled, it has since been corrected.
}
\begin{gather*}
    \infer{(f|_U \o \pi_0, D[f|_U]):U \x \R^m \to V \x \R^n.}{ D[f|_U]: U \x \R^m \to \R^n,}
\end{gather*}
The tangent bundle makes this construction global; that is, there is a functor $T:\mathsf{SMan \to SMan}$ giving an assignment \[T.M \xrightarrow{T.f} T.N\] that agrees with the local derivative on coordinate patches of $M, N$.

\begin{definition}[\cite{Kolar1993}]%
    \label{def:tang-vector}
    Write the algebra of smooth functions on a manifold as $C^\infty(M):=\mathsf{SMan}(M,\R)$.
    The set $\mathsf{SMan}(\R, M)/\cong$ of tangent vectors on a smooth manifold $M$ comprises the curves $\R \to M$ subject to the equivalence relation that for a pair of curves $\phi, \theta:\R \to M$, $\phi \cong \theta$ if and only if $\phi(0) = \theta(0)$ and for every $f \in C^\infty(M)$,
    \[
        \frac{\partial f \o \phi}{\partial x}(0) = \frac{\partial f \o \theta}{\partial x}(0).
    \] 
    The set $\mathsf{SMan}(\R,M)/\cong$ has a naturally determined smooth manifold structure which we call the \emph{tangent bundle} over $M$, $TM$.\pagenote{
    The original definition forgot to set the notation $TM$ for the tangent bundle of $M$.}
\end{definition}

\begin{example}
    ~\begin{enumerate}[(i)]
        \item For a vector space, the space of linear paths crossing through a point $v \in V$ is isomorphic to $V$, so $TV \cong V \x V$. 
        \item The tangent bundle above the circle is diffeomorphic to the cylinder. This is follows from the classical result that a tangent vector on the circle must be perpendicular to its position vector.
    \end{enumerate}
\end{example}

The tangent bundle lifts the ``local derivative'' into a globally defined construction, so the tangent bundle construction is functorial.\pagenote{
   This proof has been cleaned up, it was originally quite messy.
}
\begin{proposition}
    The tangent bundle is a product-preserving endofunctor on the category of smooth manifolds.
\end{proposition}
\begin{proof}
    Functoriality follows by showing that a morphism of smooth manifolds preserves the equivalence relation on curves that defines a tangent vector: 
    \[
        \forall g \in C^\infty(M),  \frac{\partial g \o \phi}{\partial x}(0) = \frac{\partial g \o \theta}{\partial x}(0)
    \]
    Note that if $f:N \to M \in C^\infty(N)$, so that $g \o f \in C^\infty(M)$, the chain rule ensures that
    \[
        \forall g \in C^\infty(N), 
        \frac{\partial f\o g \o \phi}{\partial x}(0) = \frac{\partial f\o g \o \theta}{\partial x}(0)
    \]

    To show that $T$ is product-preserving, it suffices to show that the equivalence classes of curves are stable under pairing. First, note that for any $M$ and $\phi \cong \theta:\R \to M$, 
    \[
        (\phi,id) \cong (\theta,id): \R \to M\x\R
    \] 
    Given a pair of curves $\theta_M, \psi_M:\R \to M$, where $\theta_M \cong \psi_M$ and similarly $\theta_N \cong \psi_N$ for $N$, this implies that
    \pagenote{Cleaned up proof per Kristine's comments.}
    \begin{gather*}
        f \o (\phi_M, \phi_N) \\
        = f \o (\phi_M \x id) \o (id, \phi_N) \\
        = f \o (\phi_M \x id) \o (id, \theta_N) \\
        = f \o (id, \theta_N) \o (\phi_M \x id) \\
        = f \o (id, \theta_N) \o (\theta_M \x id) \\
        = f \o (\theta_M \x \theta_N)
    \end{gather*}
    so $(\phi_M, \phi_N) \cong (\theta_M, \theta_N)$.
\end{proof}

The scalar action by $\R$ on tangent vectors and a partially defined addition additionally give the tangent bundle the structure of a fibered $\R$-module (that is, an $\R$-module in the slice category $\mathsf{SMan}/M$ whose objects are morphisms into $M$, $f:X \to M$, and morphisms are commuting triangles).
\begin{proposition}
    The tangent bundle over $M$ is an $\R$-module in $\mathsf{SMan}/M$, as follows:
    let $\gamma, \omega$ be tangent vectors on $M$, and define
    \begin{itemize}
        \item $p: TM \to M; p(\gamma) = \gamma(0)$.
        \item $0: M \to TM; 0(m) =  [r \mapsto m]$ (the constant map $\R \to M$ sending all $r\in \R$ to $m\in M$)
        \item $\cdot_p: TM \x \R \to TM; \gamma \cdot_p r = [x \mapsto \gamma(r \cdot x)]$
        \item $+: TM \ts{p}{p} TM \to TM := [\gamma],[\omega] \mapsto [\gamma + \omega]$ (where addition around $\gamma(0)=\omega(0)$ is defined using local coordinates). \pagenote{The notation in this proposition has been changed to be more clear, and it was clarified that addition is defined using local coordinate charts around points in $TM$.}
    \end{itemize}
\end{proposition}

The second derivative is involved in the more nuanced axioms for a cartesian differential category, namely linearity in the vector argument and the symmetry of mixed partial derivatives.  First, set $f|_U$ to be the restriction of $f:M \to N$ to a map between local coordinate patches $U \subseteq M, V \subseteq N$, and then define
\[
  f_0 = f|_U \o \pi_0 \o \pi_1, \hspace{0.25cm},
\]
and
\[
   f_1 = D[f] \o (\pi_0 \o \pi_0, \pi_1 \o \pi_0), \hspace{0.25cm}
   f_2 = D[f] \o (\pi_0 \o \pi_0, \pi_0 \o \pi_1) \hspace{0.25cm}.
\]
The axioms $[CDC.6],[CDC.7]$ then give:\pagenote{
    I have added the definition of the maps ahead of the diagrams to avoid any confusion, and made a more direct reference to the CDC axioms.
}
\begin{center}
\begin{tikzcd}
	{(U \x \R^n) \x (\R^n \x \R^n) } &&& {(V \x \R^m) \x (\R^m \x \R^m) } \\
	{} & {T^2M} & {T^2N} \\
	& TM & {T.N} \\
	{(U \x \R^n) } &&& {(V \x \R^m)}
	\arrow[hook, from=1-1, to=2-2]
	\arrow[hook, from=4-1, to=3-2]
	\arrow["{T^2.f}"', from=2-2, to=2-3]
	\arrow["{T.f}"', from=3-2, to=3-3]
	\arrow["{((f_0,f_1),(f_2, D[D[f|_U]]))}", from=1-1, to=1-4]
	\arrow["{(f, D[f|_U])}"', from=4-1, to=4-4]
	\arrow[hook', from=1-4, to=2-3]
	\arrow[hook', from=4-4, to=3-3]
	\arrow["\ell", from=3-2, to=2-2]
	\arrow["\ell", from=3-3, to=2-3]
	\arrow["{((\pi_0,0),(0,\pi_1))}"{description}, from=4-1, to=1-1]
	\arrow["{((\pi_0,0),(0,\pi_1))}"{description}, from=4-4, to=1-4]
\end{tikzcd}
    \begin{tikzcd}
	{(U \x \R^n) \x (\R^n \x \R^n) } &&& {(V \x \R^m) \x (\R^m \x \R^m) } \\
	{} & {T^2M} & {T^2N} \\
	& {T^2M} & {T^2.N} \\
	{(U \x \R^n) \x (\R^n \x \R^n) } &&& {(V \x \R^m) \x (\R^m \x \R^m) }
	\arrow["c", from=2-2, to=3-2]
	\arrow[hook, from=1-1, to=2-2]
	\arrow["{((\pi_0\o \pi_0,\pi_0 \o \pi_1), (\pi_1 \o \pi_0, \pi_1\o \pi_1)}"{description}, from=1-1, to=4-1]
	\arrow[hook, from=4-1, to=3-2]
	\arrow["{T^2.f}"', from=2-2, to=2-3]
	\arrow["{T^2.f}"', from=3-2, to=3-3]
	\arrow["c", from=2-3, to=3-3]
	\arrow["{((\pi_0\o \pi_0,\pi_0 \o \pi_1), (\pi_1 \o \pi_0, \pi_1\o \pi_1)}"{description}, from=1-4, to=4-4]
	\arrow["{((f_0,f_1),(f_2,D[D[f|_U]])}", from=1-1, to=1-4]
	\arrow["{((f_0,f_2),(f_1,D[D[f|_U]])}"', from=4-1, to=4-4]
	\arrow[hook', from=1-4, to=2-3]
	\arrow[hook', from=4-4, to=3-3]
\end{tikzcd}
\end{center}

The two natural transformations $\ell$ and $c$---the vertical lift and canonical flip---capture these coherences.
Locally, $\ell$ is the map inserting zeros into the second and third coordinates, while $c$ flips the second and third arguments, leading to the coherences established in the next proposition.
To capture these coherences on the tangent bundle, first note that a tangent vector on $TM$ is equivalent to an equivalence class of surfaces on $M$,\pagenote{I have clarified the equivalence relation that defines the second tangent bundle by using normal calculus notation.}
\begin{align*}
    \phi \cong \theta :\R^2 \to M
    \iff   &  \phi(0,0) = \theta(0,0) \\
    \text{and } &  \forall f \in C^\infty(M), \frac{\partial f \o \phi}{\partial x_i}(0,0) = \frac{\partial f \o \theta}{\partial x_i}(0,0), i = 0,1
\end{align*}

\begin{proposition}[\cite{Cockett2014}]
    There are two natural transformations\pagenote{I have fixed the notation in this proposition.}
    \[
        \ell: TM \to T^2M; \ell([\gamma]) = [\gamma \o (\pi_0 \cdot_\R \pi_1)] \hspace{0.5cm}
        c: T^2M \to T^2M; c([\gamma]) = [\gamma \o (\pi_1, \pi_0)]
    \]
    satisfying the following coherences:
    \begin{enumerate}[(i)]
        \item $\ell.T \o \ell = T.\ell \o \ell$\footnote{Recall that we are using the 2-categorical notation described in the front-matter}
        \item The following maps are morphisms of fibred $\R$-modules. 
        \[\begin{tikzcd}
            TM & {T^2M} \\
            M & TM
            \arrow["\ell", from=1-1, to=1-2]
            \arrow["p", from=1-1, to=2-1]
            \arrow["{p.T}", from=1-2, to=2-2]
            \arrow["0", from=2-1, to=2-2]
        \end{tikzcd}
        \begin{tikzcd}
            TM & {T^2M} \\
            M & TM
            \arrow["\ell", from=1-1, to=1-2]
            \arrow["p", from=1-1, to=2-1]
            \arrow["{T.p}", from=1-2, to=2-2]
            \arrow["0", from=2-1, to=2-2]
        \end{tikzcd}\]
        \item $c \o c = id$
        \item $T.c \o c.T \o T.c = c.T \o T.c \o c.T$
        \item $c \o \ell = \ell$
        \item $T.c \o c \o T.\ell = \ell.T \o c$.
    \end{enumerate}
\end{proposition}
\pagenote{Removed the proofs - they are in the literature and they proved to be a distraction.}
\begin{observation}
\label{obs:yb-eq}
    Equation (iv) is known as the \emph{Yang--Baxter} equation. 
    It is one of the coherences for a symmetric monoidal category, and states that the twisting operation between two variables is coherent.
    We may regard the category of endofunctors on a category as a strict monoidal category and use string diagram notation (see e.g. \cite{selinger2010survey}).
    Interpreting the map $c$ as twisting two strings, the coherence becomes \pagenote{I added a reference to string calculus.}
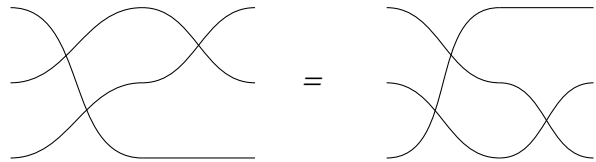
    \begin{center}
        \begin{tikzpicture}
	\begin{pgfonlayer}{nodelayer}
		\node [style=none] (0) at (-6, 1) {};
		\node [style=none] (1) at (-6, 0) {};
		\node [style=none] (2) at (-6, -1) {};
		\node [style=none] (3) at (-4.25, 0) {};
		\node [style=none] (4) at (-4.25, 1) {};
		\node [style=none] (5) at (-4.25, -1) {};
		\node [style=none] (6) at (-2.75, 1) {};
		\node [style=none] (7) at (-2.75, 0) {};
		\node [style=none] (8) at (-2.75, -1) {};
		\node [style=none] (9) at (-1, 1) {};
		\node [style=none] (10) at (-1, 0) {};
		\node [style=none] (11) at (-1, -1) {};
		\node [style=none] (12) at (0.5, 1) {};
		\node [style=none] (13) at (0.5, 0) {};
		\node [style=none] (14) at (0.5, -1) {};
		\node [style=none] (15) at (1.75, 1) {};
		\node [style=none] (16) at (1.75, -1) {};
		\node [style=none] (17) at (1.75, 0) {};
		\node [style=none] (18) at (-2, 0) {};
		\node [style=none] (19) at (-2, 0) {};
		\node [style=none] (20) at (-2, 0) {};
		\node [style=none] (21) at (-2, 0) {=};
	\end{pgfonlayer}
	\begin{pgfonlayer}{edgelayer}
		\draw [in=180, out=0] (0.center) to (5.center);
		\draw [in=180, out=0] (1.center) to (4.center);
		\draw [in=180, out=0] (2.center) to (3.center);
		\draw (5.center) to (8.center);
		\draw [in=180, out=0] (4.center) to (7.center);
		\draw [in=-180, out=0] (3.center) to (6.center);
		\draw [in=180, out=0] (11.center) to (12.center);
		\draw [in=-180, out=0] (9.center) to (13.center);
		\draw [in=180, out=0] (10.center) to (14.center);
		\draw [in=180, out=0] (14.center) to (17.center);
		\draw [in=180, out=0] (13.center) to (16.center);
		\draw (12.center) to (15.center);
	\end{pgfonlayer}
\end{tikzpicture}

    \end{center}
\end{observation}

The projection $p:TM \to M$ is \emph{locally trivial}: for each connected component of $M$ that is modeled on $\R^n$, each point $m$ lies in an open subset $U_m$ so that $p^{-1}(U_m) \cong U_m \x \R^n$. 
This local triviality property leads to the following universality condition.
\begin{proposition}\label{prop:ell-universal}
    The following diagram is an equalizer:
    \[\begin{tikzcd}
        TM & {T^2M} & TM
        \arrow["p"{description}, from=1-2, to=1-3]
        \arrow["{T.p}", shift left=2, from=1-2, to=1-3]
        \arrow["{0\o p \o p}"', shift right=2, from=1-2, to=1-3]
        \arrow["\ell"', from=1-1, to=1-2]
    \end{tikzcd}\]
\end{proposition}
Therefore the diagram in the following corollary is a pullback: \pagenote{In the original draft it was unclear whether there was a missing diagram or I was referring to the diagram in the following lemma. This change addresses that ambiguity.}
\begin{corollary}\label{cor:mu-universal}
    Write the map $\mu:TM \ts{p}{p} TM \to T^2M$ to be $T.+ \o (\ell \x 0)$. The following diagram is a pullback:
    \[\begin{tikzcd}
        {TM \ts{p}{p} TM} & {T^2M} \\
        M & TM
        \arrow["0"', from=2-1, to=2-2]
        \arrow["T.p", from=1-2, to=2-2]
        \arrow["\mu", from=1-1, to=1-2]
        \arrow[from=1-1, to=2-1]
        \arrow["\lrcorner"{anchor=center, pos=0.125}, draw=none, from=1-1, to=2-2]
    \end{tikzcd}\]
\end{corollary}

\pagenote{I now introduce the definition of the operational tangent bundle at this point addresses confusions about notation that comes up later in the chapter/thesis regarding the $C^\infty(M)$-module structure on $\chi(M)$}
The five maps $(p, 0, +, c, \ell)$, along with their coherences and universal properties, characterize the \emph{kinematic} tangent bundle, axiomatized as a tangent structure in the next section. However, there is an equivalent characterization of the tangent bundle for a finite-dimensional smooth manifold that will be important throughout this thesis: the \emph{operational} tangent bundle. We first need to define the module of vector fields on a manifold, where a vector field is essentially an ordinary differential equation defined on a manifold rather than a cartesian space. 
\begin{definition}[3.1,3.3 of \cite{Kolar1993}]\label{def:operational-tang}
    A \emph{vector field} on a manifold $M$ is a section $X:M \to TM$ of the projection $p:TM \to M$ so that $p \o X = id_M$. 
    The set of vector fields on a manifold $M$ is written $\chi(M)$ and carries a $C^\infty(M)$-module structure using the fibered $\R$-module structure on $p:TM \to M$:
    \[
        X +_{\chi(M)} Y := +.M \o (X,Y), \hspace{0.5cm} 0_{\chi(M)} := 0.M,\hspace{0.5cm} f \cdot_{\chi(M)} X (m) := f(m) \cdot_{TM} X(m)
    \]
    where $X,Y \in \chi(M), f \in C^\infty(M)$.
\end{definition}
The module $\chi(M)$ has an important universal property as a $C^\infty(M)$-module---it is precisely the module of \emph{derivations} of $C^\infty(M)$:
 \[
        \chi(M) = \{ X: C^\infty(M) \to C^\infty(M) : \forall f, g \in C^\infty(M), X(f\cdot g) = X(f)\cdot g + f \cdot X(g) \}
\]
\begin{proposition}[3.4 of \cite{Kolar1993}]\label{prop:derivations-tangent-bundle}
    There is an isomorphism of $C^\infty(M)$ modules:
    \[
        \mathsf{Der}(C^\infty(M)) \cong \chi(M).  
    \]
\end{proposition}
Finally, we observe that there is a $C^\infty(M)$-Lie algebra structure on $\chi(M)$, with two equivalent definitions. First, there is the kinematic definition of the bracket, which is induced using the universality of the vertical lift. Given vector fields $X,Y$ on $M$, note that
\[
   X = p.T \o (T.Y \o X -_{TM} c \o T.X \o Y), \hspace{0.25cm}
   0 = T.p \o (T.Y \o X -_{TM} c \o T.X \o Y)
\]
so by Corollary \ref{cor:mu-universal}, there is a unique map $[X,Y]:M \to TM$ so that
\begin{equation}\label{eq:hard}
    (T.Y \o X -_{TM} c \o T.X \o Y) = \mu([X,Y], X).
\end{equation}
Similarly, there is a bracket defined on $\mathsf{Der}(M)$ using the \emph{anticommutator} of derivations:
\begin{equation}\label{eq:anticom}
    [X,Y](f) = X(Y(f)) - Y(X(f)).
\end{equation}
These brackets are equivalent for finite-dimensional smooth manifolds.
\begin{proposition}\label{prop:anti-commm-lie}[\cite{Mackenzie2013}]
    Recall that a Lie algebra over a ring $R$ is an $R$-module $A$ equipped with a bilinear map
    \[
        [-.-]: A \ox A \to A   
    \]
    that is alternating and satisfies the Jacobi identity:
    \[
        [X,[Y,Z]] + [Z,[X,Y]] + [Y, [Z,X]] = 0.  
    \]
    The two brackets on $\chi(M)$ (viewed as a Lie algebra) from Equations \ref{eq:hard} and \ref{eq:anticom} coincide.
\end{proposition}

\section{Tangent structures}%
\label{sec:tangent-structure}
This section develops the categorical framework to study more general categories of smooth manifolds by axiomatizing the tangent bundle, tangent categories, which first appeared in \cite{Rosicky1984}. An arbitrary tangent category is significantly more general than smooth manifolds and captures examples of categories with ``tangent bundle'' from computer science and logic, as developed in \cite{Cockett2014}. First, observe that tangent categories forget the base ring from the previous section and only consider the fibered commutative monoid structure of the tangent bundle.
\begin{definition}\label{def:add-bun}
    An additive bundle in a category $\C$ is a triple $E \xrightarrow{q} M, +:E\ts{q}{q}E \to E, \xi:M \to E$ which gives $(q,+,\xi)$ the structure of a commutative monoid in the slice category $\C/M$.
    If $(q,\xi,+), (q',\xi',+')$ are both additive bundles, a bundle morphism
    \[
        \begin{tikzcd}
            E \dar[swap]{q} \rar{f} & E' \dar{q} \\
            M \rar{f_0} & M'
        \end{tikzcd}
    \]
    is \textit{additive} if $f \o + = +' \o (f\o \pi_0,f\o \pi_1)$ and $f \o \xi = \xi' \o f_0$. The category of additive bundles in a a category $\C$ is given by additive bundles in $\C$ and additive bundle morphisms.
\end{definition}
We will often write pullback powers of an additive bundle $E \ts{q}{q} \dots \ts{q}{q} E$ as $E_n$, and use infix notation to write addition so $+ \o (a,b)$ becomes $a + b$.
In the category of smooth manifolds, the tangent bundle functor gives a well-behaved functorial vector bundle. 
A tangent category has a functorial additive bundle and axiomatizes the coherences and universal properties of the tangent bundle from the category of smooth manifolds.
\begin{definition}[\cite{Rosicky1984,Cockett2014}]\label{def:tangent-cat}
 A tangent structure consists of a functor $T:\C \to \C$ equipped with natural transformations
  \begin{gather*}
    p:T \Rightarrow id, \hspace{0.15cm} 0:id \Rightarrow T, +: T\ts{p}{p} T \Rightarrow T \\
    \ell: T \Rightarrow T.T \hspace{0.30cm} c:T.T \Rightarrow T.T
  \end{gather*}
  satisfying the following axioms; we call a category equipped with a tangent structure a \emph{tangent category}.\pagenote{I have added "."'s to the diagrams to keep my notation consistent.}
  \begin{enumerate}[{[TC.1]}]
  \item Additive bundle axioms:
    \begin{enumerate}[(i)]
        \item Pullback powers of $p$ exist and are preserved by $T$; write these $T_n$.
        \item Each triple $(p.M: TM \to M, 0.M:M \to TM, +.M:T_2M \to TM)$ is an additive bundle.
        \begin{equation}
            \label{eq:abun-axioms}
\begin{tikzcd}
	T & T & {T_2} &[-2em]&[-2em] T \\
	id & id && id
	\arrow["{+}", from=1-3, to=1-5]
	\arrow["{p \o \pi_i}"', from=1-3, to=2-4]
	\arrow["p", from=1-5, to=2-4]
	\arrow["p", from=1-1, to=2-1]
	\arrow["0"', from=2-2, to=1-2]
\end{tikzcd}
        \end{equation}
    \end{enumerate}
  \item Symmetry axioms:
    \begin{enumerate}[(i)]
        \item Involution: 
        \begin{equation}\label{eq:tc-2-1}
\begin{tikzcd}
	T.T & T.T \\
	& T.T
	\arrow["c", from=1-1, to=1-2]
	\arrow["c", from=1-2, to=2-2]
	\arrow[Rightarrow, no head, from=1-1, to=2-2]
\end{tikzcd}
        \end{equation}
        \item Yang--Baxter: $c.T \o T.c \o c.T = T.c \o c.T \o T.c$
        \begin{equation}\label{eq:tc-2-2}
\begin{tikzcd}
	& T.T.T & T.T.T \\
	T.T.T &&& T.T.T \\
	& T.T.T & T.T.T
	\arrow["{T.c}", from=2-1, to=1-2]
	\arrow["{c.T}", from=1-2, to=1-3]
	\arrow["{T.c}", from=1-3, to=2-4]
	\arrow["{c.T}"', from=2-1, to=3-2]
	\arrow["{c.T}"', from=3-3, to=2-4]
	\arrow["{T.c}"', from=3-2, to=3-3]
\end{tikzcd}
        \end{equation}
        \item Naturality equations:
        \begin{equation}\label{eq:tc-2-3}
\begin{tikzcd}
	T.T && T.T & T & T \\
	T.T.T & T.T.T & T.T.T & T.T & T.T \\
	{T.T_2} && {T_2T} & T.T & T.T \\
	T.T && T.T & T & T
	\arrow["{T.\ell}"', from=1-1, to=2-1]
	\arrow["c", from=1-1, to=1-3]
	\arrow["{c.T}", from=2-1, to=2-2]
	\arrow["{T.c}", from=2-2, to=2-3]
	\arrow["{\ell.T}", from=1-3, to=2-3]
	\arrow["{(c \o T.\pi_0, c \o T.\pi_1)}", from=3-1, to=3-3]
	\arrow["{T.+}"', from=3-1, to=4-1]
	\arrow["{+.T}", from=3-3, to=4-3]
	\arrow["c"', from=4-1, to=4-3]
	\arrow["c", from=3-4, to=3-5]
	\arrow["{p.T}"', from=3-4, to=4-4]
	\arrow["{T.p}", from=3-5, to=4-5]
	\arrow[Rightarrow, no head, from=4-4, to=4-5]
	\arrow["{0.T}"', from=1-4, to=2-4]
	\arrow["c", from=2-4, to=2-5]
	\arrow[from=1-4, to=1-5]
	\arrow["{T.0}", from=1-5, to=2-5]
\end{tikzcd}
        \end{equation}
    \end{enumerate}
  \item Lift axioms:
    \begin{enumerate}[(i)]
        \item Naturality with addition:
        \begin{equation}
            \label{eq:tc-3-0}
\begin{tikzcd}
	T & T.T &[-1em] T.T & T.T &[-1em] {T_2} & {T.T_2} \\
	id & T & T & T & T & T.T
	\arrow["{(\ell\o\pi_0,\ell\o\pi_1)}", from=1-5, to=1-6]
	\arrow["{+}"', from=1-5, to=2-5]
	\arrow["{+}", from=1-6, to=2-6]
	\arrow["\ell"', from=2-5, to=2-6]
	\arrow["{T.\ell}", from=1-3, to=1-4]
	\arrow["0", from=2-3, to=1-3]
	\arrow["0"', from=2-3, to=2-4]
	\arrow["{0.T}"', from=2-4, to=1-4]
	\arrow["\ell", from=1-1, to=1-2]
	\arrow["p", from=1-1, to=2-1]
	\arrow["0"', from=2-1, to=2-2]
	\arrow["{T.p}", from=1-2, to=2-2]
\end{tikzcd}
        \end{equation}
        \item Coassociativity:
        \begin{equation}
            \label{eq:tc-3-1}
\begin{tikzcd}
	T & T.T \\
	T.T & T.T.T
	\arrow["\ell", from=1-1, to=1-2]
	\arrow["\ell"', from=1-1, to=2-1]
	\arrow["{T.\ell}"', from=2-1, to=2-2]
	\arrow["{\ell.T}", from=1-2, to=2-2]
\end{tikzcd}
        \end{equation}
        \item Symmetric co-multiplication:
        \begin{equation}
            \label{eq:tc-3-2}
\begin{tikzcd}
	T & T.T \\
	& T.T
	\arrow["\ell", from=1-1, to=1-2]
	\arrow["c", from=1-2, to=2-2]
	\arrow["\ell"', from=1-1, to=2-2]
\end{tikzcd}
        \end{equation}
        \item Universality: for $\mu:= T.+ \o (0\o \pi_0, \ell \o  \pi_1)$, the following diagram is a pullback for all $X$:
        \[\begin{tikzcd}
    T_2X \ar[r, "\mu"] \dar[description]{p\o \pi_i} & T^2X \dar{T.p}\\
    X \rar{0} & TX
\end{tikzcd}\]
    \end{enumerate}
  \end{enumerate}
\end{definition}
\begin{definition}[\cite{Cockett2014}]
    \pagenote{I pulled out the definition of a cartesian tangent category to make it clear that it had been defined, and added a note regarding notation to handle any confusion.}
    A tangent category is \textit{monoidal} whenever $\C$ is a monoidal category, $T$ is a monoidal functor, and $+,p,c,\ell$ are monoidal natural transformations. A tangent category is \textit{cartesian} whenever $\C$ is cartesian and is a strict monoidal tangent category for products.  
\end{definition}
\begin{notation}
Throughout this thesis, two key pieces of notation apply:
    \begin{itemize}
        \item $T_n$ denotes pullback powers of $p:T \Rightarrow id$ (and more generally $E_n$ for $q:E \to M$); iterated powers of $T$ are written $T^n$.
        \item There will often be long strings of $T_n$ pullbacks and functors $F:A \to B$, so a 2-categorical notation where functor composition is written with a period, $T_n.F.T'_m$, will often be adopted. While the natural transformation $c$ at $T_n.T_m.M$ under the image of the functor $T$ would often be written $T(c_{T_n.T_m.M})$, in this thesis it will be written as
        \[
            T.c.T_n.T_m.M:T.T^2.T_n.T_m \Rightarrow T.T^2.T_n.T_m
        \]  
        while the composition of 2-cells will be written using $\o$ in applicative order, rather than the diagrammatic order typically used in the tangent category literature. That is, the composition
        \[
            A \xrightarrow[]{g} B \xrightarrow[]{f} C
        \]
        would be written $f \o g$ rather than $gf$.
    \end{itemize}
\end{notation}
\begin{example}\label{example:tangcat-sman}
    Applying the results from \cref{sec:smooth-manifolds}, the category of smooth manifolds is a tangent category. 
    Recall that if $M$ is a smooth manifold, there is a coordinate patch $m \in U \hookrightarrow M$ around each point $m \in M$ so that $U \cong U' \subseteq \mathbb{R}^n$.
    fibre above $U$, $p^{-1}(U)$, is locally isomorphic to $U' \x \mathbb{R}^n$ and similarly $(p \o p)^{-1}(U) \cong U' \x (\R^n)^3$, so that:
    \begin{center}
        \begin{tabular}{|l|l|}
            \hline
            $p:U \x \R^n \to U$  & $(m,x) \mapsto m$ \\ \hline
            $0: U \to U \x \R^n$ & $m \mapsto (m,0)$ \\ \hline
            $+:U \x \R^n \x R^n \to U \x \R^n$                         & $(m,x,y) \mapsto (m, x+y)$      \\ \hline
            $\ell:U \x \R^n \to U \x \R^n \x \R^n\x \R^n$              & $(m,x) \mapsto (m,0,0,x)$       \\ \hline
            $c: U \x \R^n \x \R^n\x \R^n \to U \x \R^n \x \R^n\x \R^n$ & $(m, x,y,z) \mapsto (m, y,x,z)$ \\ \hline
        \end{tabular}
    \end{center}
\end{example}
\pagenote{removed the example of Lex - there are some coherence issues to sort out, and that would further distract from the main point of this chapter (introducing the tangent categories and smooth manifolds).}

The study of tangent categories is closely related to Lawvere's \emph{synthetic differential geometry}, first introduced in \cite{Lawvere1979} and later developed in \cite{Kock2006, Lavendhomme1996, Moerdijk1991}. The setting of synthetic differential geometry is a topos $\e$ (for our purposes, we need only a complete, cartesian closed category) equipped with a chosen ring object $R$ that satisfies the \emph{Kock-Lawvere axiom}: given the object of nilpotent elements in $R$, $D = [d : R | d^2 = 0]$, the following map is an isomorphism:
\[
    \alpha: R \x R \to [D,R]; \hspace{0.25cm} \alpha(a,b) = (d \mapsto a + db).  
\]
One can find a class of objects that form a tangent category: the infinitesimally linear objects.\pagenote{
   This passage has been edited substantially to clear up potential confusion about the definition of the Kock-Lawvere axiom (it originally seemed like it was defined twice), and the notions of infinitesimal linearity (the language of ``perceiving'' a diagram as a colimit is a needlessly confusing bit of jargon in synthetic differential geometry that doesn't need to be in this thesis). 
}
\begin{definition}%
\label{def:inf-linear}
    Let $\C$ be a model of synthetic differential geometry where the ring object is $(R, \cdot, 1,+, 0)$.
    An object $M$ in a model of synthetic differential geometry is \emph{infinitesimally linear} when it satisfies the following axioms.
    \begin{enumerate}[(i)]
        \item The following natural morphism must be an isomorphism: 
            \[
                [D(2),M] \cong [D,M]\ts{0^*}{0^*} [D,M];
                \text{ where } 
                D(2) := [(d_0,d_1) \in D^2 : d_i\cdot d_j = 0].
            \]
        \item $M$ satisfies \emph{property W} (credited by  \cite{Kock2006} to Gavin Wraith), namely that the following diagram is a ternary equalizer:
\[\begin{tikzcd}
    {[D,M]} & {[D\x D,M]} &&& {[D,M]}
    \arrow["{[0\o !\x D, M]}", shift left=4, from=1-2, to=1-5]
    \arrow["{[(0 \x 0) \o !, M]}"{description}, from=1-2, to=1-5]
    \arrow["{[D\x 0\o !, M]}"', shift right=5, from=1-2, to=1-5]
    \arrow["{[\cdot,M]}", from=1-1, to=1-2]
\end{tikzcd}\]
        (We will often use $M$ as shorthand for $id_M$ in diagrams).
    \end{enumerate}
\end{definition}
\pagenote{The treatment of the operational tangent bundle at the end of the last section made the observation regarding the Lie bracket here unnecessary.}

An \emph{infinitesimal object} generalizes the object $D$ in the category of infinitesimally linear objects in a model of synthetic differential geometry. The definition adopted in this thesis is a strict generalization of the definition given in \cite{Cockett2014}: here, we work with a symmetric monoidal closed category rather than a cartesian closed category in order to capture some examples from logic.\pagenote{I have expanded on this definition to fix  ambiguities in the original draft.}
\begin{definition}[Definition 5.6 \cite{Cockett2014}]
    \label{def:inf-object}
    An \emph{infinitesimal object} in a symmetric monoidal category \[(\C, \ox, I, \alpha, \rho, \sigma)\] is a tuple $(\odot:D \ox D \to D, 0:I \to D, \epsilon: D \to I, \delta: D \to D(2))$ (where $D(n)$ denotes pushout powers of $0$) so that:
    \begin{enumerate}[{[{I}O.1]}]
        \item Pushout powers $D(n)$ of $\hspace{0.05cm} 0:I \to D$ exist, and $\epsilon \o 0 = id_I$.
        \item $\odot$ is a commutative semigroup with zero, so that the following diagrams commute:
\[\begin{tikzcd}
	{D \ox D} & {D \ox D} &[-2em] {D \ox D \ox D} & {D\ox D} &[-2em] D & {D\ox D} \\
	& D & D & D && D
	\arrow["\odot"{description}, from=1-2, to=2-2]
	\arrow["\sigma"{description}, from=1-1, to=1-2]
	\arrow["\odot"{description}, from=1-1, to=2-2]
	\arrow["{\odot \ox D}", from=1-3, to=1-4]
	\arrow["\odot"{description}, from=1-4, to=2-4]
	\arrow["{D \ox \odot}"', from=1-3, to=2-3]
	\arrow["\odot"', from=2-3, to=2-4]
	\arrow["{(D,0)}", from=1-5, to=1-6]
	\arrow["\odot"{description}, from=1-6, to=2-6]
	\arrow["{0_D \o \epsilon}"', from=1-5, to=2-6]
\end{tikzcd}\]
        The third diagram shows that $0$ is an absorbing element rather than a unit (think of $0$ as being in the commutative semigroup on the set $[0,1)$ given by multiplication).
        \item The map $\delta: D \to D(2)$ makes $(0:I \to D, \delta, \epsilon)$ into a commutative comonoid in the coslice $I/\C$:
\[\begin{tikzcd}
	D & {D(2)} & D & {D(2)} & D & {D(2)} \\
	& D & {D(2)} & {D(3)} & {D(2)}
	\arrow["\delta"{description}, from=1-3, to=2-3]
	\arrow["\delta"{description}, from=1-3, to=1-4]
	\arrow["{D+_I\delta}", from=1-4, to=2-4]
	\arrow["{\delta+_ID}"', from=2-3, to=2-4]
	\arrow["\delta"', from=1-5, to=2-5]
	\arrow["\delta", from=1-5, to=1-6]
	\arrow["{(\iota_1|\iota_0)}"', from=2-5, to=1-6]
	\arrow["{(0 \o \epsilon +_I D)}", from=1-2, to=2-2]
	\arrow["\delta", from=1-1, to=1-2]
	\arrow[Rightarrow, no head, from=1-1, to=2-2]
\end{tikzcd}\]
        \item The following diagram commutes ($\odot$ is coadditive):
\[\begin{tikzcd}
    {D\ox D} & D \\
    {D\ox D(2)} & {D(2)}
    \arrow["{1 \x \delta}", from=1-1, to=2-1]
    \arrow["{(\iota_0\o \odot|\iota_1\o \odot)}"', from=2-1, to=2-2]
    \arrow["\delta", from=1-2, to=2-2]
    \arrow["\odot", from=1-1, to=1-2]
\end{tikzcd}\]
        The notation $(f | g):A + B \to C$ for pushouts/coproducts is dual to pairing $(f',g'):C \to A \x B$ for pullbacks/products, just as $\iota_i$ is dual to projection. Therefore, $(f | g) \o \iota_0 = f$ just as $\pi_0 \o (f,g) = f$.
        \item The following diagram is a coequalizer:
\[\begin{tikzcd}
    {D\ox I} & {D\ox D} && {} & {D(2)}
    \arrow["{0 \o \epsilon \ox 0}", shift left=1, from=1-1, to=1-2]
    \arrow["{D \ox 0}"', shift right=1, from=1-1, to=1-2]
    \arrow["{(\odot|_I\epsilon \ox D)\o (\delta \ox id)}"', from=1-2, to=1-5]
\end{tikzcd}\]
    \end{enumerate} 
\end{definition}
There are two tangent structures associated to an infinitesimal object in a symmetric monoidal category $\C$. The first relies on the \emph{exponentiability} of the infinitesimal object, while the second tangent structure is on the opposite category of $\C$. The enriched perspective on tangent structure in Section \ref{sec:tang-cats-enrichment} will clarify the relationship between these two tangent structures.
\begin{proposition}%
    \label{prop:inf-object-tangent-structures}
    Let $(\C, \ox, I, \alpha, \rho, \sigma)$ be a symmetric monoidal category, and \[\odot:D \ox D \to D, 0:I \to D, \epsilon: D \to I, \delta: D \to D(2)\] define an infinitesimal object in $\C$.\pagenote{
        The original proposition was unclear, so this passage has been changed to fix some ambiguities. The full structure of the symmetric monoidal category was added, the order of (i) and (ii) are swapped, the wording in part (ii) has been clarified, as has the proof itself.
    }
    \begin{enumerate}[(i)]
        \item There is a tangent structure on $\C^{op}$, where $T = D \ox (-)$.
        \item If $\C$ is also a symmetric monoidal closed category, then it is a tangent category with $T = [D, -]$.
    \end{enumerate}
\end{proposition}
\begin{proof}
   ~\begin{enumerate}[(i)]
        \item The opposite category of $\C$ is a symmetric monoidal category:
        \[
            (\C^{op}, \ox, I, \alpha^{-1}, \rho^{-1}, \sigma).
        \]
        The tangent functor is $D \ox (-)$, and the projection is \[p = D \ox (-) \xrightarrow[]{0^{op}} I \ox (-) \xrightarrow[]{\rho^{-1}} (-).\] The zero map is given by
        \[
            (-) \xrightarrow[]{\rho} I \ox (-) \xrightarrow[]{\epsilon^{op}} D \ox (-).
        \]
        Addition in $\C^{op}$ is given by co-addition in $\C$:
        \[
            D(2) \ox (-) \xrightarrow[]{\delta^{op} \ox (-)} D \ox (-). 
        \]
        The lift is given by the semigroup structure (along with the monoidal coherences):
        \[
            D \ox (-) \xrightarrow[]{\odot \ox (-)} (D \ox D) \ox (-) \xrightarrow[]{\alpha^{-1}} D \ox (D \ox (-)).
        \]
        The flip is also given by monoidal coherences, combined with the symmetry on the monoidal category:
        \[
            D \ox (D \ox (-))  \xrightarrow[]{\alpha} (D \ox D) \o (-) \xrightarrow[]{\sigma \ox (-)}
            (D \ox D) \ox (-) \xrightarrow[]{\alpha^{-1}} D \ox (D \ox (-)).
        \]
      \item First, we identify the following natural isomorphisms:
        \[
            u: id \Rightarrow [I,-],\hspace{0.5cm}
            b: [A,[B,-]] \Rightarrow [A \ox B, -].
        \]
        \begin{itemize}
            \item The tangent functor is $[D,-]$, and the triple $(0:I \to D, \delta:D \to D(2), \linebreak \epsilon:D \to I)$ gives the additive bundle structure
            \begin{gather*}
                p:[D,-] \xrightarrow{0^*} [I,-] \xrightarrow{u} id, \hspace{1cm}
                0: id \xrightarrow{u} [I,-] \xrightarrow{[\epsilon, -]} [D,-] \\
                +:[D(2),-] \xrightarrow{\delta^*} [D,-].
            \end{gather*}
    
            Note that by the continuity of $[-, M]: \C^{op} \to \C$, we have $[D(2),M] \cong [D,M] \ts{p}{p} [D,M]$.
            \item The lift is given by
                \[
                    \ell: [D, -] \xrightarrow{\odot^*} [D \ox D, -] \xrightarrow{b^{-1}} [D,[D,-]]. 
                \]
            \item The canonical flip is given by
                \[
                    c: [D,[D,-]] \xrightarrow{b} [D\ox D, -] \xrightarrow{\sigma^*} [D\ox D, -]\xrightarrow{b^{-1}} [D,[D,-]].
                \]
        \end{itemize}
        The coherences and couniversality properties of an infinitesimal object, along with the continuity of 
        \[
            [-,M]: \C^{op} \to \C 
        \]
        then induce the coherences and universality properties for the tangent bundle. This completes the proof.
    \end{enumerate}
\end{proof}
The tangent structure on $\C^{op}$ induced by the infinitesimal object is the \emph{dual} tangent structure on $\C$. This will be revisited in Chapters \ref{chap:weil-nerve} and \ref{ch:inf-nerve-and-realization} when looking at tangent categories from the enriched perspective.

Returning to synthetic differential geometry, the co-universality conditions on an infinitesimal object corresponds to infinitesimal linearity (Definition \ref{def:inf-linear}). In some sense, the category of infinitesimally linear objects is the largest subcategory of $\e$ for which $D$ is an infinitesimal object. 
\begin{corollary}
    In a model of synthetic differential geometry $(\e, R)$, the object $D = [d \in R | d^2 = 0]$ is an infinitesimal object in the category of infinitesimally linear objects in $\e$.
\end{corollary}

This thesis makes use of the 2-category of tangent categories (Section 2.3 of \cite{Cockett2014}), which formalizes the notion of a morphism of tangent structure and 2-cells between them. This 2-categorical framework is a departure from the classical theory of synthetic differential geometry, where the literature only really addresses morphisms of tangent structure in the form of fully faithful embeddings $\mathsf{SMan} \hookrightarrow \mathsf{Microl}(\e)$.\pagenote{Removed ``the formal theory of tangent categories, added Microl to $\e$.''}

\begin{definition}\label{def:tang-functor}
    Let $(\C, \mathbb{T}), (\D, \mathbb{T}')$ be a pair of tangent categories. A pair $(F: \C \to \D, \alpha:F.T \Rightarrow T'.F)$ is a \emph{tangent functor} if the following diagrams commute:\pagenote{I have fixed up the notation in this definition by adding a "." to $F.T$, and by setting the notation $(F,\alpha)$ for tangent functors. }
    \begin{center}
\begin{tikzcd}
	F & {F.T} & {F.T_2} & {F.T} & {F.T} & F \\
	& {T.F} & {T_2.F} & {T.F} & {T.F} \\
	{F.T} && {T.F} & {F.T^2} & {T.F.T} & {T^2.F} \\
	{F.T^2} & {T.F.T} & {T^2.F} & {F.T^2} & {T.F.T} & {T^2.F}
	\arrow["{F.+}", from=1-3, to=1-4]
	\arrow["{\alpha_2}"{description}, from=1-3, to=2-3]
	\arrow["{+.F}"', from=2-3, to=2-4]
	\arrow["\alpha"{description}, from=1-4, to=2-4]
	\arrow["\alpha"{description}, from=1-5, to=2-5]
	\arrow["{F.p}", from=1-5, to=1-6]
	\arrow["{p.F}"', from=2-5, to=1-6]
	\arrow["{0.F}"', from=1-1, to=2-2]
	\arrow["{F.0}", from=1-1, to=1-2]
	\arrow["\alpha"{description}, from=1-2, to=2-2]
	\arrow["{\ell.F}", from=3-3, to=4-3]
	\arrow["{F.\ell}"', from=3-1, to=4-1]
	\arrow["\alpha", from=3-1, to=3-3]
	\arrow["{\alpha.T}"', from=4-1, to=4-2]
	\arrow["{T.\alpha}"', from=4-2, to=4-3]
	\arrow["{\alpha.T}", from=3-4, to=3-5]
	\arrow["{T.\alpha}", from=3-5, to=3-6]
	\arrow["{\alpha.T}"', from=4-4, to=4-5]
	\arrow["{T.\alpha}"', from=4-5, to=4-6]
	\arrow["{F.c}"', from=3-4, to=4-4]
	\arrow["{c.F}", from=3-6, to=4-6]
\end{tikzcd}
    \end{center}
    A tangent functor is \emph{strong} whenever $\tnat$ is a natural isomorphism; for a sub-(tangent category) inclusion, $\alpha = id$.  A tangent functor $(F, \alpha)$ between cartesian tangent categories is a \emph{cartesian tangent functor} if $F$ is an isomonoidal functor and $\tnat$ is a monoidal natural transformation. 
\end{definition}
\begin{example}\label{ex:composition-of-tangent-functors}
    ~\begin{enumerate}[(i)]
        \item The coherences on the canonical flip $c$ guarantee that $(T:\C \to \C,c:T.T \Rightarrow T.T)$ is a strong tangent endofunctor on any tangent category. 
        \item Given a pair of tangent functors $(A,\alpha):\C \to \D, (B, \beta): \D \to \mathbb{E}$, the composition 
        \[
            (B.A:\C \to \mathbb{E}, \beta.A \o B.\alpha: 
            \begin{tikzcd}
                {B.A.T^C} & {B.T^D.A} & {T^E.B.A}
                \arrow["{B.\alpha}", Rightarrow, from=1-1, to=1-2]
                \arrow["{\beta.A}", Rightarrow, from=1-2, to=1-3]
            \end{tikzcd})
        \]
        is a tangent functor.
        \item A model of synthetic differential geometry $(\e,R)$ is \emph{well-adapted} whenever there is a fully faithful, strict tangent functor from $\mathsf{SMan}$ to the category of microlinear spaces of $\e$,  $\mathsf{SMan} \hookrightarrow \mathsf{Microl}(\e)$. The original development of well-adapted models for synthetic differential geometry may be found in \cite{Dubuc1981}, and the reader may check \cite{Bunge2018} for a recent account of the construction of such models, or section 3 of \cite{Kock2006}.\pagenote{Clarified that the inclusion into $\mathsf{Microl}(\e)$ is a tangent functor.}
    \end{enumerate}
\end{example}

\begin{definition}\label{def:tang-nat}
    A \emph{tangent natural transformation} $\gamma$ between tangent functors $(F,\alpha), (G,\beta)$ is a natural transformation so that the following diagram commutes:
    \[
        \begin{tikzcd}
    F . T(A) \rar{\gamma_{TA}} \dar{\tnat_A} & G . T(A) \dar{\beta_{TA}} \\
    T . F(A) \rar{T\gamma_A} & T . G(A)
\end{tikzcd}
    \] 
    If $F, G$ are cartesian tangent functors, then $\gamma$ is cartesian whenever it is an isomonoidal natural transformation.
\end{definition}
\begin{definition}\label{def:tang-2cat}
    We will call the 2-category of tangent categories, tangent functors, and tangent natural transformations $\mathsf{TangCat}$. The 2-category of cartesian tangent categories is $\mathsf{CartTangCat}$.
\end{definition}

\section{Local coordinates in a tangent category}%
\label{sec:diff-and-tang-struct}

This section develops some structures to facilitate reasoning about higher tangent bundles, using ``local coordinates.'' In the case of a cartesian differential category,  $T^n(A) = \prod_{2^n} A$, whereas for an open subset $U \subseteq \R^m$ the tangent bundle decomposes as $T^nU \cong U \x (\R^m)^{2^n - 1}$. This section introduces two structures that allow for these arguments in an arbitrary tangent category: differential objects and connections.\pagenote{
   When I originally wrote this thesis, differential objects figured into the story more prominently. I have cut a few results here that were messily written and were no longer used in subsequent chapters.
}

\begin{definition}[\cite{Cockett2018}]
    \label{def:differential-object}
    A\pagenote{
        This definition was very unclear/inconsistent. I have since removed the term ``linear'', and fixed the notation on the commutative monoid structure on $A$, which addresses the main issued raised. 
    } \emph{ differential object}\footnote{Not to be confused with 4.1 from \cite{Barr2002}.} in a cartesian tangent category is a commutative monoid $(A, +_A,0_A)$ such that there is a section-retract pair $A \xrightarrow{\lambda} TA \xrightarrow{\hat{p}} A$ which exhibits $T(A)$ as a biproduct in the category of commutative monoids:
        \[A \oplus A \cong TA.\]
    Concretely, a differential object is a commutative monoid equipped with $\lambda:A \to TA, \phat:TA \to A, \phat \o \lambda = id$ so that the following axioms hold:
    \begin{enumerate}[DO.1]
    \item Coherence between $+$ and $\lambda, \phat$:
        \[
\begin{tikzcd}
	{T(A \x A)} & TA & {A\x A} & A \\
	{TA \x TA} && {TA \x TA} \\
	{A \x A} & A & {T(A \x A)} & TA
	\arrow["\phat"{description}, from=1-2, to=3-2]
	\arrow["{T.+_A}", from=1-1, to=1-2]
	\arrow["{(T\pi_0,T\pi_1)}"{description}, from=1-1, to=2-1]
	\arrow["{\phat \x \phat}"{description}, from=2-1, to=3-1]
	\arrow["{+_A}"{description}, from=3-1, to=3-2]
	\arrow["{+_A}", from=1-3, to=1-4]
	\arrow["{\lambda \x \lambda}", from=1-3, to=2-3]
	\arrow["{(T\pi_0,T\pi_1)^{-1}}"{description}, from=2-3, to=3-3]
	\arrow["{T.+_A}", from=3-3, to=3-4]
	\arrow["\lambda"{description}, from=1-4, to=3-4]
\end{tikzcd}\]
    \item Coherence between $0_A$ and $0.A$:
      \[ 
\begin{tikzcd}
	A & 1 & A & 1 \\
	TA & A & TA & A
	\arrow["0"', from=1-1, to=2-1]
	\arrow["{!}"{description}, from=1-1, to=1-2]
	\arrow["{0_A}", from=1-2, to=2-2]
	\arrow["\phat"{description}, from=2-1, to=2-2]
	\arrow["\lambda"{description}, from=1-3, to=2-3]
	\arrow["p"{description}, from=2-3, to=2-4]
	\arrow["{!}"{description}, from=1-3, to=1-4]
	\arrow["{0_A}"{description}, from=1-4, to=2-4]
\end{tikzcd} \]    
    \item Coherence between $+_A$ and $+.A$:
        \[
\begin{tikzcd}
	{T_2A} & {A \x A} & {A \x A} & {T_2A} \\
	TA & A & A & TA
	\arrow["{\phat \x \phat}", from=1-1, to=1-2]
	\arrow["{+_A}", from=1-2, to=2-2]
	\arrow["{+}"', from=1-1, to=2-1]
	\arrow["\phat"', from=2-1, to=2-2]
	\arrow["{\lambda \x \lambda}", from=1-3, to=1-4]
	\arrow["{+_A}"', from=1-3, to=2-3]
	\arrow["\lambda"', from=2-3, to=2-4]
	\arrow["{+}", from=1-4, to=2-4]
\end{tikzcd}\]
    \item Coherence between $\lambda$ and $\phat$ with $\ell$:
      \[
\begin{tikzcd}
	TA & {T^2A} & A & TA \\
	A & TA & TA & {T^2A}
	\arrow["\ell", from=1-1, to=1-2]
	\arrow["{T.\phat}"{description}, from=1-2, to=2-2]
	\arrow["\phat", from=2-2, to=2-1]
	\arrow["\phat"{description}, from=1-1, to=2-1]
	\arrow["\lambda"{description}, from=1-3, to=2-3]
	\arrow["\lambda", from=1-3, to=1-4]
	\arrow["\ell"{description}, from=1-4, to=2-4]
	\arrow["{T.\lambda}"', from=2-3, to=2-4]
\end{tikzcd}\]
  \end{enumerate}
  A map $f:A \to B$ between differential objects $(A,\lambda_A, \phat_A, +_A,0_B) \to (B,\lambda_B,\phat_B, +_B, 0_B)$ is \emph{linear} whenever $f$ preserves the lifts and projections
  \[
      (T.f \o \lambda_A = \lambda_B \o f) \text{ and } (\phat_B \o T.f = f \o \phat_A).
  \]
\end{definition}
 Following the work in Section 3 of \cite{Cockett2018}, it is sufficient to check that $f$ preserves $\lambda$ or $\phat$ (each condition implies the other).
\begin{example}\label{ex:diffob-sman}
    ~\begin{enumerate}[(i)]
        \item In the category of smooth manifolds, the real numbers are a differential object, as $T.\R = \R[x]/x^2$, so the lift map in this case is
        \[
            \lambda(a) = 0 + a\cdot x,  \hspace{1cm} \phat(a + b\cdot x) = b. 
        \]
        More generally, finite-dimensional real vector spaces, with their canonical smooth manifold structure, are exactly differential objects in the category of smooth manifolds. The lift map is defined using $T\R \cong \R[x]/x^2$, so that
        \[ V \xrightarrow[]{(0,1^\R)} T(V \x \R) \xrightarrow[]{T\cdot_V} TV. \]
        This is equivalent to the isomorphism $T.V = R[x]/x^2 \ox V$.
        \pagenote{I have removed the convenient manifolds example because I had not introduced that example tangent category.}
        \item Every object in a cartesian differential category has a canonical differential object structure, as $TA := A \x A$.
    \end{enumerate}
\end{example}

Classically, the tangent space above each point of a smooth manifold is a vector space. 
We have a similar result for differential objects from \cite{Cockett2018}.
\begin{lemma}\label{lem:tang-space-diff-ob}
    Suppose we have the following pullback in a cartesian tangent category $\C$, and all powers of $T$ preserve it.
    \[
        \begin{tikzcd}
            E \rar{\iota} \dar{!} & TM \dar{p} \\
            1 \rar{m} & M
        \end{tikzcd}
    \]
    There is a unique differential object structure $(E,\lambda,\phat)$ so that $\ell \o \iota = T.\iota\o \lambda$.\footnote{Some diagrammatic notation had creeped into the original draft here, I have fixed it.}
\end{lemma}

There are two natural classes of morphisms between differential objects, linear and ``smooth'' (that is, arbitrary morphisms).
\begin{definition}[\cite{Cockett2018}]\label{def:diff-dlin}
    Let $\C$ be a cartesian tangent category. 
    We define the following categories:
    \begin{enumerate}[(i)]
        \item $\Diff(\C)$ is the category of differential objects and arbitrary morphisms, so for any differential objects $A,B$, we have $\Diff(\C)(A,B) := \C(A,B)$.
        \item $\Dlin(\C)$ is the category of differential objects and linear morphisms.
    \end{enumerate}
\end{definition}

The category of differential objects and smooth maps is a cartesian differential category, exhibiting differential calculus as a specialized logic in a tangent category. 
\begin{proposition}[Section 3.5 of \cite{Cockett2018}]\label{prop:diff-obs-cdc}
    Let $\C$ be a cartesian tangent category.
    Then:
    \begin{enumerate}[(i)]
        \item $\Diff(\C)$ is a cartesian differential category, where we define the differential combinator $D$ to be
        \[
            \infer{A \x A \xrightarrow{\lambda \x 0} T(A \x A) \xrightarrow[]{T.+_A} T(A) \xrightarrow{T(f)} T(B) \xrightarrow{\phat_B} B}{A \xrightarrow{f} B}
        \]
        (where $\lambda_A, +_A$ are the lift and addition for the differential object structure on $A$, and $\phat_B$ is the projection map on the differential object structure on $B$).
        \item There is an equality of categories, $\mathsf{Lin}(\Diff(\C)) = \Dlin(\C)$, meaning that a morphism between differential objects in the cartesian tangent category $\C$ is linear if and only if it is linear in the cartesian differential category $\Diff(\C)$.
    \end{enumerate}
\end{proposition}
Recall that if $M$ is an open subset of $\R^n$, the second tangent bundle of an open subset $U \subseteq \R^n$ splits as
\[
    T^2(U) = T(U \x \R^n) = (U \x \R^n) \x (\R^n \x \R^n) = T_3U.
\]
Every smooth manifold admits such a decomposition on its second tangent bundle; these are known as a \emph{affine connections}\footnote{The prefix ``affine'' differentiates these from more general connections on differential bundles, which are introduced in Section \ref{sec:connections-on-a-differential-bundle}.} and provide a way to reason about an object as though it has local coordinates in an arbitrary tangent category.\pagenote{Added a footnote to explain the prefix ``affine''.}
\pagenote{slightly changed the passage between differential objects -> affine connections}


\begin{definition}[\cite{Cockett2017}]
    In a tangent category $\C$, define the following:
    \begin{enumerate}[(i)]
        \item An affine vertical connection is a map $\kappa:T^2M \to TM$ so that 
        \begin{enumerate}
            \item $\kappa$ is a \emph{vertical descent}, namely a section of the vertical lift $\ell:TM \to T^2M$, so $\kappa \o \ell = id$;
            \item $\kappa$ is compatible with both lifts on $T^2M$: $T.\kappa \o \ell.T = T.\kappa \o T.\ell = \ell\o \kappa$.
        \end{enumerate}
        \item An affine horizontal connection is a map $\nabla:T_2 \to T^2M$ so that
        \begin{enumerate}
            \item $\nabla$ is a \emph{horizontal lift}, namely a section to the horizontal descent $(p.T, T.p): T^2 \Rightarrow T_2$, so that $(p.T,T.p) \o \nabla = id$;
            \item $\nabla$ is compatible with the linear structures on $T^2, T_2$:$ T.\nabla\o (\ell\x 0) \linebreak = \ell \o \nabla and T.\nabla(0\x \ell) = T.\ell \o \nabla$.
        \end{enumerate}
        \item An \emph{affine connection} is a pair $(\kappa, \nabla)$ comprising a vertical and horizontal connection on $M$ satisfying the compatibility conditions: 
        \begin{enumerate}
            \item $T.+ \o (+.T \o (\ell \o \kappa, T.0 \o p.T), \nabla(p.T, T.p)) = id_{T^2M}$, 
            \item $\kappa \o \nabla = 0 \o p \o \pi_i$.
        \end{enumerate}
        An affine connection is \emph{torsion-free} if $\kappa \o c = \kappa$.
    \end{enumerate}
\end{definition}


The data of a vertical connection is sufficient to define a full connection, as observed in \cite{LucyshynWright2018}.
\begin{lemma}[\cite{LucyshynWright2018}]
    A full connection is equivalent to a vertical connection in which the following diagram is a fiber product:
\[\begin{tikzcd}
    & TM \\
    {T^2M} & TM & M \\
    & TM
    \arrow["p", from=1-2, to=2-3]
    \arrow["p"{description}, from=2-2, to=2-3]
    \arrow["p"', from=3-2, to=2-3]
    \arrow["{p.T}", from=2-1, to=1-2]
    \arrow["{T.p}"', from=2-1, to=3-2]
    \arrow["\kappa"{description}, from=2-1, to=2-2]
\end{tikzcd}\]
\end{lemma}

\begin{example}
    ~\begin{enumerate}[(i)]
        \item Every differential object in a tangent category has a canonical vertical connection given by $(p \o p.T, \phat \o T.\phat)$. A morphism of differential objects will preserve this vertical connection.
        \pagenote{Fixed the notation for the affine vertical connection on a differential object induced by a differential object's $\phat$ map. }
        \item Every smooth manifold has a non-natural choice of Riemannian metric. By the fundamental theorem of Riemannian geometry, there is a torsion-free connection associated with the metric. (See any standard reference on Riemannian geometry, e.g. \cite{Carmo1992}.)
    \end{enumerate}
\end{example}

Connections allow for arguments on higher powers of the tangent bundle to be pushed down to pullback powers of $T$.
\begin{observation}
    Suppose $M, N$ each have connections $(\kappa_{-}, \nabla_{-})$. 
    The map $T^2.f: T^2M \to T^2N$ can be written using local coordinates $\widehat{T^2.f}: T_3M \to T_3N$ as
    \[\begin{tikzcd}
        {T^2M} & {T^2N} \\
        {T_3M} & {T_3N}
        \arrow["{T.+ \o (+.T \o (\ell \o \pi_2, T.0 \o \pi_0), \nabla(\pi_0, \pi_1))}", from=2-1, to=1-1]
        \arrow["{T^2.f}", from=1-1, to=1-2]
        \arrow["{(p.T, T.p, \kappa_N)}", from=1-2, to=2-2]
        \arrow[dashed, from=2-1, to=2-2]
    \end{tikzcd}\]
    where $\widehat{T^2.f}$ is given by
    \[
        (T.f\o\pi_0, T.f \o\pi_1, T.f \o \pi_2 +_N \nabla[f](\pi_0,\pi_1))
    \]
    with $\nabla[f] := \kappa_N \o T^2.f \o \nabla_M$.
\end{observation}
Note that in the case that $f$ preserves the connections, $\nabla[f] = 0$, as 
\[
    \kappa_N \o T^2.f \o \nabla_M = Tf \o \kappa_M \o \nabla_M = Tf \o 0 \o p = 0 \o f \o p.
\]


\section{Submersions}\label{sec:submersions}

The category of smooth manifolds is incomplete: there are cospans\footnote{Following a general convention in category theory, where the prefix "co"-X means an X in the opposite category, a cospan in $\C$ is a span in the dual category of $\C$.} $X \xrightarrow[]{f} M \xleftarrow[]{g} Y$ for which the pullback fails to exist. Following \cite{thom1954}, the pullback of a cospan exists and is preserved by $T$ (i.e. it is a \emph{T-limit}) whenever for each point $f(x) = g(y)$, the direct sum of the images of $T_xf$ and $T_yg$ is the full vector space $T_{f(x)}M$, such cospans are called \emph{transverse}. Submersions, then, form a convenient class of maps, as any cospan where one map is a submersion will be transverse. More precisely:
\begin{definition}
    \label{def:submersion-sman}
    If $A$ and $B$ are smooth manifolds, a smooth function $f:\linebreak A \rightarrow B$ is a \emph{submersion} if and only if the derivative $Df|_a$ of $f$ at every point $a\in A$ is a surjective linear map.
\end{definition}
With this definition we have the following result:
\begin{proposition}%
    \label{prop:submersion-properties}
    In the category of smooth manifolds, let the class of submersions be denoted by $\sh$.
    \begin{enumerate}[(i)]
        \item Submersions are closed under the tangent functor: $f \in \sh \Rightarrow T.f \in \sh$.
        \item Submersions are closed to pullback along arbitrary maps:
\[\begin{tikzcd}
	& X & {g^*X} & X \\
	N & M & N & M
	\arrow["g"', from=2-1, to=2-2]
	\arrow[""{name=0, anchor=center, inner sep=0}, "{f\in\mathcal{S}}"', from=1-2, to=2-2]
	\arrow[""{name=1, anchor=center, inner sep=0}, "{g^*f\in \mathcal{S}}", from=1-3, to=2-3]
	\arrow["g"', from=2-3, to=2-4]
	\arrow["f", from=1-4, to=2-4]
	\arrow["{\bar{g}}", from=1-3, to=1-4]
	\arrow["\lrcorner"{anchor=center, pos=0.125}, draw=none, from=1-3, to=2-4]
	\arrow[shorten <=10pt, shorten >=10pt, Rightarrow, from=0, to=1]
\end{tikzcd}\] 
        This will often be referred to as \emph{$T$-stability under reindexing}, as it induces a functor between slice categories: 
        \[g^*: \mathsf{Submersions}/M \to \mathsf{Submersions}/N.\]
    \end{enumerate}
    \pagenote{ Based on a comment from Michael, it seemed appropriate to move the T-prefix to stability, as the reindexing operation is unchanged.}
\end{proposition}
The properties of the class of submersions in the category of smooth manifolds were studied in \cite{Cockett2018} and axiomatized as a \emph{tangent display system}.

\begin{definition}\label{def:display-system}
    A \emph{tangent display system} in a tangent category $\C$ is a class of maps $\mathcal{D}$ in $\C$ that is
    \begin{itemize}
        \item stable under the tangent functor, $d \in \d \Rightarrow T.d \in \d$,
        \item $T$-stable under reindexing (as in Proposition \ref{prop:submersion-properties}).
    \end{itemize}
    We call any tangent display system that is closed to retracts in the arrow category a \textit{retractive} display system. If for all $M$, $p_M \in \mathcal{D}$, we call $\mathcal{D}$ a proper (retractive)\pagenote{Brackets are a common notation for optional prefixes, e.g. (co)limits.} display system. 
\end{definition}
This section will show that the submersions in the category of smooth manifolds give a retractive display system, yielding a general construction of retractive display systems from display systems.

The definition of a submersion may be rephrased as follows: $f$ is a submersion if and only if for all $a\in A$ and all $v\in T(B)$ such that $fa = pv$, there exists a $w\in T(A)$ such that $T.f \o w = v$.
This is a \textit{weakly} universal cone over $A \xrightarrow{f} B \xleftarrow{p} TB$: there exists \textit{at least} one morphism into it for any other cone over the diagram. 
\begin{definition}\label{def:weak-pullback}
    A commuting square is a \textit{weak pullback} if for any $x:X \to A$ and $y:X \to B$ so that $fx = gy$, there exists a map $X \to W$ making the following diagram commute:
    
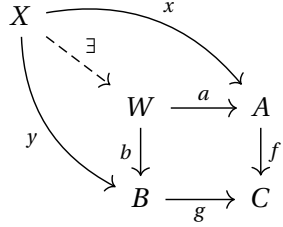
\[
        \begin{tikzcd}
        X \ar[bend left]{rrd}{x} \ar[bend right]{ddr}[swap]{y} \ar[dashed]{rd}{\exists} \\
            & W \rar{a} \dar[swap]{b} & A \dar{f} \\
            & B \rar[swap]{g} & C
        \end{tikzcd}
    \]
    If the above diagram is a weak pullback for each $T^n$, then it is a \emph{weak $T$-pullback}.
\end{definition}
\begin{lemma}\label{lem:pb-retract}
    Should the pullback of $A \xrightarrow{f} C \xleftarrow{g} B$ exist, Definition \ref{def:weak-pullback} is equivalent to asking that the induced map $(a, b):W \to A \ts{f}{g} B$ be a split epimorphism.
\end{lemma}
\begin{proof}
    Let $r$ be a retract of $(a,b):W \to A\ts{f}{g} B$. For any $X \xrightarrow[]{(x,y)}  A\ts{f}{g} B$, the map $ r \o (x,y)$ exhibits the diagram as a weak pullback. For the converse, the unique map $(a,b):W \to A \ts{f}{g} B$ will be a section of any map $A \ts{f}{g} B \to W$ induced by weak univerality.\pagenote{
    This result was originally stated without proof.
    }
\end{proof}
We now restate the submersion property for a map $f$ using global elements (for all $a\in A$ and all $v\in T(B)$ such that $fa = pv$, there exists a $w\in T(A)$ such that $T(f)w = v$) using generalized elements.
\begin{definition}\label{def:tangent-submersion}
    An arrow $f:A \rightarrow B$ in a tangent category is a \emph{tangent submersion} if and only if the naturality diagram
    \[
        \begin{tikzcd}
            TA \rar{Tf} \dar[swap]{p} & TB \dar{p} \\
            A \rar[swap]{f} & B
        \end{tikzcd}
    \] 
    is a weak $T$-pullback. 
\end{definition}{}
Following Lemma \ref{lem:pb-retract}, in the case that the pullback exists this is equivalent to asking for a section $h: A \ts{f}{p} TB \to TA$ of the horizontal descent $(p,Tf): TA \to A \ts{f}{p} TB$ (this section is sometimes called a \textit{horizontal lift} in differential geometry literature \cite{Cordero1989}). 
In smooth manifolds, the $T$-pullback along the projection $p:T \Rightarrow id$ always exists, so to prove that every submersion is a tangent submersion it suffices to show the existence of a horizontal lift.
\begin{proposition}
    In the category of smooth manifolds, the tangent submersions are precisely the submersions (Definition \ref{def:submersion-sman})
\end{proposition}
\begin{proof}
    There is an explicit construction of a horizontal lift for a classical smooth submersion in VII.1 of \cite{Cushman2015}.
\end{proof}
It is possible to show that the $T$-stability properties for submersions in the category of smooth manifolds follow from the general theory of weak pullbacks.
We begin by showing that weak pullbacks satisfy a weakened version of the pullback lemma and then show that the retract of a weak pullback is a weak pullback (the second lemma is Lemma 2.1 of \cite{Adamek2010}).
\begin{lemma}[Pullback lemma]\label{lem:weak-pullback}
    Consider the diagram
    \[
        \begin{tikzcd}
            \bullet \rar \dar \ar[rd, phantom, "(A)"] & \bullet \rar{f}\dar{g} \ar[rd, phantom, "(B)"] & \bullet\dar \\
            \bullet \rar & \bullet \rar & \bullet 
        \end{tikzcd}
    \]
    \begin{enumerate}[(i)]
        \item If $f,g$ are jointly monic and $(A)+(B)$ is a weak pullback, then $(A)$ is a weak pullback.
        \item If $(A),(B)$ are weak pullbacks then $(A)+(B)$ is a weak pullback.
    \end{enumerate}
\end{lemma}    
\begin{proof}
    ~\begin{enumerate}[(i)]
        \item If $(A)+(B)$ is a weak pullback, a map can be induced for a cone over $(A)$ by concatenating it with $(B)$; the jointly monic condition on $f,g$ guarantees that the map induced for $(A)+(B)$ will commute for $(A)$. \pagenote{caught a broken sentence here.}
        \item Given a cone for $(A)+(B)$ induce a map for $(B)$, which then induces a cone for $(A)$.
    \end{enumerate}
\end{proof}

\begin{lemma}\label{lem:wpb-retract}
    (Weak) pullbacks are closed to retracts.
\end{lemma}
\begin{proof}
    Suppose that $S'$ is a weak pullback, and $S$ is a retract of it in the category of commuting squares. Consider the following diagram (suppressing the subscripts for $s,r$):
        \[
        \begin{tikzcd}[column sep = small, row sep = small]
            Z \ar[bend right]{rrd} \ar[bend right]{ddr}  \ar[bend left, dashed]{rrr} & A \ar{rd}{x} \ar{dd}[swap, near end]{y} \ar{rr}{s} & & A'\ar{rd}{x'} \ar{dd}[swap, near end]{y'} \ar{rr}{r} & & A\ar{rd}{x} \ar{dd}[near end]{y} \\
            & & B \ar{dd}[near start]{w} \ar{rr}[near start]{s} 
            & & B' \ar{dd}[near start]{w'} \ar{rr}[near start]{r} & & B \ar{dd}[near start]{w} \\
            & C \ar{rd}[swap]{z} \ar{rr}[near start]{s} & & C' \ar{rd}[swap]{z'} \ar{rr}[near start]{r} & & C \ar{rd}[swap]{z}  \\
            & & D \ar{rr}{s} & & D' \ar{rr}{r}& & D
        \end{tikzcd}
    \]
    Given a cone for $S$, there is a corresponding cone for $S'$ which induces a map $Z \to A'$ and postcomposition with $r_A$ gives the desired map into $A$.
\end{proof}
Using these lemmas, it is straightforward to prove that the following $T$-stability properties hold for tangent submersions.
\begin{lemma}\label{lem:submersion}
    In any tangent category $\mathbb{X}$,
    \begin{enumerate}[(a)]
        \item tangent submersions are closed to composition;
        \item tangent submersions are closed to retracts; 
        \item any $T$-pullback of a tangent submersion is a tangent submersion.
    \end{enumerate}
\end{lemma}
\begin{proof}
    (a) follows from Lemma \ref{lem:weak-pullback} while (b) follows from Lemma \ref{lem:wpb-retract}. It remains to prove (c).     Consider a $T$-pullback, where $u$ is a tangent submersion:
    \[
    \begin{tikzcd}
        A \rar{f} \dar{v} & M \dar{u} \\
        B \rar{g} & N
    \end{tikzcd}
    \]
    By naturality, the outer paths of the following two diagrams are equal: 
    \pagenote{The equality of the two diagrams was originally asserted without explaining why they were equal.}
    \[
        \begin{tikzcd}
            TA \rar{Tf} \dar{Tv} & TM \rar{p} \dar{Tu} & M \dar{u} \\
            TB \rar{Tg} & TN \rar{p} & N 
        \end{tikzcd}
        =
        \begin{tikzcd}
            TA \rar{p} \dar{Tv} & A \rar{f} \dar{v} & M \dar{u} \\
            TB \rar{p} & B \rar{g} & N
        \end{tikzcd}
    \]
   Note that the left diagram is a weak pullback by composition.  Therefore the outer perimeter of the right diagram is a weak pullback, and the right square is a pullback, so the left square is a weak pullback by Lemma \ref{lem:weak-pullback}, as desired.
\end{proof}
The class of tangent submersions is closed to retracts in the arrow category and is conditionally $T$-stable under reindexing (if the $T$-pullback of a tangent submersion exists, it is a tangent submersion). This stability property leads to the following result:
\begin{proposition}\label{prop:display-submersions-are-r-display}
    Let $\mathbb{X}$ be a tangent category that allows for reindexing of the class of tangent submersions $\mathcal{R}$. Then the class of tangent submersions is a display system.
\end{proposition}
\begin{proof}
    Any class of maps that is closed to reindexing is a tangent display system, and the class of submersions is closed to retracts in the arrow category.
\end{proof}
\begin{corollary}\label{cor:sman-r-display}
    The class of submersions in the category of smooth manifolds is a  proper retractive display system.
\end{corollary}

\chapter{Differential bundles}%
\label{ch:differential_bundles}

Our principal aim in this thesis is to provide an abstract tangent-categorical axiomatization for Lie algebroids.  To accomplish this, we must provide an axiomatization for Lie algebroids which is essentially algebraic (in the sense of \cite{Freyd1972}).  However, in the category of smooth manifolds, Lie algebroids are defined in terms of vector bundles and these are prima facie a highly non-algebraic notion.  

In addition to algebraic axioms which make it an $\R$-module in the slice over its base $M$, a vector bundle $q:E \to M$ satisfies a crucial topological requirement:  it must be locally trivial.  This means that the projection $q : E \to M$ must be locally isomorphic to a projection $\pi_0:U \x \R^n \to U$ for some open subset $U$ of $M$ and natural number $n$.  It is this property that permits calculations using local coordinates, an approach deeply enshrined in the culture of differential geometry.

\cite{Cockett2017} introduced the algebraic notion of a differential bundle.  Evidence that differential bundles are the appropriate generalization of vector bundles was provided by showing how classical results for vector bundles could be generalized to differential bundles in any tangent category  \cite{Cockett2017,Cockett2018}.  However, the precise relationship in the category of smooth manifolds between vector bundles and differential bundles was left open. The main result of this chapter (see \cite{MacAdam2021}) is that vector bundles and differential bundles coincide in the category of smooth manifolds.

The axiomatization of differential bundles focuses on another important property of vector bundles: given a vector bundle $q:E \to M$ and a vector $v$ in the fibre $E_x$ above $x \in M$, the tangent space $T_v(E_x)$ can be naturally identified with $E_x$.    This gives a lift map $\lambda:E \to TE$ which can be axiomatized.   While the lift map had long been noted in the differential geometry literature in the guise of the Euler vector field (see 6.11 of \cite{Kolar1993} and also Section 1 of \cite{Michor1996} which explicitly uses the term "lift"), it had not been adopted as the basis of an abstract axiomatization.  

More recently, in the differential geometry literature, \cite{Grabowski2009} and \cite{Bursztyn2016} realized that the multiplicative $\R^+$-action on the total space $E$ determines the vector bundle structure of $q:E \to M$, and conversely such a multiplicative action determines a vector bundle precisely when its Euler vector field (Definition \ref{def:evf}) satisfies an additional ``non-singular'' property.  This chapter extends these more recent observations on vector bundles to differential bundles.

The chapter begins by reviewing vector bundles, describing the Euler vector field construction that sends a vector bundle $q:E \to M$ to a "lift" map $\lambda:E \to TE$ from \cite{Grabowski2009}, whereas the rest of the chapter contains new results developed in collaboration with Matthew Burke. The second section establishes that these lifts are associative coalgebras for the weak comonad $(T,\ell)$, and that there is a fully faithful functor from vector bundles into the category of lifts for smooth manifolds. The third section identifies the universal property satisfied by the lift (or equivalently Euler vector field), while the fourth section shows that a non-singular lift corresponds precisely to a differential bundle. The fifth section proves the main theorem of the chapter: vector bundles are precisely differential bundles for smooth manifolds. The final section contains some remarks on extending affine connections to arbitrary differential bundles, which will be useful in Chapter \ref{ch:involution-algebroids}. 

\pagenote{Each chapter has been given a new introduction.}

\section{Vector bundles}%
\label{sec:vector-bundles}

A vector bundle over a manifold $M$ axiomatizes the notion of a smoothly varying family of vector spaces indexed by the points $m \in M$.
The driving example is that of the tangent bundle over a smooth manifold $M$, where the fibre above each point $m \in M$ is the tangent space $T_mM$.
The manifold structure guarantees that the projection is \emph{locally trivial}: given a chart $U \hookrightarrow M$, the next pullback splits as a product:
\[
\begin{tikzcd}
	{U \x \R^n} & TM \\
	U & M
	\arrow["p", from=1-2, to=2-2]
	\arrow[hook, from=2-1, to=2-2]
	\arrow[from=1-1, to=2-1]
	\arrow[from=1-1, to=1-2]
	\arrow["\lrcorner"{anchor=center, pos=0.125}, draw=none, from=1-1, to=2-2]
\end{tikzcd}\]
The local triviality of the tangent bundle is essential for various constructions and is part of the definition of a vector bundle.

\begin{definition}\label{def:vector-bundle}
  A \emph{vector bundle} is a tuple \[(q:E \to M, \xi:M \to E, +: E \ts{q}{q} E \to E, \cdot: \R \x E \to E)\] of morphisms in $\mathsf{SMan}$ so that
  \begin{enumerate}[(i)]
    \item the tuple $(q,\xi,+, \cdot)$ defines an $\R$-module in $\mathsf{SMan}/M$;
    \item the map $q:E \to M$ is locally trivial.
  \end{enumerate}
  The fibred $\R$-module structure means $E$ is a family of vector spaces indexed by $M$, $\{ E_m | m \in M \}$.
\end{definition}
It is important to note that the local triviality axiom guarantees that the projection of a vector bundle is a submersion (Definition \ref{def:submersion-sman}); thus pullback powers of $q:E \to M$ exist and are preserved by the tangent functor.\pagenote{This addresses a concern Kristine had raised about the preservation of $E_2$ by $T$.}
\begin{example}
    Consider the cylinder, defined as  the subset of $\R^3$ spanned by $C = \{ (x,y,z) | x^2 + y^2 = 1, z \in \R \}$:
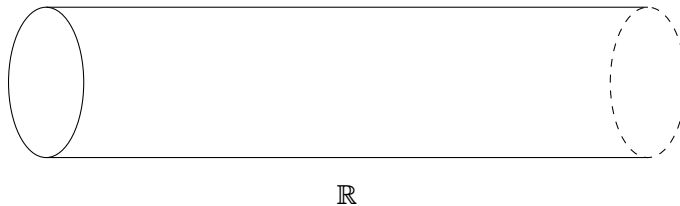
    \[ 
\begin{tikzpicture}
    \draw[rounded corners=35pt](0,0)--(8,0);
    \draw[rounded corners=35pt](0,-2)--(8,-2);
    \draw[dashed] (8.5,-1) arc (0:360:0.5cm and 1cm);
    \draw (0.5,-1) arc (0:360:0.5cm and 1cm);
    \node (b) at (4,-2.5) {$\mathbb{R}$};
\end{tikzpicture}
\]
    Above each point $i \in S^1 = \{ (x,y) |  x^2 + y^2 = 1\}$ the fibre over $i$ is $\R$. For each point $i$, we can choose a sufficiently small $\epsilon$ and take the open set \[U_i = \{ (x,y) \in S^1 |  (i_x - x)^2 + (i_y - y)^2 \le \epsilon \},\] which may be flattened to $(-(1+\epsilon), 1+\epsilon) \x \R$.
\end{example}
\pagenote{Carried forward the change in chapter 1 re:notation for $\chi(M)$, then defined the $C^\infty$ module structure for the sections of a vector bundle projection (this is a more appropriate place to introduce this notation than the Lie algebroids section)}
The sections of a vector bundle also give rise to a $C^\infty(M)$-module, generalizing that aspect of the tangent bundle's fibred $\R$-module structure.
\begin{lemma}\label{lem:Cinfty-module-vbun}
  Given a vector bundle $q:E \to M$, write the set of sections of $q$ as $\Gamma(q)$; as, for example, $\Gamma(p.M) = \chi(M)$ (recall the notation from Definition \ref{def:operational-tang}). The set $\Gamma(q)$ has a $C^\infty(M)$-module structure in much the same way as $\chi(M)$:
  \[
      X +_{\Gamma(q)} Y := + \o (X,Y), \hspace{0.25cm}
      0_{\Gamma(q)} := \xi, \hspace{0.25cm}
      (f \cdot_{\Gamma(q)} X) (m) := f(m) \cdot X(m)
  \]
\end{lemma}
There are also a variety of general constructions that yield vector bundles.
\begin{example}
  ~\begin{enumerate}[(i)]
    \item The tangent bundle is a vector bundle: the construction in \Cref{sec:smooth-manifolds} makes it clear that the projection $p:TM \to M$ is a locally trivial, fibred $\R$-module over the base space $M$.
    \item A \emph{trivial} vector bundle over $M$ with fibres in $V$ is the product $M \x V$. In particular, every vector space is a trivial vector bundle above the one-point space $\{*\}$.
    \item Each $T_kM$ will be locally trivial; locally it looks like the $k$-fold product of the tangent space $p_k^{-1}(U) \cong U \x (\R^n)^k$ for an $n$-dimensional manifold $M$. More generally, one can take the fibrewise pullback $E_k = E \ts{q}{q} E \ts{q}{q} \dots \ts{q}{q} E$ and discover a vector bundle over $M$.
    \item The \emph{cotangent bundle} of $M$, $T^*M$,  has the \emph{dual} vector space of $T_mM$ above each point $M$: $T^*_m(M) = (T_mM)^*$. This space can be appropriately topologized to be smooth, and a set of sections of $\Gamma(T^*M)$ is isomorphic to the set of morphisms $TM \to \R$ that are linear in each fibre. This construction may be applied to any vector bundle and is called the dual vector bundle.
    \item Consider the space $\Lambda^n(E)$, the alternating tensor product of $E^*$. The set of sections of this vector bundle is equivalent to the alternating $n$-linear morphism $E_n \to \R$; when restricted to the tangent bundle, this is the space of differential $n$-forms.
  \end{enumerate}
\end{example}
There are two constructions on vector bundles that will be necessary to prove the main theorem of this section.
\begin{proposition}%
  \label{prop:retracts-reindexing-of-vbuns}
  Let $(q:E \to M, \xi, +_q, \cdot_q)$ be a vector bundle. 
  ~\begin{enumerate}[(i)]
    \item For any map $f:N \to M$, the $T$-reindexing of $q$ by $f$ is a vector bundle:
    \begin{equation}%
      \label{eq:reindex-vbun}
\begin{tikzcd}
	{f^*E} & E \\
	N & M
	\arrow["q", from=1-2, to=2-2]
	\arrow["f"', from=2-1, to=2-2]
	\arrow["{f^*q}"', from=1-1, to=2-1]
	\arrow["{\bar{f}}", from=1-1, to=1-2]
	\arrow["\lrcorner"{anchor=center, pos=0.125}, draw=none, from=1-1, to=2-2]
\end{tikzcd}
    \end{equation}
    \item Any retract of $q$ in the space of arrows is a vector bundle; that is, given
    \begin{equation}%
      \label{eq:ret-of-idemp}
\begin{tikzcd}
	F & E & F \\
	N & M & N
	\arrow["q"', from=1-2, to=2-2]
	\arrow["{r'}"', from=1-2, to=1-3]
	\arrow["r", from=2-2, to=2-3]
	\arrow["\pi", from=1-3, to=2-3]
	\arrow["s", from=2-1, to=2-2]
	\arrow["{s'}"', from=1-1, to=1-2]
	\arrow["\pi"', from=1-1, to=2-1]
	\arrow[curve={height=-12pt}, Rightarrow, no head, from=1-1, to=1-3]
	\arrow[curve={height=12pt}, Rightarrow, no head, from=2-1, to=2-3]
\end{tikzcd}
    \end{equation}
    if there is a vector bundle structure on $q$, then there is a vector bundle structure on $\pi$.
  \end{enumerate}
\end{proposition}

The category of vector bundles has ``locally linear'' bundle morphisms as its maps.
\begin{definition}
  A \emph{morphism of vector bundles} between $q:E \to M$ and $\pi:F \to N$ is a commuting square
\[\begin{tikzcd}
	E & F \\
	M & N
	\arrow["q", from=1-1, to=2-1]
	\arrow["f", from=1-1, to=1-2]
	\arrow["v"', from=2-1, to=2-2]
	\arrow["\pi", from=1-2, to=2-2]
\end{tikzcd}\]
  that is \emph{fibrewise linear}, so that above each fibre \[  f|_m: E_m \to F_{v(m)} \] is a linear morphism of vector spaces. This may equivalently be stated as a morphism of fibred $\R$-modules, so that the following diagrams commute:
  \[
\begin{tikzcd}
	E & F &[-2em] E & F \\
	M & N & M & N \\[-1em]
	{E_2} & {F_2} & {\R \x E} & {\R\x F} \\
	E & F & E & F
	\arrow["q"{description}, from=1-1, to=2-1]
	\arrow["f", from=1-1, to=1-2]
	\arrow["v", from=2-1, to=2-2]
	\arrow["\pi"{description}, from=1-2, to=2-2]
	\arrow["\xi"{description}, from=2-3, to=1-3]
	\arrow["\zeta"{description}, from=2-4, to=1-4]
	\arrow["f", from=1-3, to=1-4]
	\arrow["v", from=2-3, to=2-4]
	\arrow["{+}"{description}, from=3-1, to=4-1]
	\arrow["{+}"{description}, from=3-2, to=4-2]
	\arrow["f", from=4-1, to=4-2]
	\arrow["{f_2}", from=3-1, to=3-2]
	\arrow["{\cdot_E}"{description}, from=3-3, to=4-3]
	\arrow["{\R \x f}", from=3-3, to=3-4]
	\arrow["{\cdot_F}"{description}, from=3-4, to=4-4]
	\arrow["f", from=4-3, to=4-4]
\end{tikzcd}\]
\end{definition}
\begin{example}
  ~\begin{enumerate}[(i)]
    \item For the pullback vector bundle in Diagram \ref{eq:reindex-vbun}, the pair $(\bar{f},f)$ is a linear bundle morphism.
    \item For the section/retract vector bundle structure from Diagram \ref{eq:ret-of-idemp}, the section and retract are linear morphisms. 
    Note that this is exactly the splitting of a linear idempotent on $q:E \to M$.
  \end{enumerate}
\end{example}

The lift on the tangent bundle was defined in Section \ref{sec:smooth-manifolds} as
\[
  [\gamma]_\sim \mapsto [\gamma \o \cdot_\R]_\sim. 
\]
Instead, consider the action of $\R$ on a tangent vector:
\[
  ([\gamma]_\sim, r) \mapsto [\gamma \o (r \cdot x)]_\sim.
\]
Note that $T.\cdot$ gives the equation
\[
  T.\cdot \o ([\omega \o (x,y)], [(a,b) \mapsto a + b\cdot x]) = (\omega \o (a \cdot x, a \cdot b\cdot y));
\]
so the lift map $\ell$ can be rederived as follows:
\[
\begin{tikzcd}[row sep = tiny]
	TM & {TM \x \R} & {T(TM \x \R)} & {T^2M} \\
	{[\gamma]} & {([\gamma], 1)} & {([\gamma \o \pi_0], [r \mapsto 1\bullet r])} & {[\gamma \o \bullet_\R]}
	\arrow["{(id, 1^\R)}", from=1-1, to=1-2]
	\arrow["{0 \x \lambda}", from=1-2, to=1-3]
	\arrow[from=1-3, to=1-4]
	\arrow[shorten <=3pt, shorten >=3pt, maps to, from=2-1, to=2-2]
	\arrow[shorten <=3pt, shorten >=3pt, maps to, from=2-2, to=2-3]
	\arrow[shorten <=2pt, shorten >=2pt, maps to, from=2-3, to=2-4]
\end{tikzcd}\]
This general construction is known as the \emph{Euler vector field} of a multiplicative action by $\R^+$.
\begin{definition}%
  \label{def:evf}
  Consider a multiplicative monoid action $a:\R^+ \x E \to E$. The \emph{Euler vector field}\footnote{Somewhat confusingly, the Euler vector field is almost never a vector field.} of the action is the morphism $\lambda:E \to TE$ constructed as follows:
  \[
    \lambda:= E \xrightarrow[]{(id, 1^\R \o !)} E \x \R \xrightarrow{0 \x \lambda} TE \x T\R \cong T(E \x \R) \xrightarrow{T.a} TE.
  \]
  \pagenote{This definition has been tidied up so that is clear that $\lambda$ is defined by the diagram in the definition.}
\end{definition}
Local triviality for the tangent bundle is encoded by the universality of the vertical lift condition. A similar universality condition holds for vector bundles.
\begin{proposition}%
  \label{prop:ros-for-vbun}
  Let $q:E \to M$ be a vector bundle with corresponding Euler vector field $\lambda$.
  Then the following diagram is a $T$-pullback:
  \begin{equation*}
\begin{tikzcd}
	E & TE \\
	M & {E\x TM}
	\arrow["{(\xi,0)}"', from=2-1, to=2-2]
	\arrow["q", from=1-1, to=2-1]
	\arrow["\lambda", from=1-1, to=1-2]
	\arrow["{(p,T.q)}", from=1-2, to=2-2]
\end{tikzcd}
  \end{equation*}
\end{proposition}
Exploiting the fact that fibred $\R$-modules have subtraction, the following result holds.
\begin{corollary}%
  \label{cor:two-pullbacks-from-ros}
  The following two diagrams are T-equalizers:
  \[
\begin{tikzcd}
	{E_2} &[6em] TE &[2em] TM \\
	{TM \ts{p}{q}E} & TE & E
	\arrow["{ \mu^E:= 0 \o\pi_0 +_{Tq} \lambda \o \pi_1}", from=1-1, to=1-2]
	\arrow["{\nu^E := T.\xi \o \pi_0 +_p \lambda \o \pi_1}", from=2-1, to=2-2]
	\arrow["{T.q}", shift left=1, from=1-2, to=1-3]
	\arrow["{T.q \o 0 \o p}"', shift right=1, from=1-2, to=1-3]
	\arrow["p", shift left=1, from=2-2, to=2-3]
	\arrow["{p \o T.\xi \o T.q}"', shift right=1, from=2-2, to=2-3]
\end{tikzcd}
  \]
  (recall that $E_2$ is the pullback of a submersion along a submersion and is therefore guaranteed to exist and be preserved by the tangent functor).
\end{corollary}
\begin{proof}
  Given $v: X \to TE$ so that
  \[ T.q \o v = T.q \o 0 \o p \o v \]
  then 
  \[ p \o (v -_{T.q} 0 \o p \o v) =  p \o v -_q p \o v = \xi \o q \o v. \]
  So there is a unique $v'$ so that
  \[
    \lambda \o v' = (v -_{T.q} 0 \o p \o v)
  \]
  meaning that
  \[
    v = 0 \o p \o v  +_{T.q} \lambda \o v'  = \mu(0 \o p \o v, v')
  \] as required. The projection $q$ is a submersion, so the pullback $E_2$ is preserved by the tangent functor, as is the pullback in Proposition \ref{prop:ros-for-vbun}, and the same calculation may be applied for each $T^n$. The proof for $\nu$ follows by the same argument.
\end{proof}
Recall that the class of submersions forms a retractive display system in the category of smooth manifolds (Definition \ref{def:display-system}), so they are stable under reindexing and closed to retracts.\pagenote{
   I had originally used the fact that the pullback $E \ts{q}{q} E$ was preserved without explaining why it was preserved, so I have added a reference to the section on submersions and pointed out that vector bundle projections are submersions (as they are local projections).
}
We may now infer the following:
\begin{corollary}
  The projection for a vector bundle is a submersion.
\end{corollary}
\begin{proof}
  This follows from the fact that $\pi:TE \to E$ is a submersion, so that $q \o \pi_0: E \ts{q}{p} TM \to M$ is a submersion, so the map $q:E \to M$ is a retract of the projection $p.E:TE \to E$ in the arrow category.
\end{proof}
Preservation of the Euler vector field is also sufficient to guarantee that a morphism $f:E \to F$ determines a vector bundle morphism.
\begin{proposition}%
  \label{prop:evf-is-ff}
  Let $q: E \to M, \pi: F \to N$ be a pair of vector bundles with Euler vector fields $\lambda^E,\lambda^F$.
  Then a bundle morphism $(f,v):q \to \pi$ is a vector bundle morphism if and only if \[\lambda^F \o f = T.f \o \lambda^E\]
\end{proposition}
\begin{proof}
  Note that the $\nu^F$ map from Corollary \ref{cor:two-pullbacks-from-ros} is monic, and if $f$ preserves the lift, it preserves $\nu$:
  \[
    \nu^F \o (T.v, f) = + \o (T.\zeta \o T.v, \lambda^F \o f) = + \o (T.f \o T.xi, T.f \o \lambda^E) = T.f \o \nu^E. 
  \]
  Next, observe that $T.v \x f$ is the unique map making the following diagram commute:
  \[
\begin{tikzcd}
	M &&& E \\
	& {TM \ts{p}{q} E} & TE \\
	& {TN \ts{p}{q} F} & TF \\
	N &&& F
	\arrow["{\exists!}"', dashed, from=2-2, to=3-2]
	\arrow["\nu", from=2-2, to=2-3]
	\arrow["\nu"', from=3-2, to=3-3]
	\arrow["{T.f}", from=2-3, to=3-3]
	\arrow["p"{description}, from=2-3, to=1-4]
	\arrow["p"{description}, from=3-3, to=4-4]
	\arrow["f"{description}, from=1-4, to=4-4]
	\arrow["{q \o \pi_1}"{description}, from=2-2, to=1-1]
	\arrow["{\pi \o \pi_1}"{description}, from=3-2, to=4-1]
	\arrow["\zeta"{description}, from=4-1, to=4-4]
	\arrow["\xi"{description}, from=1-1, to=1-4]
	\arrow[from=1-1, to=4-1]
	\arrow["\lrcorner"{anchor=center, pos=0.125, rotate=90}, draw=none, from=2-2, to=1-4]
	\arrow["\lrcorner"{anchor=center, pos=0.125}, draw=none, from=3-2, to=4-4]
\end{tikzcd}\] 
  Now $\nu$ is a vector bundle morphism and monic, and $T.f$ is a vector bundle morphism, so it follows that $T.v \x f$ is a vector bundle morphism and hence $f$ is also a vector bundle morphism. The reverse implication is immediate.
\end{proof}

\section{Lifts for the tangent weak comonad}%
\label{sec:lifts}
The lift $\ell: T \Rightarrow T^2$ gives rise to a \emph{weak comonad}.
Weak comonads were introduced in \cite{Wisbauer2013} and have a natural notion of an associative algebra\footnote{
This is not \emph{strictly} true. Wisbauer has a more nuanced hierarchy of almost-monads, and in his language $(T,\ell)$ would be an endofunctor with an associative product.
}. An associative coalgebra of the weak comonad $(T,\ell)$ is called a \emph{lift}; this chapter will demonstrate that lifts provide the essential structure necessary to formulate vector bundles.\pagenote{
Included some material on weak comonads to clear up confusion, as the map $p$ does have the right type for a counit it should be noted that it is not a counit in general.
}
\begin{definition}
  A weak comonad on a category $\C$ is an endofunctor $S: \C \to \C$ equipped with a coassociative map $\delta:S \Rightarrow S.S$:
  \[\begin{tikzcd}
    S & {S.S} \\
    {S.S} & {S.S.S}
    \arrow["\delta"{description}, from=1-1, to=1-2]
    \arrow["\delta"{description}, from=1-1, to=2-1]
    \arrow["{S.\delta}"{description}, from=1-2, to=2-2]
    \arrow["{\delta.S}"{description}, from=2-1, to=2-2]
  \end{tikzcd}\]
  An \emph{associative algebra} of a weak comonad is an object $E$ equipped with a map $\lambda:E \to SE$ so that
  \[\begin{tikzcd}
    E & {S.E} \\
    {S.E} & {S.S.E}
    \arrow["\lambda"{description}, from=1-1, to=1-2]
    \arrow["\lambda"{description}, from=1-1, to=2-1]
    \arrow["{\delta.E}"{description}, from=1-2, to=2-2]
    \arrow["{S.\lambda}"{description}, from=2-1, to=2-2]
  \end{tikzcd}\]
  A morphism of these algebras is a map $f:(E,\lambda) \Rightarrow (D,\gamma)$ so that
  \[\begin{tikzcd}
    E & D \\
    S.E & S.D
    \arrow["f"{description}, from=1-1, to=1-2]
    \arrow["{S.f}"{description}, from=2-1, to=2-2]
    \arrow["\lambda"{description}, from=1-1, to=2-1]
    \arrow["\gamma"{description}, from=1-2, to=2-2]
  \end{tikzcd}\]
\end{definition}
Recall that for a full (co)monad, there is an adjunction between the base category and the category of (co)algebras. That result is weakened in this case: 
\begin{lemma}
  For every weak comonad on a category $\C$, there is a free coalgebra functor
  \[
      F: \C \to \mathsf{CoAlg}(\C); E \mapsto (S.E, \delta:S.E \to S.S.E)
  \]  
  an underlying object functor
  \[
      U: \mathsf{CoAlg}(\C) \to \C; (E, \lambda) \mapsto E
  \]
  and a natural transformation
  \[
      \lambda:id \Rightarrow F.U; 
      \begin{tikzcd}
        E & S.E \\
        S.E & S.S.E
        \arrow["\lambda"{description}, from=1-1, to=1-2]
        \arrow["{S.\lambda}"{description}, from=2-1, to=2-2]
        \arrow["\lambda"{description}, from=1-1, to=2-1]
        \arrow["\delta"{description}, from=1-2, to=2-2]
      \end{tikzcd}
  \]
\end{lemma}
\begin{definition}\label{def:lift}
  A \emph{lift} in a tangent category $\C$ is an associative coalgebra of $(T,\ell)$, namely, a pair $(E, \lambda: E \to TE)$ so that the following diagram commutes:
  \[
\begin{tikzcd}
	E & TE \\
	TE & {T^2E}
	\arrow["\lambda", from=1-1, to=1-2]
	\arrow["\lambda"', from=1-1, to=2-1]
	\arrow["{\ell.E}"', from=2-1, to=2-2]
	\arrow["{T.\lambda}", from=1-2, to=2-2]
\end{tikzcd}
  \]
  A morphism of lifts is a coalgebra morphism. The category of lifts and lift morphisms in a tangent category $\C$ is written $\mathsf{Lift}(\C)$.
\end{definition}
Note that the tangent bundle is \emph{not}, in general, a comonad: while the tangent projection has the correct type for a counit, $p:T \Rightarrow id$, it does not satisfy $p \o \ell = id$. In fact, if $p$ were a counit, this would force $id = p \o \ell = 0 \o p$ so that $0$ = $p^{-1}$, thus if $(T,\ell,p)$ is a comonad then $T$ is naturally isomorphic to the identity functor.
\begin{example}
  \label{ex:lift-examples}
  ~\begin{enumerate}[(i)]
    \item For every object $M$ in a tangent category $\C$, the pair $(TM, \ell:TM \to T^2M)$ is a lift, called the \emph{free lift} on $M$.
    \item Every object $M$ in a tangent category has a \emph{trivial} lift, $0:M \to TM$, where $T.0 \o 0 = \ell \o 0.$
    \item Every differential object has a lift $\lambda$; the coherence is equivalent to axiom $[D0.3]$ in Definition \ref{def:differential-object}. \pagenote{
        I have added the map $\lambda:E \to TE$ from the definition of a differential object, and the exact differential object axiom that makes it a lift.
    }
    \item 
    \pagenote{I moved the Euler vector field into the examples, fixed notation, and tried to clarify some points Kristine brought up. I make explicit use of the Kock-Lawvere axiom here, so that $T\R$ is $R[x]/x^2$ and $T^2\R$ is $R[x]/(x^2,y^2)$.}
    The Euler vector field of a multiplicative $\R^+$-action $h:R^+ \x E \to E$ in $\mathsf{SMan}$ is a lift.
      Recall that the \textit{Euler vector field} over the scalar action $s_M: TM \x \R \to \R$ induces the vertical lift on a manifold:
      \[
          \ell = T.s_M \o (0, \lambda^\R \o 1^{\R} \o !).
      \]
      In the category of smooth manifolds, $T\R \cong \R[x]/x^2$, where $\lambda(r) = [x \mapsto r\cdot x]$ corresponds to the map $\lambda'(r) = 0 + r\cdot x$. Similarly, there is an isomorphism $\R[x,y]/(x^2,y^2)$ so that for the maps $0.T and T.0$,
      \[
          0.T(a + r\cdot x) \cong a + r\cdot x + 0y + 0xy,\hspace{0.5cm}
          T.0(a + r\cdot x) \cong a + 0x + r\cdot y + 0xy.
      \]
      Since we know that $(0 + x)(0 + y) = (0 + xy)$ in $\R[x,y]/(x^2,y^2)$ (following \ref{ex:diffob-sman}), we can use these isomorphisms to see that
      \[
          \ell \o \lambda^\R \o 1^\R \o != 
          (0.T \o \lambda^\R \o 1^\R \o !) \cdot_{T^2.\R} (T.0\o \lambda^\R\o 1^\R \o !).
      \]
  
      Consider a monoid action $(\R^+, h)$ on a manifold $E$. The Euler vector field of this action, 
      \[
          \lambda: E \xrightarrow[]{(id, 1^\R \o !)} E \x \R \xrightarrow{(0,\lambda^\R)} TE \x T\R \xrightarrow{T.h} TE
      \]
      will define an algebra if the induced scalar action on $TE$ commutes with the natural scalar action:
      \begin{align*}
          T.\lambda \o \lambda
          &= T^2.h \o (T.0 \o\lambda, T.\lambda \o 0\o 1^\R \o !) \\
          &= T^2.h \o (T.0 \o T.h \o (0, \lambda \o 1^\R \o !), 0 \o \lambda \o 1^\R \o !) \\
          &= T^2.h \o (T^2.h\o (T0\o 0, T0 \o \lambda \o 1^\R \o !), 0 \o \lambda \o 1^\R \o !) \\
          &= T^2.h \o (T.0 \o 0, (T.0 \o \lambda \o 1^\R \o !) \cdot_{T^2.\R}( 0.T \o  \lambda \o 1^\R \o !) ) \\
          &= T^2.h \o (\ell \o 0, \ell \o \lambda \o 1^\R \o !)  \\
          &= \ell \o T.h \o  (0, \lambda \o 1^\R \o !) \\
          &= \ell \o \lambda.
      \end{align*}
  \end{enumerate}
\end{example}
\begin{observation}\label{obs:evf-is-ff}
  Recall that by Proposition \ref{prop:evf-is-ff}, morphisms preserve a monoid action if and only if they preserve the associated Euler vector field of the action. This means the Euler vector field construction gives a fully faithful functor from monoid actions to lifts in the category of smooth manifolds, and therefore from the category of vector bundles to the category of lifts in $\mathsf{SMan}$.
\end{observation}

The following proposition gives a pair of constructions on lifts---closure under the tangent functor and finite $T$-limits---that will be useful in this section.
\begin{lemma}%
  \label{lem:T-limits-of-lifts}
  Let $\C$ be a tangent category.
  ~\begin{enumerate}[(i)]
    \item The tangent functor lifts to an endofunctor on the category of lifts in $\C$.
    \item Given a diagram \pagenote{
       I have clarified the statement and make the proof more concrete, I have also moved this result to a lemma and made the next result a proposition rather than corollary.
    }
    \[ D: \d \to \mathsf{Lift}(\C)\]
    in the category of lifts of $\C$, if the $T$-limit of $U.D$ exists in $\C$, then $\lim U.D$ has a natural lift $\lambda'$ associated to it so that $(\lim U.D, \lambda')$ is the limit of $D$ in $\mathsf{Lifts}(\C)$. (That is, $T$-limits of lifts are computed pointwise in the base category.)
    \pagenote{
       The proof for the 
    }
  \end{enumerate}
\end{lemma}
\begin{proof}
  ~\begin{enumerate}[(i)]
    \item Simply check that 
    \begin{gather*}
      T.(c \o T.\lambda) \o c \o T.\lambda = T.c \o T^2.\lambda \o c \o T.\lambda  = T.c \o c.T \o T^2.\lambda  \\= T.c \o c.T \o T.\ell \o T.\lambda = \ell.T \o c \o T.\lambda. 
    \end{gather*}
    \item 
    Concretely, a tangent terminal object will have a lift:
    \[
        (1, 1 \xrightarrow[]{0} T.1 \cong 1).
    \] 
    Given $(E,\lambda)$ and $(F,l)$, if the tangent product $E \x F$ exists there is a lift
    \[
      (E \x F, E \x F \xrightarrow[]{\lambda \x l} TE \x TF \cong T(E \x F)).
    \]
    Given the $T$-equalizer of a fork $f,g:(E,\lambda) \to (F,l)$, the equalizer has a lift induced as follows:
    \[
\begin{tikzcd}
	C & E & F \\
	{T.C} & {T.E} & {T.F}
	\arrow["f", shift left=1, from=1-2, to=1-3]
	\arrow["g"', shift right=1, from=1-2, to=1-3]
	\arrow["k", dashed, from=1-1, to=1-2]
	\arrow["\lambda", dotted, from=1-1, to=2-1]
	\arrow["k", dashed, from=2-1, to=2-2]
	\arrow["\lambda", from=1-2, to=2-2]
	\arrow["{T.f}", shift left=1, from=2-2, to=2-3]
	\arrow["{T.g}"', shift right=1, from=2-2, to=2-3]
	\arrow["l", from=1-3, to=2-3]
\end{tikzcd}
    \]
  \end{enumerate}
\end{proof}
\begin{proposition}\label{prop:lifts-is-tangent}
  The category of lifts is a tangent category.
\end{proposition}
\begin{proof}
  The tangent functor sends 
  \[
    \infer{T.f: (TE, c \o T.\lambda) \to (TF, c \o T.l)}{f:(E,\lambda) \to (F,l)}.
  \]
  To see that this is still an algebra morphism, compute
  \[
    c \o T.l \o T.f = c \o T^2.f \o T.\lambda = T^2.f \o c \o T.\lambda
  \]
  The structure maps are the structure maps on the underlying object of the lift; the universality conditions follow by Proposition \ref{lem:T-limits-of-lifts}.
\end{proof}


The following idempotent is key in the theory of lifts and will be used in defining non-singular lifts (Definition \ref{def:non-singular-lift}), and its splitting will present the projection and zero-section of a vector bundle (Definition \ref{def:pdb}).

\begin{proposition}%
  \label{prop:idempotent-natural}
  The category of lifts in a tangent category $\C$ has a natural idempotent:
  \[
      e: id \Rightarrow id; e_{(E,\lambda)}: (E,\lambda) \xrightarrow[]{p \o \lambda} (E,\lambda).
  \]\pagenote{
      This proposition originally included some ambiguities, these have been handled by earlier changes that make the notiong of lift and lift-morphism more concrete. 
  }
\end{proposition}
\begin{proof}
  First, we see that $e = p \o \lambda$ is an idempotent:
  \[ p \o \lambda \o p \o \lambda
  = p\o p.T \o T.\lambda \o \lambda \\
  = p\o p.T \o\ell\o \lambda \\
  = p \o 0 \o p\o \lambda \\
  = p \o \lambda.\]
  Moreover, every $f:(E,\lambda) \to (F,l)$ preserves the idempotent:
  \[
    f \o p \o \lambda = p \o T.f \o \lambda = p \o l \o f.
  \]
  Finally, note that the idempotent is a lift morphism.:
  \[
    T.\lambda \o \lambda = \ell \o \lambda = c \o \ell \o \lambda = c \o T.\lambda \o \lambda
  \]
  which implies that 
  \[
    \lambda \o e = \lambda \o p \o \lambda = p \o T.\lambda \o \lambda = p \o c \o T.\lambda \o \lambda = T.p \o T.\lambda \o \lambda = T.e \o \lambda.
  \]
\end{proof}


\section{Non-singular lifts}%
\label{sec:non-singular-lifts}
\cite{Grabowski2009} introduced the notion of a non-singular lift as a means to axiomatize the Euler vector field of a vector bundle's multiplicative $\R^+$-action. While it is not immediately clear that our definition is the same as Grabowski's, the results of Section \ref{sec:iso-vbun-dbun} will justify the use of this language as they are necessarily the same.
\begin{definition}%
  \label{def:non-singular-lift}
  A lift $(E,\lambda)$ in a tangent category $\C$ is \emph{non-singular} whenever the following diagram is a $T$-equalizer:
  \[
\begin{tikzcd}
	E & TE & TE
	\arrow["\lambda", from=1-1, to=1-2]
	\arrow["{T.e}", shift left=1, from=1-2, to=1-3]
	\arrow["{e.E}"', shift right=1, from=1-2, to=1-3]
\end{tikzcd}
  \]
  where $e.E = p \o \ell$ (the idempotent associated to the free lift on $E$) and $T.e = T.p \o T.\lambda$ (the image of the idempotent associated to $(E,\lambda)$ under the tangent functor). The category of non-singular lifts is written $\mathsf{NonSing}(\C)$.\pagenote{
  I have added comments clarifying what maps $e.E$ and $T.e$ are. 
  }
\end{definition}
The most prominent class of examples is given by the Euler vector field of the $\R$-action on a vector bundle.
\begin{proposition}%
  \label{prop:evf-vbun-is-nonsingular}
  The Euler vector field of a vector bundle is a non-singular lift.
\end{proposition}
\begin{proof}
  \pagenote{I made it clear that the diagram being discussed is a $T$-limit, as it is not actually an equalizer diagram. I have tried to clarify the wording that ``one diagram is universal if and only if the other is''.}
  Let $(q:E \to M, +, \xi, \cdot)$ be a vector bundle with Euler vector field $\lambda$. By Proposition \ref{prop:ros-for-vbun}, the diagram \[
\begin{tikzcd}
	&& TM \\
	E & TE \\
	&& E
	\arrow["\lambda", from=2-1, to=2-2]
	\arrow["{T.q}", shift left=1, from=2-2, to=1-3]
	\arrow["{p \o T.\xi \o T.q}"{pos=1}, shift left=1, from=2-2, to=3-3]
	\arrow["p"', shift right=1, from=2-2, to=3-3]
	\arrow["{T.q \o 0 \o p}"'{pos=1}, shift right=1, from=2-2, to=1-3]
\end{tikzcd}\] is a $T$-limit. The $T$-universality of this diagram will hold if and only if the diagram is universal after each parallel pair of arrows is post-composed by a $T$-monic. A section is a $T$-monic, so the previous diagram is $T$-universal if and only if the following diagram is $T$-universal:
  \[
  \begin{tikzcd}
    && E & TE \\
    E & TE \\
    && TM & TE
    \arrow["p"', shift right=2, from=2-2, to=1-3]
    \arrow["{\xi\o q \o p}", shift left=2, from=2-2, to=1-3]
    \arrow["{T.q}", shift left=2, from=2-2, to=3-3]
    \arrow["{0 \o q \o p}"', shift right=2, from=2-2, to=3-3]
    \arrow["\lambda", from=2-1, to=2-2]
    \arrow["0", from=1-3, to=1-4]
    \arrow["{T.\xi}", from=3-3, to=3-4]
  \end{tikzcd}
\]
  Now simplify this diagram using the fact that $T.(\xi \o q) = T.e, 0 \o p = e.E$:
\[\begin{tikzcd}
  && TE \\
  E & TE \\
  && TE
  \arrow["{e.T}"', shift right=2, from=2-2, to=1-3]
  \arrow["{e.E \o T.e}", shift left=2, from=2-2, to=1-3]
  \arrow["{T.e}", shift left=2, from=2-2, to=3-3]
  \arrow["{e.E \o T.e}"', shift right=2, from=2-2, to=3-3]
  \arrow["\lambda", from=2-1, to=2-2]
\end{tikzcd}\]
  Note that the the two pairs of parallel arrows have a common arrow, implying that they may be pulled together into a single ternary equalizer. All that remains to check, then, is that for any $x:X \to TE$,
  \[
  (e.e \o x = T.e \o x = e.E \o x) \iff 
  (T.e \o x = e.E \o x). 
  \]
  The forward implication is trivial, so it remains to prove the reverse. 
  Suppose $T.e \o x = e.E \o x$; then
  \[
    e.e \o x = e.E \o T.e \o x = e.E \o e.E \o x = e.E \o x
  \]
  giving the result, namely that the diagram \[
\begin{tikzcd}
	E & TE & TE
	\arrow["\lambda", from=1-1, to=1-2]
	\arrow["{T.e}", shift left=1, from=1-2, to=1-3]
	\arrow["{e.E}"', shift right=1, from=1-2, to=1-3]
\end{tikzcd}\] is a $T$-equalizer.
\end{proof}
Every map $f:E \to F$ gives a map of free coalgebras $T.f: (TE,\ell) \to (TF,\ell)$ and the idempotent $e$ is a coalgebra morphism by Proposition \ref{prop:idempotent-natural}, so the following is immediate:
\begin{proposition}%
  \label{def:non-singular-lift-eq-of-lifts}
  A non-singular lift is an equalizer in the category of lifts:
  \[
\begin{tikzcd}
	{(E,\lambda)} & {(TE,\ell)} & {(TE,\ell)}
	\arrow["\lambda", from=1-1, to=1-2]
	\arrow["{T.e}", shift left=1, from=1-2, to=1-3]
	\arrow["{e.E}"', shift right=1, from=1-2, to=1-3]
\end{tikzcd}
  \]
\end{proposition}

Observe that the category of non-singular lifts is closed under finite limits in the category of lifts.
\begin{proposition}%
  \label{prop:nonsing-closed-under-t-limits}
  The category of non-singular lifts in $\C$ is closed under $T$-limits:
  \begin{enumerate}[(i)]
    \item The tangent functor on lifts preserves non-singular lifts, so that if $(E,\lambda)$ is non-singular then $(TE,c\o T.\lambda)$ is non-singular. \pagenote{
        I clarified that this was in fact referring to the tangent functor \emph{for the category of lifts}.
    }
    \item The trivial lift on an object, $0:M \to TM$, is non-singular. 
    \item $T$-products of non-singular lifts are non-singular lifts.
    \item $T$-equalizers of non-singular lifts are non-singular lifts.
  \end{enumerate}
\end{proposition}
\begin{proof}
  ~\begin{enumerate}[(i)]
    \item This follows from the fact that the non-singularity condition is a $T$-limit.
    \item The zero map splits the idempotent $0 \o p$, so it is the equalizer of $0 \o p, p \o 0 = id$.
    \item This follows by stability of limits under products.
    \item The following diagram commutes by naturality:
    \[
\begin{tikzcd}
    C & E & F \\
    {T.C} & {T.E} & {T.F} \\
    \\
    {T.C} & {T.E} & {T.F}
    \arrow["\lambda", from=1-2, to=2-2]
    \arrow["\lambda", from=1-3, to=2-3]
    \arrow["f", shift left=1, from=1-2, to=1-3]
    \arrow["{T.f}", shift left=1, from=2-2, to=2-3]
    \arrow["{T.g}"', shift right=1, from=2-2, to=2-3]
    \arrow["{e.E}"', shift right=1, from=2-2, to=4-2]
    \arrow["{T.e}", shift left=1, from=2-2, to=4-2]
    \arrow["{e.F}"', shift right=1, from=2-3, to=4-3]
    \arrow["{T.e}", shift left=1, from=2-3, to=4-3]
    \arrow["{T.f}", shift left=1, from=4-2, to=4-3]
    \arrow["{T.g}"', shift right=1, from=4-2, to=4-3]
    \arrow["g"', shift right=1, from=1-2, to=1-3]
    \arrow[dashed, from=1-1, to=1-2]
    \arrow["\lambda", from=1-1, to=2-1]
    \arrow["{e.C}"', shift right=1, from=2-1, to=4-1]
    \arrow["{T.e}", shift left=1, from=2-1, to=4-1]
    \arrow[dashed, from=2-1, to=2-2]
    \arrow[dashed, from=4-1, to=4-2]
\end{tikzcd}\]
    Each horizontal diagram is a $T$-equalizer, and the two columns on the right are $T$-equalizers, so the column on the left is a $T$-equalizer.
  \end{enumerate}
\end{proof}

Finally, when a tangent category has certain $T$-equalizers, there is an idempotent monad on the category of lifts, whose algebras are non-singular lifts:
\begin{theorem}%
  \label{thm:idemp-monad-nonsingular}
  Let $\C$ be a tangent category with chosen $T$-equalizers of idempotents. Then the following equalizer determines a left-exact idempotent monad on the category of lifts, whose algebras are non-singular lifts.\pagenote{
      I have reworded this proposition to make it clear that the functor is defined by sending a lift to the lift defined by the equalizer in the following diagram.
  }
\[\begin{tikzcd}
	{(F,l)} & {(T.E,\ell)} & {(T.E, \ell)}
	\arrow[from=1-1, to=1-2]
	\arrow["{e.E}", shift left=1, from=1-2, to=1-3]
	\arrow["{T.e}"', shift right=1, from=1-2, to=1-3]
\end{tikzcd}\] 
\end{theorem}
\begin{proof}
  First, take the equalizer in $\mathsf{Lift}(\C)$; the functor sends $(E,\lambda)$ to the chosen limit $(F,l)$.The unit of the monad is the unique morphism from $(E,\lambda)$ to the equalizer $(F,l)$ induced by universality:
  \[
\begin{tikzcd}
	{(F,l)} & {(TE,\ell)} & {(TE,\ell)} \\
	{(E,\lambda)}
	\arrow["\lambda"', from=2-1, to=1-2]
	\arrow["{T.e}", shift left=1, from=1-2, to=1-3]
	\arrow["{e.E}"', shift right=1, from=1-2, to=1-3]
	\arrow[from=1-1, to=1-2]
	\arrow["{\exists!}", dashed, from=2-1, to=1-1]
\end{tikzcd}\]
  Note that non-singular lifts are closed under finite limits, so $(F,l)$ is a non-singular lift. If $(E,\lambda)$ is a nonsingular lift then $\lambda: (E,\lambda) \to (TE, \ell)$ equalizes the diagram, so there is a unique isomorphism $(F,l) \cong (E,\lambda)$, making the multiplication of the monad a natural isomorphism (and thus yielding an idempotent monad).
  The functor is defined as a $T$-limit and therefore preserves all $T$-limits of lifts, so it is left-exact.
\end{proof}
In the category of smooth manifolds, $\ell$ is the Euler vector field of an $\R^+$-action, this guarantees that every $\lambda$ is the Euler vector field of a multiplicative $\R^+$-action.  We note the following corollary.
\begin{corollary}%
  \label{cor:non-singular-implies-ract}
  In the category of multiplicative $\mathbb{R}^+$ actions in $\mathsf{SMan}$, multiplication by $0$ is equivalent to the natural idempotent $e$ in the fully faithful functor sending an $\mathbb{R}^+$-action to its Euler vector field. By non-singularity, the following diagram is an equalizer, giving $E$ a multiplicative action by $\mathbb{R}^+$ whose Euler vector field is $\lambda$: 
  \[
\begin{tikzcd}
	{(E, \cdot_E)} & {(TE, \cdot_T)} & {(TE, \cdot_T)}
	\arrow["{T.e}", shift left=1, from=1-2, to=1-3]
	\arrow["{e.T}"', shift right=1, from=1-2, to=1-3]
	\arrow["\lambda", from=1-1, to=1-2]
\end{tikzcd}
  \]
  Moreover, $\lambda$ is the Euler vector field of this lift, and the following diagram commutes:
  \[
\begin{tikzcd}
	E & TE & TE \\
	{\R^+ \x E} & {\R^+ \x TE} & {\R^+ \x TE} \\
	{T(\R^+ \x E)} & {T(\R^+ \x TE)} & {T(\R^+ \x TE)} \\
	TE & TTE & TTE
	\arrow["{T(\R^+ \x e.E)}", shift left=1, from=3-2, to=3-3]
	\arrow["{T(\R^+\x T.e)}"', shift right=1, from=3-2, to=3-3]
	\arrow["{\R^+ \x \lambda}", from=3-1, to=3-2]
	\arrow["{\cdot_p}", from=3-3, to=4-3]
	\arrow["{T\cdot_p}"', from=3-2, to=4-2]
	\arrow["{e.E}", shift left=1, from=4-2, to=4-3]
	\arrow["{T.e}"', shift right=1, from=4-2, to=4-3]
	\arrow["T\lambda"', from=4-1, to=4-2]
	\arrow["{T\cdot_E}"', dashed, from=3-1, to=4-1]
	\arrow["{(1^\R,id)}"', from=1-1, to=2-1]
	\arrow["{(1^\R,id)}"{description}, from=1-2, to=2-2]
	\arrow["{(1^\R,id)}"{description}, from=1-3, to=2-3]
	\arrow["{\lambda^\R \x 0}"{description}, from=2-2, to=3-2]
	\arrow["{\lambda^\R \x 0}"{description}, from=2-3, to=3-3]
	\arrow["{\lambda^\R \x 0}"{description}, from=2-1, to=3-1]
	\arrow["{e.E}", shift left=1, from=1-2, to=1-3]
	\arrow["{T.e}"', shift right=1, from=1-2, to=1-3]
	\arrow["{\R^+ \x e.E}", shift left=1, from=2-2, to=2-3]
	\arrow["{\R^+ \x T.e}"', shift right=1, from=2-2, to=2-3]
	\arrow["\lambda", from=1-1, to=1-2]
	\arrow["{\R^+ \x \lambda}", from=2-1, to=2-2]
\end{tikzcd}
  \]
\end{corollary}

\section{Differential bundles}%
\label{sec:lifts-pdbs-dbs}

This section introduces (pre-)differential bundles, which provided the rest of the data for a vector bundle: namely the projection, the zero section, and the addition map. The zero section and projection data arise by splitting the natural idempotent $e: id \Rightarrow id$, and non-singularity will induce the addition map. Every differential bundle satisfies a pair of universality diagrams, linking this presentation of differential bundles to the original definition in \cite{Cockett2018}.

\begin{definition}%
  \label{def:pdb}
  ~\begin{enumerate}[(i)]
    \item A \emph{pre-differential bundle} is a lift $\lambda:E \to TE$ equipped with a chosen splitting of the natural idempotent $e = p \o \lambda$ from Proposition \ref{prop:idempotent-natural}. Pre-differential bundles are formally written $(q:E \to M, \xi, \lambda)$, where $q:E \to M$ is the retract, $\xi:M \to E$ the section, and $\lambda:E \to TE$ the lift (the types are only necessary for the projection as the rest may be inferred, and will generally be suppressed to save space).
    \item A \emph{differential bundle} is a pre-differential bundle $(q:E \to M, \xi, \lambda)$ with the properties that $\lambda$ is non-singular and $T$-pullback powers of $q$ exist.
  \end{enumerate}
  Morphisms of (pre-)differential bundles are exactly morphisms of their underlying lifts. The categories of (pre-)differential bundles are exactly (pre-)differential bundles and lift morphisms, and are written $\mathsf{Pre}(\C), \mathsf{DBun}(\C)$ respectively.\pagenote{
      The definition of differential bundles has been expanded so that it is clear that $\xi,q$ are and idempotent splitting of $p \o \lambda$, and the category names $\mathsf{Pre}(\C), \mathsf{DBun}(\C)$ are defined explicitly to help draw attention to the definition of bundle morphisms.
  }
\end{definition}

Recall that as $e$ is a natural idempotent in the category of lifts, any lift morphism will preserve $e$ and will consequently preserve its idempotent splitting. Preserving the idempotent means that every differential bundle morphism is a bundle morphism, where the base map is given by \[
\begin{tikzcd}
	E & F \\
	M & N
	\arrow["q", from=1-1, to=2-1]
	\arrow["f", from=1-1, to=1-2]
	\arrow["\pi", from=1-2, to=2-2]
	\arrow["{m:= \pi\o f \o \xi}"', dashed, from=2-1, to=2-2]
\end{tikzcd}\]
We now look at the limits of (pre-)differential bundles.
\begin{observation}%
  \label{obs:T-limits-pdbs}
  ~\begin{enumerate}[(i)]
    \item The limit for a diagram of pre-differential bundles is the limit of the underlying lifts equipped with a chosen splitting of $p \o \lambda$ due to basic properties about idempotent splittings.
    \item The limit for a diagram of differential bundles is the limit in the category of pre-differential bundles (the lift will be universal by Proposition \ref{prop:nonsing-closed-under-t-limits}), so long as $T$-pullback powers of the resulting projection exist.
  \end{enumerate}
\end{observation}
Because there is a projection associated to a (pre-)differential bundle, the $T$-reindexing operation described in Proposition \ref{prop:submersion-properties} can now be applied. This gives a pullback differential bundle in a similar way to Proposition \ref*{prop:retracts-reindexing-of-vbuns}.
\begin{lemma}[\cite{Cockett2018}]%
  \label{lem:reindex-db}
  Let $(q:E \to M, \xi, \lambda)$ be a (pre-) differential bundle in a tangent category $\C$, and consider a $T$-pullback in $\C$:
  \[
\begin{tikzcd}
	{u^*E} & E & {} \\
	N & M
	\arrow["u"', from=2-1, to=2-2]
	\arrow["q", from=1-2, to=2-2]
	\arrow["{u^*q}"', from=1-1, to=2-1]
	\arrow["{\bar{u}}", from=1-1, to=1-2]
	\arrow["\lrcorner"{anchor=center, pos=0.125}, draw=none, from=1-1, to=2-2]
\end{tikzcd}\]
  Then the induced triple maps
  \[
\begin{tikzcd}
	{T.u^*E} &&& TE & N \\
	& {u^*E} & E &&& {u^*E} & E \\
	& N & M &&& N & M \\
	TN &&& TM
	\arrow["u"', from=3-2, to=3-3]
	\arrow["q", from=2-3, to=3-3]
	\arrow["{u^*q}"', from=2-2, to=3-2]
	\arrow["{\bar{u}}", from=2-2, to=2-3]
	\arrow["\lrcorner"{anchor=center, pos=0.125}, draw=none, from=2-2, to=3-3]
	\arrow["{T.\bar u}", from=1-1, to=1-4]
	\arrow["\lambda"{description}, from=2-3, to=1-4]
	\arrow["0"{description}, from=3-3, to=4-4]
	\arrow["{T.q}", from=1-4, to=4-4]
	\arrow["{T.u}"', from=4-1, to=4-4]
	\arrow["0"{description}, from=3-2, to=4-1]
	\arrow["{T.u^*q}"', from=1-1, to=4-1]
	\arrow["{u^*\lambda}"{description}, dashed, from=2-2, to=1-1]
	\arrow["q", from=2-7, to=3-7]
	\arrow["u"', from=3-6, to=3-7]
	\arrow[from=2-6, to=3-6]
	\arrow[from=2-6, to=2-7]
	\arrow["\lrcorner"{anchor=center, pos=0.125}, draw=none, from=2-6, to=3-7]
	\arrow[curve={height=6pt}, Rightarrow, no head, from=1-5, to=3-6]
	\arrow["{\xi \o u}"{description}, curve={height=-12pt}, from=1-5, to=2-7]
	\arrow["{u^*\xi}"{description}, dashed, from=1-5, to=2-6]
\end{tikzcd}\]
  induce a (pre-)differential bundle $(u^*q: u^*E \to N, u^*\xi, u^*\lambda)$. If $\lambda$ is non-singular and $T$-pullback powers of $u^*q$ exist, then $(u^*q, u^*\xi, u^*\lambda)$ is a differential bundle.
\end{lemma}
\begin{proof}
  Note that a pre-differential bundle $(q:E \to M, \xi, \lambda)$ in $\C$ may also be regarded as a pre-differential bundle $(q: (E,\lambda) \to (M, 0), \xi, \lambda)$ in $\mathsf{Lift}(\C)$. Take the following pullback in $\mathsf{Lift}(\C)$:
  \[
\begin{tikzcd}
	{(u^*E, u^*\lambda)} & {(E,\lambda)} \\
	{(N, 0)} & {(M,0)}
	\arrow["u", from=2-1, to=2-2]
	\arrow["q"', from=1-2, to=2-2]
	\arrow[from=1-1, to=2-1]
	\arrow[from=1-1, to=1-2]
	\arrow["\lrcorner"{anchor=center, pos=0.125}, draw=none, from=1-1, to=2-2]
\end{tikzcd}\]
  It follows by construction that $\iota \o u^*\lambda = \lambda \o \iota$. Thus the result holds.
\end{proof}

\begin{proposition}%
  \label{prop:induce-abun}
  Let $(q:E \to M, \xi, \lambda)$ be a non-singular pre-differential bundle in a tangent category $\C$, so that $T$-pullback powers of $q$ exist.
  Then there is an additive bundle structure $(q,\xi,+_q)$ so that the differential bundle morphisms are additive:
  \[
    (\lambda,\xi): (q,\xi,+_q) \to (p,0,+) \hspace{0.5cm}
    (\lambda,0): (q,\xi,+_q) \to (T.q,T.\xi,T.+_q).
  \]
  Furthermore, every differential bundle morphism preserves addition.
\end{proposition}
\begin{proof}
  Non-singularity forces the existence of an addition map: 
  \[
\begin{tikzcd}
	{E_2} & {T_2E} & {T_2E} \\
	E & TE & TE
	\arrow["{\exists! +_q}"{description}, dashed, from=1-1, to=2-1]
	\arrow["{+}"{description}, from=1-2, to=2-2]
	\arrow["{\lambda \x \lambda}", from=1-1, to=1-2]
	\arrow["\lambda"', from=2-1, to=2-2]
	\arrow["{e.E}", shift left=1, from=2-2, to=2-3]
	\arrow["{T.e}"', shift right=1, from=2-2, to=2-3]
	\arrow["{+}"{description}, from=1-3, to=2-3]
	\arrow["{(e \x e).E}", shift left=1, from=1-2, to=1-3]
	\arrow["{T_2.e}"', shift right=1, from=1-2, to=1-3]
\end{tikzcd}
  \]
  Note that this diagram commutes because $T_2E$ is a pre-differential bundle whose lift is $\ell \x \ell$, and also $+$ is a linear morphism, so it commutes with the addition map $e\o a + e\o b = e \o (a\x b)$. Post-composition with $\lambda$ ensures that $\xi$ is the unit and that associativity holds. A differential bundle morphism will induce a morphism of the equalizer diagrams that induce each addition map to preserve addition.
\end{proof}

Recall that by Proposition \ref{prop:evf-vbun-is-nonsingular}, the Euler vector field for every vector bundle is a non-singular lift.\pagenote{
    The reference to Proposition \ref{prop:evf-vbun-is-nonsingular} has been put into the middle of the sentence to help. Also, rather than ``addition induced by singularity'', we give an explicit reference to the above proposition that constructs the additive bundle structure.
} If $TM \ts{q}{q} E$ exists, the map
\[
  \nu^E: TM \ts{p}{q} E \xrightarrow[]{T.\xi \x \lambda} T_2E \xrightarrow[]{+.E} TE
\] may be formed. Similarly, using the additive bundle structure from Proposition \ref{prop:induce-abun}, the $\mu$ map may be formed: 
\[
  \mu^E: E \ts{q}{q} E \xrightarrow[]{0 \x \lambda} T(E\ts{q}{q}E) \xrightarrow[]{T.+_q} TE
\]
Note that any differential bundle morphism will preserve $\mu$ and $\nu$.
\begin{lemma}%
  \label{lem:pre-mu-nu}
  Differential bundle maps preserve $\mu(x,y) := 0 \o x +_{T.q} \lambda \o y$ and $\nu(v,y) := T.\xi \o v +_p \lambda \o y$.\pagenote{
     The word linear had crept into this sentence in the original draft, it has been correct to ``differential bundle morphism''.
  }
\end{lemma}
\begin{proof}
  The following diagram demonstrates that lift maps preserve $\nu$:
  \[
\begin{tikzcd}
	{E \ts{q}{p} TM} & {F \ts{q'}{p}TN} \\
	{T.E} & {T.F} \\
	\\
	{T.E} & {T.F}
	\arrow["\nu", from=1-1, to=2-1]
	\arrow["\nu", from=1-2, to=2-2]
	\arrow["{f \x T.m}", shift left=1, from=1-1, to=1-2]
	\arrow["{T.f}", shift left=1, from=2-1, to=2-2]
	\arrow["{T.g}"', shift right=1, from=2-1, to=2-2]
	\arrow["{e.e}"', shift right=1, from=2-1, to=4-1]
	\arrow["{T.e}", shift left=1, from=2-1, to=4-1]
	\arrow["{e.e}"', shift right=1, from=2-2, to=4-2]
	\arrow["{T.e}", shift left=1, from=2-2, to=4-2]
	\arrow["{T.f}", shift left=1, from=4-1, to=4-2]
	\arrow["{T.g}"', shift right=1, from=4-1, to=4-2]
	\arrow["{g \x T.n}"', shift right=1, from=1-1, to=1-2]
\end{tikzcd}\]
  Similarly, this diagram demonstrates that lift maps preserve $\mu$:
  \[
\begin{tikzcd}
  TE & TF \\
  {TE_2} & {TF_2} \\
  {E_2} & {F_2}
  \arrow["{(T.f \x T.f)}", from=2-1, to=2-2]
  \arrow["{(\lambda\x 0)}", from=3-1, to=2-1]
  \arrow["{(\lambda' \x 0)}"', from=3-2, to=2-2]
  \arrow["{(f\x f)}"', from=3-1, to=3-2]
  \arrow["{T.+}", from=2-1, to=1-1]
  \arrow["{T.+'}"', from=2-2, to=1-2]
  \arrow["Tf", from=1-1, to=1-2]
\end{tikzcd}\]
\end{proof}

\begin{lemma}%
  \label{lem:univ-props}
  Consider the full subcategory of differential bundles in $\C$ whose objects are differential bundles $(E,\lambda)$ so that the forks
  \begin{equation}%
    \label{eq:universality}
\begin{tikzcd}
	{E_2} &[6em] TE &[2em] TM \\
	{TM \ts{p}{q}E} & TE & E
	\arrow["{ \mu^E:= 0 \o\pi_0 +_{Tq} \lambda \o \pi_1}", from=1-1, to=1-2]
	\arrow["{\nu^E := T.\xi \o \pi_0 +_p \lambda \o \pi_1}", from=2-1, to=2-2]
	\arrow["{T.q}", shift left=1, from=1-2, to=1-3]
	\arrow["{T.q \o 0 \o p}"', shift right=1, from=1-2, to=1-3]
	\arrow["p", shift left=1, from=2-2, to=2-3]
	\arrow["{p \o T.\xi \o T.q}"', shift right=1, from=2-2, to=2-3]
\end{tikzcd}
  \end{equation}
  are $T$-equalizers. This subcategory is closed under $T$-equalizers.\pagenote{
     The original wording for this lemma did not properly quantify the objects over which the theorem was stated for.
  }
\end{lemma}
\begin{proof} 
  Start with the $T$-equalizer of lifts:
  \[
\begin{tikzcd}
	C & {E } & {F }
	\arrow["f", shift left=1, from=1-2, to=1-3]
	\arrow["{g }"', shift right=1, from=1-2, to=1-3]
	\arrow["k", dashed, from=1-1, to=1-2]
\end{tikzcd}\]
  Observe that $k$ is a lift map, so by Lemma \ref{lem:pre-mu-nu}, the following diagram commutes:
\[\begin{tikzcd}
  {C\ts{q^c}{p} TP} & {E \ts{q}{p} TM} & {F \ts{q'}{p}TN} \\
  {T.C} & {T.E} & {T.F} \\
  \\
  {T.C} & {T.E} & {T.F}
  \arrow["\nu", from=1-2, to=2-2]
  \arrow["\nu", from=1-3, to=2-3]
  \arrow["{f \x T.m}", shift left=1, from=1-2, to=1-3]
  \arrow["{T.f}", shift left=1, from=2-2, to=2-3]
  \arrow["{T.g}"', shift right=1, from=2-2, to=2-3]
  \arrow["{e.e}"', shift right=1, from=2-2, to=4-2]
  \arrow["{T.e}", shift left=1, from=2-2, to=4-2]
  \arrow["{e.e}"', shift right=1, from=2-3, to=4-3]
  \arrow["{T.e}", shift left=1, from=2-3, to=4-3]
  \arrow["{T.f}", shift left=1, from=4-2, to=4-3]
  \arrow["{T.g}"', shift right=1, from=4-2, to=4-3]
  \arrow["{k \x T.k'}"', dashed, from=1-1, to=1-2]
  \arrow["\nu", from=1-1, to=2-1]
  \arrow["{e.e}"', shift right=1, from=2-1, to=4-1]
  \arrow["{T.e}", shift left=1, from=2-1, to=4-1]
  \arrow["{T.k}", dashed, from=2-1, to=2-2]
  \arrow["{T.k}", dashed, from=4-1, to=4-2]
  \arrow["{g \x T.n}"', shift right=1, from=1-2, to=1-3]
\end{tikzcd}\]
  Next, since $k$ is a morphism of pre-differential bundles, by Lemma \ref{lem:pre-mu-nu} the following diagram commutes:
\[\begin{tikzcd}
  {C_2} & {E_2} & {F_2} \\
  {T.C} & {T.E} & {T.F} \\
  \\
  {T.C} & {T.E} & {T.F}
  \arrow["\mu", from=1-2, to=2-2]
  \arrow["\mu", from=1-3, to=2-3]
  \arrow["{f_2}", shift left=1, from=1-2, to=1-3]
  \arrow["{T.f}", shift left=1, from=2-2, to=2-3]
  \arrow["{T.g}"', shift right=1, from=2-2, to=2-3]
  \arrow["{e.e}"', shift right=1, from=2-2, to=4-2]
  \arrow["{T.e}", shift left=1, from=2-2, to=4-2]
  \arrow["{e.e}"', shift right=1, from=2-3, to=4-3]
  \arrow["{T.e}", shift left=1, from=2-3, to=4-3]
  \arrow["{T.f}", shift left=1, from=4-2, to=4-3]
  \arrow["{T.g}"', shift right=1, from=4-2, to=4-3]
  \arrow["{g_2}"', shift right=1, from=1-2, to=1-3]
  \arrow[dashed, from=1-1, to=1-2]
  \arrow["\mu", from=1-1, to=2-1]
  \arrow["{e.e}"', shift right=1, from=2-1, to=4-1]
  \arrow["{T.e}", shift left=1, from=2-1, to=4-1]
  \arrow[dashed, from=2-1, to=2-2]
  \arrow[dashed, from=4-1, to=4-2]
\end{tikzcd}\]
  The top row follows because maps satisfying Rosicky's universality condition preserve $\mu$, and the bottom row by the naturality of $e.C$ and the fact that linear maps preserve the natural idempotent. Thus, if $E,F$ satisfy the universality diagrams in Diagram \ref{eq:universality}, the equalizer $C$ will as well, because $T$-equalizers are closed to $T$-limits in the category of fork diagrams.
\end{proof}
\begin{theorem}%
  \label{thm:universal-prop-differential-bundles}
  For every differential bundle $(q:E \to M, \xi,\lambda)$ in a tangent category, the diagram \[
\begin{tikzcd}
	{E_2} & TE & TM
	\arrow["{ \mu^E}", from=1-1, to=1-2]
	\arrow["{T.q}", shift left=1, from=1-2, to=1-3]
	\arrow["{T.q \o 0 \o p}"', shift right=1, from=1-2, to=1-3]
\end{tikzcd}\] is a $T$-equalizer, and if $TM \ts{p}{q} E$ exists, then \[
\begin{tikzcd}
	{E\ts{q}{p} TM} &[5em] TE & E
	\arrow["{\nu^E := \lambda \o \pi_0 +_p T.\xi \o \pi_0 }", from=1-1, to=1-2]
	\arrow["p", shift left=1, from=1-2, to=1-3]
	\arrow["{p \o T.\xi \o T.q}"', shift right=1, from=1-2, to=1-3]
\end{tikzcd}\] is a $T$-equalizer.
\end{theorem}
\begin{proof}
  The tangent bundle satisfies both universality conditions, and by Lemma \ref{lem:univ-props} every differential bundle will satisfy these conditions by the non-singularity of the lift.
\end{proof}
\begin{corollary}%
  \label{cor:idemp-dbun}
  In a complete tangent category, the category of differential bundles is precisely the category of algebras of the monad in Theorem \ref{thm:idemp-monad-nonsingular} on pre-differential bundles.
\end{corollary}
\begin{remark}
  The universality conditions in Theorem \ref{thm:universal-prop-differential-bundles} demonstrate that this definition of differential bundle agrees with that of \cite{Cockett2018}.
  \pagenote{I have removed a reference to the published version of this chapter, it was unnecessary.}
\end{remark}

\section{The isomorphism of categories}%
\label{sec:iso-vbun-dbun}
It is now straightforward to show that the main theorem of this chapter holds. First, observe that for every differential bundle $(q:E \to M, \xi, \lambda)$ in the category of smooth manifolds, the $T$-pullback $E \ts{p}{q} TM$ exists\pagenote{I have switched the order of the pullback for consistency}
, as $p$ is a submersion. Note that the universality of the two diagrams is equivalent:
\[
\begin{tikzcd}
	{E\ts{q}{p} TM} &[5em] TE & E
	\arrow["{\nu^E := \lambda \o \pi_0 +_p T.\xi \o \pi_0 }", from=1-1, to=1-2]
	\arrow["p", shift left=1, from=1-2, to=1-3]
	\arrow["{p \o T.\xi \o T.q}"', shift right=1, from=1-2, to=1-3]
\end{tikzcd} 
  \hspace{0.15cm}
\begin{tikzcd}
	{E \ts{q}{p}TM} & TE \\
	M & E
	\arrow["{\nu^E}"{description}, from=1-1, to=2-1]
	\arrow["{\lambda \o \pi_0 +_p T.\xi \o \pi_1}", from=1-1, to=1-2]
	\arrow["p", from=1-2, to=2-2]
	\arrow["\xi"', from=2-1, to=2-2]
	\arrow["\lrcorner"{anchor=center, pos=0.125}, draw=none, from=1-1, to=2-2]
\end{tikzcd}
\]
Thus the following holds.
\begin{theorem}%
  \label{iso-of-cats-dbun-sman}
  There is an isomorphism of categories between vector bundles and differential bundles in smooth manifolds. 
\end{theorem}
\begin{proof}
  Note that Proposition \ref{prop:evf-vbun-is-nonsingular} gives a fully faithful functor from the category of vector bundles to differential bundles of smooth manifolds, as the lift associated to the vector bundle is non-singular and the projection and zero section give the rest of the structure of a differential bundle: as remarked after Definition \ref{def:vector-bundle}, local triviality guarantees that the projection is a submersion, so pullback powers of the projection exist, yielding a differential bundle.

  To see there is an isomorphism on objects, recall that every differential bundle is a fibred $\R$-module by Corollary \ref{cor:non-singular-implies-ract}, and that this identification is a bijective mapping (it recovers the original $\R$-action from the Euler vector field of the $\R$-action, and vice versa). Last, note that the universality condition
  \[
\begin{tikzcd}
	{E \ts{q}{p}TM} & TE \\
	M & E
	\arrow["{\nu^E}"{description}, from=1-1, to=2-1]
	\arrow["{\lambda \o \pi_0 +_p T.\xi \o \pi_1}", from=1-1, to=1-2]
	\arrow["p", from=1-2, to=2-2]
	\arrow["\xi"', from=2-1, to=2-2]
	\arrow["\lrcorner"{anchor=center, pos=0.125}, draw=none, from=1-1, to=2-2]
\end{tikzcd}\]
  along with Proposition \ref{prop:retracts-reindexing-of-vbuns} ensures the local triviality of $q$, so that the unique fibred $\R$-module structure associated to a differential bundle is indeed a vector bundle.
  \pagenote{I have added details to this proof to clarify that it does give a bijection on the classes of objects.}
\end{proof}

\section{Connections on a differential bundle}%
\label{sec:connections-on-a-differential-bundle}

The connections discussed in this section generalize the notion of an affine connection to a differential bundle, giving a ``local coordinates'' presentation for $TE$ similar to the presentation of $T^2M$ as $T_3M$ induced by an affine connection. Chapter \ref{ch:involution-algebroids} makes extensive use of connections on vector bundles, so it is useful to set out the basic definitions before embarking on algebroid theory.  
\begin{definition}[\cite{Cockett2017}]
\label{def:lin-connection}
  Let $(q:E \to M, \xi, \lambda)$ be a differential bundle in a tangent category $\C$. 
  \begin{itemize}
    \item A vertical connection is a map $\kappa:TE \to E$ so that 
    \begin{enumerate}[(i)]      
      \item $\kappa$ is a vertical descent, and hence a retract of the lift; thus, $\kappa \o \lambda = id$;
      \item $\kappa$ is compatible with both differential bundle structures on $TE$, so that the maps
      \begin{gather*}
        \kappa: (TE,\ell) \to (E,\lambda) \\ \kappa:(TE, c\o T.\lambda) \to (E,\lambda)
      \end{gather*}
      are lift morphisms.
    \end{enumerate}
    \item A horizontal connection is a map $\nabla: E \ts{q}{p} TM \to TE$ so that
    \begin{enumerate}[(i)]
      \item $\nabla$ is a horizontal lift, and so is a retract of $(p.E, T.q):TE \to E \ts{q}{p} TM$\footnote{
        That $(p.E,T.q)$ land in the pullback $E \ts{q}{p} TM$ is a consequence of naturality, as $p \o T.q = q \o p$.
      }; thus, $(p, T.q)\o \nabla = id$;
      \item $\nabla$ is compatible with each pair of lifts, so that the maps
      \begin{gather*}
        \nabla: ( E \ts{q}{p} TM, \lambda \x \xi) \to (TE, c \o T.\lambda) \\
        \nabla: ( E \ts{q}{p} TM, 0 \x \ell) \to (TE, \ell)
      \end{gather*}
       are lift morphisms.
    \end{enumerate}
    \item A full connection is a pair $(\kappa,\nabla)$ that satisfies the following compatibility relations:\pagenote{
        I have extended the exposition on the compatibility relation between $\kappa, \nabla$ to make it clear that it is about constructing an isomorphism $TE \cong E_2\x TM$.
    }
    \begin{enumerate}[(i)]
      \item $\kappa \o \nabla = \xi \o q \o \pi_0$,
      \item $\nabla(p, T.q) +_{T.q} \mu(p, \kappa) = id$, so that there is an isomorphism $E \ts{q}{p} TM \ts{p}{q} E \cong TE$. \pagenote{caught a typo}
    \end{enumerate}
    
  \end{itemize}
\end{definition}
A notion that will be useful when dealing with classical differential geometry is that of a \emph{covariant derivative}, whose definition is equivalent to that of a vertical connection in the category of smooth manifolds.
\begin{definition}%
  \label{def:covariant-derivative}
  Let $(q:E \to M, \xi, \lambda, \kappa)$ be a vertical connection.\footnote{
     The definition of a covariant derivative only uses a vertical connection.
  }
  The \emph{covariant derivative} associated to $(\kappa,\nabla)$ is the map
  \[
      \nabla_{(-)}[=]:\Gamma(\pi) \x \Gamma(p) \to \Gamma(\pi); (A,X) \mapsto (\kappa \o TA \o X).
  \]
\end{definition}
\cite{LucyshynWright2018} drastically simplified the notion of a full connection by showing that it is exactly a vertical connection satisfying a universal property.
\begin{proposition}[\cite{LucyshynWright2018}]\label{prop:rory-connection}
  A connection on a differential bundle is equivalently specified by a vertical connection \[\kappa:TE \to E\] so that the following diagram exhibits $TE$ as a biproduct in the category of differential bundles over $M$:
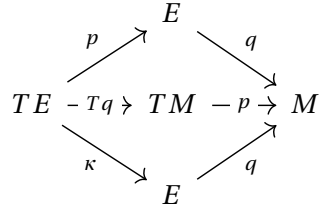
\[\begin{tikzcd}
	& E \\
	TE & TM & M \\
	& E
	\arrow["p", from=2-1, to=1-2]
	\arrow["Tq"{description}, from=2-1, to=2-2]
	\arrow["\kappa"', from=2-1, to=3-2]
	\arrow["q", from=1-2, to=2-3]
	\arrow["q"', from=3-2, to=2-3]
	\arrow["p"{description}, from=2-2, to=2-3]
\end{tikzcd}\]
\end{proposition}
There is also a notion of flatness for connections that extends to vertical connections on a differential bundle.
\begin{definition}
  A connection $\kappa$ on a differential bundle $(q:E \to M, \xi, \lambda)$ is \emph{flat} whenever
  \[ 
    \kappa \o T.\kappa \o c = \kappa \o T.\kappa.
  \]
\end{definition}
We can see that connections are closed under similar constructions to differential bundles, in particular idempotent splittings and the reindexing construction from Lemma \ref{lem:reindex-db}.
\begin{lemma}%
  \label{lem:conn-const}
  Let $(q:E \to M, \xi, \lambda)$ be a differential bundle equipped with a vertical connection.\pagenote{
      I have reworded the second part of this lemma so that it no longer refers to effective vertical connections. The proof has also been streamlined.
  }
  \begin{enumerate}[(i)]
    \item Any linear retract of $(q,\xi,\lambda)$ will have a vertical connection. 
    \item The pullback differential bundle induced by pulling back $q$ along $f:N \to M$ will have a vertical connection.
  \end{enumerate}
  If the vertical connection is flat or is part of a full connection (that is, it satisfies the universality condition in Proposition \ref{prop:rory-connection}), then the induced connection will be as well.
\end{lemma}
\begin{proof}
  The universal property induces the vertical connection in each case. The construction will preserve flatness as it is an equational condition, and it preserves effectiveness by the commutativity of limits.
\end{proof}
There is no reason for every differential bundle in a tangent category to have a connection (for example, the tangent bundle in the free tangent category $\wone$ in Chapter \ref{chap:weil-nerve} is a differential bundle that does not have a connection). However, if the total space of a differential bundle has an affine connection, this induces a compatible connection on the differential bundle and its base space.
\begin{theorem}
  \label{thm:linear-connection-from-total-space}
  Let $(q:E \to M, \xi, \lambda)$ be a differential bundle in a tangent category in $\C$, where $E$ has a (flat) vertical connection. Then the total space $M$ and the differential bundle $(q,\xi,\lambda)$ each have a (flat) vertical connection. If the connection is full (so there is a compatible horizontal connection), then the induced connections are likewise full.
\end{theorem}
\begin{proof}
  By the strong universality condition for differential bundles, the differential bundle $(q\o\pi_0: E \ts{q}{p} TM \to M, (\xi, 0), \lambda \x \ell)$ is the pullback differential bundle of $p:TE \to E$ along $\xi:M \to E$. 
  This gives $(q\o\pi_0, (\xi,0), (\lambda,\ell))$, a (flat, effective) vertical connection by Lemma \ref{lem:conn-const}, and so it yields $(q:E \to M, \xi, \lambda)$ and $(p:TM \to M, 0, \ell)$ as (flat, effective) vertical connections by the idempotent splitting property.
\end{proof}
Every smooth manifold has an affine connection, thus inducing a connection on any vector bundle.
\begin{corollary}
  Every vector bundle has a connection.
\end{corollary}
\pagenote{Kristine mentioned it's a bit weird to end a chapter on a piece of notation, which I makes sense once it has been pointed out to me.}




\chapter{Involution algebroids}%
\label{ch:involution-algebroids}

This chapter accomplishes the first major goal of this thesis by providing a tangent-categorical axiomatization of Lie algebroids, namely involution algebroids. Much like vector bundles, Lie algebroids are a highly non-algebraic notion (in the sense of \cite{Freyd1972}), being vector bundles equipped with a Lie algebra structure on the set of sections of the projection. Furthermore, the bracket on sections must satisfy a product rule with respect to $\mathbb{R}$-valued functions on the base space (the \emph{Leibniz law}), introducing another piece of non-algebraic structure to the definition. The tangent-categorical definition of Lie algebroids will treat the tangent bundle as the ``prototypical Lie algebroid'' in which the vertical lift $\ell:T \Rightarrow T^2$ identifies the vector bundle structure and the canonical flip $c:T^2 \Rightarrow T^2$ plays the role of the Lie bracket.

Lie algebroids are the natural many-object analogue to Lie algebras, in the same way that Lie groupoids are the many-object analogue of Lie groups. In the single-object case, a Lie group is classically thought of as a space of symmetries for some smooth manifold (one often identifies a group action $G \x M \to M$), and a Lie algebra may similarly be thought of as a space of \emph{derivations} (often identified as a sub-Lie-algebra of $\chi(M)$ for a manifold $M$).  The extension of groups to groupoids is natural; in fact, Brandt's introduction of groupoids in \cite{brandt1927verallgemeinerung} predates MacLane and Eilenberg's invention of category theory in \cite{eilenberg1945general} by nearly two decades.  The translation of Lie algebras to the many-object case is not as straightforward. The first step is to replace the vector space underlying a Lie algebra with a vector bundle $(\pi:A \to M, \xi,\lambda)$. The idea is to axiomatize this vector bundle so that each section in $\Gamma(\pi)$ corresponds to a derivation on $C^\infty(M)$. The \emph{anti-commutator} operation on derivations from Proposition \ref{prop:anti-commm-lie} suggests there should be a Lie bracket $[-,-]:\Gamma(\pi)\ox\Gamma(\pi) \to \Gamma(\pi)$ (similar to the partially-defined multiplication for a groupoid), while the correspondence with derivations on $C^\infty(M)$ suggests there be a vector bundle morphism $\anc:A \to TM$ satisfying the \emph{Leibniz law}:
\begin{equation}\label{eq:bracket-f}
    [X, f\cdot Y] = f\cdot [X, Y] + [X,f]\cdot Y; \hspace{0.5cm} [X,f] := \phat \o T.f \o \anc \o X\footnote{Recall the notation from Lemma \ref{lem:Cinfty-module-vbun}.}
\end{equation}
(the full definition of Lie algebroids may be found in \ref{def:lie-algd}). This ``operational'' definition of Lie algebroids makes it difficult to describe their morphisms, and furthermore it essentially fails to be an algebraic structure in the classical sense, as it axiomatizes structure on the \emph{set of sections} of a map rather than a morphism in the category itself.

Involution algebroids were introduced to provide a tangent-categorical presentation of Lie algebroids, similar to the relationship between differential bundles and vector bundles. Chapter \ref{ch:differential_bundles} focused on the Euler vector field construction on a vector bundle, showing that this induced a fully-faithful functor from vector bundles to associative coalgebras (lifts) of the weak comonad $(T,\ell)$, and identified vector bundles with a subcategory of $\mathsf{Lift}(\mathsf{SMan})$ satisfying a universal property. The corresponding construction for Lie algebroids, then, is the \emph{canonical involution}, which was identified by Eduardo Martinez and his collaborators (a clearly written exposition may be found in Section 4 of \cite{de2005lagrangian}). Given a Lie algebroid $(\pi:A \to M, \anc:A \to TM, [-,-]:\Gamma(\pi)\ox\Gamma(\pi) \to \Gamma(\pi))$, its canonical involution is a map
\[
    \sigma: \prolong \to \prolong.
\]
Using this $\sigma$ map, there is a straightforward characterization of Lie algebroid morphisms: a Lie algebroid morphism is precisely a vector bundle morphism $(f,m):A \to B$ that preserves the anchor and involution maps:
\[
\begin{tikzcd}
	A & B & \prolong & {B\ts{\anc^B}{T.\pi^B}TB} \\
	TM & TN & \prolong & {B\ts{\anc^B}{T.\pi^B}TB}
	\arrow["{\anc^A}"', from=1-1, to=2-1]
	\arrow["{\anc^B}", from=1-2, to=2-2]
	\arrow["f", from=1-1, to=1-2]
	\arrow["{T.m}"', from=2-1, to=2-2]
	\arrow["{\sigma^B}", from=1-4, to=2-4]
	\arrow["{f \x T.f}", from=1-3, to=1-4]
	\arrow["{\sigma^A}"', from=1-3, to=2-3]
	\arrow["{f\x T.f}"', from=2-3, to=2-4]
\end{tikzcd}\]
Furthermore, it is implicit in Martinez's work (\cite{Martinez2001}) that $\sigma$ satisfies axioms corresponding to the Lie algebroid axioms. Thus, involutivity corresponds to antisymmetry of the Lie bracket
\[
    \sigma \o \sigma = id \iff [X,Y] + [Y,X] = 0, 
\]
while the Leibniz law holds if and only if the $\anc$ map sends the algebroid involution to the canonical flip on $M$,
\[
    T.\anc \o \pi_1 \o \sigma = c \o T.\pi \o \anc \iff 
    \forall f \in C^\infty(M), [X, f\cdot Y] = f\cdot [X, Y] + [X,f]\cdot Y
\]
(using the same definition as before for $[X,f]$).

The idea of an involution algebroid, then, is to axiomatize the canonical involution directly, just as differential bundles axiomatize the Euler vector field of a differential bundle. An involution algebroid is a differential bundle equipped with a pair of structure maps
\[
    \anc:A \to TM, \hspace{0.15cm} \sigma: \prolong \to \prolong
\]
satisfying a collection of axioms. Some of them are straightforward translations of the structure equations for Lie algebroids given in \cite{Martinez2001}, for instance
\[
  T.\anc \o \lambda = \ell \o \anc, \hspace{0.15cm}
  \sigma \o \sigma = id,\hspace{0.15cm}
  T.\anc \o \pi_1 \o \sigma = c \o T.\anc \o \pi_1.
\]
However, this requires a new coherence between the Euler vector field of the underlying vector bundle and the involution map:
\[
    \sigma \o (\xi\o\pi,\lambda) = (\xi\o\pi,\lambda).
\]
The most striking new fact about this coherence is that the Jacobi identity on the bracket $[-,-]$ corresponds to the \emph{Yang--Baxter} equation on $\sigma$:
\[\begin{tikzcd}
	{\prolong \ts{T.\anc}{T^2.\pi} T^2A} & {\prolong \ts{T.\anc}{T^2.\pi} T^2A} \\
	{\prolong \ts{T.\anc}{T^2.\pi} T^2A} & {\prolong \ts{T.\anc}{T^2.\pi} T^2A} \\
	{\prolong \ts{T.\anc}{T^2.\pi} T^2A} & {\prolong \ts{T.\anc}{T^2.\pi} T^2A}
	\arrow["{\sigma \x c.A}"', from=1-1, to=1-2]
	\arrow[from=1-2, to=2-2]
	\arrow["{\sigma \x c.A}"', from=2-2, to=3-2]
	\arrow["{1 \x T.\sigma}", from=1-1, to=2-1]
	\arrow["{\sigma \x c.A}", from=2-1, to=3-1]
	\arrow["{1 \x T.\sigma}", from=3-1, to=3-2]
\end{tikzcd}\]
This is both surprising (it is a new characterization of a central object of study in differential geometry and mathematical physics) and yet in a way expected (the work in \cite{Cockett2015}, \cite{Mackenzie2013} indicates that the Lie algebra structure on the set of vector fields over a manifold follows from the Yang--Baxter equation on $c$).  This vector-field-free presentation of the Jacobi identity allows for a structural approach to Lie algebroids that drives the work presented in Chapters \ref{chap:weil-nerve} and \ref{ch:inf-nerve-and-realization}.

As with \Cref{ch:differential_bundles}, the first section is expository, and is concerned with introducing the category of Lie algebroids. 
The second section introduces anchored bundles, together with the space of prolongations of an anchored bundle. The relationship between anchored bundles and involution algebroids is equivalent to that between reflexive graphs and groupoids (the subject of Chapter \ref{ch:inf-nerve-and-realization}), the space of prolongations of an anchored bundle being equivalent to the set of composable arrows for a groupoid. This section is mostly a translation of Martinez's prolongation construction to a general tangent category. The rest of the chapter contains new results, developed in collaboration with Matthew Burke and Richard Garner.

Section 3 introduces involution algebroids, which are anchored bundles equipped with an involution map on their space of prolongations. Section 4 considers an anchored bundle in a tangent category with negatives that is equipped with a connection. The connection gives an involution algebroid a ``local coordinates'' presentation (in the sense of Section \ref{sec:diff-and-tang-struct}) that is equivalent to the local characterization of Lie algebroids from Section 1. The final section of this chapter establishes the main result: the category of Lie algebroids is isomorphic to that of involution algebroids in smooth manifolds.\pagenote{
   The introduction has been substantially expanded to better explain what is happening in this chapter and why we are doing it. 
}


\section{Lie algebroids}\label{sec:Lie_algebroids}

This section reviews the basic theory of Lie algebroids: their definition and that of their morphisms, along with some introductory examples. The classical definition will not appear elsewhere in this chapter, however, as we quickly introduce Martinez's \emph{structure equations} for a Lie algebroid (\cite{Martinez2001}), then translate them into tangent-categorical terms using a connection. \pagenote{
   The substantial revisions to the introduction made the original preamble to this section redundant. There have also been substantial revisions to the thesis up until this point, and those changes have percolated into this section. In particular, the notation for the $C^\infty(M)$-module of sections $\Gamma(\pi)$ was set in the previous chapter, and we stick to those conventions.
}

\begin{definition}\label{def:lie-algd}
    A \emph{Lie algebroid} is a vector bundle $\pi:A \rightarrow M$ equipped with an anchor $\anc:A \to TM$ and a bracket $[-,-]:\Gamma(\pi)\ox\Gamma(\pi) \rightarrow \Gamma(\pi)$ satisfying the following axioms:
    \begin{itemize}
        \item bilinear: $[aX_1+bX_2, Y] = a[X_1, Y]+ b[X_2, Y]$ and $[X, aY_1+bY_2] = a[X, Y_1]+b[X, Y_2]$
        \item anti-symmetric: $[X, Y]+[Y, X] = 0$
        \item Jacobi: $[X, [Y, Z]] = [[X, Y], Z] + [Y, [X, Z]]$
        \item Leibniz: $[X, f\cdot Y] = f\cdot [X, Y] + [X,f]\cdot Y$\pagenote{Changed the notation of the Lie derivative to this bracket version.}
    \end{itemize}
    (where $[X,f]$ is defined as in Equation \ref{eq:bracket-f}).
\end{definition}

\begin{example}%
    \label{ex:lie-algebroids}
    ~\begin{enumerate}[(i)]
        \item The canonical example of a Lie algebroid is, of course, the tangent bundle using the operational tangent bundle from Definition \ref{def:operational-tang}.
        \item A \emph{Lie algebra} is a Lie algebroid over the terminal object: for a group $G$, the bundle of source-constant tangent vectors is the usual Lie functor from Lie groups to Lie algebras, because a groupoid is a one-object group.
        \item The bundle of \emph{source-constant} tangent vectors $s,t: G \to M$ of a Lie groupoid forms a Lie algebroid. This bundle is defined by the pullback
\[\begin{tikzcd}
    A & TG & TM \\
    M & {TM \x G}
    \arrow["{(T.s,p)}", from=1-2, to=2-2]
    \arrow["{(0,e)}"', from=2-1, to=2-2]
    \arrow["\pi"', from=1-1, to=2-1]
    \arrow[hook, from=1-1, to=1-2]
    \arrow["{T.t}", from=1-2, to=1-3]
    \arrow["\anc", curve={height=-18pt}, from=1-1, to=1-3]
    \arrow["\lrcorner"{anchor=center, pos=0.125}, draw=none, from=1-1, to=2-2]
\end{tikzcd}\]
        where the projection is $\pi$ and the target is given by $\anc$ in the diagram.
        There is an injective $\R$-module morphism from sections of $\pi$, $\Gamma(\pi)$ to vector fields on $G$, $\chi(G)$, and the Lie bracket on $G$ is closed over the image of this lift, putting a Lie bracket on $\Gamma(\pi)$. In particular, we can see that $TM$ is the bundle of source-constant tangent vectors for the pair groupoid on a manifold $M$:
\[\begin{tikzcd}
	TM & {T(M \x M)} \\
	M & {(M \x M) \x TM}
	\arrow["{(p, T.\pi_0)}", from=1-2, to=2-2]
	\arrow["{(\Delta, 0)}"', from=2-1, to=2-2]
	\arrow["p"', from=1-1, to=2-1]
	\arrow["{(id, 0\o p)}", from=1-1, to=1-2]
	\arrow["\lrcorner"{anchor=center, pos=0.125}, draw=none, from=1-1, to=2-2]
\end{tikzcd}\]
        Given $(u,v):X \to T(M \x M)$ above $(m,m):X \to TM$ with $u = 0 \o m$, it follows that $u = 0 \o p \o u$, so $v$ is the unique map induced into $TM$.
        \item Every group is a groupoid over a single object. The Lie algebroid associated with a group $G$, then, is the usual Lie algebra.
        \item From the Hamiltonian formalism of mechanics, every \emph{Poisson manifold} has an associated Lie algebroid.
        A Poisson manifold is a manifold $M$ equipped with a \emph{Poisson algebra} structure on $C^\infty(M)$, namely a Lie algebra that is also a derivation:
        \[
            [f\cdot g, h] = f\cdot [g,h] + [f,h]\cdot g
        \]
        where $\cdot$ is the multiplication in the algebra $C^\infty(M)$ as in Lemma \ref{lem:Cinfty-module-vbun}. The cotangent bundle over a Poisson manifold $M$ is canonically a Lie algebroid, called a Poisson Lie algebroid \cite{Courant1994}. 
        \item   Any Lie algebra bundle---that is, a vector bundle equipped with a Lie bracket on its space of sections---is a Lie algebroid, with anchor map $\xi \o \pi$.     
    \end{enumerate}
\end{example}
Morphisms of Lie algebroids are notoriously difficult to work with, and have an involved definition.
\begin{definition}\label{def:lie-algd-morphism}
    Let $A, B$ be a pair of anchored bundles over $M,N$, and $\Phi: A \to B$ an anchored bundle morphism over a map $\phi:M \to N$. A $\Phi$-decomposition of $X \in \Gamma(\pi_A)$ is a set of $X_i \in \Gamma(\pi_B)$ and $f_i \in C^\infty(M)$ so that
    \[
        \Phi \o X = \sum_i f_i \cdot X_i \o \phi.
    \]
    An anchor-preserving vector bundle morphism $\Phi$ is a Lie algebroid morphism if and only if for any $X,Y \in \Gamma(\pi_A)$ and $\Phi$-decompositions $\{X_i, f_i\}, \{Y_j, g_j\}$ of $X,Y$, the following equation holds:
    \[        
        \Phi \o [X,Y] 
        = 
        \sum f_i\cdot g_j \cdot ([X_i,Y_j] \o \phi)
        + \sum [X, f_i] \cdot (X_i \o \phi)
        - \sum [Y, g_j](Y_j \o \phi).
    \]
\end{definition}
The equation defining a Lie algebroid morphism holds independently of the choice of $\Phi$-decomposition (see \cite{Higgins1990} for a proof).\pagenote{Here I have tidied up the definition by explaining what a $\Phi$-decomposition is (in particular, quantifying over the $X_i, f_i$, and also clarified some wording.}

\begin{example}
    ~\begin{enumerate}[(i)]
        \item If $A, B$ are Lie algebroids over a base manifold $M$, $\Phi: A \to B$ is a Lie algebroid morphism if and only if it preserves the anchor and Lie bracket.\pagenote{I removed the remark about tangent Lie algebroids, as it refers to De Rham cohomologies which have disappeared from this thesis.}
        \item Given two Lie algebroid morphisms $f: A \to B, g: B \to C$, their composition $g \o f$ is also a Lie algebroid morphism.
        \item A \emph{Poisson Sigma model} (see \cite{Bojowald2005}) is a morphism of Lie algebroids 
        \[
            \phi: T\Sigma \to T^*M
        \]
        for which $\Sigma$ is a 2-dimensional manifold and $T\Sigma$ denotes its tangent Lie algebroid, while $M$ is a Poisson manifold and $T^*M$ denotes the Lie algebroid structure on its cotangent bundle .
    \end{enumerate}
\end{example}

Lie algebroids are a natural generalization of Lie algebras to the ``multi-object'' setting, but they are ill-suited for a functorial presentation of the theory. A step in this direction is to consider the coordinate-based presentation of the Lie bracket and its coherences due to  \cite{Martinez2001}.
Let $A$ be an anchored bundle over $M$ equipped with a bilinear bracket on its space of sections, and choose a pair of bases for $\Gamma(\pi)$ and $\chi(M)$: for $\Gamma(\pi)$ write $\{ e_\alpha\}$, and for $\chi(M)$ write $\{ \frac{\partial}{\partial x^i}\}$. The anchor and bracket then have a presentation in local coordinates: 
\[
    \anc(e_\alpha) = \sum_i \anc^i_\alpha \frac{\partial}{\partial x^i}
    \hspace{0.5cm}
    [e_\alpha, e_\beta] = \sum_\gamma C^{\gamma}_{\alpha\beta}e_\gamma
\]
(from here on out, we use Einstein summation notation to simplify our calculations, so instead write
\[
    \anc(e_\alpha) = \anc^i_\alpha \frac{\partial}{\partial x^i}
    \hspace{0.5cm}
    [e_\alpha, e_\beta] = C^{\gamma}_{\alpha\beta}e_\gamma
\]
with $\sum$ suppressed).
The following characterization of the Lie algebroid axioms uses Martinez's \emph{structure equations}.\pagenote{This is better called the "structure equations" for a Lie algebroid.}
\begin{proposition}\label{prop:la-iff-structure-morphisms}[\cite{Martinez2001}]
    An anchored bundle $A$ over $M$ equipped with a bracket is a Lie algebroid if and only if  $\anc$ and $[-,-]$ satisfy the following structure equations:
    \begin{enumerate}[(i)]
        \item Alternating: \[C^\nu_{\alpha\beta} + C^\nu_{\beta\alpha} = 0\]
        \item Leibniz: \[\anc^j_\alpha \frac{\partial \anc^i_\beta}{\partial x^j} 
            = \anc^i_\gamma C^\gamma_{\alpha\beta} + \anc^j_\beta \frac{\partial \anc^i_\alpha}{\partial x^j}\]
        \item Bianchi: \[ 0
            = \anc^i_\alpha \frac{\partial C^{\nu}_{\beta\gamma}}{\partial x^i}
            + \anc^i_\beta \frac{\partial C^{\nu}_{\gamma\alpha}}{\partial x^i}
            + \anc^i_\gamma \frac{\partial C^{\nu}_{\alpha\beta}}{\partial x^i}
            + C^{\mu}_{\beta\gamma}C^{\nu}_{\alpha\mu}
            + C^{\mu}_{\gamma\alpha}C^{\nu}_{\beta\mu}
            + C^{\mu}_{\alpha\beta}C^{\nu}_{\gamma\mu}\]
    \end{enumerate}
\end{proposition}
This proposition is a straightforward translation of the Lie algebroid axioms into local coordinates using a covariant derivative.
First, recall that by the smooth Serre--Swan theorem (11.33 \cite{Nestruev2003}), the bilinearity of the bracket \pagenote{Clarified where this new bracket comes from/what category it lives in.}
\[
    [-,-]:\Gamma(\pi) \x \Gamma(\pi) \to \Gamma(\pi)  
\] as a morphism of $C^\infty(M)$-modules guarantees that it corresponds to a bilinear morphism $A_2 \to A$ of vector bundles, meaning that there exists a globally defined bilinear map $A_2 \to A$ that is equal to the Lie bracket when applied to sections of the projection. We record this as a lemma: \pagenote{Clarified the point that this theorem gives a global definition of the Lie bracket, extending from from $\Gamma(\pi)$ to general elements in the category}
\begin{lemma}\label{lem:ang-brack-lie}
    For every Lie algebroid $\a$, there is a bilinear morphism
    \[
       \langle -,-\rangle: A_2 \to A 
    \]
    so that for any sections $X,Y\in \Gamma(\pi)$, $\langle X,Y\rangle = [X,Y]$.
\end{lemma}

There are two structure maps derived from $\< -, -\>$ that encode the coherences of a Lie algebroid. The first map measures the extent to which the anchor maps fail to preserve the chosen connections on the vector bundle $\pi:A \to M$ and the tangent bundle $p:TM \to M$.\pagenote{Added some explanation of what the curly bracket map does.}
\begin{definition}\label{def:curly-bracket}
    Let $A$ be a Lie algebroid, and for a chosen horizontal connection $\nabla$ on $A$ and vertical connection $\kappa'$ on $TM$, set
    \[
        \{ v, x\}_{\kappa',\nabla} := TM \ts{p}{q} A \xrightarrow[]{\nabla} TA \xrightarrow[]{T.\anc} T^2M \xrightarrow[]{\kappa'} TM.
    \]
    When the choice of connection is evident by context, we will suppress the subscript.
\end{definition}
\begin{observation}
    The above parentheses bracket corresponds to the symbol
    \[
        \{-,-\} = \anc^j_\beta \frac{\partial \anc^i_\alpha}{\partial x^j}.  
    \]
\end{observation}
The Leibniz coherence may be rewritten as follows:
\begin{lemma}\label{lem:curly-bracket-coh}
    Let $A$ be an anchored bundle with a bilinear bracket (inducing an involution $\sigma$).
    Choose connections $(\kappa, \nabla), (\kappa', \nabla')$ on $A$ and $TM$ respectively.
    The bracket and anchor map satisfy the Leibniz law if and only if
    \[
        \anc \o \< x, y\>_{(\kappa, \nabla)} 
        + \{ \anc x, y\}_{(\kappa', \nabla)}
        = \{ \anc y, x\}_{(\kappa', \nabla)}.
    \]
\end{lemma}
\begin{proof}
    The condition is equivalent to the identity
    \[\anc^j_\alpha \frac{\partial \anc^i_\beta}{\partial x^j} 
            = \anc^i_\gamma C^\gamma_{\alpha\beta} + \anc^j_\beta \frac{\partial \anc^i_\alpha}{\partial x^j}.\]
\end{proof}
The Bianchi axiom measures the failure of the Jacobi identity in local coordinates, and states that it must be corrected for by the curvature of the brackets. 
We see that
\[
    C^{\nu}_{\alpha\mu} C^{\mu}_{\beta\gamma} = [e_\alpha, [e_\beta, e_\gamma]]
\]
while 
\[
    \frac{\partial C^{\nu}_{\beta\gamma}}{\partial x^i} e_\nu
    = \kappa \o T(\langle -,-\rangle_{(\kappa, \nabla)})\o \nabla^{A_2} \o ( \frac{\partial}{\partial x^i}, e_\beta, e_\gamma)
\]
determines a trilinear map
\[
    \{\frac{\partial}{\partial x^i}, e_\beta, e_\gamma\}_{(\kappa, \nabla)} 
    := \kappa \o T(\langle -,-\rangle_{(\kappa, \nabla)})\o \nabla^{A_2} \o ( \frac{\partial}{\partial x^i}, e_\beta, e_\gamma).
\]
This is the second derived map used in the structure equations for a Lie algebroid.
\begin{definition}
    Let $(\pi:A \to M, \xi, \lambda, \anc)$ be an anchored bundle equipped with a bilinear map
    \[
        \langle -, - \rangle: A_2 \to A.
    \]
    The derived ternary bracket $\{ -, -, - \}: TM \ts{p}{\pi} A \ts{\pi}{\pi} A \to A$ is defined as
    \[
        \{v,x,y\}_{(\kappa, \nabla)} := 
        TM\ts{p}{\pi} A \ts{\pi}{\pi} A \xrightarrow[]{\nabla[A2]} TA_2 \xrightarrow[]{T.\<-,-\>} TA \xrightarrow{\kappa} A
    \]
    where $\nabla[A2]$ is the pairing $(\nabla(\pi_0,\pi_1), \nabla(\pi_0,\pi_2))$.
\end{definition}
\begin{observation}
    The ternary bracket corresponds to the following symbol:
    \[
        \{-,-,-\}_{(\kappa, \nabla)} :=
        \anc^i_\alpha \frac{\partial C^{\nu}_{\beta\gamma}}{\partial x^i}.
    \]
\end{observation}
\begin{lemma}\label{lem:bianchi-connection}
    Let $A$ be an anchored bundle over $M$ with a bilinear bracket and denote the induced involution by $\sigma$.
    Choose a pair of connections and write the derived maps $\langle-,-\rangle, \{-,-,-\}$.
    Then
    \begin{enumerate}[(i)]
        \item the bracket is antisymmetric if and only if its globalization is; that is, $\langle e_\alpha, e_\beta\rangle + \langle e_\beta, e_\alpha\rangle = 0$;
        \item the bracket satisfies the Jacobi identity if and only if it is alternating in the last two arguments and
            \[
                0
                = \sum_{\gamma \in \mathsf{Cy}(3)} \< x_{\gamma_0}, \< x_{\gamma_1}, x_{\gamma_2}\> \>
                + \sum_{\gamma \in \mathsf{Cy}(3)} \{ \anc x_{\gamma_0}, x_{\gamma_1}, x_{\gamma_2}\}.
            \]
    \end{enumerate}
\end{lemma}
\begin{proof}~
    \begin{enumerate}[(i)]
        \item This is equivalent to $C^\nu_{\alpha\beta} + C^\nu_{\beta\alpha} = 0$.
        \item This is equivalent to \[ 0
            =   C^{\mu}_{\beta\gamma}C^{\nu}_{\alpha\mu}
            + C^{\mu}_{\gamma\alpha}C^{\nu}_{\beta\mu}
            + C^{\mu}_{\alpha\beta}C^{\nu}_{\gamma\mu}
            + \anc^i_\alpha \frac{\partial C^{\nu}_{\beta\gamma}}{\partial x^i}
            + \anc^i_\beta \frac{\partial C^{\nu}_{\gamma\alpha}}{\partial x^i}
            + \anc^i_\gamma \frac{\partial C^{\nu}_{\alpha\beta}}{\partial x^i}. \] 
    \end{enumerate}
\end{proof}

\begin{proposition}%
    \label{prop:lie-alg-bil-defn}
    A Lie algebroid is exactly an anchored vector bundle $(\pi:A \to M, \xi, \lambda, \anc)$ with a bilinear, alternating map
    \[
        \langle -, -\rangle: A_2 \to A; \hspace{0.5cm} \< x, y \> + \< y, x\> = 0
    \]
    so that for any connection $(\nabla,\kappa)$ on $A$, the maps 
    \begin{equation}
        \label{eq:lie-algd-structure-maps}
        \{ v, x\}_{(\kappa, \nabla)} := \kappa' \o T.\anc \o \nabla(v, x), \hspace{0.5cm}
        \{v,x,y\}_{(\kappa, \nabla)} := \kappa \o T(\langle -,-\rangle_{(\kappa, \nabla)})\o \nabla^{A_2} \o (v,x,y)
    \end{equation}
    satisfy the equations
    \begin{enumerate}[(i)]
        \item $ \anc \o \< x, y\> 
        + \{ \anc x, y\}_{(\kappa', \nabla)}
        = \{ \anc y, x\}_{(\kappa', \nabla)}$,
        \item $\sum_{\gamma \in \mathsf{Cy}(3)} \< x_{\gamma_0}, \< x_{\gamma_1}, x_{\gamma_2}\> \>
        + \sum_{\gamma \in \mathsf{Cy}(3)} \{ \anc x_{\gamma_0}, x_{\gamma_1}, x_{\gamma_2}\}_{(\kappa, \nabla)} = 0$.
    \end{enumerate}
\end{proposition}

There is also a local coordinates presentation of morphisms as in Section 2 of \cite{Martinez2018}. An anchored bundle morphism $A \to B$ is a Lie algebroid morphism whenever
\[
    f^\beta_\gamma A^\gamma_{\alpha \delta} + \anc^i_\delta \frac{\partial f^\beta_\alpha}{\partial x^i}
    = B^\beta_{\theta\sigma}f^\theta_{\alpha}f^{\sigma}_{\delta} + \anc^i_\alpha \frac{\partial f^\beta_\delta}{\partial x^i}.
\]
The $A$ and $B$ arguments are understood as the brackets, so this condition can be rewritten as
\[
    f^\beta_\gamma A^\gamma_{\alpha\delta} = f \o \< \alpha, \delta \>, \hspace{0.25cm}
    B^\beta_{\theta\sigma}f^\theta_{\alpha}f^{\sigma}_{\delta} = \< f \o \alpha, f \o \delta \>.
\]
Set the following notation for maps between vector bundles with connection:\pagenote{Reintroduce the $\nabla$ notation here after removing it.}
\begin{equation}\label{eq:nabla-notation}
    \infer{\nabla[f]:A_2 \to B := \kappa^B \o T.f \o \nabla^A}{f:A \to B & (\kappa^A,\nabla^A,A, \lambda^A) & (\kappa^B,\nabla^B, B, \lambda^B)}
\end{equation}
The $\anc$ terms are understood to be the torsion, so that
\begin{gather*}
        \anc^i_\delta \frac{\partial f^\beta_\alpha}{\partial x^i} 
    = \kappa \o T.f \o \nabla(\anc e_\delta, e_\alpha) = \nabla[f](\anc e_\delta, e_\alpha),\\
    \anc^i_\alpha \frac{\partial f^\beta_\delta}{\partial x^i} 
    = \kappa \o T.f \o \nabla(\anc \alpha, \delta) = \nabla[f](\anc e_\alpha, e_\delta),
\end{gather*}
using the notation set up in Equation \ref{eq:nabla-notation}.
The notion of a Lie algebroid morphism, then, has the following presentation:
\begin{proposition}[\cite{Martinez2018}]%
    \label{prop:lie-algd-morphism-defn}
    Let $(\pi:A \to M,\anc^A, \<-,-\>^A), (q:B \to N,\anc^B,\<-,-\>^B)$ be a pair of Lie algebroids with chosen connections $(\kappa^{-},\nabla^{-})$.
    An anchor-preserving vector bundle morphism $f:A \to B$ is a Lie algebroid morphism if and only if
    \[
        f \o \< e_\alpha, e_\delta \> + \nabla[f] \o (\anc e_\delta, e_\alpha)
        = \< f \o \alpha, f \o \delta \> + \nabla[f] \o(\anc e_\alpha, e_\delta).
    \]
\end{proposition}

\section{Anchored bundles}%
\label{sec:anchored-bundles}
Anchored bundles are to Lie algebroids what reflexive graphs are to groupoids. Each theory has (mostly) the same structure, but while a reflexive graph is missing a groupoid's composition operation, an anchored bundle lacks a Lie algebroid's bracket operation. This section reviews the basic theory of anchored bundles and their prolongations (see \cite{Mackenzie2005} for more details).
\begin{definition}\label{def:anchored_bundles}
    An anchored bundle in a tangent category is a differential bundle $(A \xrightarrow{\pi} M, \xi, \lambda)$ equipped with a linear morphism 
    \begin{equation*}
        \begin{tikzcd}[column sep = tiny]
    A \ar{rr}{\anc} \ar{rd}[swap]{\pi} & & TM \ar{ld}{p} \\
    & M
 \end{tikzcd}
    \end{equation*}
    A morphism of anchored bundles is a linear bundle morphism $(f,v)$ that preserves the anchors
    \[
\begin{tikzcd}
	TA & TB & A & B \\
	A & B & TM & TN
	\arrow["\lambda"{description}, from=2-1, to=1-1]
	\arrow["l"{description}, from=2-2, to=1-2]
	\arrow["f"{description}, from=2-1, to=2-2]
	\arrow["{T.f}"{description}, from=1-1, to=1-2]
	\arrow["\rho"{description}, from=1-4, to=2-4]
	\arrow["\anc"{description}, from=1-3, to=2-3]
	\arrow["f"{description}, from=1-3, to=1-4]
	\arrow["{T.v}"{description}, from=2-3, to=2-4]
\end{tikzcd}\]
    The category of anchored bundles and anchored bundle morphisms in a tangent category $\C$ is written $\mathsf{Anc}(\C)$, and a generic anchored bundle is written $(A \xrightarrow{\pi} M, \xi, \lambda,\anc)$.
    %
\end{definition}
There are two pullbacks that are associated with every anchored bundle.
These play the role of the spaces of composable arrows $G_2 := G \ts{t}{s} G, G_3 := G \ts{t}{s} G \ts{t}{s} G$ for a reflexive graph $s,t:G \to M$.
\begin{definition}\label{def:prolongations}
    Let $(A \xrightarrow{\pi} M, \xi, \lambda, \anc)$ be an anchored bundle. Its \emph{first and second prolongations} are given by the limits
    \[
\begin{tikzcd}
	{\prol(A)} & TA & {\prol^2(A)} && {T^2A} \\
	A & TM && TA & {T^2M} \\
	&& A & TM
	\arrow["\anc"{description}, from=2-1, to=2-2]
	\arrow["{T.\pi}", from=1-2, to=2-2]
	\arrow[from=1-1, to=2-1]
	\arrow[from=1-1, to=1-2]
	\arrow["\lrcorner"{anchor=center, pos=0.125}, draw=none, from=1-1, to=2-2]
	\arrow["{T^2.\pi}"{description}, from=1-5, to=2-5]
	\arrow["{T.\anc}"{description}, from=2-4, to=2-5]
	\arrow["{T.\pi}"{description}, from=2-4, to=3-4]
	\arrow["\anc"{description}, from=3-3, to=3-4]
	\arrow[from=1-3, to=3-3]
	\arrow[from=1-3, to=1-5]
	\arrow["\lrcorner"{anchor=center, pos=0.125}, draw=none, from=1-3, to=2-4]
\end{tikzcd}\]
    (The notation for the fibre product is slightly non-standard, as it is not technically a pullback.)
    Throughout this chapter, it will always be assumed that the first and second prolongations of an anchored bundle exist (although no choice of prolongation is explicitly made).\pagenote{I have explicitly defined the category of anchored bundles, and split the definition of prolongations from the definition of an anchored bundle. Also clarified that we are simply assuming these first two $T$-limits always exist, but we do not make a specific choice of prolongation.}
\end{definition}

\begin{remark}
    It is not strictly necessary that the prolongations of an anchored bundle exist; this condition is primarily a matter of convenience when discussing involution algebroids. Every result in this section, that does not explicitly mention prolongations, holds for an anchored bundle independently of their existence.
\end{remark}
\begin{example}%
    \label{ex:anchored-bundles}\pagenote{
       This set of examples has been substantially revised to fix some clumsy wording in the original version. It made sense to included the full result about $\C$ be a reflective subcategory of $\mathsf{Anc}(\C)$ and the coreflection from the category of differential bundles to anchored bundles rather than having them as separate results, as they each describe classes of anchored bundles.
    }
    ~\begin{enumerate}[(i)]
        \item For any object $M$, $id: TM \to TM$ is an anchor for the tangent differential bundle, and every $f: M \to N$ yields a morphism of anchored bundles. Moreover, for any anchored bundle over $M$ the anchor is itself a morphism of anchored bundles $(A,\pi,\xi,\lambda,\anc) \to (TM, p, 0, \ell, id)$. This induces a fully faithful functor:
        \[
            \C \hookrightarrow \mathsf{Anc}(\C).
        \]
        This inclusion has a left adjoint, which sends an anchored bundle $(\pi:A \to M, \xi,\lambda,\anc)$ to its base space $M$ (the unit is the anchor map $\anc$), so that $\C$ is a reflective subcategory of $\mathsf{Anc}(\C)$.
        \item For any differential bundle $(A, \pi, \xi, \lambda)$, the map $0\o\pi:A \to TM$ is an anchor and every morphism $f: A \to B$ of differential bundles again yields a morphism of anchored bundles. The naturality of $0$ ensures that every differential bundle morphism will preserve this trivial anchor map, giving a fully faithful functor
        \[
            \mathsf{DBun}(\C) \hookrightarrow \mathsf{Anc}(\C).
        \]
        This functor has a right adjoint that replaces the anchor map with the trivial anchor map
        \[
            (\pi:A \to M, \xi, \lambda, \anc) \mapsto (\pi:A \to M, \xi, \lambda, 0 \o \pi)  
        \]
        where the counit is given by the natural idempotent $e = \xi \o \pi:(A,\lambda) \to (A,\lambda)$. It is trivial to check that the anchor map is preserved by the bundle morphism $(e,id)$:
        \[
            \anc \o \xi \o \pi = 0 \o T.id \o \pi 
        \]
        and so the following diagram commutes:\pagenote{I took the opportunity to clarify a point brought up in my defense about this relationship.}
        \[\begin{tikzcd}
            TM & TM \\
            A & A
            \arrow[Rightarrow, no head, from=1-1, to=1-2]
            \arrow["{0 \o \pi}", from=2-1, to=1-1]
            \arrow["{p \o \lambda}"', from=2-1, to=2-2]
            \arrow["\anc"', from=2-2, to=1-2]
        \end{tikzcd}\]
        This means that  differential bundles are a \emph{coreflective} subcategory of anchored bundles.
        \item Any reflexive graph $(s,t: C \to M, e:M \to C)$ in a tangent category has an anchored bundle (when sufficient limits exist), constructed as
        \[
\begin{tikzcd}
	{C^\partial} & {T.C} & {T.C}
	\arrow["{e.T_1}", shift left=1, from=1-2, to=1-3]
	\arrow["{T.e^s}"', shift right=1, from=1-2, to=1-3]
	\arrow[from=1-1, to=1-2]
\end{tikzcd}\]
        (where $e^s:C \to C = e \o s$). Construct a lift on $C^\partial$:
        \[
\begin{tikzcd}
	{T.C^\partial} & {T^2C_1} & {T^2C_1} \\
	{C^\partial} & {T.C_1} & {T.C_1}
	\arrow["{e.C_1}", shift left=1, from=2-2, to=2-3]
	\arrow["{T.e^s}"', shift right=1, from=2-2, to=2-3]
	\arrow[from=2-1, to=2-2]
	\arrow["\ell", from=2-2, to=1-2]
	\arrow["\ell", from=2-3, to=1-3]
	\arrow["{T.e.C_1}", shift left=1, from=1-2, to=1-3]
	\arrow["{T.T.e^s}"', shift right=1, from=1-2, to=1-3]
	\arrow[from=1-1, to=1-2]
	\arrow["\lambda", dashed, from=2-1, to=1-1]
\end{tikzcd}\]    
        This lift will be non-singular by the commutativity of $T$-limits. The pre-differential bundle data is given by the projection
        \[ C^\partial \hookrightarrow T.C \xrightarrow[]{p.e^s} M.\]
        The section is induced by
        \[
            M \xrightarrow[]{(0.e)}  T.C
        \]
        while post-composition with $T.t$ gives the anchor map:
        \[
            C^\partial \hookrightarrow T.C \xrightarrow[]{T.t} T.C. 
        \]
        The diagram is a pullback by composition, and the outer perimeter defines the pullback $\prol(A)$. Note that any reflexive graph morphism will give rise to an anchored bundle morphism by naturality, making the construction of an anchored bundle from a reflexive graph functorial.\pagenote{This was originally in the examples of prolongations, which was not correct.}
        \item For any object $M$ in a tangent category, $c_M: T^2M \to T^2M$ is an anchor on $(T.p, T.0, c \o T.\ell)$, and every map $f: M \to N$ gives a morphism of anchored bundles.
        \item For any anchored bundle $(q:E \to M, \xi, \lambda, \anc)$, the differential bundle $(T.q, T.\xi, \newline c \o T.\lambda)$ has an anchor 
        \[\anc_T: TE \xrightarrow[]{T.\anc} T^2M \xrightarrow[]{c}T^2M.\]
    \end{enumerate} 
\end{example}
The first prolongation of an anchored bundle $\prol(A)$ behaves similarly to the second tangent bundle, except that it does not have a canonical flip. In the definition of an affine connection, the tangent bundle played a similar role to the ``arities'' of a theory. There is a lift map that makes this connection stronger:
\begin{definition}\label{def:lhat}
    Let $(\pi:A \to M, \xi, \lambda, \anc)$ be an anchored bundle in a tangent category $\C$.
    We define a \emph{generalized lift}:
    \[
        \hat{\lambda}: A \to \prolong := (\xi \o \pi, \lambda).   
    \]
\end{definition}
This generalized lift satisfies the same coherences as the lift on the second tangent bundle:
\begin{proposition}\label{prop:lift-axioms-anchor}
    Let $(\pi:A \to M, \xi, \lambda, \anc)$ be an anchored bundle in a tangent category $\C$.
    It follows that
    \begin{enumerate}[(i)]
        \item (Coassociativity of $\hat{\lambda}$) $(\hat{\lambda}\x\ell)\o\hat{\lambda} = (id \x T.\hat{\lambda}) \o \hat{\lambda}$;
        \item (Universality) the following diagram is a $T$-pullback:
        \[
\begin{tikzcd}
	{A_2} & \prolong \\
	M & A
	\arrow["{\hat \mu}", from=1-1, to=1-2]
	\arrow["{\pi_0}", from=1-2, to=2-2]
	\arrow["{\pi \o \pi_i}"', from=1-1, to=2-1]
	\arrow["\xi"', from=2-1, to=2-2]
	\arrow["\lrcorner"{anchor=center, pos=0.125}, draw=none, from=1-1, to=2-2]
\end{tikzcd}\]
        where $\hat \mu := (\xi \o \pi \o \pi_0, \mu)$.
    \end{enumerate}
\end{proposition}
\begin{proof}
    ~\begin{enumerate}[(i)]
        \item Compute
            \begin{align*}
                (\hat{\lambda}\x \ell)\o \hat{\lambda} 
                &= (\xi\o\pi\o\xi\o\pi, \lambda\o\xi\o\pi, \ell\o\lambda) \\
                &= (\xi\o\pi, T.(\xi\o\pi)\o\lambda, T.\lambda \o \lambda) \\
                &= (\pi_0, T.(\xi\o\pi)\o\pi_1, T.\lambda \o \pi_1)\o (\xi\o\pi, \lambda) \\
                &= (id \x T(\hat{\lambda}))\o\hat{\lambda}.
            \end{align*}
        \item  Use the pullback lemma to observe that the following diagram is universal for any anchored bundle:
    \[\begin{tikzcd}
        {A_2} & {\prol(A)} & TA \\
        M & A & TM
        \arrow["\xi", from=2-1, to=2-2]
        \arrow["\anc", from=2-2, to=2-3]
        \arrow["0"', curve={height=12pt}, from=2-1, to=2-3]
        \arrow["T\pi", from=1-3, to=2-3]
        \arrow["{\pi_0}", from=1-2, to=2-2]
        \arrow["{\pi\o\pi_1}"', from=1-1, to=2-1]
        \arrow["{\hat{\mu}}"', from=1-1, to=1-2]
        \arrow["{\pi_1}"', from=1-2, to=1-3]
        \arrow["\mu", curve={height=-12pt}, from=1-1, to=1-3]
    \end{tikzcd}\]
            The top triangle of the diagram commutes by definition\pagenote{Removed a redundancy pointed out by Kristine.}. 
            The right square and outer perimeter are pullbacks by definition, and the bottom triangle also commutes by definition. The pullback lemma ensures that the left square is a pullback, so for every anchored bundle, the general lift is universal for $\prol(A)$. Now post-compose with the involution:
\[\begin{tikzcd}
    {A_2} & {\prol(A)} & {\prol(A)} \\
    M & A & A
    \arrow["\xi", from=2-1, to=2-2]
    \arrow["{\pi_0}", from=1-2, to=2-2]
    \arrow["{\pi\o\pi_1}"', from=1-1, to=2-1]
    \arrow["{\hat{\mu}}"', from=1-1, to=1-2]
    \arrow["id", from=2-2, to=2-3]
    \arrow["{p\pi_1}", from=1-3, to=2-3]
    \arrow["\sigma"', from=1-2, to=1-3]
    \arrow["{\hat{\nu}}", curve={height=-12pt}, from=1-1, to=1-3]
\end{tikzcd}\]
            It suffices to check that the top triangle commutes, so $\sigma \o \hat{\mu} = \nu$:
            \[
                \sigma \o \hat{\mu}\o (a,b) = \sigma \o   ((\xi\o \pi,0)\o a +_{\pi_0} (\xi\o \pi,\lambda)\o b) = 
                 (id, T.\xi \o \anc \o a) +_{p\pi_1} (\xi\o \pi,\lambda)\o b.
            \]
            Thus, the lift $(\xi\pi,\lambda)$ involution algebroid is universal for $\prol(A)$.
    \end{enumerate}
\end{proof}

\begin{example}%
    \label{ex:prolongations}
    ~\begin{enumerate}[(i)]
        \item For $T(M) = T(M)$, the space of prolongations is $T(M) \ts{id}{Tp} T^2M = T^2M$, and the second prolongation is given by $T^2M$.
        \item   In a tangent category with a tangent display system, if $\pi \in \d$ then the prolongations for $(\pi, \xi, \lambda, \anc)$ automatically exist.
        In particular, for every anchored bundle in the category of smooth manifolds, all prolongations exist because the projection is a submersion (see \Cref{sec:submersions}).
        \item   For any differential bundle with the anchor $0\o \pi$, it follows that $\prol(A) \cong A\ts{\pi}{\pi\o\pi_i} (A_2) \cong A_3$. 
            The universality of the vertical lift factors $b$ into $(a_1,a_2)$, as the following diagram is a pullback:
            \begin{equation*}
\begin{tikzcd}
    A_3 \ar{d}{\pi_0} \ar[dotted]{r}{(\pi_1,\pi_2)}
    &[3em] A \ts{\pi}{\pi} A \rar{\mu} \dar{\pi\pi_i} & TA \dar{T\pi} \\
    A \rar{\pi} & M \rar{0} & TM
\end{tikzcd}
\end{equation*}

        \pagenote{The functor from reflexive graphs to anchored bundles was here for some reason, rather than in the examples of anchored bundles piece.}
        \item Returning to the anchored bundle constructed from a graph, the space of prolongations $\prol(A)$ embeds into the second tangent bundle of the space of composable arrows:
        \[ \prolong \hookrightarrow T^2(G \ts{t}{s} G).\]
    \end{enumerate}
\end{example}
\pagenote{It seemed cleaner to include this proposition in the above examples, since this is really just a remark about }

The category of anchored bundles is, in a sense, ``tangent monadic'' over the category of differential bundles: the forgetful functor from anchored bundles to differential bundles ``creates'' $T$-limits and the tangent structure (this is all made precise in \Cref{ch:inf-nerve-and-realization} using an enriched perspective on tangent categories).
\begin{observation}%
    \label{obs:t-limits-of-anc}
    A limit of anchored bundles is the limit of the underlying differential bundles. Because the anchor is preserved by every map in the diagram, this induces a natural transformation for any $D: \d \to \mathsf{Anc}(C)$:
    \[
\begin{tikzcd}
	\d & {\mathsf{Anc}(\C)} & {\mathsf{Anc}} & {\mathsf{DBun}(\C)}
	\arrow["D"',from=1-1, to=1-2]
	\arrow[""{name=0, anchor=center, inner sep=0}, "S"', curve={height=12pt}, from=1-2, to=1-3]
	\arrow["U"', from=1-3, to=1-4]
	\arrow[""{name=1, anchor=center, inner sep=0}, curve={height=-12pt}, Rightarrow, no head, from=1-2, to=1-3]
	\arrow["\anc", shorten <=3pt, shorten >=3pt, Rightarrow, from=1, to=0]
\end{tikzcd}\]
    Thus, $\lim U.S.D$ computes the limit of the underlying objects, while $\lim U.D$ computes the limit of the underlying differential bundles in the diagram. The anchor/unit map, then, induces a differential bundle map:
    \[
        \lim \anc_i: (\lim A_i, \lim \lambda_i) \to (T.(\lim_i M_i), \ell.(\lim_i M_i)).
    \]
    This is the limit in the category of anchored bundles (so long as pullback powers of the limit projection and the two prolongations of the limit anchor bundle exist).
\end{observation}
The tangent structure on lifts defined in Proposition \ref{prop:lifts-is-tangent} lifts to a tangent structure on anchored bundles. 
\begin{lemma}\label{lem:tangent-anchored-bundle}
    The category of anchored bundles in a tangent category has a tangent structure that maps objects as follows:
    \[
        (\pi:A \to M, \xi, \lambda, \anc) \mapsto (T.\pi:TA \to TM, T.\xi, c \o T.\lambda, c \o T.\anc)  
    \]
    where the structure maps are all defined using the pointwise structure maps in $\C$. 
    \pagenote{There exists multiplie tangent-like structures on anchored bundles and involution algebroids, so I have added more detail in the statements of those results.}
\end{lemma}
\begin{proof}
    Given an anchored bundle $(q:A \to M,\xi, \lambda, \anc)$, there is an anchor on the differential bundle $(Tq, T\xi, c \o T\lambda)$ given by $c \o T\anc$; the diagram commutes by naturality and the coherences on $c, \ell$.
    \[
\begin{tikzcd}
	{T^2A} & {T^3M} & {T^3M} \\
	TA & {T^2M} & {T^2M} \\
	TM & TM & TM
	\arrow["Tq", from=2-1, to=3-1]
	\arrow["Tp", from=2-2, to=3-2]
	\arrow["{c \o T\lambda}", from=2-1, to=1-1]
	\arrow["{c \o T\ell}"', from=2-2, to=1-2]
	\arrow["\ell"', from=2-3, to=1-3]
	\arrow["c"', from=2-2, to=2-3]
	\arrow["p"', from=2-3, to=3-3]
	\arrow[Rightarrow, no head, from=3-1, to=3-2]
	\arrow[Rightarrow, no head, from=3-2, to=3-3]
	\arrow["T\anc"', from=2-1, to=2-2]
	\arrow["{T^2\anc}", from=1-1, to=1-2]
	\arrow["Tc", from=1-2, to=1-3]
\end{tikzcd}\]
    The tangent structure maps and universality properties all follow from the forgetful property of the functor from anchored bundles to differential bundles as a consequence of Observation \ref{obs:t-limits-of-anc}.
\end{proof}

For any anchored bundle $A$, there are two differential bundles associated to $\prol(A)$. The first is the usual pullback differential bundle given by pulling back $T.\pi$ along $\anc$ as in Lemma \ref{lem:reindex-db}:
\begin{equation*}
    \begin{tikzcd}
    \prolong \rar{\pi_1} \dar{\pi_0} & TA \dar{T\pi} \\
    A \rar{\anc} & TM
\end{tikzcd}
\end{equation*}
This gives the differential bundle structure 
\[
    (\prolong \xrightarrow{\pi_0} A, A \xrightarrow{(id, T.\xi \o \anc)} \prolong, \prolong \xrightarrow{(0, c \o T.\lambda)} T(\prolong)).
\]
Taking the fibre product in the category of anchored bundles as in Observation \ref{obs:t-limits-of-anc} yields the second differential bundle structure:
\[
    (A, \pi, \xi, \lambda) \xrightarrow{(\anc, id)} (TM, p, 0, \ell) \xleftarrow{(T\pi, \pi)} (TA, p, 0, \ell).
\]
The two lifts behave similarly to the pair $(T.\ell, \ell.T)$ on the second tangent bundle, and $(\ell,\lambda_T)$ on the tangent bundle of a pre-differential bundle.
\begin{lemma}\label{prop:anc-prol-fun}
    Let $\C$ be a tangent category and $\mathsf{Anc}(\C)$ the category of anchored bundles in $\C$. If $T$-pullback powers of $p \o \pi_0, \pi_1: \prolong \to A$ exist, then there are two differential bundles on $\prol(A)$, with lifts induced as
    \[
\begin{tikzcd}
	{(\prol(A), \lambda \x \ell)} & {(TA, \ell)} & {(\prol(A), 0 \x (c \o T.\lambda))} & {(TA, c \o T.\lambda)} \\
	{(A,\lambda)} & {(TM, \ell)} & {(A, 0)} & {(TM, 0)}
	\arrow[from=1-1, to=2-1]
	\arrow["\anc"', from=2-1, to=2-2]
	\arrow["{T.\pi}", from=1-2, to=2-2]
	\arrow[from=1-1, to=1-2]
	\arrow["\lrcorner"{anchor=center, pos=0.125}, draw=none, from=1-1, to=2-2]
	\arrow[from=1-3, to=2-3]
	\arrow["{T.\pi}"{description}, from=1-4, to=2-4]
	\arrow[""{name=0, anchor=center, inner sep=0}, "\anc"{description}, from=2-3, to=2-4]
	\arrow[from=1-3, to=1-4]
	\arrow["\lrcorner"{anchor=center, pos=0.125}, draw=none, from=1-3, to=0]
\end{tikzcd}\]
    with structure maps
    \begin{enumerate}
        \item $(\prol(A) \xrightarrow{p\o\pi_1} A, A \xrightarrow{(\xi\pi,0)} \prol(A), \prol(A) \xrightarrow{\lambda \x \ell} T \o \prol(A))$,
        \item $(\prol(A) \xrightarrow{\pi_0} A, A \xrightarrow{(id, T.\xi\o\anc)} \prol(A), \prol(A) \xrightarrow{0 \x c \o T.\lambda} T.\prol(A))$.
    \end{enumerate}
    Furthermore, the two lifts $\lambda_\prol$ and $\lambda_\anc$ commute: 
    \[c \o T.(\lambda \x \ell) \o ( 0 \x c \o T.\lambda ) = T.( 0 \x c \o T.\lambda ) \o (\lambda \x \ell).\]
\end{lemma}
\begin{proof} 
    The two differential bundles exist as a consequence of Observation \ref{obs:T-limits-pdbs} and Lemma \ref{lem:reindex-db}, respectively. The commutativity of limits follows by postcomposition, as both $c \o T.\lambda, \ell$ and $\lambda, 0$ commute by the differential bundle axioms.
\end{proof}
The above lemma determines a functor, which we denote $\prol$, that sends an anchored bundle in $\C$ to an anchored bundle in $\mathsf{Lift}(\C)$ (equipped with the tangent structure from Proposition \ref{prop:lifts-is-tangent}).
\begin{proposition}%
    \label{cor:prol-functor}
    There is a functor $\prol$ from anchored bundles in $\C$ to anchored bundles in the category of lifts\pagenote{added more context for what $\mathsf{Lift}(\C)$ is} $\mathsf{Lift}(\C)$, that sends an anchored bundle $(\pi:A \to M, \xi, \lambda, \anc)$ to the tuple in $\mathsf{Lift}(\C)$
    \begin{align*}
        &\pi^\prol:= (\prol(A), 0 \x c \o T.\lambda) \xrightarrow[]{p \o \pi_1} (A, 0),\\
        &\xi^\prol := (A,0) \xrightarrow[]{(\xi\o\pi,0)} (\prol(A),0 \x c \o T.\lambda)\\ 
        &\lambda^\prol := (\prol(A), 0\x c\o T.\lambda) \xrightarrow[]{(\lambda \x \ell)} (T(\prol(A)), c \o T(0 \x c \o T.\lambda), \\
        &\anc^\prol := (\prol(A), 0 \x c \o T.\lambda) \xrightarrow[]{\pi_1} (TA, c \o T.\lambda) 
    \end{align*}
    (note that this functor lands in anchored bundles of \emph{non-singular} lifts; see Definition \ref{def:non-singular-lift}).
    Morphisms of anchored bundles \[(f,m):(\pi:A \to M, \xi, \lambda, \anc) \to (q:E \to N, \zeta, l, \delta)\] are sent to
    \[
        \prol(f,m) := f \ts{Tm}{Tm} Tf: \prolong \to E \ts{\delta}{Tq} TE
    \]
\end{proposition}
\begin{proof}
    First, note that the tuple
    \[
        (\pi^\prol: \prol(A) \to A, \xi^\prol, \lambda^\prol, \anc^\prol)  
    \]
    is an anchored bundle, and that each morphism is a lift morphism:
    \begin{itemize}
        \item $\pi^\prol: (\prol(A), 0 \x c \o T.\lambda) \to (A,0)$ follows because
        \begin{align*}
            T.p \o T.\pi_1 \o (\lambda \x \ell) = T.p \o \ell \o \pi_1 = 0 \o p \o \pi_1
        \end{align*}
        \item $\xi^\prol: (A,0) \to (\prol(A), 0 \x c \o T.\lambda)$ follows since
        \[
           ( T.\xi \o T.\pi, T.0) \o 0 = (0 \o \xi \o \pi, T.0 \o 0) = (0 \o \xi \o \pi, c \o T.\lambda \o 0) = (0 \x c \o T.\lambda) \o (\xi \o \pi, 0)
        \]
        \item $\lambda^\prol: (\prol(A), 0 \x c \o T.\lambda) \to (T.\prol(A), 0 \x c \o T.(c \o T.\lambda))$
        follows by the commutativity of $0 \x c \o T.\lambda)$ and $\lambda \x \ell$, and
        \item $\anc^\prol: (\prol(A), 0 \x c \o T.\lambda) \to (TA, c \o T.\lambda)$ is a lift by definition of $\anc^\prol = \pi_1$.
    \end{itemize}
    This gives an anchored bundle in the category of non-singular lifts in $\C$.
    
    Next, check that the mapping is functorial. To see that $\prol(f,m)$ preserves the lifts, note that $(f,m)$ gives a morphism of diagrams for each pullback in $\mathsf{Lift}(\C)$ defining the two lifts, so the induced map $\prol(f,m)$\pagenote{There was some ambiguity in map names here, so $(f,v)$ was switched to $(f,m)$.} preserves each lift. To see that $\prol(f,m)$ preserves the anchor, check that
    \[
        \anc^{\prol B} \o \prol(f,m) = \pi_1 \o (f \x T.f)  = T.f \o \pi_1 = T.f \o \anc^{\prol A}.  
    \] 
\end{proof}
\begin{corollary}%
    \label{cor:prol-endofunctor}
    In a tangent category $\C$ where pullbacks along differential bundle projections exist, such as a tangent category equipped with a proper retractive display system (Definition \ref{def:display-system}) like the category of smooth manifolds, there is an endofunctor on the category of anchored bundles in $\C$,  \[\prol': \mathsf{Anc}(\C) \to \mathsf{Anc}(\C) = U^{\mathsf{Lift}}.\prol.\]
\end{corollary}
 \begin{observation}%
    \label{obs:computation-of-lim-prol}
     $T$-Limits in $\mathsf{Anc}(\mathsf{Lift}(\C))$ are computed as pointwise limits in $\C$ by Observations \ref{obs:t-limits-of-anc} and \ref{obs:T-limits-pdbs}, and $\prol$ is constructed as a limit, so it follows that $\prol$ preserves $T$-limits.
 \end{observation}

\begin{proposition}
    \label{prop:anc-almost-tang-cat}
        The prolongation endofunctor $\prol': \mathsf{Anc}(\C) \to \mathsf{Anc}(\C)$ has natural transformations
    \begin{itemize}
        \item $p': \prol' \Rightarrow id$
        \item $0': id \Rightarrow \prol'$
        \item $+': \prol' \ts{p'}{p'} \prol' \Rightarrow \prol'$
        \item $\ell': \prol' \Rightarrow \prol'.\prol'$
    \end{itemize}
    satisfying all of the axioms of a tangent category that do not incvolve the canonical flip (Definition \ref{def:tangent-cat}).
\end{proposition}
\begin{proof}[(Sketch)]
    Note that the full argument is given in \Cref{sec:weil-nerve}, but it is not difficult to sketch it out here. For the projection, zero map, and addition, use the differential bundle structure induced by Corollary \ref{cor:prol-functor}. To see the lift axiom, note that up to a choice of pullback, we have:
    \[
        \prol'.\prol'(\pi:A \to M, \xi, \lambda, \anc) = 
        \begin{cases}
            \quad \quad \quad \quad (p \o \pi_1, p \o \pi_2):& \prol^2(A) \to \prol(A) \\
            (\xi \o \pi \o \pi_0, 0 \o \pi_1, 0 \o \pi_2):& \prol(A) \to \prol^2(A) \\
            \quad \quad \quad \quad \quad \quad (\lambda \x \ell \x \ell):&\prol^2(A) \to T.\prol^2(A) \\
            \quad \quad \quad \quad \quad \quad \quad (\pi_1, \pi_2):& \prol^2(A) \to T.\prol(A)
        \end{cases}  
    \]
    The ``lift map'' is then
    \[
        \ell': \prol' \Rightarrow \prol'.\prol'; \hspace{0.15cm}
        \prol(A) \xrightarrow[]{A \x T.\hat\lambda} \prol^2(A) 
    \]
    (where we recall that $\hat\lambda = (\xi\o\pi,\lambda):A \to \prol(A)$).
\end{proof}

\begin{remark}
    The structure described in Proposition \ref{cor:prol-functor} leads to the theory of \emph{double vector bundles}, developed by MacKenzie and his collaborators \cite{flari2019warps,Mackenzie1992}. A double vector bundle is a commuting square
    \[
\begin{tikzcd}
	D & B \\
	A & M
	\arrow["{q^D_A}", from=1-1, to=2-1]
	\arrow["{q^D_B}", from=1-1, to=1-2]
	\arrow["{q^A}", from=2-1, to=2-2]
	\arrow["{q^B}", from=1-2, to=2-2]
\end{tikzcd}\]
    where each projection is a vector bundle projection, and ``vertical'' and ``horizontal'' orientations of the square are each vector bundle homomorphisms. It was observed by \cite{Grabowski2009} that this is equivalent to a pair of commuting $\R^+$-actions on the total space $E$; following the development in \Cref{ch:differential_bundles}, this is a \emph{commuting} pair of non-singular lifts on $E$,
    \[
        \lambda^A,\lambda^B:E \to TE, \hspace{0.15cm}
        T.\lambda^B \o \lambda^A = c \o T.\lambda^A \o \lambda^B.  
    \]
    A proper exposition of the so-called ``Ehresmann doubles'' (\cite{Mackenzie2011}) for the structures in Lie theory would substantially expand the scope of this thesis, and so it has been relegated to the margins.
\end{remark}

Finally, observe that a connection (see Section \ref{sec:connections-on-a-differential-bundle}) on an anchored bundle's underlying differential bundle behaves similarly to an affine connection.
\begin{lemma}\label{lem:anchored-bundle-conn}
    Let $(\pi:A \to M, \xi, \lambda, \anc)$ be an anchored bundle with $\prol(A)$ existing in a tangent category $\C$.\pagenote{The original draft had a typo - it was meant to be ``anchored bundle equipped with a connection''.}
    If $(\pi, \xi, \lambda)$ has a connection, then define \[\hat{\kappa} := \kappa \pi_1, \quad \hat{\nabla} := \nabla(\pi_0, \anc\pi_1)\] and note that
    \begin{enumerate}[(i)]
        \item $\hat{\kappa}: \prol(A) \to A$ is a retract of $(\xi\pi, \lambda)$, and $\hat{\kappa}: (\prol(A),l) \to (A, \lambda)$ is linear for both $l = (\lambda \x \ell), (0 \x c \o T.\lambda)$;
        \item $\hat{\nabla}:A_2 \to \prol(A)$ is a section of $(\pi_0, p \o \pi_1)$ and is bilinear;
        \item there is an isomorphism $\prol(A) \cong A_3$:\pagenote{added the explicit isomorphism}
        \[
            \nabla(p \o \pi_1, \anc \o \pi_0) +_{(\pi, T.\pi)} \hat{\mu}(p \o \pi_1, \hat{\kappa}) = id. 
        \]
    \end{enumerate}
\end{lemma}
\begin{example}
    Every vector bundle in the category of smooth manifolds has a connection; it follows that  Lemma \ref{lem:anchored-bundle-conn} holds for every anchored bundle in the category of smooth manifolds.
\end{example}

\begin{remark}
    The category of anchored bundles in a tangent category is almost a tangent category, except that it lacks a symmetry map. The ``differential objects'' in such a category will act like a cartesian differential category, except that the symmetry of mixed partial derivatives fails. There has been some interest in settings for differentiable programming where the symmetry of mixed partial derivatives need not hold (see Definition 3.4 along with the discussion at the end of Section 6 in \cite{cruttwell2021categorical});\pagenote{I have included an actual reference to quantify some interest.} the category of anchored bundles in a tangent category appears to be a source of examples.
\end{remark}
\section{Involution algebroids}%
\label{sec:involution-algebroids}


Involution algebroids are a tangent-categorical axiomatization of Lie algebroids. Recall that a Lie algebroid has almost all of the structure of the operational tangent bundle, in particular a Lie bracket on sections that satisfies the Leibniz law so that it gives a directional derivative. 
In Proposition \ref{prop:anc-almost-tang-cat}, it was demonstrated that the category of anchored bundles have almost all of the same maps as the tangent bundle, the only structure map missing being the canonical flip
\[
    c:T^2M \Rightarrow T^2M
\]
which, we may recall from the chapter introduction (and \cite{Cockett2015}, \cite{Mackenzie2013}), is used in constructing the Lie bracket of vector fields in a tangent category with negatives. Thus, a natural next step is to add an involution map to an anchored bundle and require that it satisfies the same coherences as $c$ from the tangent bundle.

\begin{definition}\label{def:involution-algd}
    An involution algebroid is an anchored bundle $(A, \pi,\xi,\lambda, \anc)$ equipped with a map $\sigma: \prol(A) \to \prol(A)$ satisfying the following axioms:
    \begin{enumerate}[(i)]
        \item Involution: \[
\begin{tikzcd}
	{\prol(A)} & {\prol(A)} \\
	& {\prol(A)}
	\arrow[Rightarrow, no head, from=1-1, to=2-2]
	\arrow["\sigma", from=1-1, to=1-2]
	\arrow["\sigma", from=1-2, to=2-2]
\end{tikzcd}\]
        \item Double linearity:\[
        \begin{tikzcd}
            {T.\prol(A)} & {T.\prol(A)} \\
            {\prol(A)} & {\prol(A)}
            \arrow["\sigma", from=2-1, to=2-2]
            \arrow["{0 \x c \o T.\lambda}"', from=2-2, to=1-2]
            \arrow["{T.\sigma}", from=1-1, to=1-2]
            \arrow["{\lambda \x \ell}", from=2-1, to=1-1]
        \end{tikzcd}\]
        \item Symmetry of lift: \[
\begin{tikzcd}
	A & {\prol(A)} \\
	& {\prol(A)}
	\arrow["{\hat \lambda = (\xi\o\pi, \lambda)}", from=1-1, to=1-2]
	\arrow["\sigma", from=1-2, to=2-2]
	\arrow["\hat\lambda"', from=1-1, to=2-2]
\end{tikzcd}\]
        \item Target: \[
\begin{tikzcd}
	{\prol(A)} & TA & {T^2M} \\
	{\prol(A)} & TA & {T^2M}
	\arrow["\sigma", from=1-1, to=2-1]
	\arrow["{\pi_1}", from=1-1, to=1-2]
	\arrow["{T.\anc}", from=1-2, to=1-3]
	\arrow["{\pi_1}", from=2-1, to=2-2]
	\arrow["{T.\anc}", from=2-2, to=2-3]
	\arrow["c", from=1-3, to=2-3]
\end{tikzcd}\]
        \item Yang--Baxter: \[
\begin{tikzcd}
	& {\prol^2(A)} & {\prol^2(A)} \\
	{\prol^2(A)} &&& {\prol^2(A)} \\
	& {\prol^2(A)} & {\prol^2(A)}
	\arrow["{\sigma \x c}", from=2-1, to=1-2]
	\arrow["{\sigma \x c}", from=1-3, to=2-4]
	\arrow["{\sigma \x c}"', from=3-2, to=3-3]
	\arrow["{A \x T.\sigma}"', from=2-1, to=3-2]
	\arrow["{A \x T.\sigma}"', from=3-3, to=2-4]
	\arrow["{A \x T.\sigma}", from=1-2, to=1-3]
\end{tikzcd}\]
    \end{enumerate}
    $A$ is an \emph{almost}-involution algebroid if the involution does not satisfy the Yang--Baxter equation.
    A morphism of involution algebroids is a morphism of anchored bundles, so that $\prol(f)\o \sigma_A = \sigma_B\o\prol(f)$.
    (Note that because $\sigma$ is an isomorphism and $\sigma = \sigma^{-1}$, 
    \[\sigma: (\prol(A), 0 \x c\o T.\lambda) \to (\prol(A),\lambda\x \ell)\]
    is linear as well.) Write the category of involution algebroids and involution algebroid morphisms in $\C$ is written $\mathsf{Inv}(\C)$.\pagenote{
    The category $\mathsf{Inv}(\C)$ has been explicitly defined - this addresses a few later remarks in this section.
}
\end{definition}

\begin{observation}
    It is not immediately clear that the Yang--Baxter equation is well-typed. This follows from the target axiom (iv) and the double linearity axiom (ii). Starting with $(u,v,w): \prolong \ts{T.\anc}{T^2.\pi} T^2A$, we see that $\sigma \x c$ is well-typed if and only if
    \[
       T.\anc \o \pi_1 \o \sigma(u,v) = T^2.\pi \o c \o w = c \o T^2.\pi \o w = c \o T.\anc \o v.
    \]
    Similarly, $1 \x T.\sigma$ is well-typed if $T.\anc \o u = T.\pi \o \pi_0 \o T.\sigma(v,w)$; then use the double linearity axiom to compute
    \begin{gather*}
        T.\pi \o \pi_0 \o T.\sigma(v,w) = T.\pi \o T.p \o T.\pi_1 (v,w) = T.\pi \o T.p \o w 
        \\ = T.p \o T^2.\pi \o w = T.p \o T.\anc \o v = T.\pi \o v.
    \end{gather*}
\end{observation}
This perspective on Lie algebroids has already appearad in the work of Martinez and his collaborators in \cite{Leon2005}, where a ``canonical involution'' was derived on space of prolongations of a Lie algebroid using the formula
\[
    \sigma: \prol(A) \to \prol(A); \sigma(x,y,z) = \sigma(y,x,z + \< x, y\>).
\]
The structure of this map has been largely unexplored; helpfully, involution algebroids succeed in reverse-engineering axioms for an involution map that will induce a Lie bracket on the sections of the projection map. \pagenote{I removed the remark about the Yang-Baxter equation, as that is discussed more thoroughly in the introduction of this chapter. Identifying axioms on the involution to identify Lie brackets is original, I think I may have muddied the waters with my original remark.} 
The bracket from the original Lie algebroid is induced using the same formula as for the bracket of vector fields on the tangent bundle:
\[
    \lambda \o [X,Y]^* 
    =
    \left( 
        (\pi_1 \o \sigma \o (id, T.X \o \anc) \o Y -_p T.Y \o X \o \anc) -_{T.\pi} 0Y
    \right).
\]
Furthermore, a morphism of anchored bundles is a Lie algebroid morphism if and only if it preserves the derived involution map.

\begin{example}\label{ex:inv-algds}
    \pagenote{I have updated this example following to clarify the relationships between $\C, \mathsf{Anc}(\C), \mathsf{Inv}(\C)$.}
    ~\begin{enumerate}[(i)]
        \item For any $M$ in $\C$, $(TM, p, 0, id, c)$ is an involution algebroid. Furthermore, for any involution algebroid anchored on $M$, $(\anc, id)$ is a morphism of involution algebroids (by the target axiom). This defines a fully faithful functor $\C \hookrightarrow \mathsf{Inv}(\C)$. The same construction as the anchored bundle case in Example \ref{ex:anchored-bundles}(i) exhibits $\C$ as a reflective subcategory of $\mathsf{Inv}(\C)$.
        \item For any differential bundle, the trivial anchored bundle $(\pi:A \to M,\xi,\lambda, 0\o\pi)$ has an involution using the isomorphism $\prol(A) \cong A_3$, given by $\sigma := (\pi_1,\pi_0,\pi_2)$ (proving this map satisfies the involution algebroid axioms is just an exercise in combinatorics). It follows that every differential bundle morphism gives rise to a morphism of these \emph{trivial} involution algebroids. The same construction from the anchored bundle case (Example \ref{ex:anchored-bundles}(ii)) exhibits $\mathsf{Diff}(\C)$ as a \emph{coreflective} subcategory of involution algebroids.
        \item Consider a groupoid 
        \[s,t:G \to M, \hspace{0.15cm} e: M \to G, \hspace{0.15cm} (-)^{-1}:G \to G, \hspace{0.15cm} m:G_2 \to G.\] 
        The underlying reflexive graph has an associated anchored bundle, constructed in Example \ref{ex:anchored-bundles}, and the space of prolongations of this anchored bundle includes into $(T^2.G_2)$. Note that there is a well-formed involution map:
        \[
            \sigma(u,v) = 
            c \o ((0 \o p \o v)^{-1}; (T.0 \o u);v)
        \] 
        The direct proof of this involves a more conceptual construction, which is the focus of Section \ref{sec:inf-nerve-of-a-gpd}.
    \end{enumerate}
\end{example}

An involution algebroid resembles a generalized tangent bundle, and so the lift $(\xi\o\pi, \lambda):A \to \prol(A)$ satisfies the same universality conditions as $\ell:T \Rightarrow T^2$. The double linearity condition is equivalent to the naturality condition for $c, \ell$ in Definition \ref{def:tangent-cat}:
\begin{equation}
    \label{eq:symmetry-of-lift}
\begin{tikzcd}
	{\prol(A)} && {\prol(A)} &[-1em] {T^2} && {T^2} \\
	{\prol^2(A)} & {\prol^2(A)} & {\prol^2(A)} & {T^3} & {T^3} & {T^3}
	\arrow["{T.\ell}"{description}, from=1-4, to=2-4]
	\arrow["{c.T}"', from=2-4, to=2-5]
	\arrow["{T.c}"', from=2-5, to=2-6]
	\arrow["c", from=1-4, to=1-6]
	\arrow["{\ell.T}"{description}, from=1-6, to=2-6]
	\arrow["{id \x T.\hat\lambda}", from=1-1, to=2-1]
	\arrow["\sigma", from=1-1, to=1-3]
	\arrow["{\sigma \x c}"', from=2-1, to=2-2]
	\arrow["{id \x T.\sigma}"', from=2-2, to=2-3]
	\arrow["\hat\lambda\x\ell"{description}, from=1-3, to=2-3]
\end{tikzcd}
\end{equation}
which, using string diagrams for monoidal categories (\cite{selinger2010survey}), is the equation
\begin{equation*}
    \begin{tikzpicture}
	\begin{pgfonlayer}{nodelayer}
		\node [style=none] (0) at (0, 0) {};
		\node [style=none] (1) at (0, 0.75) {};
		\node [style=new style 0] (2) at (0.5, 0) {};
		\node [style=none] (3) at (1.75, 1) {};
		\node [style=none] (4) at (1.75, 0.5) {};
		\node [style=none] (5) at (1.75, 0) {};
		\node [style=none] (6) at (2.75, 0.5) {$=$};
		\node [style=none] (7) at (3.5, 0.75) {};
		\node [style=none] (8) at (3.5, 0) {};
		\node [style=new style 0] (9) at (4.75, 0.75) {};
		\node [style=none] (10) at (5.75, 0) {};
		\node [style=none] (11) at (5.75, 0.5) {};
		\node [style=none] (12) at (5.75, 1) {};
	\end{pgfonlayer}
	\begin{pgfonlayer}{edgelayer}
		\draw (0.center) to (2);
		\draw [in=-180, out=15] (1.center) to (5.center);
		\draw [in=180, out=15, looseness=0.75] (2) to (3.center);
		\draw [in=180, out=0] (2) to (4.center);
		\draw [in=-180, out=0, looseness=1.25] (8.center) to (9);
		\draw [in=-150, out=30, looseness=0.75] (7.center) to (10.center);
		\draw [in=180, out=45] (9) to (12.center);
		\draw [in=-165, out=-45, looseness=0.75] (9) to (11.center);
	\end{pgfonlayer}
\end{tikzpicture}

\end{equation*}
where the circle denotes the lift and $c$ is the crossing of two lines. 
\begin{proposition}%
    \label{prop:nat-of-sigma-ell}
    Let $(\pi:A \to M, \xi, \lambda, \anc)$ be an anchored bundle, and suppose that
    \[
        \sigma: \prol(A) \to \prol(A)  
    \]
    satisfies the involution axiom (i) and the symmetry of lift axiom (iii), and furthermore that the involution ``exchanges'' the idempotents associated to the two lifts $(\lambda \x \ell)$ and $(0 \x c \o T.\lambda)$ on $\prol(A)$ (Proposition \ref{prop:idempotent-natural})
    \[
        \sigma \o ((p \o \lambda) \x (p \o \ell)) = (p \o 0) \x (p \o c \o T.\lambda) = id \x (T.p \o T.\lambda)
    \]
    Then the double linearity axiom is equivalent to the left-hand commuting diagram in Diagram \ref{eq:symmetry-of-lift}:
    \[
      (\hat{\lambda} \x \ell) \o \sigma = (id \x T.\sigma) \o (\sigma \x c) \o (id \x T.\hat{\lambda}).  
    \]
\end{proposition}
\begin{proof}
    Starting with the left-hand side of the equation, 
    \begin{align*}
        &\quad (id \x T.\sigma) \o (\sigma \x c) \o (id \x T.\hat\lambda) \o (u,v) \\
        &= (id \x T.\sigma) \o (\sigma \x c) \o (u, T.\xi \o T.\pi \o v, T.\lambda v) \\
        &= (id \x T.\sigma) \o (\sigma \o (u, T.\xi \o \anc \o u), T.\lambda v) \\
        &= (id \x T.\sigma) \o (\xi \o \pi \o u, 0 \o \anc \o u, T.\lambda v) \\
        &= (id \x T.\sigma) \o (\xi \o \pi \o u, 0 \o T.\pi \o v, T.\lambda v) \\
        &= (\xi \o \pi \o u, T.\sigma \o \hat \lambda \o v) .
    \end{align*}
    For the right-hand side:
    \begin{align*}
        &\quad (\hat \lambda, \ell) \o \sigma \o (u,v) \\
        &= (\xi \o \pi \o \pi_0 \o \sigma\o (u,v), (\lambda \x \ell) \o \sigma\o (u,v)) \\
        &= (\xi \o \pi \o p \o v, (\lambda \x \ell) \o \sigma\o (u,v)) \\
        &= (\xi \o p \o T.\pi \o v, (\lambda \x \ell) \o \sigma\o (u,v)) \\
        &= (\xi \o p \o \anc \o u, (\lambda \x \ell) \o \sigma\o (u,v)) \\
        &= (\xi \o \pi \o u, (\lambda \x \ell) \o \sigma\o (u,v)) \\
        &= (\xi \o \pi \o u, T.\sigma \o (0 \x c \o T.\lambda) \o (u,v)).
    \end{align*}\pagenote{Added a step to clarify the calculation.}
    So the naturality equation \ref*{eq:symmetry-of-lift} is equivalent to
    \[
        T.\sigma \o (0 \x c \o T.\lambda) = (\lambda \x \ell) \o \sigma.
    \]
\end{proof}

The category of involution algebroids is ``tangent monadic'' over the category of anchored bundles, in the same sense that anchored bundles are tangent monadic over the category of differential bundles, or internal categories over reflexive graphs in a category. The ``tangent monadicity'' leads to a similar observation about $T$-limits of involution algebroids as in Observation \ref{obs:t-limits-of-anc}.
\begin{observation}%
    \label{obs:inv-tmonad-anc}
    The forgetful functor from involution algebroids to anchored bundles creates limits; that is to say, the $T$-limits of the underlying anchored bundles give the limits of involution algebroids. Recall that by Observation \ref{obs:computation-of-lim-prol}, the limit $\prol(\lim A_i) = \lim \prol(A_i)$ in the category of anchored bundles, and this induces a map between objects in $\C$
    \[
        \lim \sigma_i: \lim \prol(A_i) \to \lim \prol(A_i).  
    \]
    The axioms for an involution algebroid are induced by universality. 
\end{observation}
Note that a point-wise tangent structure may be defined following the anchored bundles example from Lemma \ref{lem:tangent-anchored-bundle}:\pagenote{Added some preamble to this result to help differentiate it from the prolongation tangent structure.}
\begin{proposition}%
    \label{prop:pointwise-tangent-structure-inv}
    For a tangent category $\C$, the category of involution algebroids has a ``point-wise'' tangent structure that maps objects as follows:
    \[
        (\pi:A \to M, \xi, \lambda, \anc, \sigma) \mapsto
        (T.\pi, T.\xi, c \o T.\lambda, T.\anc, \sigma_T)  
    \]
    where $\sigma_T$ is defined as
    \begin{equation}\label{eq:sigmaTdeff}
        \sigma_T := (1 \x_c c)\o T.\sigma\o (1 \x_c c): \prol(A_T) \to \prol(A_T).
    \end{equation}
\end{proposition}
\begin{proof}    
    The tangent structure on involution algebroids is inherited from the functor $\mathsf{Inv}(\C) \to \mathsf{Anc}(\C)$. Thus, it suffices to give the involution map for the tangent involution algebroid.

    Note that given $(\pi:A \to M, \xi, \lambda, \anc, \sigma)$, the space of prolongations on $(T.\pi, T.\xi, \newline c\o T.\lambda, c\o T.\anc)$ is $TA \ts{c\o T.\anc}{T^2\pi} T^2A$. We can construct an isomorphism between the objects $\prol(TA)$ and $T(\prol(A))$ in $\C$ using the cospan isomorphism
\[\begin{tikzcd}
    TA & {T^2M} & {T^2A} \\
    TA & {T^2M} & {T^2A}
    \arrow["c \o T.\anc", from=1-1, to=1-2]
    \arrow["{T^2.\pi}"', from=1-3, to=1-2]
    \arrow["T.\anc"', from=2-1, to=2-2]
    \arrow["{T^2.\pi}", from=2-3, to=2-2]
    \arrow["id"', from=1-1, to=2-1]
    \arrow["c"', from=1-2, to=2-2]
    \arrow["c"', from=1-3, to=2-3]
\end{tikzcd}\]
    thus inducing a map
    \[
       TA \ts{c\o T.\anc}{T^2\pi} T^2A
       \xrightarrow[]{id \x c} T(\prolong)
       \xrightarrow[]{T.\sigma} T(\prolong)
       \xrightarrow[]{id \x c} TA \ts{c\o T.\anc}{T^2\pi} T^2A
    \]
    which we call $\sigma_T$. The linearity and involution axioms follow by construction.
    
    Now check the rest of the axioms. For the unit:
    \begin{gather*}
        (1 \x_c c)\o T.\sigma\o (1 \x_c c) \o (T.(\xi\o\pi), c \o T.\lambda)\\
        = (1 \x_c c) \o T.\sigma \o (T(\xi\pi), T\lambda) 
        = (1 \x_c c) \o (T.(\xi\o\pi), T.\lambda) = (T.(\xi\o\pi), c \o T.\lambda).
    \end{gather*}
    For the anchor:
    \begin{gather*}
        T.\anc_T \o \pi_1 \o \sigma_T = T(c\o T\anc) \o \pi_1 \o (1 \x_c c)\o T.\sigma\o (1 \x_c c)\\
        = T.c\o c \o  T^2.\anc \o T.\pi_1 \o T.\sigma\o (1 \x_c c)
        = T.c \o c \o T.c \o T^2.\anc \o T.\pi_1 \o (1 \x_c c)\\
        = c \o T.c \o c \o T^2\anc \o c \o \pi_1  = c \o Tc \o T^2.\anc \o \pi_1
        = c \o T.\anc_T \o \pi_1.
    \end{gather*}
    The Yang--Baxter equation is straightforward to check.
\end{proof}
The tangent bundle is a canonical involution algebroid on every object in a tangent category, and the anchor induces a morphism from an involution algebroid to the tangent involution algebroid on its base space.
The anchor acts as a reflector from involution algebroids in $\C$ to $\C$ itself.
\begin{proposition}%
    \label{prop:c-refl-in-inv}
    Any tangent category $\C$ is a reflective subcategory of the category of involution algebroids in $\C$.
\end{proposition}
\begin{proof}
    First, observe that the inclusion of $\C$ into $\mathsf{Inv}(\C)$ (Definition \ref{def:involution-algd})\pagenote{The notation $\mathsf{Inv}(\C)$ had not yet been introduced at this point, it has been added to the definition of involution algebroids and a reference has been made to that definition.} is fully faithful because the anchor on the tangent involution algebroid is $id: TM \to TM$. This means that the only involution algebroid morphisms $TA \to TB$ are pairs $(Tf,f)$, $f: A \to B$.
    Now consider the functor $\mathsf{Inv}(\C) \to \C$ that sends $(\pi:A \to M, \xi, \lambda, \anc, \sigma)$ to $M$: this gives an endofunctor $S: \mathsf{Inv}(\C) \to \mathsf{Inv}(\C)$ along with a natural transformation $\anc: id \Rightarrow S$, so that $S.\anc = \anc.S = id$, given by the anchor map.
    Thus,  the category $\C$ is the category of algebras for an idempotent monad on $\mathsf{Inv}(\C)$.
\end{proof}
\begin{corollary}
    Let \[ \widehat{A} = (\pi:A \to M, \xi, \lambda, \anc, \sigma) \]   be an involution algebroid in a tangent category $\C$.
    \begin{enumerate}[(i)]
        \item The morphism
        \[(T.\pi,\pi): TA \to TM\]
        is an involution algebroid morphism from the tangent involution algebroid on $A$ to the tangent involution algebroid on $M$.
        \item The morphism
        \[(\anc,id):\widehat{A} \to TM\]  
        is an involution algebroid morphism.
    \end{enumerate}
    If pullback powers of $p \o \pi_1: \prol(A) \to A$ exist, then
    \[
        (p\o \pi_1: \prol(A) \to A, (\xi \o \pi, 0), (\lambda \x \ell))
    \] is a differential bundle\pagenote{A ``differentiable bundle'' slipped in.}; note that $\pi_0$ acts as an anchor. If the prolongations exist, then 
    \[
        (p \o \pi_1:\prol(A) \to A, (\xi \o \pi, 0), (\lambda \x \ell), \pi_0, 
        \sigma')
    \] is an involution algebroid; this follows from computing the pullback in the category of involution algebroids ($\sigma'$ is induced as in Observation \ref{obs:t-limits-of-anc}).
\end{corollary}
The above corollary puts an involution algebroid structure on the differential bundle $(\prol(A), \lambda \x \ell)$. Note that the map $\sigma$ gives an isomorphism of differential bundles
\[
  (\prol(A),\lambda \x \ell) \to (\prol(A),0 \x c \o T.\lambda).
\]
Martinez observed that the canonical involution $\sigma$ puts a unique Lie algebroid structure on $(\prol(A),0 \x c \o T.\lambda)$, and $\sigma$ is a uniquely determined isomorphism of involution algebroids (see \cite{Leon2005}):
\begin{corollary}\label{cor:second-lie-algd-unique-map}
    For an involution algebroid $(\pi:A \to M, \xi, \lambda, \anc, \sigma)$, the isomorphism of differential bundles
    \[
        \sigma: (\prol(A), 0 \x c \o T.\lambda) \to (\prol(A), \lambda \x \ell)
    \]
    induces a second involution algebroid structure on $\prol(A)$.
\end{corollary}

Recall that Proposition \ref{prop:anc-almost-tang-cat} sketched out a proof that the category of anchored bundles in $\C$ has an endofunctor $\prol'$ and natural transformations $p', 0', +', \ell'$ satisfying the axioms of a tangent structure; this endofunctor and the natural transformations all lift to $\mathsf{Inv}(\C)$. The involution map $\sigma$ is the missing piece that gives a tangent structure on $\mathsf{Inv}(\C)$. The construction may be spelled out here at a big-picture level, but the actual proof brings up tricky coherence issues that make up the bulk of Chapter \ref{chap:weil-nerve}.
\begin{proposition}%
    \label{prop:second-tangent-structure-inv-algds}
    The category of involution algebroids in a tangent category $\C$ has a second tangent structure, where the tangent functor is given by\pagenote{Changed wording to help differentiate this structure against the pointwise tangent structure defined beforehand.}
    \[
        \prol': \mathsf{Inv}(\C) \to \mathsf{Inv}(\C)  
    \]
    and the tangent natural transformations are given as in Proposition \ref{prop:anc-almost-tang-cat}, with the canonical flip
    \[
        \sigma': \prol'.\prol' \to \prol'.\prol' := \prol^2(A) \xrightarrow[]{1 \x T.\sigma} \prol^2(A).  
    \]
    Starting with an involution algebroid, $\prol'(A)$ and $\prol'.\prol'(A)$ are given by
    \[
       \prol'(A) 
       \begin{cases}
            \pi':& \prol(A) \xrightarrow[]{p \o \pi_1} A \\
            \xi':& A \xrightarrow[]{(\xi \o \pi, 0)} \prol(A) \\
            \lambda':& \prol(A) \xrightarrow[]{\lambda \x \ell} T.\prol(A) \\
            \anc':& \prol(A) \xrightarrow[]{\pi_1} TA \\
            \sigma':& \prol^2(A) \xrightarrow{\sigma \x c} \prol^2(A)
        \end{cases}  
        \hspace{0.15cm}
        \prol'.\prol'(A)
        \begin{cases}
            \pi''& \prol^2(A) \xrightarrow[]{(p \o \pi_1, p \o \pi_2):} \prol(A) \\
            \xi'':& \prol(A) \xrightarrow[]{(\xi \o \pi \o \pi_0, 0 \o \pi_1, 0 \o \pi_2)} \prol^2(A) \\
            \lambda'':&\prol^2(A) \xrightarrow[]{(\lambda \x \ell \x \ell)} T.\prol^2(A) \\
            \anc'':& \prol^2(A) \xrightarrow[]{(\pi_1, \pi_2)} T.\prol(A) \\
            \sigma'':& \prol^3(A) \xrightarrow[]{(\sigma \x c \x c)}\prol^3(A) 
        \end{cases}  
    \]
    (where $\prol^3(A) = \prolong \ts{T.\anc}{T^2.\pi} T.(\prolong)$).
    The tangent natural transformations are given by
    \begin{itemize}
        \item $p: \prol' \Rightarrow id; \hspace{0,15cm} \prol(A) \xrightarrow[]{\pi_0} A$ 
        \item $0: id \Rightarrow \prol'; \hspace{0,15cm} A \xrightarrow[]{(id,T.\xi \o T.\pi)} \prol'(A)$
        \item $+: \prol'_2 \Rightarrow \prol'; \hspace{0,15cm} A \ts{\anc}{T.\pi \o \pi_i} T_2A \xrightarrow[]{id \x +} \prol(A)$
        \item $\ell:\prol'(A) \Rightarrow \prol'.\prol'; \hspace{0,15cm} \prol(A) \xrightarrow[]{1 \x T.(\xi\o\pi, \lambda)} \prol^2(A)$
        \item $\anc': \prol' \Rightarrow T$
        \item $c:\prol'.\prol'(A) \Rightarrow \prol'.\prol';\hspace{0,15cm}  \prol^2(A) \xrightarrow[]{1 \x T.\sigma} \prol^2(A)$.
    \end{itemize}
\end{proposition}
\begin{proof}
    Deferred to \Cref{sec:prol_tang_struct}.
\end{proof}

\section{Connections on an involution algebroid}%
\label{sec:connections_on_an_involution_algebroid}
In this section we take an involution algebroid with a chosen linear connection on its underlying anchored bundle (Definition \ref{def:involution-algd}, \ref{def:lin-connection}), and rederive Martinez's structure equations for a Lie algebroid (Proposition \ref{prop:la-iff-structure-morphisms}). 

In a tangent category $\C$  with negatives, there is a natural transformation\pagenote{Clarifying notation for this section as there is substantially more fibered linear algebra, including the use of substraction, than is present in other sections/chapters.}
\[
    n:T \Rightarrow T  
\]
making each fibred commutative monoid $(p:TM \Rightarrow M, 0, +, n)$ a fibred abelian group. Because the additive bundle structure on differential bundles is induced via universality (Proposition \ref{prop:induce-abun}), in a tangent category with negatives the additive bundle structure for differential bundles will likewise have negatives.
We adopt the following notation for the ``fibred linear algebra'' used in this section, as there are a significant number of computations done on bundles with multiple choices of additive bundle structure (e.g. the second tangent bundle of a differential bundle has three additive bundles structures).
\begin{notation}
    Let $E$ be an object with multiple differential bundle structures $(\pi^i:E \to M^i, \xi^i, \lambda^i)$ in a tangent category with negatives. Addition over a specific differential bundle is written using infix notation, where a subscript is added to the symbol denoting the projection of the differential bundle. Letting $x,y:X \to E$ denote a pair of generalized elements for which the $\pi^i$-addition operation is well-defined; we set
    \[
        x +_{\pi[i]} y = X \xrightarrow[]{(x,y)} E \ts{\pi[i]}{\pi[i]} E \xrightarrow[]{+^i} E.
    \]
    Similar notation is used for subtraction:
    \[
        x -_{\pi[i]} y  = X 
        \xrightarrow[]{(x,y)} E \ts{\pi[i]}{\pi[i]} E 
        \xrightarrow[]{id \x n[i]} E \ts{\pi[i]}{\pi[i]} E 
        \xrightarrow[]{+^i} E.
    \]
\end{notation}

Throughout this section, we work in a tangent category $\C$ with negatives\pagenote{I had  originally forgotten to add the ``with negatives'' caveat} and a chosen anchored bundle $(\pi:A \to M, \xi, \lambda, \anc)$, equipped with a connection $(\kappa,\nabla)$, whose base object has a torsion-free connection $(\kappa',\nabla')$ and a morphism
\[
    \sigma: \prol(A) \to \prol(A)
\]
that exchanges the two projection maps $p \o \pi_1, \pi_0$, meaning that the following diagram commutes:
\begin{equation*}
    \begin{tikzcd}[column sep = small, row sep = small]
        & \prol(A) \ar{ld}[swap]{\pi_0} \ar{rd}{p \o \pi_1} \ar[dd,"\sigma" description] \\
    A & & A \\
        & \prol(A) \ar{lu}{p \o \pi_1} \ar{ru}[swap]{\pi_0} \\
\end{tikzcd}
\end{equation*}
The following notation will be useful when working with local coordinates.
\begin{notation}
    First, recall the $\nabla$-notation for morphisms of differential bundle where each morphism has a choice of connection from Equation \ref{eq:nabla-notation}:
    \[
        \infer{\nabla[f] := \kappa^B \o T.f \o \nabla^A:A_2 \to B}{f:A \to B & (\kappa^A,\nabla^A,A, \lambda^A) & (\kappa^B,\nabla^B, B, \lambda^B)}
    \]
    Let \[(\pi:A \to M, \xi, \lambda, \anc), (q:B \to N, \zeta, l, \rho)\] be a pair of anchored bundles equipped with connections.
    ``Hatting'' a map $f:\prol(A) \to \prol(B)$ refers to the map:
    \[
        \widehat{f}:A_3 \xrightarrow[]{\hat\nu(\pi_0,\pi_2) + \hat\nabla(\pi_0,\pi_1)} \prol(A) \xrightarrow[]{f} \prol(B) \xrightarrow[]{(\pi_0,p \o \pi_1, \kappa \o \pi_1)} B_3.  
    \]
    Similarly, for $f:TA \to TB$,
    \[
        \widehat{f}: TM \ts{p}{\pi} A \ts{\pi}{\pi} A \xrightarrow[]{\nu(\pi_0,\pi_2) + \nabla(\pi_0,\pi_1)} TA \xrightarrow[]{f} TB \xrightarrow[]{(\pi_0,p \o \pi_1, \kappa \o \pi_1)} TM \ts{p}{q} B \ts{\pi}{q} B  
    \]
    Clearly, $\widehat{f \o g} = \widehat{f} \o \widehat{g}$. Similarly, a map $A_3 \to B_3$ may be ``barred'' to form a map $\prol(A) \to \prol(B)$:
    \[
        \overline{g}:\prol(A) \xrightarrow[]{(\pi_0,p \o \pi_1, \kappa \o \pi_1)} A_3  \xrightarrow[]{g} B_3 \xrightarrow[]{\hat\nu(\pi_0,\pi_2) + \hat\nabla(\pi_0,\pi_1)} \prol(B).
    \]
    It is straightforward to see that $\overline{\widehat{f}} = f, \widehat{\overline{g}} = g$.
\end{notation}

\begin{lemma}\label{lem:intertwining-induces-bracket}
    $\sigma: \prol(A) \to \prol(A)$ induces a bracket on $\Gamma(\pi)$:
    \[
        \lambda \o [X,Y]
        = (\sigma \o (id, TX \o \anc) \o Y -_{T\pi} (id, TY \o \anc) \o X) - 0\o Y.
    \]    
\end{lemma}
\begin{proof}
    Let $X,Y \in \Gamma(\pi)$ and compute:
    \begin{align*}
         &\quad p \o \pi_1 \o (\sigma \o (id, TX \o \anc) \o Y -_{\pi_0} (id, TY \o \anc) \o X) \\
        &= p \o (\pi_1 \o \sigma \o (id, TX \o \anc) \o Y -_{T\pi} TY \o \anc \o X) \\
        &= p \o \pi_1 \o \sigma \o (id, TX \o \anc) \o Y -_{\pi} p\o TY \o \anc \o X \\
        &= \pi_0 \o (id, TX \o \anc) \o Y - Y \o p \o \anc \o X \\
        &= Y - Y = \xi, \\
         & \pi_0 \o ( \sigma \o (id, TX \o \anc) \o Y -_{\pi_0} (id, TY \o \anc) \o X) \\
        &= \pi_0 \o \sigma \o (id, TX \o \anc) \o Y -_{\pi} \pi_0 \o (id, TY \o \anc) \o X \\
        &= p \o \pi_1 \o (id, TX \o \anc) \o Y -_{\pi} X \\
        &= p \o TX \o \anc \o Y -_{\pi} X = X -_{\pi} X = \xi. 
    \end{align*}
    The universality of the lift induces a new section $[X,Y]$ so that
    \[
        \lambda \o [X,Y]
        = (\sigma \o (id, TX \o \anc) \o Y -_{T\pi} (id, TY \o \anc) \o X) - 0Y.
    \]    
\end{proof}

\begin{definition}\label{def:linear-intertwining}
    The map $\sigma:\prol(A) \to \prol(A)$ is \emph{linear} whenever the two bundle morphisms
    \[
        \sigma: (\prol(A),\lambda \x \ell) \to (\prol(A),0 \x c \o T\lambda)
        \hspace{0.25cm}
        \sigma: (\prol(A),0 \x c \o T\lambda) \to (\prol(A),\lambda \x \ell)
    \]
    are linear, and \emph{cosymmetric} if \[\sigma \o \widehat{\lambda} = \widehat{\lambda} = (\xi\pi,\lambda).\]
    Note that whenever $\sigma$ is linear and cosymmetric, $\sigma \o \hat\mu (u,v) = \hat\nu (u,v)$.
\end{definition}
Linearity and unit axioms, along with the connection on the differential bundle, force the existence of a bilinear bracket $\<-,-\>$ as in the definition of a Lie algebroid in Proposition \ref{prop:lie-alg-bil-defn}.
\begin{proposition}\label{prop:intertwining-linear-iff-bracket-linear}
    For an anchored bundle $(\pi:A \to M, \xi, \lambda, \anc)$ with connection $(\kappa, \nabla)$, a cosymmetric and bilinear $\sigma$ is equivalent to a bilinear $\<-,-\>:A_2 \to A$, with the correspondence given by
    \[
        \infer{\sigma: \prol(A)\xrightarrow[\overline{(\pi_1,\pi_0, \pi_2 + \< \pi_0,\pi_1 \>)}]{} \prol(A)}
        {\< -,-\>: \kappa \o \sigma \o \nabla -_\pi \kappa \o \nabla}
    \]
\end{proposition}
\begin{proof}
    Derive
    \begin{align*}
    \widehat{\sigma} \o (x,y,z)
        &= (y,x, \kappa \o \sigma \o (\nabla(\anc x,y) +_{p\o\pi_1} \mu(y, z)))\\
        &= (y,x, \kappa \o \sigma \o \nabla(\anc x, y) +_\pi \kappa \o \sigma \o \mu (y,z))\\
        &= (y,x, \kappa \o \sigma \o \nabla(\anc x, y) +_\pi \kappa \o \nu \o (y,z))\\
        &= (y,x, \kappa \o \sigma \o \nabla(\anc x, y) +_\pi z)\\
        &=: (y,x, \langle x, y\rangle +_\pi z )
    \end{align*}
    where $\<-,-\>$ is certainly bilinear. The converse is immediate: take
    \[
        \widehat{\sigma}(u,v,w) := (v,u,w + \< u,v\>).
    \]
    It is easy to see that $\sigma$ is linear and cosymmetric.    
\end{proof}
The linear bracket must be involutive for the bilinear bracket to be alternating.
\begin{proposition}\label{prop:intertwining-involutive}
    If $\sigma$ is cosymmetric and linear, then the bilinear bracket $\< -,-\>$ is alternating if and only if $\sigma\o \sigma = id$.
\end{proposition}
\begin{proof}
    First, note that $\sigma \o \sigma = id$ if and only if $\widehat{\sigma} \o \widehat{\sigma} = id$.
    Then check that
    \[
        \widehat{\sigma} \o \widehat{\sigma} (u,v,w) = \widehat{\sigma}(v,u,w+\<u,v\>) = (u,v,w + \< u,v \> + \< v, u\>).
    \]
    By the cancellativity of addition on $A$, 
    \[
    w = w+ \< u,v \> + \< v,u \> \iff 0 = \< u,v \> + \< v,u \>.
    \]
\end{proof}
\begin{observation}
    A bilinear $\<-,-\>$ on an anchored bundle with a connection induces the maps
    \[
        \{ v, x\}_{(\kappa, \nabla)} := \kappa' \o T.\anc \o \nabla(v, x), \quad
        \{v,x,y\}_{(\kappa, \nabla)} := \kappa \o T(\langle -,-\rangle_{(\kappa, \nabla)})\o \nabla^{A_2} \o (v,x,y)
    \]
    from \Cref{eq:lie-algd-structure-maps} in Proposition \ref{prop:lie-alg-bil-defn}.
\end{observation}
\begin{proposition}\label{prop:cosymmetric-leibniz}
    Let $\sigma$ be cosymmetric and bilinear, with the associated bracket $\<-,-\>$. Then
    \[
        T\anc \o \pi_1 \o \sigma = c \o T\anc \o \pi_1
    \] 
    if and only if the Leibniz equation is satisfied:
    \begin{equation}
        \label{eq:liebniz-local}
        \anc \o \langle u,v \rangle + \{ \anc v, u\} = \{ \anc u, v.\}
    \end{equation}
\end{proposition}
\begin{proof}
    Following the given notation and using the hypothesis that the connection is torsion-free on $M$, 
    \[
        \widehat{T\anc \o \pi_1}(u,v,w) = (\anc u, \anc v, w + \{u,v\})
        \hspace{0.5cm}
        \widehat{c}(x,y,z) = (y,x,z). 
    \]
    Computing each side,
    \begin{align*}
        \widehat{T\anc} \o \widehat{\pi_1} \o \widehat{\sigma} \o (u,v,w)
        &= \widehat{T\anc} \o \widehat{\pi_1} \o (v,u,w + \langle u,v\rangle) \\
        &= \widehat{T\anc} \o (\anc v,u,w + \langle u,v\rangle) \\
        &= (\anc v,\anc u, \anc w + \anc \langle u,v\rangle + D[\anc](\anc v, u)) \\
        &= (\anc v,\anc u, \anc w + \anc \langle u,v\rangle + \{ v, u\} ),\\
        \widehat{c} \o \widehat{T\anc} \o \widehat{\pi_1}(u,v,w)
        &= \widehat{c} (\anc u, \anc v, \anc w + \{u, v\}) \\
        &= (\anc v, \anc u, \anc w + \{ u, v\}),
    \end{align*}
    so it follows that the two terms are equal if and only if the desired equality holds.
\end{proof}

\begin{lemma}\label{lem:tang-of-sigma}
    Let $\sigma$ be linear and cosymmetric.
    Then
    \begin{enumerate}[(i)]
        \item $T(\<-,-\>):TA_2 \to TA$ satisfies 
        \begin{align*}
            &\quad \widehat{T(\<-,-\>)}(a_x, u_y, u_z,u_{xy}, u_{xz})\\
            &= (a_x, \< u_y, u_z \>, \{a_x, u_y, y_z\} + \<u_y, u_{xz}\> + \langle u_{xy}, u_{z}\rangle); 
        \end{align*}
        \item $T.\sigma$ satisfies
        \begin{align*}
                &\quad \widehat{(id \x T(\widehat{\sigma}))}(u_{x}, u_{y}, u_{xy}, u_{z}, u_{xz}, u_{yz}, u_{xyz})\\
                &= (u_x, u_z, u_{xz}, u_y, u_{xy},  u_{yz} + \< u_y, u_z\>, \\&\quad u_{xyz} + \{a_x, u_y, y_z\} + \<u_y, u_{xz}\> 
                + \langle u_{xy}, u_{z}\rangle).
        \end{align*}
    \end{enumerate}
\end{lemma}
\begin{proof}~
    \begin{enumerate}[(i)]
        \item By the universality of the vertical lift and bilinearity of $\<-,-\>$, the outer squares below are pullbacks:
            \begin{equation*}
\begin{tikzcd}
    A_4 \ar{rd}{\mu^{A_2}} \ar[dotted]{rrr}{(\< \pi_0, \pi_1\>, \langle\pi_0,\pi_3\rangle + \langle \pi_1, \pi_2\rangle)} \ar{ddd}  
    & & & A_2 \ar{ddd} \ar{ld}{\mu^{A}} \\
    & T(A) \rar{T\langle,\rangle} \dar{T\pi \o T\pi_i} & T(A)\dar{T\pi}\\
    & TM \rar[equals] & TM \\
    M \ar{ru}{0} \ar[equals]{rrr} & & & M \ar{lu}{0}
\end{tikzcd}
\end{equation*}
            so that 
               \[T(\langle,\rangle)(0,u_y, u_z, u_{xy},u_{xz}) 
               = \mu^A(\< u_y,u_z \>, \<u_y, u_{xz}\> + \langle u_{xy}, u_{z}\rangle).\]
               Now we compute
               \begin{align*}
                   &\quad T(\langle,\rangle)(\mu^{A_2}((u_y,y_z), (u_{xy}, u_{xz})) +_p \nabla^{A_2}(a_x, (u_y,u_z))) \\
                   &= T(\langle,\rangle)\mu^{A_2}((u_y, u_z),(u_{xy}, u_{xz})) +_p T(\langle,\rangle)\o\nabla(a_x, (u_y,u_z))\\
                   &= \mu^A(\< u_y,u_z \>, \<u_y, u_{xz}\> + \langle u_{xy}, u_{z}\rangle) 
                   +_p T(\langle,\rangle)\o\nabla(a_x, (u_y,u_z)) 
               \end{align*}                   
               and then postcompose this with $(T\pi, p, \kappa)$ to obtain
               \begin{align*}
                   &\quad (0, \< u_y,u_z \>, \ \<u_y, u_{xz}\> + \langle u_{xy}, u_{z}\rangle) 
                   +_p (a_x, \< u_y, u_z \>, \{a_x, u_y, y_z\})\\
                   &=  
                   (a_x, \< u_y, u_z \>, \{a_x, u_y, y_z\} +  \<u_y, u_{xz}\> + \< u_{xy}, u_{z}\>). 
               \end{align*}
        \item Consider the following diagram:
               \begin{equation*}
\begin{tikzcd}
    TM \ts{p}{\pi} A_3 \ts{\pi}{\pi} A_3\rar{\widehat{T\widehat{\sigma}}} \dar[swap]{\nabla^{A^3} +_p \mu^{A^3}} 
    & TM \ts{p}{\pi} A_3 \ts{\pi}{\pi} A_3 \\
    T(A_3) \rar{T\widehat{\sigma}}[swap]{T(\pi_1, \pi_0, \pi_2 + \langle \pi_0, \pi_1 \rangle)}
    \dar[swap]{T(\nabla' +_{\pi_0} \mu'))} 
    & T(A_3) \ar{u}[swap]{(T(\pi^3),p, \kappa^{A_3})} \\
    T\o\prol(A) \rar{T\sigma} & T\o\prol(A) \ar{u}[swap]{T(\pi_0, p \o \pi_1, \kappa \o\pi_1)}
\end{tikzcd}
\end{equation*}
            We want to find $\widehat{T\widehat{\sigma}} = \widehat{T(\pi_1, \pi_0, \pi_2 + \< \pi_0, \pi_1 \> )}$.
            Note that
            \[
                T(\pi_1, \pi_0, \pi_2 +_\pi \< \pi_0, \pi_1 \>)
                = T(\pi_1, \pi_0, \pi_2) +_{T\pi} T(\xi\pi\pi_i, \pi_0, \< \pi_0, \pi_1\>).
            \]
            The left term is straightforward:
            \begin{align*}
                &\widehat{T(\pi_0, \pi_1, \pi_2)}(a_x, (u_y, u_{xy}), (u_z, u_{xz}), (u_{yz}, u_{xyz}))\\
                &=(a_x,  (u_z, u_{xz}), (u_y, u_{xy}),(u_{yz}, u_{xyz}))
            \end{align*}
            and for the right term, use part (i) of this lemma:
            \begin{align*}
                &\quad \widehat{T\<-,-\>}\o \widehat{T(\pi_0, \pi_1))}(a_x, (u_y, u_{xy}), (u_z, u_{xz}), (u_{yz}, u_{xyz})) \\
                &= 
                \widehat{T\<-,-\>}(a_x, u_y, u_{xy}, u_z, u_{xyz}) \\
                &= (a_x, \< u_y, u_z \>, 
                \{a_x, u_y, y_z\} + \<u_y, u_{xz}\> + \< u_{xy}, u_{z}\>) 
            \end{align*}

            Then compute
            \begin{align*}
                &\quad \widehat{T\widehat{\sigma}}(a_x, u_y, u_{xy}, u_z, u_{xz}, u_{yz}, u_{xyz}) \\
                &= (a_x, u_z, u_{xz}, u_y, u_{xy}, u_{yz} + \< u_y, u_z\>, u_{xyz} + q)
            \end{align*}
            where $q = \{a_x, u_y, y_z\} + \<u_y, u_{xz}\> + \langle u_{xy}, u_{z}\rangle$, 
            giving the desired equation.
    \end{enumerate}
\end{proof}

\begin{proposition}\label{prop:yang-baxter-iff-bianchi}
    Let $\sigma$ be cosymmetric, doubly linear, and involutive, and satisfy the target axiom.
    Then $\sigma$ satisfies the Yang--Baxter equation if and only if $\<-,-\>$ and $\{-,-,-\}$ satisfy the Bianchi identity: 
    \begin{equation}
        \label{eq:bianchi-identity}
        0
        =\sum_{\gamma \in \mathsf{Cy}(3)} \langle x_{\gamma_0}, \langle x_{\gamma_1}, x_{\gamma_2}\rangle\rangle 
        + \sum_{\gamma \in \mathsf{Cy}(3)} \{ \anc x_{\gamma_0}, x_{\gamma_1}, x_{\gamma_2}.\}
    \end{equation}
\end{proposition}
\begin{proof}
    We expand $\widehat{id \x T\sigma}$ and  $\widehat{\sigma \x c}$.
    Start with $T(id \x \sigma)$, which was derived in Lemma \ref{lem:tang-of-sigma}:
    \begin{align*}
        \sigma_1(u) &= \widehat{(id \x T(\widehat{\sigma}))}(u_{x}, u_{y}, u_{xy}, u_{z}, u_{xz}, u_{yz}, u_{xyz})\\
        &= (u_x, u_z, u_{xz}, u_y, u_{xy}
        ,  u_{yz} + \< u_y, u_z\>, \\
        &\quad u_{xyz} + \langle u_y, u_{xz} \rangle  + \langle u_{xy}, u_z \rangle  + \{  \anc \o u_x, u_y, u_z \}).
    \end{align*}
    Using the fact that $\kappa'$ is torsion free, so 
    $\hat{c}(u_{xz}, u_{yz}, u_{xzy}) = (u_{yz}, u_{xz}, u_{xyz})$, it follows that
    \[
       \sigma_2(u_x, u_y, u_{xy}, u_z, u_{xz}, u_{yz}, u_{xyz})
        = (u_y, u_x, u_{xy}+ \< u_x, u_y \>, u_z, u_{yz}, u_{xz}, u_{xyz}).
    \]
    Then compute
    \begin{align*}
         &\quad \sigma_2\sigma_1\sigma_2(u) \\
        &= \big( u_z
                , u_y
                , u_{yz} + \langle u_y, u_z \rangle 
                , u_x
                , u_{xz} + \langle u_x, u_z \rangle 
                , u_{xy} + \langle u_x, u_y \rangle , z_1 \big),
                \\
        &\quad  \sigma_1\sigma_2\sigma_1(u) \\
        &= \big( u_z
              , u_y
              , u_{yz} + \langle u_y, u_z \rangle 
              , u_x
              , u_{xz} + \langle u_x, u_z \rangle 
              , u_{xy} + \langle u_x, u_y \rangle , z_2\big).
    \end{align*}
   Note that the first five terms are equal, so it suffices to check $z_1 = z_2$ for $z_1 = \pi_6\sigma_1\sigma_2\sigma_1(u), z_2 = \pi_6\sigma_2\sigma_1\sigma_2(u)$.
    \begin{align*}
    z_1 &= u_{xyz} + \langle u_y, u_{xz} \rangle  + \langle u_{xy}, u_z \rangle  
    +    \{  \anc \o u_x, u_y, u_z \} + \{  \anc \o u_z, u_x, u_y \}  \\
    &\quad +\langle u_x, u_{yz} + \langle u_y, u_z \rangle  \rangle  
    + \langle u_{xz} + \langle u_x, u_z \rangle , u_y \rangle \\ 
    &= u_{xyz} + \langle u_y, u_{xz} \rangle  + \langle u_{xy}, u_z \rangle  + \{  \anc \o u_x, u_y, u_z \} + \{  \anc \o u_z, u_x, u_y \}  \\
    &\quad +\langle u_x, u_{yz} \rangle + \langle u_x, \langle u_y, u_z \rangle  \rangle
    + \langle u_{xz}, u_y \rangle + \langle \langle u_x, u_z \rangle , u_y \rangle \\ 
    &= u_{xyz} + \langle u_{xy}, u_z \rangle  + \{  \anc \o u_x, u_y, u_z \} + \{  \anc \o u_z, u_x, u_y \}  \\
    &\quad +\langle u_x, u_{yz} \rangle + \langle u_x, \langle u_y, u_z \rangle  \rangle
    + \langle \langle u_x, u_z \rangle , u_y \rangle, \\ 
    z_2 &= u_{xyz} 
    + \langle u_x, u_{yz} \rangle + \{  \anc \o u_y, u_x, u_z \} + \langle u_{xy} + \langle u_x, u_y \rangle , u_z \rangle \\
    &= u_{xyz} + \langle u_x, u_{yz} \rangle + \{  \anc \o u_y, u_x, u_z \}    + \langle u_{xy} , u_z \rangle + \langle \langle u_x, u_y \rangle , u_z \rangle .
    \end{align*}    
    So $z_1 = z_2$ is equivalent to requiring
    \begin{align*}
     &\quad  u_{xyz} + \langle u_{xy}, u_z \rangle + \{  \anc \o u_x, u_y, u_z \} + \{  \anc \o u_z, u_x, u_y \}  \\
     &\quad +\langle u_x, u_{yz} \rangle + \langle u_x, \langle u_y, u_z \rangle  \rangle
    + \langle \langle u_x, u_z \rangle , u_y \rangle \\ 
    &= u_{xyz} + \langle u_x, u_{yz} \rangle + \{  \anc \o u_y, u_x, u_z \}     + \langle u_{xy} , u_z \rangle + \langle \langle u_x, u_y \rangle , u_z \rangle;
    \end{align*}
    cancelling alike terms, this is equivalent to
    \begin{align*}
     &\quad \langle u_x, \langle u_y, u_z \rangle  \rangle + \langle \langle u_x, u_z \rangle , u_y \rangle + \{  \anc \o u_x, u_y, u_z \} 
    + \{  \anc \o u_z, u_x, u_y \} \\
    &=  \{  \anc \o u_y, u_x, u_z \}   +  \langle \langle u_x, u_y \rangle , u_z \rangle.
    \end{align*}
    Using the fact that $\< -, -\>$ and $\{-,-,-\}$ are alternating in the last two arguments, this is equivalent to
    \begin{align*}
        0&=\langle u_x, \langle u_y, u_z \rangle  \rangle + \langle \langle u_x, u_z \rangle , u_y \rangle +  \langle u_z , \langle u_x, u_y \rangle \rangle \\
         &\quad + \{  \anc \o u_x, u_y, u_z \} + \{  \anc \o u_z, u_x, u_y \} + \{  \anc \o u_y, u_z, u_x \} \\
         &=\langle u_x, \langle u_y, u_z \rangle  \rangle + \langle u_y, \langle u_z, u_x \rangle\rangle +  \langle u_z , \langle u_x, u_y \rangle \rangle \\
         &\quad + \{  \anc \o u_x, u_y, u_z \} + \{  \anc \o u_z, u_x, u_y \} + \{  \anc \o u_y, u_z, u_x \} 
    \end{align*}
    giving the desired identity.
\end{proof}

\begin{corollary}\label{cor:inv-con-def}
    Let $(\pi:A \to M, \xi, \lambda, \anc)$ be an anchored bundle in a tangent category with negatives, with anchored connection $(\nabla,\kappa)$ on $A$ and torsion-free affine connection $(\nabla',\kappa')$ on $M$.
    An involution algebroid structure on $A$ is equivalent to a bilinear map
    \[
        \<-,-\>:A_2 \to A 
    \]
    with derived maps 
    \begin{gather*}
        \{-,-\}: A \ts{\pi}{p} TM \to TM := \kappa' \o T\anc \o (\pi_0,\pi_1), \\
        \{-,-,-\}: TM \ts{p}{\pi \o\pi_i} A_2 \to A; \{a, u_1, u_2\} := \kappa \o T(\<-,-\>) \o (\nabla(a,u_1), \nabla(a,u_2))
    \end{gather*}
    satisfying 
    \begin{enumerate}[(i)]
        \item $\< -, -\>$ is linear and cosymmetric,
        \item $\< -, -\>$ is alternating,
        \item $\< -, - \>$ and $\{-,-\}$ satisfy the Leibniz equation, Equation \ref{eq:liebniz-local}.
        \item $\< -, -\>, \{-,-\},$ and $\{-,-,-\}$ satisfy the Bianchi identity, \Cref{eq:bianchi-identity}.
    \end{enumerate}
\end{corollary}
Morphisms of involution algebroids may also be characterized by preservation of the tensor.
\begin{proposition}%
    \label{prop:inv-algd-mor-conn}
    Let $A, B$ be a pair of involution algebroids with chosen connections in a tangent category with negatives.
    Then an anchored bundle morphism $f: A \to B$ is an involution algebroid morphism if and only if (recalling the notation from Equation \ref{eq:nabla-notation})
    \[
        \nabla[f](x,y) + \< f\o x, f\o y\> = \nabla[f](y,x) + f \o \< x, y\>.
    \]
\end{proposition}
\begin{proof}
    Note that
    \begin{gather*}
        (\sigma \o \prol(f) = \prol(f) \o \sigma) \\
        \iff (\pi_1,\pi_0, \pi_2 + \< \pi_0, \pi_1 \> ) \o (f\o x, f\o y, f \o z + \nabla[f](x,y))
        \\
        = (f, f, f \o \pi_2 + \nabla[f](\pi_0,\pi_1)) \o (y,x, z + \< x, y \> ) 
    \end{gather*}
    while the second condition reduces to
    \[
        \nabla[f](x,y) + \< f\o x, f\o y\> = \nabla[f](y,x) + f \o \< x, y\>.
    \]
\end{proof}

\section{The isomorphism of Lie and involution algebroid categories}
\label{sec:the-isomorphism-of-categories}

Sections \ref{sec:Lie_algebroids} and \ref{sec:connections_on_an_involution_algebroid} have made the relationship between involution algebroids and Lie algebroids clear. It is important to note that while the proofs used connections as a tool to identify the local coherences satisfied by involution and Lie algebroids, the construction of a Lie algebroid from an involution algebroid (or vice versa) is independent of the choice of connection.

\begin{theorem}%
    \label{thm:iso-of-cats-Lie}
    There is an isomorphism of categories between Lie algebroids and involution algebroids in smooth manifolds.
\end{theorem}
\begin{proof}
    For the equivalence of categories, note that by Propositions \ref{prop:lie-alg-bil-defn} and \ref{prop:lie-algd-morphism-defn}, Corollary \ref{cor:inv-con-def}, and Proposition \ref{prop:inv-algd-mor-conn} there is an isomorphism of categories between involution algebroids with a choice of connection and Lie algebroids with a choice of connection (morphisms are \emph{not} restricted to connection preserving morphisms). This allows us to chain together isomorphisms
    \[
        \mathsf{Inv}(\mathsf{SMan}) 
        \cong \mathsf{Inv}(\mathsf{SMan})_{\mathsf{ChosenConn}}
        \cong \mathsf{LieAlgd}_{\mathsf{ChosenConn}}
        \cong \mathsf{LieAlgd}.
    \]
    \pagenote{One thing that was unclear in the original proof of this statement is that while connections are helpful for the nuts and bolts of this proof, the actual constructions do not rely on connections.}

    To complete the proof, we must show that the assignment that sends an involution algebroid to a Lie algebroid whose bracket is given by
    \begin{equation}\label{eq:inv-to-lie}
                \lambda \o [X,Y]^* = 
        \left( (\pi_1 \o \sigma \o (id, T.X \o \anc) \o Y -_p T.Y \o \anc \o) -_{T.\pi} 0 \o Y
        \right),
    \end{equation}
    is a bijection on objects, which brings up some subtleties. First, while an involution map 
    \[
        \sigma:\prolong \to \prolong
    \]
    is defined with respect to a particular choice of pullback $\prolong$, the category of involution algebroids \emph{does not} distinguish between different choices of this pullback (and therefore different representations of the map $\sigma$), and it is not part of the data of an involution algebroid. It is immediate by universality that the left-hand-side of Equation \ref{eq:inv-to-lie} is independent of the choice of pullback $\prolong$.
    
    Now, recall that the canonical involution of a Lie algebroid is uniquely by Theorem 4.7 of \cite{de2005lagrangian} (this was also mentioned in Corollary \ref{cor:second-lie-algd-unique-map}). Once we make a choice of prolongation $\prolong$, we have made a choice of pullback $\prol(A)$ in $\mathsf{LieAlgd}$, which uniquely determines the canonical involution
    \[
    \sigma: \prolong \to \prolong.
    \]
    While the exact map $\sigma$ depends on the choice of pullback $\prolong$, they all determine the same involution algebroid, thus proving the bijection correspondence of objects.
\end{proof}



%

\chapter{The Weil nerve of an algebroid}%
\label{chap:weil-nerve}\pagenote{
   This chapter has been given a new introduction to help with its exposition.
   Various typos have been fixed, but the most substantial changes are in Section \ref{sec:weil-nerve}, where the proof has been restructured to address some concerns brought up in Michael's comments.
}

The first three chapters of this thesis demonstrated that tangent categories allow for an essentially algebraic description of Lie algebroids by axiomatizing the behaviour of the tangent bundle, and showing that a Lie algebroid over a manifold $M$ is a ``generalized tangent bundle'', namely an \emph{involution algebroid}. This chapter will make precise the sense in which an involution algebroid is a generalized tangent bundle, by showing that the category of involution algebroids in a tangent category $\C$ is equivalent to a certain tangent-functor category from the free tangent category over a single object to $\C$, or more generally that there is a fully faithful functor
\[
	\mathsf{Inv}(\C) \hookrightarrow \mathsf{Tang_{Lax}}(\mathsf{FreeTangCat}(*),\C).
\]
This functor, the \emph{Weil nerve} of an involution algebroid, builds a functor from the free tangent category over a single object to $\C$ using a span construction. This chapter primarily builds on two pieces of work: Leung's construction of the free tangent category $\wone$ (\cite{Leung2017}) , and Grothendieck's original nerve construction (first published in \cite{Segal1974}).

To understand Leung's construction of the free tangent category, and more generally his actegory-theoretic presentation of tangent categories (Section \ref{sec:tang-struct-as-wone}), we first look at Weil's original insight relating the kinematic and operational descriptions of the tangent bundle in $\mathsf{SMan}$. The definition of a tangent vector on a manifold $M$ as an equivalence class of curves (Definition \ref{def:tang-vector})  puts a bijective correspondence between tangent vectors and $\R$-algebra homomorphisms from the ring of smooth functions $C^\infty(M)$ to the ring of dual numbers:
\[
	C^\infty(M) \to \R[x]/x^2.
\] The hom-set $\R\mathsf{Alg}(C^\infty(M),\R[x]/x^2)$ is precisely the set of \emph{derivations} on $C^\infty(M)$, which defines the operational tangent bundle discussed in Definition \ref{def:operational-tang}: there is a natural smooth manifold structure on this set. The Weil functor formalism, most notably developed in \cite{Kolar1993}, extends this observation to a general class of endofunctors on $\mathsf{SMan}$. For example, the fibre product  $T_2M$ corresponds to $\R$-algebra morphisms,
\[
C^\infty(M) \to R[x,y]/(x^2, y^2, xy), 
\] while the second tangent bundle corresponds to $\R$-algebra morphisms into the tensor product,
\[
C^\infty(M) \to R[x]/(x^2) \ox R[y]/(y^2) = R[x,y]/(x^2,y^2).
\] By applying Milnor's exercise (Problem 1-C \cite{Milnor1974}), which states that the $C^\infty$ functor
\[
    \mathsf{SMan} \to \R\mathsf{Alg^{op}};\hspace{0.15cm} M \mapsto C^\infty(M)=\mathsf{SMan}(M,\R)
\] is fully faithful, the structure maps occur as natural transformations. For example, the tangent projection is induced by the morphism
\[
	{p}: \R[x]/x^2 \xrightarrow[]{a + bx \mapsto a} \R,  
\] 
so that 
\[
    TM \xrightarrow[]{p} M = [C^\infty(M),\R[x]/x^2] \xrightarrow[]{(p)_*} [C^\infty(M), \R].
\]
The zero map and addition are similarly induced by
\[
	{0}: \R \xrightarrow[]{a \mapsto a + 0x} R[x]/x^2 \text{ and } +:R[x,y]/(x^2,y^2,xy) \xrightarrow[\mapsto a + (b+c)x]{a + bx + cy} R[x]/x^2,
\] 
respectively, while the lift and flip are induced by the morphisms
\[
	{\ell}: \R[x]/x^2 \xrightarrow[\mapsto a + bxy]{a+bx} \R[x,y]/(x^2,y^2) 
\]
and
\[
    {c}: \R[x,y]/(x^2,y^2) \xrightarrow[\mapsto a + cx + by + dxy]{a + bx + cy + dxy} \R[x,y]/(x^2,y^2).
\] 
More generally, there is a monoidal category of \emph{Weil algebras} (Definition \ref{def:weil-algebra-and-prol}) which has a monoidal action on the category of smooth manifolds. The Weil functor formalism, then, studies differential geometric structures from the perspective of the endofunctors and natural transformations induced by this action. Leung's insight is that there is an analogous category of commutative rigs\footnote{A rig is a ri\emph{n}g without \emph{n}egatives, i.e. a commutative monoid equipped with a bilinear multiplication.} built by replacing $\R[x]/x^2$ with $\N[x]/x^2$, called $\wone$ (Definition \ref{def:Weil-algebra}); a tangent structure is precisely a monoidal action by $\wone$ satisfying some universal properties. In particular, this category $\wone$ is precisely the free tangent category over a point, $\mathsf{FreeTang}(*)$, so that every object $A$ in a tangent category $\C$ determines a strict tangent functor
\[
    T^{(-)}A: \wone \to \C; V \mapsto T^VA
\]
and morphisms $f:A \to B$ are in bijective correspondence with tangent-natural transformations $T^{(-)}A \Rightarrow T^{(-)}(B)$.

The axioms of an involution algebroid in a tangent category $\C$ correspond bijectively with those of the tangent bundle - this suggests that an involution algebroid should determine a tangent functor from $\wone$ to $\C$. A first guess would lead one to think that $p:\N[x]/x^2 \to \N$ is sent to $\pi:A \to M$, $0$ to $\xi$, and $+$ to $+_A$. As the space of prolongations $\prol(A) = \prolong$ plays the role of the second tangent bundle, we can see that
\[
  \ell:N[x]/x^2 \to N[x,y]/(x^2,y^2) \mapsto (\xi\o\pi,\lambda): A \to \prol(A), 
\]
and 
\[
  c:N[x,y]/(x^2,y^2) \to N[x,y]/(x^2,y^2) \mapsto \sigma \prol(A) \to \prol(A).
\]
This pattern may be neatly summed up using span composition - we will construct a functor that sends $\N[x]/x^2$ to the span
\[\begin{tikzcd}
	& A \\
	M && TM
	\arrow["\anc"{description}, from=1-2, to=2-3]
	\arrow["\pi"{description}, from=1-2, to=2-1]
\end{tikzcd}\]
and the tensor product $\N[x]/x^2 \ox \N[x]/x^2$ to the span composition (e.g. the pullback)
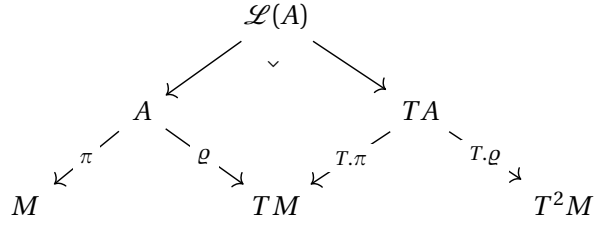
\[\begin{tikzcd}
	&& {\prol(A)} \\
	& A && TA \\
	M && TM && {T^2M}
	\arrow["\anc"{description}, from=2-2, to=3-3]
	\arrow["\pi"{description}, from=2-2, to=3-1]
	\arrow["{T.\pi}"{description}, from=2-4, to=3-3]
	\arrow["{T.\anc}"{description}, from=2-4, to=3-5]
	\arrow[from=1-3, to=2-2]
	\arrow[from=1-3, to=2-4]
	\arrow["\lrcorner"{anchor=center, pos=0.125, rotate=-45}, draw=none, from=1-3, to=3-3]
\end{tikzcd}\]
which is the space of prolongations. This leads to the first major result of this chapter, the \emph{Weil Nerve} (Theorem \ref{thm:weil-nerve}), which states there is a fully faithful embedding
\[
    N_{\mathsf{Weil}}:\mathsf{Inv}(\C) \hookrightarrow [\wone, \C].
\]
This bears a strong similarity to Grothendieck's original nerve theorem, which takes an internal category $s,t:C \to M$ and constructs a functor $\Delta^{op} \to \C$ (where $\Delta^{op}$ is the monoidal theory of an internal monoid) by sending tensor to span composition, and the composition and unit maps given span morphisms
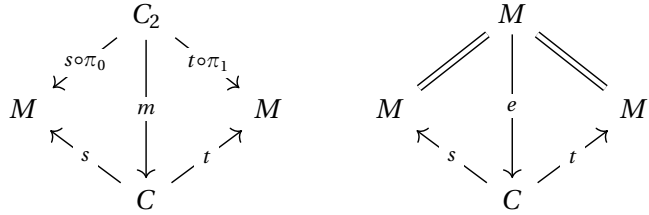
\[\begin{tikzcd}
	& {C_2} &&& M \\
	M && M & M && M \\
	& C &&& C
	\arrow["{s \o \pi_0}"{description}, from=1-2, to=2-1]
	\arrow["{t \o \pi_1}"{description}, from=1-2, to=2-3]
	\arrow["s"{description}, from=3-2, to=2-1]
	\arrow["t"{description}, from=3-2, to=2-3]
	\arrow["m"{description}, from=1-2, to=3-2]
	\arrow[Rightarrow, no head, from=1-5, to=2-4]
	\arrow[Rightarrow, no head, from=1-5, to=2-6]
	\arrow["s"{description}, from=3-5, to=2-4]
	\arrow["t"{description}, from=3-5, to=2-6]
	\arrow["e"{description}, from=1-5, to=3-5]
\end{tikzcd}\]
while the unit and associativity axioms for a category are exactly the unit and associativity laws for a monoid in this setting. The \emph{Segal conditions} identify exactly the functors $C:\Delta^{op} \to \C$ that lie in the image of the nerve functor $N$ as those whose $[n]^{th}$ object is sent to the wide pullback $C([n]) = C[2] \ts{t}{s} C[2]\dots \ts{t}{s} C[2]$. The corresponding result for involution algebroids is found in Theorem \ref{thm:iso-of-cats-inv-emcs}, which states that a tangent functor $(A,\alpha):\wone \to \C$ is the nerve of an involution algebroid if and only if $A$ preserves tangent limits and $\alpha$ is a $T$-cartesian natural transformation (this forces $A.(W^{\ox n}) = A \ts{\anc}{T.\pi} TA \dots \ts{T^{n-1}\anc}{T^n\pi} T^nA$). The similarity between the Weil nerve and Grothendieck/Segal's nerve runs deep, and in Chapter \ref{ch:inf-nerve-and-realization} we demonstrate that the \emph{enriched} perspective on tangent categories puts these both into the same formal framework.

Sections \ref{sec:weil-algebras-tangent-structure} and \ref{sec:tang-struct-as-wone} give a detailed introduction to the Weil functor formalism (\cite{Kolar1993},\cite{Bertram2014a}) and Leung's unification of Weil functors with tangent categories \cite{Leung2017}. The rest of the chapter contains contains new results developed by the author. Section \ref{sec:weil-nerve} proves the embedding part of the Weil nerve theorem, that the category of involution algebroids embeds into the category of tangent functors and tangent natural transformations $[\wone, \C]$. Section \ref{sec:identifying-involution-algebroids} identifies exactly those tangent functors $(A,\alpha):\wone \to \C$ that are the nerve of an involution algebroid, completing the proof of the Weil nerve theorem. Section \ref{sec:prol_tang_struct} uses the Weil nerve to develop a novel tangent structure on the category of involution algebroids in a tangent category (in particular, the category of Lie algebroids will have this novel tangent structure).

\section{Weil algebras and tangent structure}%
\label{sec:weil-algebras-tangent-structure}

This section gives a more thorough introduction to the Weil functor formalism of \cite{Kolar1993}, and in particular how the structure maps of a tangent category may be teased out of it. We begin by introducing  Weil algebras, and the \emph{prolongation} of a smooth manifold by a Weil algebra. (The relationship with prolongations of involution algebroids from Definition \ref{def:anchored_bundles} will be made clear in Section \ref{sec:weil-nerve}.)

\begin{definition}%
    \label{def:weil-algebra-and-prol}
    An $\R$-Weil algebra\footnote{Not to be confused with the normal usage of ``Weil algebra'' in Lie theory, e.g. \cite{Meinrenken2019}.} is a finite-dimensional $\R$-algebra $V$ so that $V = \R \oplus \dot{V}$ as $\R$-modules and $\dot{V}$ is a nilpotent ideal. The category $\R\mathsf{Weil}$ is the full subcategory of $\R\mathsf{Alg}$ spanned by the $\R$-Weil algebras. The \emph{prolongation} of a manifold by a Weil algebra $V$ is given by the manifold
    \[
        T^VM := \R\mathsf{Alg}(C^\infty(M,\R), V).
    \]
\end{definition}
\cite{Eck1986} showed that every product-preserving endofunctor on the category of smooth manifolds is constructed as the Weil prolongation by some Weil algebra. Consequently, $\R$-algebra homomorphisms induce natural transformations between these product-preserving  endofunctors on the category of smooth manifolds.
\begin{example}\label{ex:weil-algebras-and-maps}
    Consider the following $\R$-Weil algebras and their associated prolongation functors.
    \begin{enumerate}[(i)]
        \item Prolongation by $\R$ induces the identity functor, and the tangent bundle is given by $\R[x]/x^2$. The tangent projection, then, is equivalent to the $\R$-algebra morphism
        \[
            p: \R[x]/x^2 \to \R; \hspace{0.2cm}
            p(a + bx) = a
        \]
        while the 0-map induces the zero vector field:
        \[
            0: \R \to \R[x]/x^2; \hspace{0.2cm}
            0(a) = a + 0x.
        \]
        \item The algebra $\R[x_i]_{1 \le i \le n}/(x_ix_j)_{1 \le i \le j \le n} = (\R[x]/x^2)^n$ is the wide pullback $T_nM = TM \ts{p}{p} TM \dots \ts{p}{p} TM$. 
        In particular, prolongation by $\R[x,y]/(x^2,y^2,xy)$ gives the bundle $T_2M = TM \ts{p}{p} TM$. The $\R$-algebra morphism
        \[
            +:\R[x,y]/(x^2,y^2,xy) \to R[x]/x^2; \hspace{0.2cm}
            +(a_0 + a_1x + a_2y) = a_0 + (a_1 + a_2)x
        \]corresponds to the addition of tangent vectors.
        \item The algebra $\R[x,y]/(x^2,y^2) = (\R[x]/x^2)\ox (\R[x]/x^2)$ is the second tangent bundle $T^2M = TTM$. The vertical lift $T \to T^2$ is induced by the morphism 
        \[
            \ell: \R[x]/x^2 \to \R[x,y]/(x^2,y^2);\hspace{0.2cm}
            \ell(a + bx) = a + bxy.
        \]
        \item The monoidal symmetry map induces $c: T^2 \Rightarrow T^2$, as follows: 
        \begin{gather*}
                        c: (\R[x]/x^2)\ox (\R[y]/y^2) \to (\R[y]/y^2)\ox (\R[x]/x^2);\\ c(a + b_1x + b_2y + b_3xy) = a + b_2x + b_1y + b_3xy.        
        \end{gather*}
        \item For $n \ge 2$, the algebra $\R[x]/x^n$ gives the $n$-jet bundle.
        Note that this is the equalizer of $\ox^n \R[x]/x^2$ by the symmetry actions of $S_n$.
    \end{enumerate}
\end{example}
Further examples may be found in the monograph \cite{Kolar1993}. Tangent categories bridge the gap between the Weil functor approach to studying the differential geometry of smooth manifolds and the synthetic differential geometry approach of axiomatizing a tangent bundle using nilpotent infinitesimals. The main structure axiomatized here is that of \emph{monoidal action} of a symmetric monoidal category on a category $M \x \C \to \C$, or equivalently, a lift from a category to the category of complexes $\C \to [M,\C]$, which involves translating a bit of classical category theory to the 2-categorical setting.

Example  \ref{ex:weil-algebras-and-maps} leads to the classical theorem that the category of smooth manifolds has an action by the category of $\R$-Weil algebras that preserves all connected limits that exist. These ``natural'' universal properties (in the sense of \cite{Kolar1993}) is foundational to synthetic differential geometry; see, for example, Chapter Two of \cite{Lavendhomme1996}. Unfortunately, Weil algebras are not an ideal syntactic presentation: they are not a finitely presentable category, and it is not immediately clear when a diagram is a connected limit.\footnote{It should be noted that \cite{Nishimura2007} made progress applying techniques from computer algebra to latter problem.} Moving from $\R$-algebras to commutative rigs and restricting to an appropriate subcategory solves this problem.\pagenote{Clarified a point raised by Michael}

\begin{definition}[Definition 3.1 \cite{Leung2017}]
    \label{def:Weil-algebra}
    The category $\wone$ is defined to be the full subcategory of commutative rigs, $\mathsf{CRig}$, generated by the rig of dual numbers $W := \N[x]/x^2$, constructed as follows:
    \begin{enumerate}
        \item Start with finite product powers of $W$ in $\mathsf{CRig}$, and make a strict choice of presentation:
        \[
            W_0 = \N, \hspace{0.15cm} W_n := \N[x_i]/(x_ix_j)_{i \le j}, {0 \le i < n}.
        \]
        \item Then take the closure of $W_n$ under coproduct of commutative rigs, written $\ox$. Again, make a strict choice of presentation:
        \[
            W_{n(0)} \ox \dots \ox W_{n(m-1)} :=
            \N[x_{i,j}]/(x_{ij}x_{ik})_{j \le k}, 0 < i < m, 0 < j < n(i).
        \]
    \end{enumerate}
    Note that we will often suppress the tensor product $\ox$ and simply write\pagenote{explained notation used throughout this chapter}
    \[
        UV := U\ox V.
    \]
\end{definition}
\begin{proposition}[Definition 3.3 \cite{Leung2017} ]
    ~\begin{enumerate}[(i)]
        \item The category $\wone$ is a symmetric strict monoidal category with unit $\N$ and coproduct $\ox$.
        \item $\N$ is a terminal object in $\wone$.
    \end{enumerate}
\end{proposition} 
Note that there is a forgetful functor
\[
    \wone \to (\mathsf{CMon}/\N) \to \mathsf{CMon}    
\]
that reflects connected limits. This gives the following class of limits, identified in \cite{Leung2017}.
\begin{definition}%
    \label{def:transverse-limit}
    We say the following pullback diagrams in $\wone$ are \emph{transverse}:\pagenote{switched superscript to subscript in second diagram}
\[\begin{tikzcd}
    {W} & {W} &[-2em] {W^2} & {W} &[-2em] {W^2} & {WW} \\
    {W} & {W} & {W} & {\N} & {\N} & {W}
    \arrow["{id}", from=1-1, to=1-2]
    \arrow["{id}", from=1-2, to=2-2]
    \arrow["{id}"', from=2-1, to=2-2]
    \arrow["{id}"', from=1-1, to=2-1]
    \arrow["\lrcorner"{very near start, rotate=0}, from=1-1, to=2-2, phantom]
    \arrow["{\pi_0}"', from=1-3, to=2-3]
    \arrow["{\pi_1}", from=1-3, to=1-4]
    \arrow["{p}", from=1-4, to=2-4]
    \arrow["{p}"', from=2-3, to=2-4]
    \arrow["\lrcorner"{very near start, rotate=0}, from=1-3, to=2-4, phantom]
    \arrow["{p \o \pi_i}"', from=1-5, to=2-5]
    \arrow["{\mu}", from=1-5, to=1-6]
    \arrow["{0}"', from=2-5, to=2-6]
    \arrow["{pW}", from=1-6, to=2-6]
\end{tikzcd}\]

    where $\mu(a + a_1x + a_2y) = a + a_1x + a_2xy$. The $\ox$-closure of these three pullbacks is the set of \emph{transverse squares}, and they are also pullback squares by \cite{Leung2017}.
\end{definition}
To see that each transverse square in the $\ox$-closure is a pullback diagram, take the two non-identity squares and rewrite them in $\mathsf{CMon}$:
\[
\begin{tikzcd}
	{\N \x (\N \x \N)} & \N\x\N &[-1em] {\N \x (\N \x \N)} &[3em] {\N \x (\N\x\N\x\N)} \\
	\N\x\N & \N & \N & {\N \x \N}
	\arrow["{\N \x \pi_0}"', from=1-1, to=2-1]
	\arrow["{\N \x \pi_1}", from=1-1, to=1-2]
	\arrow["{\pi_0}", from=1-2, to=2-2]
	\arrow["{\pi_0}"', from=2-1, to=2-2]
	\arrow["\lrcorner"{anchor=center, pos=0.125}, draw=none, from=1-1, to=2-2]
	\arrow[draw=none, from=1-1, to=2-2]
	\arrow["{\pi_0}"', from=1-3, to=2-3]
	\arrow["{\N \x (\pi_0,0 \circ !, \pi_1)}", from=1-3, to=1-4]
	\arrow["{(id,0 \o !)}"', from=2-3, to=2-4]
	\arrow["{\N \x \pi_1}", from=1-4, to=2-4]
	\arrow["\lrcorner"{anchor=center, pos=0.125}, draw=none, from=1-3, to=2-4]
\end{tikzcd}\]
The coproduct of Weil algebras is the tensor product of the underlying commutative monoids, which are finite-dimensional and free, so these limits are closed under $\ox$.
\begin{proposition}[\cite{Leung2017} Proposition 4.1]
    The category $\wone$ is a tangent category, where the tangent functor is
    \[
        T := W \ox (\_): \wone \to \wone  
    \]
    and the natural transformations are given by 
    \begin{gather*}
        p: W \ox (\_) \xrightarrow[]{p \ox (\_)} (\_), \hspace{0.25cm}
        0: (\_) \xrightarrow[]{0 \ox (\_)} W \ox (\_), \hspace{0.25cm}
        +: W_2 \ox (\_) \xrightarrow[]{+ \ox (\_)} W \ox (\_), \\
        \ell: W \ox (\_) \xrightarrow[]{\ell \ox (\_)} W \ox W \ox (\_), \hspace{0.25cm}
        c: W \ox W \ox (\_) \xrightarrow[]{p \ox (\_)} W \ox W \ox (\_)).
    \end{gather*}
\end{proposition}

The category $\wone$ is, in some sense, a finitely presented theory. It is precisely the free tangent category on a single object:
\begin{proposition}[Proposition 9.5, \cite{Leung2017}]
    \label{thm:leung}
    The category $\wone$ is generated by the maps $\{p, 0, +, \ell, c\}$ from Example  \ref{ex:weil-algebras-and-maps}, closed under composition, tensor, and maps induced by transverse limits.
\end{proposition}
\begin{corollary}
    The category $\wone$ is the \emph{free} tangent category over a single object: every object $C$ in a tangent category $\C$ determines a strict tangent functor $T_{-}.C: \wone \to \C$, mapping
    \[
        V = W^{n{1}} \ox \dots \ox W^{n(k)}  
        \mapsto 
        T_{n(1)}.(\dots).T_{n(k)}.C = T^VC
    \]
    so that there is an isomorphism of categories between $\C$ and the category of strict tangent functors $\wone \to \C$ with tangent natural transformations as morphisms.\pagenote{Original statement was incomplete.}
\end{corollary}
\begin{notation}
    Throughout this section, the notation $T^VC$ will refer to the action of the Weil algebra $V$ on an object $C$ in a tangent category.
    In particular, we will make use of the isomorphism $T^U.T^VC = T^{UV}C$.
\end{notation}

\section{Tangent structures as monoidal actions}%
\label{sec:tang-struct-as-wone}

The presentation of $\wone$ as the free tangent category situates the formal theory of tangent categories as an instance of more general categorical machinery, namely monoidal actions. Recall that in a symmetric monoidal category $(\C, \ox, I)$, an internal monoid $(C, \bullet, i)$ determines a monad:
\[
    (C \ox \_: \C \to \C, \mu: \C \ox (\C \ox \_) \xrightarrow{\bullet \ox \_} \C \ox \_, 
    \eta: \_ \xrightarrow{\rho} I \ox \_ \xrightarrow{e \ox \_} C \ox \_).
\]
The category of algebras for this monad is exactly the category of $C$-modules, objects with an associative and unital action by $C$. A morphism will be a map on the base object that preserves the action:
\[\begin{tikzcd}
	C & {M \ox C} & {M \ox M \ox C} & {M \ox C} \\
	& C & {M \ox C} & C \\
	& {M\ox C} & {M \ox D} \\
	& C & D
	\arrow["M\ox\propto", from=1-3, to=1-4]
	\arrow["{\bullet\ox M}"', from=1-3, to=2-3]
	\arrow["\propto"', from=2-3, to=2-4]
	\arrow["\propto", from=1-4, to=2-4]
	\arrow["{(u,C)}", from=1-1, to=1-2]
	\arrow["\propto", from=1-2, to=2-2]
	\arrow[Rightarrow, no head, from=1-1, to=2-2]
	\arrow["f", from=4-2, to=4-3]
	\arrow["{\propto_C}", from=3-2, to=4-2]
	\arrow["{\propto_D}", from=3-3, to=4-3]
	\arrow["{M \ox f}", from=3-2, to=3-3]
\end{tikzcd}\]
Strict actegories are the 2-categorical generalization of modules over a monoid. The coherences for a 2-monad follow from the coherences from a strict monoidal category in the 2-category of categories. The following proposition relies on a few facts from enriched category theory (treating the cartesian closed category $\mathsf{Cat}$ as a $\mathsf{Cat}$-enriched category, per \cite{Kelly2005}) but a more general treatment of non-strict actegories may be found in \cite{janelidze2001note}:\pagenote{Since I'm not referencing any sort of weakness, this is just enriched category theory so I think Kelly is a fair reference. However, I have included the reference to Janelidze and Kelly's work on actegories.} 
\begin{itemize}
    \item A 2-functor and 2-natural transformations are exactly a functor and natural transformations that satisfy extra coherences. These coherences follow for free by constructing the monad and comonad $\m \x \_, [\m, \_]$.
    \item An algebra of the underlying 1-monad is exactly an algebra of the 2-monad (the same result holds for comonads).
\end{itemize}
When working with algebras of a 2-monad, four different notions of morphisms can come into play (\cite{Lack2009}). These arise through using the 2-categorical data to weaken the notion of a morphism:
\begin{enumerate}[(i)]
    \item Strict: this is exactly a morphism of the underlying algebras. Write the 2-category of strict $\m$-actegories.
    \item Strong: the morphisms preserve the action to an isomorphism:
    \[
\begin{tikzcd}
	{\m \x \C} & {\m \x \D} \\
	\C & \D
	\arrow["{\m \x F}", from=1-1, to=1-2]
	\arrow["{\propto_C}"', from=1-1, to=2-1]
	\arrow["{\propto_D}", from=1-2, to=2-2]
	\arrow["F"', from=2-1, to=2-2]
	\arrow["\alpha"{description}, draw=none, from=1-2, to=2-1]
\end{tikzcd}\]
    \item Lax: the 2-cell is no longer an isomorphism:
    \[
\begin{tikzcd}
	{\m \x \C} & {\m \x \D} \\
	\C & \D
	\arrow["{\m \x F}", from=1-1, to=1-2]
	\arrow["{\propto_C}"', from=1-1, to=2-1]
	\arrow["{\propto_D}", from=1-2, to=2-2]
	\arrow["F"', from=2-1, to=2-2]
	\arrow["\alpha"{description}, shorten <=2pt, shorten >=2pt, Rightarrow, from=1-2, to=2-1]
\end{tikzcd}\]
    \item Oplax: the 2-cell travels in the opposite direction (these will not figure into this account).
\end{enumerate}
2-cells between actegory morphisms must satisfy a coherence between the natural transformation parts of the actegory morphisms.
\begin{definition}\label{def:actegory-natural}
    In the case of strict, strong, and lax tangent functors, the same notion of a 2-cell applies: a natural transformation $\gamma:F \Rightarrow G$ satisfying the following coherences with the natural transformations $\alpha$ and $\beta$:
\[\begin{tikzcd}
	\m\x\C & \m\x\D & \m\x\C & \m\x\D \\
	\C & \D & \C & \D
	\arrow["{\m\x F}", from=1-1, to=1-2]
	\arrow[""{name=0, anchor=center, inner sep=0}, "{\propto^\D}", from=1-2, to=2-2]
	\arrow["{\propto^\C}"', from=1-1, to=2-1]
	\arrow[""{name=1, anchor=center, inner sep=0}, "F"{description}, from=2-1, to=2-2]
	\arrow["\alpha"{description}, Rightarrow, from=1-2, to=2-1]
	\arrow[""{name=2, anchor=center, inner sep=0}, "G"', curve={height=18pt}, from=2-1, to=2-2]
	\arrow[""{name=3, anchor=center, inner sep=0}, "{\m\x G}"'{pos=0.2}, from=1-3, to=1-4]
	\arrow[""{name=4, anchor=center, inner sep=0}, "{\m\x F}", curve={height=-18pt}, from=1-3, to=1-4]
	\arrow[""{name=5, anchor=center, inner sep=0}, "{\propto^\C}"', from=1-3, to=2-3]
	\arrow["{\propto^\D}", from=1-4, to=2-4]
	\arrow["G"', from=2-3, to=2-4]
	\arrow["\beta", Rightarrow, from=1-4, to=2-3]
	\arrow["\gamma", shorten <=2pt, shorten >=2pt, Rightarrow, from=1, to=2]
	\arrow["{\m\x \gamma}"', shift left=2, shorten <=2pt, shorten >=2pt, Rightarrow, from=4, to=3]
	\arrow["{=}"{description}, Rightarrow, draw=none, from=0, to=5]
\end{tikzcd}\]
We call these \emph{actegory natural transformations}. 
\end{definition}
Note that for strict actegory morphisms, this condition holds for any natural transformation $\gamma:F \Rightarrow G$. Now consider the following three 2-categories.\pagenote{Added the actual definition of the 2-categories referenced later on. I also added the coherence for actegory-natural transformation}
\begin{definition}\label{def:2-categories}
    Let $(\m,\ox,I)$ be a strict monoidal category. Define the following three 2-categories.
    \begin{enumerate}
        \item $\m\mathsf{Act}_{\mathsf{strict}}$: the 2-category of strict $\m$-actegories, strict actegory morphisms, and natural transformations.
        \item $\m\mathsf{Act}_{\mathsf{strong}}$: the 2-category of strict $\m$-actegories, strong actegory morphisms, and actegory natural transformations.
        \item $\m\mathsf{Act}_{\mathsf{lax}}$: the 2-category of strict $\m$-actegories, lax actegory morphisms, and actegory natural transformations.
    \end{enumerate}
    Note that the inclusions of these 2-categories are \emph{locally} fully faithful, so only the 1-cells differ.
\end{definition}

The case where the action preserves certain limits in the monoidal category is of particular interest. A small category equipped with a class of chosen limits is known as a \emph{sketch}. The previous correspondence restricts to the class of limit-preserving actions in this case.
\begin{definition}
    A \emph{sketch} is a small category with a class of chosen limits, and a sketch morphism is a functor sending chosen limits to the chosen limits in the domain (up to isomorphism). The category of models of a sketch $\c$ in a category $\C$, $\mathsf{Mod}(\c, \C)$, is the full subcategory $[\c, \C]$ whose functors preserve the chosen limits.
    A \emph{monoidal sketch}, then, is a sketch $(\c, \prol)$ equipped with a symmetric monoidal category structure on $\c$ so that $\_ \ox \_$ preserves limits in each argument.
\end{definition}

Now use the fact that the category $\wone$ is a monoidal sketch, since it is a small, strict monoidal category equipped with a class of limits stable under the tensor product.
\begin{theorem}[Theorem 14.1, \cite{Leung2017}]
    Let $\C$ be a category. The following are equivalent:
    \begin{enumerate}[(i)]
        \item A tangent structure on $\C$,
        \item A sketch action $\propto: \wone \x \C \to \C$.
    \end{enumerate}
\end{theorem}

\begin{observation}%
    \label{obs:cofree-tangent-cat}
    There is a \emph{coalgebraic} perspective on tangent categories, coming from the equivalence between algebras of the 2-monad $(\wone \x (\_), \ox, I:1 \to \wone)$ and the 2-comonad $([\wone,\_], [\ox,\_], [I,\_])$.
    For any  category $\C$, there is a free tangent category given by
    \[
        \wone \x \C
    \] and this agrees with the free $\wone$-actegory. However, for the \emph{cofree} tangent category, take
    \[
        \mathsf{Mod}(\wone, \C),
    \]
    the category of transverse-limit-preserving functors $\wone \to \C$.
\end{observation}

We can use Leung's theorem to induce a monoidal functor $\wone \to \C$ when a tangent structure is induced by a single object.
\begin{corollary}\label{cor:using-leung-thm}
    Let $(\C,\ox,I)$ be a strict monoidal category.\pagenote{I've added this intermediary result to make the Weil nerve a bit more explicit.}
    If an additive bundle $(p:A \to I, +:A_2 \to A, 0:I \to A)$ equipped with morphisms
    \[
       A \ox A \xrightarrow[]{c} A \ox A \hspace{0.5cm} A \xrightarrow[]{\ell} A \ox A
    \]
    determines a tangent structure on $\C$ using the endofunctor $A \ox (-)$, then $A$ determines a strict, transverse-limit-preserving, monoidal functor 
    \[
       A(-):\wone \to \C; W_{n[1]}\ox \dots \ox W_{n[k]} \mapsto A_{n[1]}\ox \dots \ox A_{n[k]} = T_{n[1]}\dots T_{n[k]}.I
    \]
\end{corollary}
Note that this allows for a more conceptual description of representable tangent structure.
\begin{proposition}\label{prop:monoidal-functor-inf-obj}
    In a symmetric monoidal closed category, an infinitesimal object is exactly a strict symmetric monoidal functor $D:\wone \to \C$.\pagenote{The other statement was probably more general than necessary - this statement is more to the point.}
\end{proposition}
This presentation of an infinitesimal object makes it tautological that $\C^{op}$ has a tangent structure.
\begin{corollary}%
    \label{cor:dual-tangent-structure}
    Given a strict symmetric monoidal functor
    \[
        D:\wone \to \C^{op}
    \]
    there is a strict action of $\wone$ on $\C^{op}$ given by
    \[
        \wone \x \C^{op} \xrightarrow[]{D^{op} \ox \C} \C^{op} \x \C^{op} \xrightarrow[]{\otimes} \C^{op}. 
    \]
\end{corollary}

There is a clear correspondence between the notions of a (strict, strong, lax) tangent functor and a (strict, strong, lax) actegory morphism. This proposition extends to the following equivalence of 2-categories.
\begin{corollary}[Theorem 14.1 \cite{Leung2017}]
    The following pairs of 2-categories are equivalent.
    \begin{enumerate}[(i)]
        \item The 2-category of tangent categories and strict tangent functors is the full sub-2-category of $\wone\mathsf{Act_{Strict}}$ spanned by sketch actions.
        \item The 2-category of tangent categories and strong tangent functors is the full sub-2-category of $\wone\mathsf{Act_{strong}}$ spanned by sketch actions.
        \item The 2-category of tangent categories and lax tangent functors is the full sub-2-category of $\wone\mathsf{Act_{Lax}}$ spanned by sketch actions.
    \end{enumerate}
\end{corollary}

\section{The Weil nerve of an involution algebroid}
\label{sec:weil-nerve}

The construction in this section is analagous to the nerve of an internal category---hence the ``Weil nerve'' construction---and deals with similar technical issues. In particular, the construction in this section will mimic the nerve construction for internal categories by replacing the tensor product of $\wone$ with span composition in the domain category. Recall that every anchored bundle or internal category has a canonical span associated with it:
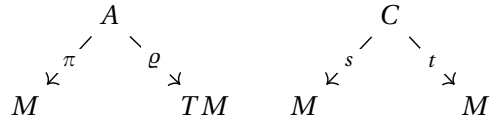
\[
\begin{tikzcd}[column sep = small]
	& A &&& C \\
	M && TM & M && M
	\arrow["\pi"{description}, from=1-2, to=2-1]
	\arrow["\anc"{description}, from=1-2, to=2-3]
	\arrow["s"{description}, from=1-5, to=2-4]
	\arrow["t"{description}, from=1-5, to=2-6]
\end{tikzcd}\]
In any category $\C$, there is a category of spans in $\C$ as well as span composition.
\begin{definition}%
    \label{def:span-stuff}
    A \emph{span from $A$ to $B$} in a category $\C$ is a diagram of the form 
    
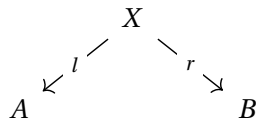
\[
\begin{tikzcd}
	& X \\
	A && B
	\arrow["l"{description}, from=1-2, to=2-1]
	\arrow["r"{description}, from=1-2, to=2-3]
\end{tikzcd}\]
    There is a notion of \emph{span composition}, so given a span $X:A \to B$ and $Y:B \to C$, then the composition of $X$ and $Y$ is the pullback (if it exists):
    
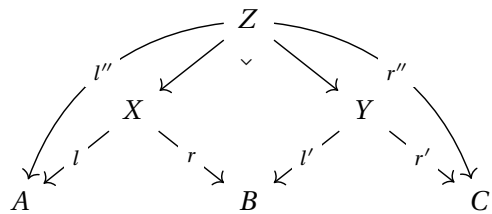
\[
\begin{tikzcd}
	&& Z \\
	& X && Y \\
	A && B && C
	\arrow["l"{description}, from=2-2, to=3-1]
	\arrow["r"{description}, from=2-2, to=3-3]
	\arrow["{l'}"{description}, from=2-4, to=3-3]
	\arrow["{r'}"{description}, from=2-4, to=3-5]
	\arrow[from=1-3, to=2-2]
	\arrow[from=1-3, to=2-4]
	\arrow["\lrcorner"{anchor=center, pos=0.125, rotate=-45}, draw=none, from=1-3, to=3-3]
	\arrow["{l''}"{description}, curve={height=24pt}, from=1-3, to=3-1]
	\arrow["{r''}"{description}, curve={height=-24pt}, from=1-3, to=3-5]
\end{tikzcd}\]
    A morphism of spans is a commuting diagram of the form
    
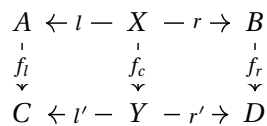
\[
\begin{tikzcd}
	A & X & B \\
	C & Y & D
	\arrow["l"{description}, from=1-2, to=1-1]
	\arrow["r"{description}, from=1-2, to=1-3]
	\arrow["{l'}"{description}, from=2-2, to=2-1]
	\arrow["{f_l}"{description}, from=1-1, to=2-1]
	\arrow["{f_r}"{description}, from=1-3, to=2-3]
	\arrow["{r'}"{description}, from=2-2, to=2-3]
	\arrow["{f_c}"{description}, from=1-2, to=2-2]
\end{tikzcd}\]
    Note that if $f$ and $g$ are span morphisms with $f_r = g_l$, then the horizontal composition may be formed if each respective span composition exists:
    
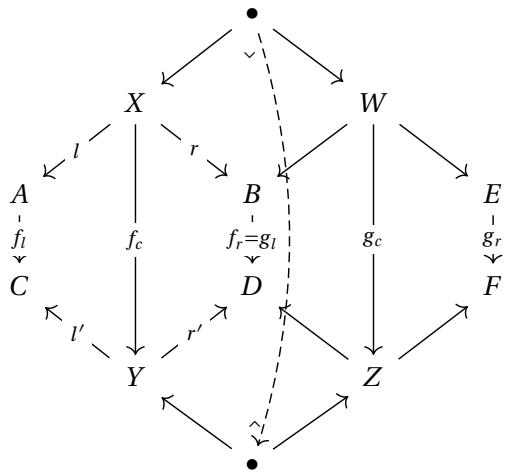
\[
\begin{tikzcd}
	&& \bullet \\
	& X && W \\
	A && B && E \\
	C && D && F \\
	& Y && Z \\
	&& \bullet
	\arrow["l"{description}, from=2-2, to=3-1]
	\arrow["r"{description}, from=2-2, to=3-3]
	\arrow["{l'}"{description}, from=5-2, to=4-1]
	\arrow["{f_l}"{description}, from=3-1, to=4-1]
	\arrow["{f_r=g_l}"{description}, from=3-3, to=4-3]
	\arrow["{r'}"{description}, from=5-2, to=4-3]
	\arrow["{f_c}"{description}, from=2-2, to=5-2]
	\arrow[from=2-4, to=3-3]
	\arrow[from=2-4, to=3-5]
	\arrow[from=5-4, to=4-3]
	\arrow[from=5-4, to=4-5]
	\arrow["{g_r}"{description}, from=3-5, to=4-5]
	\arrow["{g_c}"{description}, from=2-4, to=5-4]
	\arrow[from=1-3, to=2-2]
	\arrow[from=1-3, to=2-4]
	\arrow[from=6-3, to=5-2]
	\arrow[from=6-3, to=5-4]
	\arrow["\lrcorner"{anchor=center, pos=0.125, rotate=-45}, draw=none, from=1-3, to=3-3]
	\arrow["\lrcorner"{anchor=center, pos=0.125, rotate=135}, draw=none, from=6-3, to=4-3]
	\arrow[curve={height=-18pt}, dashed, from=1-3, to=6-3]
\end{tikzcd}\]
    When discussing span composition in a tangent category, it is assumed that the pullback is a $T$-pullback.
\end{definition}
\begin{observation}%
    \label{obs:lim-of-spans}
    Note that the category of spans in $\C$ is a functor category, so that limits are computed pointwise in $\C$. This also means that the horizontal composition operation, when it exists, preserves limits in either argument.
\end{observation}
These span constructions can be helpful in constructing functors from a monoidal category into a non-monoidal category $\C$, by forming a monoidal category from $\C$ using spans. In the case of an internal category over $M$, one takes the slice category $\C/(M \x M)$ where the tensor product is span composition. An internal category $s,t:C \to M$ is a monoid in this category of spans over $M$, so that it determines a monoidal functor
\[
    C: \Delta^{op} \to \C/(M \x M),    
\] 
remembering that $\Delta^{op}$ is the monoidal theory for monoid (every monoid in a monoidal category $\C$ determines a monoidal functor $\Delta \to \C$). The construction of the corresponding monoidal category for spans is more nuanced, as the category $\wone$ is not $\N$-indexed. Observe that the prolongation of an anchored bundle is constructed as a span composition:
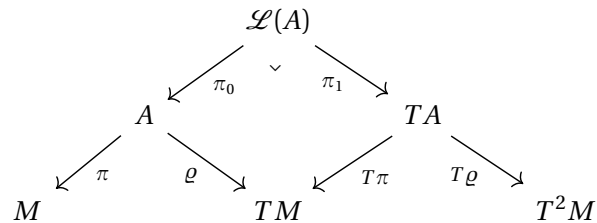
\[\begin{tikzcd}
    && {\prol(A)} \\
    & A && TA \\
    M && TM && {T^2M}
    \arrow["\pi", from=2-2, to=3-1]
    \arrow["\anc"', from=2-2, to=3-3]
    \arrow["T\pi", from=2-4, to=3-3]
    \arrow["T\anc"', from=2-4, to=3-5]
    \arrow["{\pi_0}", from=1-3, to=2-2]
    \arrow["{\pi_1}"', from=1-3, to=2-4]
    \arrow["\lrcorner"{anchor=center, pos=0.125, rotate=-45}, draw=none, from=1-3, to=3-3]
\end{tikzcd}\]
The third prolongation is given by span composition as well:
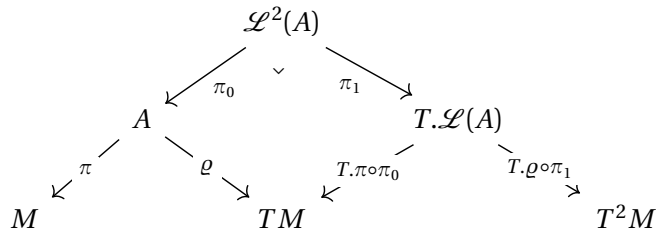
\[
\begin{tikzcd}
	&& {\prol^2(A)} \\
	& A && {T.\prol(A)} \\
	M && TM && {T^2M}
	\arrow["\pi"{description}, from=2-2, to=3-1]
	\arrow["\anc"{description}, from=2-2, to=3-3]
	\arrow["{T.\pi \o \pi_0}"{description}, from=2-4, to=3-3]
	\arrow["{T.\anc \o \pi_1}"{description}, from=2-4, to=3-5]
	\arrow["{\pi_0}", from=1-3, to=2-2]
	\arrow["{\pi_1}"', from=1-3, to=2-4]
	\arrow["\lrcorner"{anchor=center, pos=0.125, rotate=-45}, draw=none, from=1-3, to=3-3]
\end{tikzcd}\]
This horizontal composition will play the same role as the tensor product in $\C/(M\x M)$.
\begin{definition}%
    \label{def:boxtimes-span}
    In a tangent category $\C$, consider a pair of spans
    \[
        X:M \to T^UM, \hspace{0.15cm} Y:M \to T^VM.
    \]
    Define $X \boxtimes Y$ to be the horizontal composition (when it exists):
    
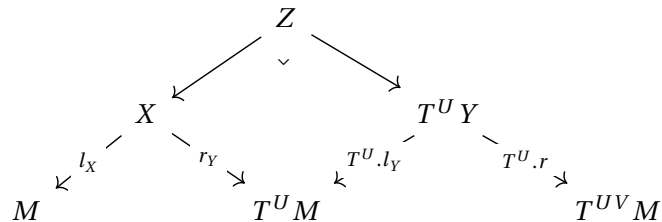
\[
\begin{tikzcd}
	&& Z \\
	& X && {T^UY} \\
	M && {T^UM} && {T^{UV}M}
	\arrow["{l_X}"{description}, from=2-2, to=3-1]
	\arrow["{r_Y}"{description}, from=2-2, to=3-3]
	\arrow["{T^U.l_Y}"{description}, from=2-4, to=3-3]
	\arrow["{T^U.r}"{description}, from=2-4, to=3-5]
	\arrow[from=1-3, to=2-2]
	\arrow[from=1-3, to=2-4]
	\arrow["\lrcorner"{anchor=center, pos=0.125, rotate=-45}, draw=none, from=1-3, to=3-3]
\end{tikzcd}\]
    (recall that we will often suppress the $\ox$ in $\wone$ to save space).\pagenote{clarifying notation}
    So the span composition is
    \[
        M \xrightarrow[]{X} T^UM \xrightarrow[]{T^U.Y} T^U.T^VM    
    \]
    \begin{equation}
        \label{eq:span-form}
\begin{tikzcd}
	M & X & {T^UM} \\
	M & Y & {T^{U}M} \\
	M & A & {T^V} \\
	M & B & {T^{V'}}
	\arrow[Rightarrow, no head, from=1-1, to=2-1]
	\arrow["{l_X}"', from=1-2, to=1-1]
	\arrow["{l_{Y}}", from=2-2, to=2-1]
	\arrow["{\theta.M}", from=1-3, to=2-3]
	\arrow["{r_X}", from=1-2, to=1-3]
	\arrow["{r_Y}"', from=2-2, to=2-3]
	\arrow["f"{description}, from=1-2, to=2-2]
	\arrow["{r_A}", from=3-2, to=3-3]
	\arrow["{r_B}"', from=4-2, to=4-3]
	\arrow["{\psi.M}", from=3-3, to=4-3]
	\arrow["g"{description}, from=3-2, to=4-2]
	\arrow["{l_B}", from=4-2, to=4-1]
	\arrow[Rightarrow, no head, from=3-1, to=4-1]
	\arrow["{l_A}"', from=3-2, to=3-1]
\end{tikzcd}
    \end{equation}
    The horizontal composition $f \boxtimes g$ is defined as $f \x_{\theta.M} \theta.g$:
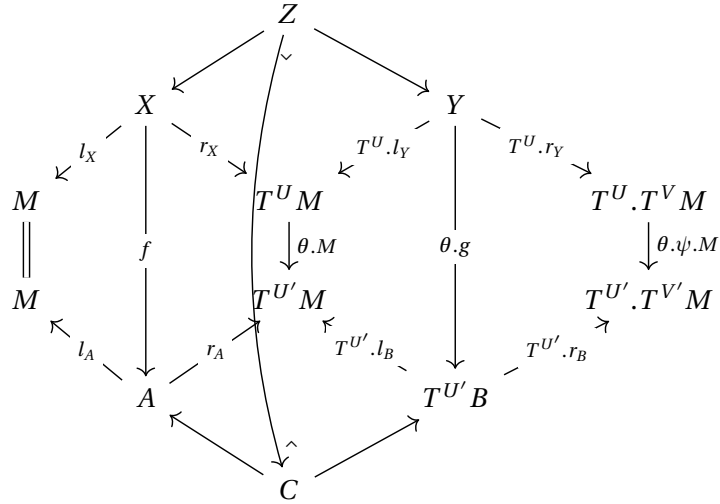
    \[
\begin{tikzcd}
	&& Z \\
	& X && Y \\
	M && {T^UM} && {T^{U}.T^{V}M} \\
	M && {T^{U'}M} && {T^{U'}.T^{V'}M} \\
	& A && {T^{U'}B} \\
	&& C
	\arrow[Rightarrow, no head, from=3-1, to=4-1]
	\arrow["{l_X}"{description}, from=2-2, to=3-1]
	\arrow["{l_A}"{description}, from=5-2, to=4-1]
	\arrow["{\theta.M}", from=3-3, to=4-3]
	\arrow["{r_X}"{description}, from=2-2, to=3-3]
	\arrow["{r_A}"{description}, from=5-2, to=4-3]
	\arrow["f"{description}, from=2-2, to=5-2]
	\arrow["{T^U.r_Y}"{description}, from=2-4, to=3-5]
	\arrow["{T^{U'}.r_B}"{description}, from=5-4, to=4-5]
	\arrow["{\theta.\psi.M}", from=3-5, to=4-5]
	\arrow["{\theta.g}"{description}, from=2-4, to=5-4]
	\arrow["{T^U.l_Y}"{description}, from=2-4, to=3-3]
	\arrow["{T^{U'}.l_B}"{description}, from=5-4, to=4-3]
	\arrow[from=1-3, to=2-2]
	\arrow[from=1-3, to=2-4]
	\arrow["\lrcorner"{anchor=center, pos=0.125, rotate=-45}, draw=none, from=1-3, to=3-3]
	\arrow[from=6-3, to=5-2]
	\arrow[from=6-3, to=5-4]
	\arrow["\lrcorner"{anchor=center, pos=0.125, rotate=135}, draw=none, from=6-3, to=4-3]
	\arrow[curve={height=18pt}, from=1-3, to=6-3]
\end{tikzcd}\]
\end{definition}
In any tangent category with a tangent display system (Definition\pagenote{there was no reason to pull this out as an observation} \ref{def:display-system}), the category of spans on $M$ whose maps are of the form given by Equation \ref{eq:span-form} with $l \in \d$ is a monoidal category.
Any anchored bundle in a tangent category gives rise to a monoidal category after a strict choice of $T$-pullbacks (assuming those $T$-pullbacks exist).
\begin{definition}\label{def:monoidal-category}
    Let $(\pi:A \to M,\xi,\lambda,\anc)$ be an anchored bundle in a tangent category $\C$. Write the span
    \[
        \widehat{A}.W_n := (M \xleftarrow[]{\pi \o \pi_i} A_n \xrightarrow[]{\anc^n} T_nM),
        \hspace{0.5cm}
        \widehat{A}.\N := (M = M = M).
    \]
    A \emph{choice of prolongations} for $(\pi,\xi,\lambda,\anc)$ is a strict choice of horizontal composition for each $V \in \wone$:
    \[
        \widehat{A}.V = \widehat{A}.(W_{n[1]}\dots W_{n[k]}) :=
        \widehat{A}.W_{n[1]} \boxtimes \dots \boxtimes \widehat{A}.W_{n[k]}.
    \]
    We will write the span as follows:
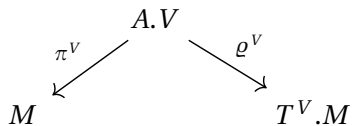
\[\begin{tikzcd}
	& {A.V} \\
	M && {T^V.M}
	\arrow["{\pi^V}"', from=1-2, to=2-1]
	\arrow["{\anc^V}", from=1-2, to=2-3]
\end{tikzcd}\]
    (notice that the apex is not hatted).
    Given a choice of prolongations for an anchored bundle $(\pi,\xi,\lambda,\anc)$, the category $\mathsf{Span}(\pi,\xi,\lambda,\anc)$ is defined as follows:
    \begin{itemize}
        \item Objects are $\widehat{A}.V$ for $V \in \wone$.
        \item Morphisms are given by pairs
        \[
            (f,\phi):\widehat{A}.V \to \widehat{A}.U
        \]
        where $f:A.V \to A.U$ and $\phi:V \to U$ determine a span morphism of the form
\[\begin{tikzcd}
	M & {A.V} & {T^V.M} \\
	M & {A.U} & {T^U.M}
	\arrow["{\pi^V}"', from=1-2, to=1-1]
	\arrow["{\anc^V}", from=1-2, to=1-3]
	\arrow["{\phi.M}", from=1-3, to=2-3]
	\arrow[Rightarrow, no head, from=1-1, to=2-1]
	\arrow["{\pi^V}", from=2-2, to=2-1]
	\arrow["{\anc^U}"', from=2-2, to=2-3]
	\arrow["f", from=1-2, to=2-2]
\end{tikzcd}\]
        as discussed in Definition \ref{def:span-stuff}.
        \item Tensor structure: The tensor product is defined using the horizontal composition $\boxtimes$ as defined in Definition \ref{def:span-stuff}.
    \end{itemize}
\end{definition}
The idea is to show that an involution algebroid structure on an anchored bundle induces a tangent structure on the monoidal category of prolongations, and then to apply Leung's theorem. 
The following two lemmas will simplify this proof.
\begin{lemma}\label{lemma:pullback-part-of-theorem}
    Let $(\pi:A \to M, \xi, \lambda, \anc)$ be an anchored bundle with chosen prolongations in a tangent category $\C$, and identify the monoidal category $\mathsf{Span}(\pi,\xi,\lambda,\anc)$.
    \begin{enumerate}[(i)]
        \item There is a functor 
        \[
           U^\anc: \mathsf{Span}(\pi,\xi,\lambda,\anc) \to \C
        \]
        constructed by sending a span morphism to the morphism between the objects at its apex:
\[\begin{tikzcd}
	M & {A.V} & {T^V.M} & {A.V} \\
	M & {A.U} & {T^U.M} & {A.U}
	\arrow["{\pi^V}"', from=1-2, to=1-1]
	\arrow["{\anc^V}", from=1-2, to=1-3]
	\arrow[""{name=0, anchor=center, inner sep=0}, "{\phi.M}"', from=1-3, to=2-3]
	\arrow[Rightarrow, no head, from=1-1, to=2-1]
	\arrow["{\pi^V}", from=2-2, to=2-1]
	\arrow["{\anc^U}"', from=2-2, to=2-3]
	\arrow["f", from=1-2, to=2-2]
	\arrow[""{name=1, anchor=center, inner sep=0}, "f", from=1-4, to=2-4]
	\arrow[shorten <=14pt, shorten >=14pt, maps to, from=0, to=1]
\end{tikzcd}\]
        \item Suppose we have a square
\[\begin{tikzcd}
	{\widehat{A}.U} & {\widehat{A}.Y} \\
	{\widehat{A}.X} & {\widehat{A}.Z}
	\arrow["{(f,\phi)}"', from=2-1, to=2-2]
	\arrow["{(g,\psi)}", from=1-2, to=2-2]
	\arrow["{(l,\alpha)}"', from=1-1, to=2-1]
	\arrow["{(r,\beta)}", from=1-1, to=1-2]
\end{tikzcd}\]
        whose image under $U^\anc$ is a $T$-pullback in $\C$, and so that the square in $\wone$ is a transverse $T$-pullback:
\[\begin{tikzcd}
	{{A}.U} & {{A}.Y} & U & Y \\
	{{A}.X} & {{A}.Z} & X & Z
	\arrow["f"', from=2-1, to=2-2]
	\arrow["g", from=1-2, to=2-2]
	\arrow["l"', from=1-1, to=2-1]
	\arrow["r", from=1-1, to=1-2]
	\arrow["\phi"', from=2-3, to=2-4]
	\arrow["\psi", from=1-4, to=2-4]
	\arrow["\beta", from=1-3, to=1-4]
	\arrow["\alpha"', from=1-3, to=2-3]
	\arrow["\lrcorner"{anchor=center, pos=0.125}, draw=none, from=1-1, to=2-2]
	\arrow["\lrcorner"{anchor=center, pos=0.125}, draw=none, from=1-3, to=2-4]
\end{tikzcd}\]
      Then $U^\anc$ reflects the limit; that is, the original square in $\mathsf{Span}(\pi,\xi,\lambda,\anc)$ is a $T$-pullback.
      \item $T$-pullbacks of the form described in (ii) are closed under $\boxtimes$.
    \end{enumerate}
\end{lemma}
\begin{proof}
    The functor in in (i) is straightforward to construct, as it simply forgets the left and right legs of the spans. For (ii), note that because the $\wone$ part of the diagram is a transverse $T$-pullback, then given a pair of maps
\[\begin{tikzcd}
	{\widehat{A}.V} \\
	& {\widehat{A}.U} & {\widehat{A}.Y} \\
	& {\widehat{A}.X} & {\widehat{A}.Z}
	\arrow["{(f,\phi)}"', from=3-2, to=3-3]
	\arrow["{(g,\psi)}", from=2-3, to=3-3]
	\arrow["{(l,\alpha)}"', from=2-2, to=3-2]
	\arrow["{(r,\beta)}", from=2-2, to=2-3]
	\arrow["{(x,\omega)}"', curve={height=12pt}, from=1-1, to=3-2]
	\arrow["{(y,\gamma)}", curve={height=-12pt}, from=1-1, to=2-3]
\end{tikzcd}\]
    a unique span morphism $\widehat{A}.V \to \widehat{A}.U$ may be induced using the apex map from $\C$, and the unique map induced in $\wone$ by the universality of transverse squares (this square is also universal in $\C$), so the span morphism diagram will commute by universality.
    
    For (iii), $T$-pullback squares of the form in (ii) are closed under $\boxtimes$ as transverse squares in $\wone$ are closed under $\ox$, so the result follows by the commutativity of limits and by applying part (ii) of this lemma. 
\end{proof}

\begin{observation}%
    \label{obs:concrete-desc-zigzag}
    It will be useful to have a ``flat'' presentation of the prolongation $A.W_{n(1)}\dots W_{n(k)}$. 
    The higher prolongations of an anchored bundle may be concretely described as the $T$-pullback of the zig-zag below:
    \[
\begin{tikzcd}
	&& {\hat{A}.W_{n(1)}\dots W_{n(k)}} \\
	{\hat A.W_{n(1)}} && {T_{n[1]}.\hat A. W_{n(2)}} && {T_{n(1)\dots n(k-1)}\hat A . W_{n(k)}} \\
	& {T_{n(1)}M} && \dots
	\arrow["{T_{n(1)}.\anc_{n(2)}}"', from=2-3, to=3-4]
	\arrow["{T_{n(1)\dots n(k)}.(\pi \o \pi_0)}", from=2-5, to=3-4]
	\arrow["{\anc_{n(1)}}"', from=2-1, to=3-2]
	\arrow["{T_n[1].(\pi \o \pi_0)}", from=2-3, to=3-2]
	\arrow[from=1-3, to=2-1]
	\arrow[from=1-3, to=2-5]
\end{tikzcd}\]
    so that the prolongation $A.W^{n[1]}\dots W^{n[k]}$ may be written concretely as
    \[
      (u_1,\dots, u_{k}) : A_{n[1]} \ts{\anc'}{T.\pi'}  T_{n[1]}.A_{n[2]} \ts{T.\anc}{T^2.\pi} \dots \ts{\anc'}{T.\pi'} T_{n[1]\dots n[k-1]}A_{n[k]}.
    \]
    Furthermore, the choice of prolongation identifies the following limits:
    \[
\begin{tikzcd}
	{\hat{A}.UV} & {T^U.\hat{A}.V} \\
	{\hat{A}.U} & {T^{U}M}
	\arrow["{\anc^U}"', from=2-1, to=2-2]
	\arrow["{T^U.\pi^V}", from=1-2, to=2-2]
	\arrow[from=1-1, to=2-1]
	\arrow[from=1-1, to=1-2]
	\arrow["\lrcorner"{anchor=center, pos=0.125}, draw=none, from=1-1, to=2-2]
\end{tikzcd}
\begin{tikzcd}
	{\hat{A}.UV} && {T^U\hat{A}.V} \\
	& {T^{U}M} & {T^UM} \\
	{\hat{A}.U} & {T^UM}
	\arrow[from=1-1, to=3-1]
	\arrow[from=1-1, to=1-3]
	\arrow["\lrcorner"{anchor=center, pos=0.125}, draw=none, from=1-1, to=2-2]
	\arrow["{\anc^U}"', from=3-1, to=3-2]
	\arrow["{T^U.\pi^V}", from=1-3, to=2-3]
	\arrow[Rightarrow, no head, from=2-2, to=2-3]
	\arrow[Rightarrow, no head, from=2-2, to=3-2]
\end{tikzcd}\]
    so that
    \[
        A. UV
        = A.U \boxtimes A.V
        = A.U \ \boxtimes id_M \boxtimes A.V
    \]
    where $id_M$ is the span $M = M = M$.
\end{observation}


%

                   
Note that a $\anc$ sends the involution algebroid structure map to its corresponding tangent structure map. Each of the structure maps, then, gives a span morphism where $\boxtimes$ is well defined:
\begin{definition}%
    \label{def:generators-for-wone-in-anc}
    Let $(\pi:A \to M, \xi, +_q, \lambda, \anc, \sigma)$ be an involution algebroid in $\C$ with chosen prolongations.
    Then define the following maps in $\mathsf{Span}(\pi,\xi,\lambda,\anc)$:
    \begin{itemize}
        \item The projection $p: \hat{A}.W \to \hat{A}.\N$,
        \[
\begin{tikzcd}
	M & A & TM \\
	M & M & M
	\arrow["\anc", from=1-2, to=1-3]
	\arrow["\pi"', from=1-2, to=1-1]
	\arrow[Rightarrow, no head, from=1-1, to=2-1]
	\arrow[Rightarrow, from=2-2, to=2-1]
	\arrow[Rightarrow, no head, from=2-2, to=2-3]
	\arrow["p", from=1-3, to=2-3]
	\arrow["\pi"{description}, shorten <=2pt, shorten >=2pt, from=1-2, to=2-2]
\end{tikzcd}\]
        \item The zero map $0: \hat{A}.\N \to \hat{A}.W$,
        \[
\begin{tikzcd}
	M & M & M \\
	M & A & TM
	\arrow[Rightarrow, no head, from=1-2, to=1-1]
	\arrow[Rightarrow, no head, from=1-2, to=1-3]
	\arrow[Rightarrow, no head, from=1-1, to=2-1]
	\arrow["0", from=1-3, to=2-3]
	\arrow["\pi", from=2-2, to=2-1]
	\arrow["\anc"', from=2-2, to=2-3]
	\arrow["\xi"{description}, shorten <=2pt, shorten >=2pt, from=1-2, to=2-2]
\end{tikzcd}\]
        \item The addition map $+:  \hat{A}.W_2 \to \hat{A}.W$,
        \[
\begin{tikzcd}
	M & {A_2} & {T_2M} \\
	M & A & TM
	\arrow["{\anc_2}", from=1-2, to=1-3]
	\arrow["{\pi\o \pi_i}"', from=1-2, to=1-1]
	\arrow[Rightarrow, no head, from=1-1, to=2-1]
	\arrow["{+.M}", from=1-3, to=2-3]
	\arrow["\pi", from=2-2, to=2-1]
	\arrow["\anc"', from=2-2, to=2-3]
	\arrow["{+}"{description}, shorten <=2pt, shorten >=2pt, from=1-2, to=2-2]
\end{tikzcd}\]
        \item The lift map $\ell: \hat{A}.W \to \hat{A}.WW$,
        \[
\begin{tikzcd}
	M & A & TM \\
	M & \prolong & {T^2M}
	\arrow["\pi"', from=1-2, to=1-1]
	\arrow["\anc", from=1-2, to=1-3]
	\arrow["{\ell.M}", from=1-3, to=2-3]
	\arrow[Rightarrow, no head, from=1-1, to=2-1]
	\arrow["{\pi\o \pi_0}"{pos=0.7}, from=2-2, to=2-1]
	\arrow["{T.\anc \o \pi_1}"'{pos=0.8}, from=2-2, to=2-3]
	\arrow["{(\xi\pi,\lambda)}"{description}, shorten <=2pt, shorten >=2pt, from=1-2, to=2-2]
\end{tikzcd}\]
        \item The flip map $c: \hat{A}.WW \to \hat{A}.WW$,
        \[
\begin{tikzcd}
	M & {\prolong} & {T^2M} \\
	M & {\prolong} & {T^2M}
	\arrow[Rightarrow, no head, from=1-1, to=2-1]
	\arrow["{c.M}", from=1-3, to=2-3]
	\arrow["{\pi\o\pi_0}"', from=1-2, to=1-1]
	\arrow["{\anc \o \pi_1}", from=1-2, to=1-3]
	\arrow["{\pi\o\pi_0}", from=2-2, to=2-1]
	\arrow["{\anc \o \pi_1}"', from=2-2, to=2-3]
	\arrow["\sigma"{description}, shorten <=2pt, shorten >=2pt, from=1-2, to=2-2]
\end{tikzcd}\]
    \end{itemize}
\end{definition}
The idea is to show that the monoidal category of chosen prolongations $\mathsf{Span}(\pi,\xi,\lambda,\anc)$ for an involution algebroid has a tangent structure generated by the structure maps in Definition \ref{def:generators-for-wone-in-anc} and the endofunctor $\widehat{A}.W \boxtimes (-)$. Using the flat presentation, we can then show that $U^\anc$ will determine a tangent functor in $\C$. The following lemma about $\anc$ will be useful in constructing the natural transformation part of a tangent functor.
\begin{definition}\label{def:anc-nat}
    Let $(\pi:A \to M, \xi, \lambda, \anc)$ be an anchored bundle with chosen prolongations. 
    Recall that by Definition \ref{def:monoidal-category}, the right leg of $A.U$ is written $\anc$, so it induces a span map:\[
\begin{tikzcd}
	M & {A.U} & {T^UM} \\
	M & {T^UM} & {T^UM}
	\arrow[Rightarrow, no head, from=2-2, to=2-3]
	\arrow["{\anc^U}", from=1-2, to=1-3]
	\arrow["{\anc^U}", from=1-2, to=2-2]
	\arrow[Rightarrow, no head, from=1-3, to=2-3]
	\arrow["{\pi^U}"', from=1-2, to=1-1]
	\arrow["{p^U}"', from=2-2, to=2-1]
	\arrow[Rightarrow, no head, from=1-1, to=2-1]
\end{tikzcd}\]
    This map has a flat presentation as
    \[
\begin{tikzcd}
	{A.UV} & {A.U \ts{\anc^U}{T^U.\pi^V}T^U.A.V} \\
	\\
	{T^U.A.V} & {T^UM \ts{id}{T^U.\pi^V} T^U.A.V}
	\arrow["{\anc^U \boxtimes A.V}"{description}, from=1-1, to=3-1]
	\arrow[Rightarrow, no head, from=3-1, to=3-2]
	\arrow[Rightarrow, no head, from=1-1, to=1-2]
	\arrow["{\pi^U \x T^U.A.V}"{description}, from=1-2, to=3-2]
\end{tikzcd}        
    \]
    We write the map
    \[
        \anc^U.V := \anc^U \boxtimes (\widehat A.V)
    \]
    which corresponds to the following span morphism:
\[\begin{tikzcd}
	&& {A.UV} \\
	& {A.U} && {T^U.A.V} \\
	M && {T^UM} && {T^{UV}M} \\
	& {T^UM} && {T^U.A.V} \\
	&& {T^U.A.V}
	\arrow["{\anc^U}", from=2-2, to=4-2]
	\arrow[Rightarrow, no head, from=2-4, to=4-4]
	\arrow["{T^U.\anc^V}"{description}, from=2-4, to=3-5]
	\arrow["{T^U.\pi^V}"{description}, from=2-4, to=3-3]
	\arrow["{\anc^U}"{description}, from=2-2, to=3-3]
	\arrow["{\pi^U}"{description}, from=2-2, to=3-1]
	\arrow[from=5-3, to=4-2]
	\arrow[Rightarrow, no head, from=5-3, to=4-4]
	\arrow[from=1-3, to=2-2]
	\arrow[from=1-3, to=2-4]
	\arrow["\lrcorner"{anchor=center, pos=0.125, rotate=-45}, draw=none, from=1-3, to=3-3]
	\arrow["{p^U.M}"{description}, from=4-2, to=3-1]
	\arrow[Rightarrow, no head, from=4-2, to=3-3]
	\arrow["{T^U.\pi^V}"{description}, from=4-4, to=3-3]
	\arrow[from=4-4, to=3-5]
\end{tikzcd}\]
\end{definition}

\begin{theorem}[The Weil Nerve]
    \label{thm:weil-nerve}
    There is a fully faithful functor
    \[
        \mathsf{N}_{\weil}: \mathsf{Inv}(\C) \to [\wone, \C]  
    \]
    that sends an involution algebroid to the transverse-limit-preserving tangent functor:
    \[
        (\widehat{A},\alpha): \wone \to \C 
    \]
\end{theorem}
\begin{proof}
    For the first step of this proof, we show that an involution algebroid structure on an anchored bundle $(\pi:A \to M, \xi, \lambda, \anc)$ determines a tangent category structure on the monoidal category $\mathsf{Span}(\pi:A \to M, \xi, \lambda, \anc)$.
    
    We check that the endofunctor $\hat{A} \boxtimes (-)$ determines a tangent structure, with the structure maps given by Definition \ref{def:generators-for-wone-in-anc}:
    
    \begin{enumerate}[{[TC.1]}]
        \item Additive bundle axioms:
        \begin{enumerate}[(i)]
            \item Use Lemma \ref{lemma:pullback-part-of-theorem} to see that \[ \hat{A}.W_2 = \hat{A}.W \ts{p}{p} \hat{A}.W  A \ts{\pi}{\pi} A;\] this is preserved by $\hat A.V \boxtimes (-)$.
            \item The triple $(\hat{A}.+,\hat{A}.p,\hat{A}.0) = (+_q, \pi, \xi)$ is an additive bundle induced by Proposition  \ref{prop:induce-abun}, and $\boxtimes$ preserves pullbacks (and therefore additive bundles), so the additive bundle axioms hold.
        \end{enumerate} 
        \item Symmetry axioms:
        \begin{enumerate}[(i)]
            \item $\hat{A}.c \o \hat{A}.c = id$ follows from the involution axiom $\sigma \o \sigma = id$.
            \item For Yang--Baxter, note that 
            \[ 
                (\hat{A} \boxtimes c) \o (c \boxtimes \hat{A}) \o (\hat{A} \boxtimes c)  = 
                (c \boxtimes \hat{A})  \o (\hat{A} \boxtimes c)  \o (c \boxtimes \hat{A}) 
            \] 
            follows from the Yang--Baxter equation on an involution algebroid
            \[
              (\sigma \x c)  \o (id \x T.\sigma) \o (\sigma \x c) =
              (id \x T.\sigma) \o (\sigma \x c)\o (id \x T.\sigma),
            \]
            since 
            \[\sigma \x c.A = (\hat{A}.c) \boxtimes (\hat{A}.W )
                \text{ and }
            id \x T.\sigma = (\hat{A}.W ) \boxtimes c. \]
            \item For the naturality conditions:
            \begin{enumerate}[(a)]
                \item The interchanges of $+,0,p$ all follow from the fact that 
                \[\sigma:(A.WW,\lambda \x \ell) \to (A.WW,id \x c \o T.\lambda)\] 
                is linear, and so is an additive bundle morphism.
                \item The axiom 
                \[\ell.T \o c = T.c \o c.T \o T.\ell\] 
                is equivalent to the equation
                \[
                    (\sigma \x c) \o (1\x T.\sigma) \o (\hat{\lambda} \x \ell) = (1 \x T\hat{\lambda}) \o \sigma
                \]
                which is equivalent to the double linearity axiom on $\sigma$ by Proposition  \ref{prop:nat-of-sigma-ell}.
            \end{enumerate}
        \end{enumerate}
        \item The lift axioms:
            \begin{enumerate}[(i)]
                \item The additive bundle equations are a consequence of $\lambda$ being a lift and $+$ being the addition induced by the non-singularity of $\lambda$.
                \item The coassociativity axiom 
                \[\ell.T \o \ell = T.\ell \o \ell\] 
                is equivalent to 
                \[(\hat{\lambda}\x\ell)\o\hat{\lambda} = (id \x T.\hat{\lambda}) \o \hat{\lambda}\] 
                proved in (i) of Proposition  \ref{prop:lift-axioms-anchor}.
                \item The symmetry of comultiplication, $c \o \ell = \ell$, is given by the unique equation for an involution algebroid, so that $\sigma \o (\xi\o\pi,\lambda) = (\xi\o\pi,\lambda)$.
                \item The universality of the lift follows from part (ii) of Proposition  \ref{prop:lift-axioms-anchor}; Lemma \ref{lemma:pullback-part-of-theorem} ensures that for any $V \in \wone$, $\widehat{A}.V \boxtimes \mu$ and $\mu \boxtimes \widehat{A}.V$ are universal.
            \end{enumerate}
    \end{enumerate}
    This lemma puts a tangent structure on $\mathsf{Span}(\pi,\xi,\lambda,\anc)$. Now consider the functor sending spans to the apex map, 
    \[
        U^\anc: \mathsf{Span}(\pi,\xi,\lambda,\anc) \to \C.
    \]
    The family of maps
    \[
        \{\anc^U.V: \anc^U \boxtimes \widehat A.V | U,V \in \wone\}
    \]
    gives a family of natural transformations
    \[
        \anc^{U}:A.T^U \Rightarrow T^U.A,
    \]
    so that the following pair constitute a tangent functor
    \[
      (U^\anc,\anc): \mathsf{Span}(\pi,\xi,\lambda,\anc) \to \C.
    \]
    Because the universality conditions on $\mathsf{Span}(\pi,\xi,\lambda,\anc)$ followed by reflecting limits in $\C$ using Lemma \ref{lemma:pullback-part-of-theorem}, it follows that $(U^\anc,\anc)$ will preserve the tangent-natural limits in $\mathsf{Span}(\pi,\xi,\lambda,\anc)$ corresponding to transverse limits in $\wone$.
    
    By Leung's Theorem  \ref{thm:leung} (by way of Corollary \ref{cor:using-leung-thm}), the tangent structure on $\mathsf{Span}(\pi,\xi,\lambda,\anc)$ induces a strict, monoidal, transverse-limit-preserving functor
    \[
        \bar{A}: \wone \to \mathsf{Span}(\pi,\xi,\lambda,\anc)    
    \]
    that sends the tensor product $\ox$ to the span composition $\boxtimes$. By composing the strict tangent functor $(\bar{A},id)$ and $(U^\anc,\anc)$, we have a lax, transverse-limit-preserving, tangent functor:
    \[
        (A,\anc):\wone \to \C; V \mapsto A.V
    \]


    Now, check the bijection on morphisms. Starting with an involution algebroid morphism $(f,m):A \to B$, note that this gives a span morphism $\hat f$:
\[\begin{tikzcd}[ampersand replacement=\&]
	M \& A \& TM \\
	N \& B \& TN
	\arrow["{T.m}", from=1-3, to=2-3]
	\arrow["m", from=1-1, to=2-1]
	\arrow[from=1-2, to=1-1]
	\arrow[from=1-2, to=1-3]
	\arrow[from=2-2, to=2-3]
	\arrow[from=2-2, to=2-1]
	\arrow["f", from=1-2, to=2-2]
\end{tikzcd}\]
    This gives a natural definition of $\hat{f}.V$ using the horizontal composition of span morphism, so that 
    \begin{equation}\label{eq:boxtimes-def-of-hat-f}
        \hat{f}.(UV) = \hat{f}.U \boxtimes \hat{f}.V \text{ and } \hat{f}.\N = m,
    \end{equation}
    giving a family of maps $\{\hat{f}_{U}: U \in \mathsf{objects}(\wone)\}$. Because $f$ will commute with the structure maps $\{ \pi,\xi,+,(\xi\o\pi,\lambda),\sigma\}$, it follow immediately that $\hat{f}$ is a natural transformation, because the following calculation holds for each $\theta:X \to Y \in \{ p,0,+,\ell,c\}$:
    \begin{align*}
        & \quad \hat{f}.UYV \o (\hat{A}.U \boxtimes \theta \boxtimes \hat{A}.V) \\
        &= (\hat{f}.U \boxtimes \hat{f}.Y \boxtimes \hat{f}.V) \o  (\hat{A}.U \boxtimes \theta \boxtimes \hat{A}.V) \\
        &= \hat{f}.U \boxtimes( \hat{f}.Y \o \theta)\boxtimes \hat{f}.V \\
        &= \hat{f}.U \boxtimes( \theta \o \hat{f}.X) \boxtimes \hat{f}.V \\
        &= (\hat{A}.U \boxtimes \theta \boxtimes \hat{A}.V) \o \hat{f}.UXV
    \end{align*}
    Tangent naturality will follow by the preservation of the anchor map by $f$. 
    The equality, for any Weil algebra $U$, of the diagrams
\[\begin{tikzcd}[ampersand replacement=\&]
	M \& {A.U} \& {T^UM} \& M \& {A.U} \& {T^UM} \\
	N \& {B.U} \& {T^UN} \& M \& {T^UM} \& {T^UM} \\
	N \& {T^UN} \& {T^UN} \& N \& {T^VN} \& {T^VN}
	\arrow["{T^U.m}", from=1-3, to=2-3]
	\arrow["m", from=1-1, to=2-1]
	\arrow["{\pi^U}"', from=1-2, to=1-1]
	\arrow["{\anc^U}", from=1-2, to=1-3]
	\arrow["{\anc^U}", from=2-2, to=2-3]
	\arrow[from=2-2, to=2-1]
	\arrow["{\hat{f}.U}"{description}, from=1-2, to=2-2]
	\arrow["{\anc^U}", from=2-2, to=3-2]
	\arrow["{p^U}", from=3-2, to=3-1]
	\arrow[Rightarrow, no head, from=2-1, to=3-1]
	\arrow[Rightarrow, no head, from=2-3, to=3-3]
	\arrow[Rightarrow, no head, from=3-2, to=3-3]
	\arrow["{\pi^U}"', from=1-5, to=1-4]
	\arrow["{p^U}"', from=2-5, to=2-4]
	\arrow["{\anc^U}"', from=1-5, to=1-6]
	\arrow[Rightarrow, no head, from=2-5, to=2-6]
	\arrow[Rightarrow, no head, from=3-5, to=3-6]
	\arrow["{p^U}", from=3-5, to=3-4]
	\arrow[Rightarrow, no head, from=1-6, to=2-6]
	\arrow["{T^U.m}", from=2-6, to=3-6]
	\arrow["{T^U.m}", from=2-5, to=3-5]
	\arrow["m"', from=2-4, to=3-4]
	\arrow[Rightarrow, no head, from=1-4, to=2-4]
	\arrow["{\anc^U}", from=1-5, to=2-5]
	\arrow["{=}"{description}, draw=none, from=2-3, to=2-4]
\end{tikzcd}\]
is precisely the tangent-naturality condition from Definitions \ref{def:tang-nat}, \ref{def:actegory-natural}.
    
    For the inverse of this mapping, consider a tangent natural transformation (Definition \ref{def:tang-nat})
    \[
        \gamma: ({A},\alpha) \to ({B}, \beta),  \hspace{0.15cm}
            \begin{tikzcd}
    A \circ T(A) \rar{\gamma_{TA}} \dar{\tnat_A} & B \circ T(A) \dar{\beta_{TA}} \\
    T \circ A(A) \rar{T\gamma_A} & T \circ B(A)
\end{tikzcd}
    \]
    where $(A,\alpha)$ and $(B,\beta)$ are tangent functors $\wone \to \C$ built out of involution algebroids with chosen prolongations. For any $U,V$, the map $\gamma.UV$ decomposes as $\gamma.U \boxtimes \gamma.V$:
    \[
\begin{tikzcd}
	{\widehat{A}.T.T} && {T.\widehat A.T} \\
	& {\widehat B.T.T} && {T.\widehat B.T} \\
	{\widehat A.T} && {T.\widehat A} \\
	& {\widehat B.T} && {T.\widehat B}
	\arrow["{\widehat A.p.T = \pi_0}"{description}, from=1-1, to=3-1]
	\arrow["{\alpha.T = \pi_1}", from=1-1, to=1-3]
	\arrow["{\widehat A.p}"{description, pos=0.7}, from=3-1, to=3-3]
	\arrow["{T.\widehat A.p}"{description, pos=0.7}, from=1-3, to=3-3]
	\arrow["{\widehat B.p}"{description}, from=4-2, to=4-4]
	\arrow["{T.\widehat B.p}"{description}, from=2-4, to=4-4]
	\arrow["{\widehat B.p.T = \pi_0}"{description, pos=0.2}, from=2-2, to=4-2]
	\arrow["{\beta.T = \pi_1}"{pos=0.2}, from=2-2, to=2-4]
	\arrow["{\gamma.T}"{description}, from=3-1, to=4-2]
	\arrow["{T.\gamma}"{description}, from=3-3, to=4-4]
	\arrow["{T.\gamma.T}"{description}, from=1-3, to=2-4]
	\arrow["{\gamma.TT}", dashed, from=1-1, to=2-2]
\end{tikzcd}   
    \]
    Applying this relationship inductively, it is clear that the base maps $\gamma.W$ and $\gamma.\N$ determine the entire morphism $\gamma.V$:
\[\begin{tikzcd}
	& M & {A.U} & {T^UM} \\
	& N & {B.U} & {T^UN} \\
	M & {A_{n[1]}} & {T_{n[1]}.M} & \dots & {T_{n[1]}\dots T_{n[k]}.M} \\
	N & {B_{n[1]}} & {T_{n[1]}.M} & \dots & {T_{n[1]}\dots T_{n[k]}.N}
	\arrow["{T_{n[1]}\dots T_{n[k]}.\gamma.\N}", from=3-5, to=4-5]
	\arrow["{\gamma.\N}", from=3-1, to=4-1]
	\arrow["{\anc_{n[1]}}", from=3-2, to=3-3]
	\arrow["{\pi\o\pi_i}"', from=3-2, to=3-1]
	\arrow["{\pi\o\pi_i}", from=4-2, to=4-1]
	\arrow["{\anc_{n[1]}}"', from=4-2, to=4-3]
	\arrow["{\gamma.W_{n[1]}}", from=3-2, to=4-2]
	\arrow[""{name=0, anchor=center, inner sep=0}, "{T_{n[1]}.\gamma.\N}", from=3-3, to=4-3]
	\arrow["{T_{n[1]}.(\pi \o \pi_i)}"', from=3-4, to=3-3]
	\arrow[from=3-4, to=3-5]
	\arrow["{T_{n[1]}.(\pi \o \pi_i)}", from=4-4, to=4-3]
	\arrow[from=4-4, to=4-5]
	\arrow["\dots"{description}, draw=none, from=3-4, to=4-4]
	\arrow["{T^U.\gamma.\N}"', from=1-4, to=2-4]
	\arrow["{\pi^U}"', from=1-3, to=1-2]
	\arrow["{\pi^U}", from=2-3, to=2-2]
	\arrow["{\anc^U}"', from=2-3, to=2-4]
	\arrow["{\anc^U}", from=1-3, to=1-4]
	\arrow["{\gamma.\N}"', from=1-2, to=2-2]
	\arrow[""{name=1, anchor=center, inner sep=0}, "{\gamma.U}"', from=1-3, to=2-3]
	\arrow["{=}"{description}, Rightarrow, draw=none, from=1, to=0]
\end{tikzcd}\]
    Thus, every tangent-natural transformation is constructed out of a pair
    \[
        (\gamma.\N: M \to N, \gamma.W:A \to B)
    \]
    using the $\boxtimes$ construction from Equation \ref{eq:boxtimes-def-of-hat-f}.
    All that remains to show is that this pair is an involution algebroid morphism. 

    Tangent naturality gives the following two coherences:
    \[
        \anc^B \o \gamma.W = T.\gamma.\N \o \anc^A \text{ and } 
        \sigma^B \o \gamma.WW = \gamma.WW \o \sigma^A
    \]
    since $\anc^B = \beta.W, \anc^A = \alpha.W, \sigma^B = B.c$, and $\sigma^A = A.c$ by construction. The following diagram proves that $\gamma.W$ preserves the lifts, so that $(\gamma.W, \gamma.\N)$ is an involution algebroid morphism:
    \[
\begin{tikzcd}
    {T.A.T} & {T.A.T} & {T.A.T} & {T.B.T} & {T.B.T} & {T.B.T} \\
    & {A.T.T} & {A.T.T} & {B.T.T} & {B.T.T} \\
    {A.T} & {A.T} & {A.T} & {B.T} & {B.T} & {B.T}
    \arrow["{A.\ell}"{description}, from=3-3, to=2-3]
    \arrow["{\alpha.T}"{description}, from=2-3, to=1-3]
    \arrow["{(\xi\pi,\lambda^A)}"{description}, from=3-2, to=2-2]
    \arrow[Rightarrow, no head, from=3-2, to=3-3]
    \arrow[Rightarrow, no head, from=2-2, to=2-3]
    \arrow[Rightarrow, no head, from=1-2, to=1-3]
    \arrow["{\pi_1}"{description}, from=2-2, to=1-2]
    \arrow["{T.\gamma}", from=1-3, to=1-4]
    \arrow["{\gamma.T}", from=2-3, to=2-4]
    \arrow["m", from=3-3, to=3-4]
    \arrow["{B.\ell}"{description}, from=3-4, to=2-4]
    \arrow["{\beta.T}"{description}, from=2-4, to=1-4]
    \arrow["{(\xi\pi,\lambda^B)}"{description}, from=3-5, to=2-5]
    \arrow["{\beta.T}"{description}, from=2-5, to=1-5]
    \arrow[Rightarrow, no head, from=1-4, to=1-5]
    \arrow[Rightarrow, no head, from=2-4, to=2-5]
    \arrow[Rightarrow, no head, from=3-4, to=3-5]
    \arrow["{\lambda^A}"{description}, from=3-1, to=1-1]
    \arrow[Rightarrow, no head, from=1-1, to=1-2]
    \arrow[Rightarrow, no head, from=3-1, to=3-2]
    \arrow[Rightarrow, no head, from=3-5, to=3-6]
    \arrow[Rightarrow, no head, from=1-5, to=1-6]
    \arrow["{\lambda^B}"{description}, from=3-6, to=1-6]
\end{tikzcd}\]
    Thus, a tangent natural transformation $(\widehat{A},\alpha) \to (\widehat{B}, \beta)$ is exactly a morphism of involution algebroids $A \to B$, proving the theorem.
    
\end{proof} 
Now, the projection for a Lie algebroid is a submersion, as we may make a choice of prolongations for each $U \in \wone$. These prolongations lead to a new observation about Lie algebroids: they embed into a category of functors into smooth manifolds.
\begin{corollary}%
    \label{cor:SMan-embedding}
    Using the Weil nerve construction, the category of Lie algebroids embeds into the tangent-functor category:
    \[
        \mathsf{LieAlgd} \hookrightarrow [\wone, \mathsf{SMan}].  
    \]
\end{corollary}

\section{Identifying involution algebroids}%
\label{sec:identifying-involution-algebroids}

This section identifies those tangent functors
\[
    (A,\alpha): \wone \to \C  
\] 
that are involution algebroids as precisely those where $A$ preserves transverse limits and $\alpha$ is a \emph{$T$-cartesian} natural transformation (Definition \ref{def:cart-nat}). 
These conditions will force each $A.V$ to be the $V$-prolongation of the underlying anchored pre-differential bundle:
\[
    (A.p: A.T \to A,\;\ A.0: A \to A.T,\;\ A.T \xrightarrow[]{A.\ell} A.T.T \xrightarrow[]{\alpha.T} T.A.T,\;\ \alpha:A.T \to T.A )  
\]
(these conditions also ensure that this tuple is an anchored differential bundle). 

Initially, it is only clear that $\alpha$ is $T$-cartesian for the projection $p$. Indeed, recall that the prolongation $A.UV$ is defined to be the $T$-pullback of the cospan:
\[ \widehat A .U \xrightarrow[]{\alpha^U} T^U.A.\N \xleftarrow[]{T^U.A.p^V} T^U.\widehat{ A}.V\]
Then consider the following diagram:
\[
\begin{tikzcd}
	{\widehat A.UWV} & {T^U. \widehat A.T.T^V} \\
	{\widehat A.U.V} & {T^U.\widehat A.V} \\
	{\widehat A.T^U} & {T^U.\widehat A.\N}
	\arrow[from=1-1, to=1-2]
	\arrow[from=2-1, to=2-2]
	\arrow["{T^U.\hat A.p.V}", from=1-2, to=2-2]
	\arrow["{\widehat A.T^U.p.T^V}"', from=1-1, to=2-1]
	\arrow["{\widehat A.T^U.p^V}"', from=2-1, to=3-1]
	\arrow["{T^U.\widehat A.p^V}", from=2-2, to=3-2]
	\arrow[from=3-1, to=3-2]
	\arrow["\lrcorner"{anchor=center, pos=0.125}, draw=none, from=2-1, to=3-2]
	\arrow["\lrcorner"{anchor=center, pos=0.125}, draw=none, from=1-1, to=2-2]
\end{tikzcd}\]
This means that every naturality square of $\alpha$ for $p$ is a $T$-pullback; natural transformations satisfying this property for every map in the domain category are called $T$-cartesian.
\begin{definition}%
    \label{def:cart-nat}
    A natural transformation $\gamma: F \Rightarrow G$ is \emph{cartesian} whenever each naturality square
    \[\begin{tikzcd}
        FC & GC \\
        FD & GD
        \arrow["{F.f}"', from=1-1, to=2-1]
        \arrow["\gamma", from=1-1, to=1-2]
        \arrow["\gamma"', from=2-1, to=2-2]
        \arrow["{G.f}", from=1-2, to=2-2]
        \arrow["\lrcorner"{anchor=center, pos=0.125}, draw=none, from=1-1, to=2-2]
\end{tikzcd}\]
    is a pullback. A natural transformation between functors into a tangent category is \emph{$T$-cartesian} whenever each component square is a $T$-pullback (we will generally suppress the $T$ when the context is clear). 
\end{definition}

Now, recall that the Weil complex determined by an involution algebroid has $A.U.V$ determined by the following $T$-pullback squares:
\[
\begin{tikzcd}
	{\widehat A.U.V} & {T^U.\widehat A.V} \\
	{\widehat A.T^U} & {T^U.\widehat A.\N}
	\arrow[from=1-1, to=1-2]
	\arrow["{\widehat A.T^U.p^V}"', from=1-1, to=2-1]
	\arrow["{T^U.\widehat A.p^V}", from=1-2, to=2-2]
	\arrow[from=2-1, to=2-2]
	\arrow["\lrcorner"{anchor=center, pos=0.125}, draw=none, from=1-1, to=2-2]
	\arrow["\lrcorner"{anchor=center, pos=0.125}, draw=none, from=1-1, to=2-2]
\end{tikzcd}\]
Then it is not difficult to show that the $T$-cartesian condition on a Weil complex forces it to be an involution algebroid.  We first need:
\begin{definition}
    A $T$-cartesian Weil complex in $\C$ is a tangent functor
    \[
        (A,\alpha): \wone \to \C  
    \]
    for which $A$ sends transverse limits to $T$-limits and $\alpha$ is a $T$-cartesian natural transformation.
\end{definition}
The first condition to check is that a $T$-cartesian Weil complex gives a natural anchored bundle $\hat{A}$ whose Weil prolongations coincide with the functor assignments on objects.
\begin{proposition}%
    \label{prop:anc-bun-cw-complex}
    Let $(A,\alpha)$ be a $T$-cartesian Weil complex. Then we have an anchored bundle
    \[
        (M  := A.\N, \hspace{0.15cm}
        \hat{A} := A.W , \hspace{0.15cm}
        \pi := A.\pi, \hspace*{0.15cm}
        \xi := A.\xi, \hspace*{0.15cm}
        \lambda := \alpha.T \o A.\ell).
    \]
    Furthermore, 
    \[
        \prol(\hat{A}) = A.WW, \hspace*{0.30cm}
        \prol^2(\hat{A}) = A.WWW. 
    \]
\end{proposition}
\begin{proof}
    Suppose we have a tangent functor $(F,\alpha): \C \to \D$ and a differential bundle $(\pi,\xi,\lambda)$ in $\C$. If $F$ preserves $T$-pullbacks of $\pi$, it preserves the additive bundle structure on $(\pi,\xi,+)$, so to show $(F.\pi, F.\xi, \alpha \o F.\lambda)$ is universal it suffices to show that the following diagram is a $T$-pullback in $\D$:
\[\begin{tikzcd}
    {F.A_2} & {T.F.A} \\
    {F.M} & {T.F.M}
    \arrow["{F.\pi_2}"', from=1-1, to=2-1]
    \arrow["{\mu^{\alpha \o F.\lambda}}", from=1-1, to=1-2]
    \arrow["{F.0}"', from=2-1, to=2-2]
    \arrow["{F.T.\pi}"', from=1-2, to=2-2]
\end{tikzcd}\]
    Expand this to
\[\begin{tikzcd}
    {F.A_2} & {F.T.A} & {T.F.A} \\
    {F.M} & {F.T.M} & {T.F.M}
    \arrow["{F.\pi_2}"', from=1-1, to=2-1]
    \arrow["{F.\mu^\lambda}", from=1-1, to=1-2]
    \arrow["{F.0}"', from=2-1, to=2-2]
    \arrow["{T.F.\pi}"', from=1-3, to=2-3]
    \arrow["{F.T.\pi}"', from=1-2, to=2-2]
    \arrow["{\alpha.A}", from=1-2, to=1-3]
    \arrow["{\alpha.M}"', from=2-2, to=2-3]
\end{tikzcd}\]
    In this case, it restricts to the diagram
\[\begin{tikzcd}
    {A.T_2} & {A.T^2} & {T.A.T} \\
    A & {A.T} & {T.A}
    \arrow["{A.\mu}", from=1-1, to=1-2]
    \arrow["{\alpha.T}", from=1-2, to=1-3]
    \arrow[from=1-1, to=2-1]
    \arrow["{A.T.p}"{description}, from=1-2, to=2-2]
    \arrow["\alpha"', from=2-2, to=2-3]
    \arrow["{T.A.p}"{description}, from=1-3, to=2-3]
    \arrow["{A.0}"', from=2-1, to=2-2]
    \arrow["\lrcorner"{anchor=center, pos=0.125}, draw=none, from=1-2, to=2-3]
\end{tikzcd}\]

    Each square is a $T$-pullback by hypothesis, so the universality of the lift follows by the $T$-pullback lemma. Because the complex is $T$-cartesian, the assignment $A.V$ gives a coherent choice of prolongations by the $T$-pullback
    \[\begin{tikzcd}
    {A.U.V} & {T^U.A.V} \\
    {A.U} & {T^U.A}
    \arrow["{A.U.p^V}"', from=1-1, to=2-1]
    \arrow["{\alpha^U}", from=1-1, to=1-2]
    \arrow["{T^U.A.p^V}", from=1-2, to=2-2]
    \arrow["{\alpha^U}"', from=2-1, to=2-2]
\end{tikzcd}\] 
\end{proof}
There is, of course, a natural candidate for the involution map.
\begin{corollary}
    Let $(\pi:A \to M, \xi, \lambda, \anc)$ be the anchored bundle induced by a $T$-cartesian Weil complex in a tangent category $\C$. Then we have an involution map
    \[
        \sigma: \prol(A) \xrightarrow{A.c} \prol(A).
    \]
\end{corollary}
The equations for an involution algebroid should follow immediately by functoriality; one need only ensure that the maps take the correct form.
\begin{lemma}\label{lem:cwm-map-structure}
    Let $A$ be a $T$-cartesian Weil complex in a tangent category $\C$, with $(\pi:A \to M, \xi, \lambda, \sigma)$ its underlying anchored bundle.
    Then we have:
    \begin{enumerate}[(i)]
        \item $A.c.T = \sigma \x c$,
        \item $A.T.c = 1 \x T.\sigma$,
        \item $A.\ell.T = \hat{\lambda} \x \ell.A$,
        \item $A.T.\ell = id \x T.\hat{\lambda}$.
    \end{enumerate}
\end{lemma}
\begin{proof}
    ~\begin{enumerate}[(i)]
        \item Consider the diagram
\[\begin{tikzcd}
    {A.T.T.T} &&& {A.T.T.T} \\
    & {T.T.A.T} & {T.T.A.T} \\
    {} & {T.T.A} & {T.T.A} \\
    {A.T.T} &&& {A.T.T}
    \arrow["{A.T.T.p}"', from=1-1, to=4-1]
    \arrow["{\alpha^{TT}}", from=1-1, to=2-2]
    \arrow["{\alpha^{TT}}"', from=4-1, to=3-2]
    \arrow["{T.T.A.p}", from=2-2, to=3-2]
    \arrow["{A.T.T.p}", from=1-4, to=4-4]
    \arrow[from=4-4, to=3-3]
    \arrow["{T.T.A.p}", from=2-3, to=3-3]
    \arrow[from=1-4, to=2-3]
    \arrow["{c.A.T}", from=2-2, to=2-3]
    \arrow["{c.A}"', from=3-2, to=3-3]
    \arrow["{A.c}"{description}, from=4-1, to=4-4]
    \arrow["{A.c.T}"{description}, from=1-1, to=1-4]
    \arrow["\lrcorner"{anchor=center, pos=0.125}, draw=none, from=1-1, to=3-2]
    \arrow["\lrcorner"{anchor=center, pos=0.125, rotate=-90}, draw=none, from=1-4, to=3-3]
\end{tikzcd}\]

        Observe that this forces $A.c.T = A.c \x c.A.T = \sigma \x c$.
        \item Likewise, the diagram
\[\begin{tikzcd}
    {A.T.T.T} &&& {A.T.T.T} \\
    & {T.T.A.T} & {T.T.A.T} \\
    {} & {T.A} & {T.T.A} \\
    {A.T} &&& {A.T}
    \arrow["{A.T.p.p}"', from=1-1, to=4-1]
    \arrow["{\alpha.TT}", from=1-1, to=2-2]
    \arrow["{\alpha^{T}}"', from=4-1, to=3-2]
    \arrow["{T.T.A.p}", from=2-2, to=3-2]
    \arrow["{A.T.p.p}", from=1-4, to=4-4]
    \arrow["{\alpha^T}", from=4-4, to=3-3]
    \arrow["{T.T.A.p}", from=2-3, to=3-3]
    \arrow["{\alpha.TT}"', from=1-4, to=2-3]
    \arrow["{T.A.c}", from=2-2, to=2-3]
    \arrow[Rightarrow, no head, from=3-2, to=3-3]
    \arrow[Rightarrow, no head, from=4-1, to=4-4]
    \arrow["{A.T.c}"{description}, from=1-1, to=1-4]
    \arrow["\lrcorner"{anchor=center, pos=0.125}, draw=none, from=1-1, to=3-2]
    \arrow["\lrcorner"{anchor=center, pos=0.125, rotate=-90}, draw=none, from=1-4, to=3-3]
\end{tikzcd}\]
        forces $A.T.c = id \x T.A.c = id \x T.\sigma$.
        \item The diagram
\[\begin{tikzcd}
    {A.T.T} & {} && {A.T.T.T} \\
    & {T.A.T} & {T.T.A.T} \\
    & {T.A} & {T.T.A} \\
    {A.T} &&& {A.T.T}
    \arrow["{A.\ell.T}"{description}, from=1-1, to=1-4]
    \arrow["{\alpha.T}"{description}, from=1-1, to=2-2]
    \arrow["{\alpha^{TT}.T}"{description}, from=1-4, to=2-3]
    \arrow["{\ell.A.T}", from=2-2, to=2-3]
    \arrow["{A.T.p}"{description}, from=1-1, to=4-1]
    \arrow["{T.A.p}", from=2-2, to=3-2]
    \arrow["{T.T.A.p}", from=2-3, to=3-3]
    \arrow["{\ell.A}"', from=3-2, to=3-3]
    \arrow["\alpha"{description}, from=4-1, to=3-2]
    \arrow["{A.\ell}"{description}, from=4-1, to=4-4]
    \arrow["{A.T.T.p}"{description}, from=1-4, to=4-4]
    \arrow["{\alpha^{TT}}"{description}, from=4-4, to=3-3]
    \arrow["\lrcorner"{anchor=center, pos=0.125}, draw=none, from=1-1, to=3-2]
    \arrow["\lrcorner"{anchor=center, pos=0.125, rotate=-90}, draw=none, from=1-4, to=3-3]
    \arrow["\lrcorner"{anchor=center, pos=0.125, rotate=45}, draw=none, from=4-1, to=3-3]
    \arrow["\lrcorner"{anchor=center, pos=0.125, rotate=45}, draw=none, from=1-1, to=2-3]
\end{tikzcd}\]

        forces $A.\ell.T = A.\ell \x \ell.A = \hat{\lambda} \x \ell$.
        \item As $\alpha$ is $T$-cartesian, the following diagram is a $T$-pullback:
\[\begin{tikzcd}
    {A.T.T} & {A.T.T.T} \\
    {T.A.T} & {T.A.T.T}
    \arrow["{\alpha.T}"', from=1-1, to=2-1]
    \arrow["{A.T.\ell}", from=1-1, to=1-2]
    \arrow["{\alpha.T.T}", from=1-2, to=2-2]
    \arrow["{T.A.\ell}"', from=2-1, to=2-2]
    \arrow["\lrcorner"{anchor=center, pos=0.125}, draw=none, from=1-1, to=2-2]
\end{tikzcd}\]
        Using previous results, this means that $A.\ell.T$ is the unique map making the following diagram commute:
\[\begin{tikzcd}
    {A \ts{\anc}{T.\pi}TA} & {A \ts{\anc}{T.\pi}TA \ts{T.\anc}{T^2.\pi}T^2A} \\
    TA & {TA \ts{T.\anc}{T^2.\pi}T^2A}
    \arrow["{(\pi_1,\pi_2)}", from=1-2, to=2-2]
    \arrow["{T.\hat\lambda}"', from=2-1, to=2-2]
    \arrow["{\pi_1}"', from=1-1, to=2-1]
    \arrow[dashed, from=1-1, to=1-2]
\end{tikzcd}\]
        which we can see is $id \x \hat\lambda$.
    \end{enumerate}
\end{proof}

Pulling together this lemma and the previous proposition, the following is now clear:
\begin{proposition}
    A $T$-cartesian Weil complex determines an involution algebroid.
\end{proposition}
However, we have not yet exhibited an isomorphism of categories between the image of the Weil nerve functor and $T$-cartesian Weil complexes. 
At first glance, the Weil nerve construction only gives a Weil complex that is $T$-cartesian for the tangent projection $p \in \wone$. Being $T$-cartesian for $p$ is, however, sufficient: a Weil complex that is $T$-cartesian for tangent projections will be $T$-cartesian for every map in $\wone$ (a similar result appears in the context of differentiable programming languages; see \cite{Cruttwell2019}).

\begin{proposition}\label{prop:mod-is-cart-if-p}
    A lax tranverse-limit-preserving tangent functor $(F,\alpha):\wone \to \C$ for which $F$ preserves pullback powers of each $T^U.p$ is $T$-cartesian if and only if each
    \begin{equation*}
        \begin{tikzcd}
        F.T.T \dar[swap]{F.T.p} \rar{\alpha_A} & T.F.T \dar{T.F.p} \\
        F.T \rar{\alpha_B} & T.F
\end{tikzcd}
    \end{equation*}
    is a $T$-pullback.
\end{proposition}
\begin{proof}
    We only check the converse since the forward implication is trivial. We make use of the $T$-pullback lemma.
    \begin{enumerate}[(i)]
        \item $c$ is an isomorphism, so its naturality square is a $T$-pullback.
        \item For projections $T_2 \to T$,  the retract of a $T$-pullback diagram is a $T$-pullback, so the following diagram is universal:
        \begin{equation*}
    \begin{tikzcd}
        F.T.T_2 \ar[equals]{rrr} \ar[equals]{rd} \ar{ddd}{\alpha} & & & F.T.T_2 \ar{ddd}{\alpha} \ar{ld}{F.T.\pi_i} \\
        & F.T.T_2 \rar{F.T.\pi_i} \dar{\alpha} & F.T.T \dar{\alpha} \\
        & T.F.T_2 \rar{T.F.\pi_i}  & T.F.T \\
        F.T.T_2 \ar[equals]{rrr} \ar[equals]{ru} & & & F.T.T_2 \ar{lu}{F.T.\pi_i} \\
    \end{tikzcd}
\end{equation*}

        \item For $0$, observe that the following two diagrams are equal:
\[\begin{tikzcd}
	{F.T} & {F.T.T} & {F.T} & {F.T} & {F.T} \\
	{T.F} & {T.F.T} & {T.F} & {T.F} & {T.F}
	\arrow["\alpha", from=1-3, to=2-3]
	\arrow["{F.T.p}", from=1-2, to=1-3]
	\arrow["{T.F.p}"', from=2-2, to=2-3]
	\arrow["{F.T.0}", from=1-1, to=1-2]
	\arrow["{T.F.0}"', from=2-1, to=2-2]
	\arrow["\alpha"', from=1-1, to=2-1]
	\arrow["\alpha"{description}, from=1-2, to=2-2]
	\arrow[Rightarrow, no head, from=1-4, to=1-5]
	\arrow[Rightarrow, no head, from=2-4, to=2-5]
	\arrow["\alpha", from=1-5, to=2-5]
	\arrow["\alpha"', from=1-4, to=2-4]
	\arrow["\lrcorner"{anchor=center, pos=0.125}, draw=none, from=1-2, to=2-3]
	\arrow["\lrcorner"{anchor=center, pos=0.125}, draw=none, from=1-4, to=2-5]
\end{tikzcd}\]
        The right diagram is a $T$-pullback, and the right square of the left diagram is a $T$-pullback by hypothesis.
        By the $T$-pullback lemma, the left square of the left diagram is a $T$-pullback.
        \item For $\ell$, observe that
\[\begin{tikzcd}
	{F.T.T} & {F.T.T.T} & {F.T.T} & {F.T.T} & {F.T} & {F.T.T} \\
	{T.F.T} & {T.F.T.T} & {T.F.T} & {T.F.T} & {T.F} & {T.F.T}
	\arrow["{F.T.\ell}", from=1-1, to=1-2]
	\arrow["{\alpha.T}"', from=1-1, to=2-1]
	\arrow["{T.F.\ell}"', from=2-1, to=2-2]
	\arrow["{\alpha.T.T}"', from=1-2, to=2-2]
	\arrow["{F.T.p.T}", from=1-2, to=1-3]
	\arrow["{T.F.p.T}"', from=2-2, to=2-3]
	\arrow["{\alpha.T}", from=1-3, to=2-3]
	\arrow["\lrcorner"{anchor=center, pos=0.125}, draw=none, from=1-2, to=2-3]
	\arrow["{F.T.p}", from=1-4, to=1-5]
	\arrow["{T.F.p}"', from=2-4, to=2-5]
	\arrow["{\alpha.T}"', from=1-4, to=2-4]
	\arrow["\alpha"', from=1-5, to=2-5]
	\arrow["\lrcorner"{anchor=center, pos=0.125}, draw=none, from=1-4, to=2-5]
	\arrow["{F.T.0}", from=1-5, to=1-6]
	\arrow["{T.F.0}"', from=2-5, to=2-6]
	\arrow["{\alpha.T}", from=1-6, to=2-6]
	\arrow["\lrcorner"{anchor=center, pos=0.125}, draw=none, from=1-5, to=2-6]
\end{tikzcd}\]
        The outer perimeter of the right diagram is a $T$-pullback (left square by hypothesis, right square by (ii)), as is the right square of the left diagram (by hypothesis). 
        By the $T$-pullback lemma, the left square of the left diagram is a $T$-pullback.
        \item For $+$, observe that
\[\begin{tikzcd}
	{F.T.T_2} & {F.T.T} & {F.T.T} & {F.T.T_2} & {F.T.T} & {F.T} \\
	{T.F.T_2} & {T.F.T} & {T.F.T} & {T.F.T_2} & {T.F.T} & {T.F}
	\arrow["{F.T.+}", from=1-1, to=1-2]
	\arrow["{\alpha.T}"', from=1-1, to=2-1]
	\arrow["{T.F.+}"', from=2-1, to=2-2]
	\arrow["{\alpha.T}"', from=1-2, to=2-2]
	\arrow["{F.T.p}", from=1-2, to=1-3]
	\arrow["{T.F.p}"', from=2-2, to=2-3]
	\arrow["{\alpha.T}", from=1-3, to=2-3]
	\arrow["\lrcorner"{anchor=center, pos=0.125}, draw=none, from=1-2, to=2-3]
	\arrow["{F.T.\pi_i}", from=1-4, to=1-5]
	\arrow["{T.F.\pi_i}"', from=2-4, to=2-5]
	\arrow["{\alpha.T_2}"', from=1-4, to=2-4]
	\arrow["{\alpha.T}"', from=1-5, to=2-5]
	\arrow["\lrcorner"{anchor=center, pos=0.125}, draw=none, from=1-4, to=2-5]
	\arrow["{F.T.p}", from=1-5, to=1-6]
	\arrow["{T.F.p}"', from=2-5, to=2-6]
	\arrow["\alpha", from=1-6, to=2-6]
	\arrow["\lrcorner"{anchor=center, pos=0.125}, draw=none, from=1-5, to=2-6]
\end{tikzcd}\]
        The outer diagram on the right is a $T$-pullback by composition, and the right square on the left diagram is a $T$-pullback by hypothesis, so the result follows. 
    \end{enumerate}
    To check that the naturality square is a $T$-pullback for \emph{every} map in $\wone$, we once again use Leung's characterization of maps in $\wone$ from Proposition \ref{thm:leung}. Inductively, the set of maps generated by $\{p,0,+,\ell,c\}$ closed under $\ox$ and $\o$ follows as $T$-pullback squares are closed to composition. For maps induced by a tranverse limit in $\wone$, $F$ preserves transverse limits so this follows by the commutativity of limits.
\end{proof}

\begin{theorem}\label{thm:iso-of-cats-inv-emcs}
    For any tangent category $\C$, the replete image of the Weil nerve functor
    \[
        \mathsf{Inv}(\C) \hookrightarrow [\wone, \C]  
    \]
    is precisely the category of $T$-cartesian Weil complexes.
\end{theorem}
\begin{corollary}\label{cor:the-prolongation-description}
    That $\alpha:A.T \Rightarrow T.A$ is $T$-cartesian is equivalent to requiring that the tangent functor
    \[
        (A,\alpha): \wone \to \C  
    \]
    restricts to an anchored bundle
    \[(\pi: A.T \xrightarrow[]{A.p} A.\N,\;\ \xi: A \xrightarrow[]{A.0} A.T,\;\ \lambda: A.T \xrightarrow[]{A.\ell} A.TT \xrightarrow[]{\alpha.T}T.A.T,\;\ \anc:A.T \xrightarrow[]{\alpha} T.A)\]
    and each $A.T^V$ is the $V$-prolongation of this anchor bundle.
\end{corollary}
\begin{remark}
    The condition in Corollary  \ref{cor:the-prolongation-description} is analogous to the Segal conditions identifying those simplicial complexes
    \[\Delta \to \C \]
    that are internal categories. Note that every simplicial object has an underlying reflexive graph
    \[
        \mathsf{tr}_1(X) := (s,t:X([1]) \to X([0]), i:X([0]) \to X([1]))
    \]
    where $X([n])$ is isomorphic to the object of $n$-composable arrows for the underlying reflexive graph.
\end{remark}
\begin{remark}
    Notably, being $T$-cartesian for $p$ is enough to force that a natural transformation is $T$-cartesian for the other tangent-structural natural transformations. This has consequences when one uses partial maps to combine \emph{topological} notions with tangent categories.
    In this context, a partial map $N \to X$ with domain $M \hookrightarrow N$  is a span
    \[
\begin{tikzcd}
	& M \\
	N && X
	\arrow["f"{description}, from=1-2, to=2-3]
	\arrow["m"{description}, from=1-2, to=2-1]
\end{tikzcd}\]
    whose right leg is monic. The intuition is that the map $f$ is defined on the subobject $M$ of $N$, which introduces a new problem: what is the proper notion of a \emph{subobject} in a tangent category?
    Such a notion should give rise to a \emph{stable class of monics}: one that is closed under horizontal span composition.
    One answer is the notion of etale monics: a morphism is \emph{etale} whenever the naturality square for $p$ is a $T$-pullbacks:
    \[
\begin{tikzcd}
    TM & TN \\
    M & TN
    \arrow["{T.f}", from=1-1, to=1-2]
    \arrow["{T.f}"', from=2-1, to=2-2]
    \arrow["p"', from=1-1, to=2-1]
    \arrow["p", from=1-2, to=2-2]
    \arrow["\lrcorner"{anchor=center, pos=0.125}, draw=none, from=1-1, to=2-2]
\end{tikzcd}\]
    Geometrically, this means that the morphism is a local diffeomorphism; for example, an etale subobject of $\R^n$ in the Dubuc topos is precisely an open subset in the usual sense.
    An endofunctor lifts to the partial map category whenever it preserves the class of monics. A natural transformation lifts to endofunctors on the partial map category whenever it is $T$-cartesian for the class of monics, and the same proof will show that this property holds for etale monics \cite{Cruttwell2019}. 
\end{remark}

\section{The prolongation tangent structure}
\label{sec:prol_tang_struct}

One of the most important consequences of the Weil Nerve Theorem  \ref{thm:weil-nerve} is that the category of involution algebroids (with chosen prolongations) may be equipped with two tangent structures. The first tangent structure is the pointwise tangent structure described in Proposition \ref{prop:pointwise-tangent-structure-inv}.
The tangent functor sends
\[
    (A,\alpha) \mapsto (T.A:\wone \to \C, c.A \o T.\alpha: T.A.T \Rightarrow T.T.A)  
\]
(recall the composition of tangent functors given in Example \ref{ex:composition-of-tangent-functors} (ii)). The structure morphisms will be given by whiskering, so in this case $\theta.A, \theta \in \{ p,0,+,\ell,c\}$. The restriction to tangent functors that preserve transverse limits along with the fact that the natural part $\alpha$ is $T$-cartesian, however, ensures that precomposition with the tangent functor
\[
    (A,\alpha) \mapsto (A.T: \wone \to \C, T.\alpha \o A.c: A.T.T \Rightarrow T.A.T)  
\]
returns an involution algebroid. The structure maps are once again given by whiskering, with the pre-composition tangent structure $A.\theta, \theta \in \{ p,0,+,\ell,c\}$. Preservation of transverse limits guarantees that this tangent structure will satisfy the necessary universality conditions. 

\begin{proposition}[Proposition \ref{prop:second-tangent-structure-inv-algds}]
\label{prop:second-tangent-structure-inv-algds-2}
    The category of involution algebroids with chosen prolongations in a tangent category $\C$ has a second tangent structure, where the action by $\wone$ is given by
    \[
        (A,\alpha) \mapsto ( A.T: \wone \to \C, \alpha.T \o \hat A.c: \hat A.T.T \Rightarrow T.\hat A.T).  
    \]
\end{proposition}
\begin{proof}
    The proposition statement means that the structure morphisms for this new involution algebroid are given by
    \[
        ( A.T, \alpha.T \o  A.c) \cong
        \begin{cases}
            \alpha.T \o  A.c = \anc':& \prol(A) \xrightarrow[]{\pi_1} TA \\
             \quad \;\; A.T.p = \pi':& \prol(A) \xrightarrow[]{p \o \pi_1} A \\
             \quad \;\;\; A.T.0 = \xi':& A \xrightarrow[]{(\xi \o \pi, 0)} \prol(A) \\
            \anc' \o  A.T.\ell = \lambda':& \prol(A) \xrightarrow[]{\lambda \x \ell} T.\prol(A) \\
             \quad \;\ A.T.c = \sigma':& \prol^2(A) \xrightarrow{\sigma \x c} \prol^2(A)
        \end{cases}  
    \]
    Similarly, we can see that
    \begin{gather*}
        ( A.T.T, \alpha.T.T \o  A.c.T \o  A.T.c) \\ = 
        \begin{cases}
            \alpha.T.T \o  A.c.T \o  A.T.c = \anc'':& \prol^2(A) \xrightarrow[]{(\pi_1, \pi_2)} T.\prol(A) \\
             \hspace{2.5cm} A.T.p = \pi'':& \prol^2(A) \xrightarrow[]{(p \o \pi_1, p \o \pi_2):} \prol(A) \\
             \hspace{2.6cm} A.T.0 = \xi'':& \prol(A) \xrightarrow[]{(\xi \o \pi \o \pi_0, 0 \o \pi_1, 0 \o \pi_2)} \prol^2(A) \\
            \hspace{1.85cm} \anc''\o  A.T.\ell = \lambda'':&\prol^2(A) \xrightarrow[]{(\lambda \x \ell \x \ell)} T.\prol^2(A) \\
             \hspace{2.15cm} A.T.T.c = \sigma'':& \prol^3(A) \xrightarrow[]{(\sigma \x c \x c)}\prol^3(A) 
        \end{cases}  
    \end{gather*}
    These coincide with the involution algebroids $\prol'(A), \prol'.\prol'(A)$ in Proposition \ref{prop:second-tangent-structure-inv-algds}: the second tangent structure follows from the fact that the natural transformations for the tangent structure there are given by
    \[
        A.\phi: A.U \Rightarrow A.V, \phi:U \to V \in \{p,0,+, \ell, c\}.  
    \]
    The result follows as a corollary of Theorem \ref{thm:weil-nerve}.
\end{proof}

\subsection*{The Jacobi identity for involution algebroids}
Classically, the theory of Lie algebroids uses the algebra of sections $\Gamma(\pi)$.  One key observation is that when using the Lie tangent structure $(\mathsf{Inv}(\C), \prol)$, sections of $\pi$ are in bijective correspondence with $\chi_\prol(A)$.
This observation allows for different statements about Lie algebroids to be translated into formal statements about the tangent bundle in $(\mathsf{Inv}(\C), \prol)$.
\begin{proposition}\label{prop:section-morphism}
    Let $A$ be an involution algebroid in $\C$.
    There is a bijection between the sections of\, $\pi$ in $\C$ and the vector fields on $A$ in $\mathsf{Inv}(\C)$:
    \[
        X \in \Gamma(\pi) \mapsto ((id, TX \o \anc), X): A \to T_L(A);
        \hspace{.2cm}
        \hat{X} \in \chi_\prol(A) \mapsto \hat{X}_R: A.R \to A.T.
    \]
\end{proposition}
\begin{proof}
    Recall the coherence for tangent natural transformations $\gamma: (H,\phi) \Rightarrow (G,\psi)$:
    \begin{equation*}
        \begin{tikzcd}
    H.T \rar{\gamma.T} \dar{\phi} & K.T \dar{\psi} \\
    T.H \rar{T.\gamma} & T.K
\end{tikzcd}
    \end{equation*}
    We specify this to a morphism $\hat{X}: (A, \alpha) \Rightarrow T_L(A,\alpha) = (A.T, \alpha.T)$ at $\N, W$:
    \begin{equation*}
        \begin{tikzcd}
    A \rar{X.T} \dar{\anc} & \prolong \dar{\pi_1} \\
    TM \rar{T.X} & TA
\end{tikzcd}
    \end{equation*}
    and so  infer that, if we set $X := X_R$, we have $\pi_1 \o X.T = T(X) \o \anc$.
    Furthermore, the condition that $p_L \o X = id$ forces $id = \pi_0 \o X.T$; thus, we can see that every section $X$ of $p_L$ is given by a morphism of the form $((id, TX \o \anc), X))$ on the underlying involution algebroids, where $\pi \o X = id$. 
    
    We now show that every $X \in \Gamma(\pi)$ gives rise to a section of $\pi_L$.
    Observe that the following is a morphism of involution algebroids:
    \begin{equation*}
        \begin{tikzcd}
    A \dar{\pi} \rar{(id, TX \o \anc)} &[2em] A \ts{\anc}{T\pi} TA \dar{p\o \pi_1} \\
    M \rar{X} & A
\end{tikzcd}
    \end{equation*}
    Note that it is well typed, as $T.\pi \o T.X \o \anc = \anc \o id$. Check that it is a bundle morphism:
    \[
        p \o \pi_1 \o (id, T.X \o \anc) = p \o T.X \o \anc = X \o p \o \anc = X \o \pi
    \]
    and that it is linear:
    \begin{gather*}
        (\lambda \o \pi_0, \ell\o\pi_1)\o(id, T.X\o \anc) = (\lambda, \ell \o T.X \o \anc) \\= (\lambda, T^2.X \o \ell \o \anc) = (\lambda, T^2.X \o \lambda) = T(id, TX)\o\lambda. 
    \end{gather*}
    Then check that it preserves the anchor:
    \[
        \pi_1 \o (id, TX \o \anc) = TX \o \anc
    \]
    and the involution:
    \begin{gather*}
        (\pi_0, \pi_1, T^2X \o T\anc \pi_1)\o \sigma = (\sigma, T^2X \o T\anc \o \pi_1 \sigma)\\ 
        = (\sigma, T^2X \o c T\anc\o \pi_1)
        = (\sigma(\pi_0, \pi_1), c\pi_3) \o (id, T^2X \o T\anc \o \pi_1).
    \end{gather*}

    Thus we have that $(id, TX \o \anc \pi_1)$ is a morphism of involution algebroids, inducing a morphism of $T$-cartesian Weil complexes.
    Lastly, we check that it is a section of $p^L_A$, but this is clear, since
    \[
        \pi_0 \o (id, TX \o \anc) = id;
    \]
    thus we have the desired bijection.
\end{proof}

Recall that given an involution on an anchored bundle, there is a bracket on its set of sections (see the explicit construction in Section \ref{sec:connections_on_an_involution_algebroid}). 
Given an $X,Y \in \Gamma(\pi)$, there is a bracket defined as follows:
\[
    \hat{\lambda} \o [X,Y]_A +_1 (\xi\pi,0)\o Y = ((\sigma \o (id, TY \o \anc) \o X -_2 (id, TX \o \anc)\o Y)).
\]
A direct proof of the Jacobi identity is a detailed calculation (see the original preprint on involution algebroids \cite{Burke2019}) and still relies on Cockett and Cruttwell's result for an arbitrary tangent category with negatives. As a result of Proposition  \ref{prop:section-morphism}, we can instead use Cockett and Cruttwell's result directly:
\begin{corollary}\label{cor:lie-bracket}
    Let $A$ be a complete involution algebroid in a tangent category $\C$ with negatives. 
    There is a Lie bracket defined on $\Gamma(\pi)$, $[-,-]$ induced by
    \[
        \hat{\lambda} \o [X,Y]_A +_1 (\xi\pi,0)\o Y = ((\sigma \o (id, TY \o \anc) \o X -_2 (id, TX \o \anc)\o Y)).
    \]
\end{corollary}
\begin{proof}
    The bracket is induced by Rosicky's universality diagram, as
    \begin{align*}
        0&=p \o  ((\sigma \o (id, TY \o \anc) \o X - (id, TX \o \anc)\o Y)) - 0Y \\
        &= Tp \o  ((\sigma \o (id, TY \o \anc) \o X - (id, TX \o \anc)\o Y)) - 0Y. 
    \end{align*}
    We look at the Lie tangent structure for $\mathsf{Inv}^*(A)$; this is precisely the vector field induced by
    \[
        ev_R( [\hat{X}, \hat{Y}] ).
    \]
    We complete the proof by using the result that for any $A$ in a tangent category with negatives, the bracket on $\chi(A)$ that is defined by
    \[
        \ell \o [X,Y] = (c\o T.X \o Y - T.Y \o X) - 0X
    \]
    satisfies the Jacobi identity.
\end{proof}

\subsection*{Identifying categories of involution algebroids}%
\label{sub:identifying-cats-of-inv}

Section \ref{sec:identifying-involution-algebroids} identified whenever a functor $\wone \to \C$ is an involution algebroid, whereas this section identifies tangent categories $\C$ that embed into the category of involution algebroids in some tangent category $\C$. We call this structure an \emph{abstract category} of involution algebroids. This notion involves some 2-category theory, using a modified notion of \emph{codescent} (see \cite{Bourke2010} for a development of codescent).

Recall that for any tangent category $\C$, the category of involution algebroids has $\C$ as a reflective subcategory. Furthermore, because limits of involution algebroids are computed pointwise, this reflector is left-exact. This left-exact reflection is the main structure we axiomatize.
\begin{definition}
    An \emph{abstract category of involution algebroids} is a tangent category $\C$ with a left-exact $T$-cartesian tangent localization $(Z,\anc): \C \to \D$, where $L$ satisfies a \emph{codescent} condition:
    \[
        \mathsf{TangCat_{Strict}}(\wone, \C)  \hookrightarrow \mathsf{TangCat_{Lax}}(\wone, \C) \xrightarrow{L_*} \mathsf{TangCat_{Lax}}(\wone, \D)
    \]
    (where $L_*$ denotes post-composition by $L$) is fully faithful. 
\end{definition}
\begin{example}
    The category of involution algebroids in any tangent category $\C$ is an abstract category of involution algebroids using $(\mathsf{Inv}(C), \prol)$. The reflector is the functor sending an involution algebroid to its base space; the $T$-cartesian natural transformation is the anchor map. Any tangent subcategory of $\mathsf{Inv}(\C)$ that contains $\C$ as a full subcategory will give rise to an abstract category of involution algebroids.
\end{example}

\begin{proposition}
    Let $\Z \hookrightarrow \C$ be an abstract category of involution algebroids.
    Then there is an embedding $\C \hookrightarrow \mathsf{Inv}(\Z)$.
\end{proposition}
\begin{proof}
    The proof follows by treating objects in $\C$ as strict tangent functors $\wone \to \C$ and morphisms as tangent natural transformations.
    \[
\begin{tikzcd}
    \wone && \C & \Z
    \arrow["Z"{description}, from=1-3, to=1-4]
    \arrow[""{name=0, anchor=center, inner sep=0}, "{\prol(-,A)}"{description}, curve={height=-24pt}, from=1-1, to=1-3]
    \arrow[""{name=1, anchor=center, inner sep=0}, "{\prol(-,B)}"{description}, curve={height=24pt}, from=1-1, to=1-3]
    \arrow["{\prol(-,f)}"{description}, shorten <=6pt, shorten >=6pt, Rightarrow, from=0, to=1]
\end{tikzcd} \]
    The natural part of $Z$ is $T$-cartesian, and the functor part preserves limits, so $Z.\prol(-,A) =: Z[A]$ determines an involution algebroid in $\Z$, and $f$ a morphism of involution algebroids.
    The embedding is guaranteed by the codescent condition so that the post-composition functor is fully faithful.
\end{proof}
\begin{corollary}
    An abstract category of involution algebroids $\Z \hookrightarrow \C$ is exactly a full subcategory $\Z \hookrightarrow \C \hookrightarrow \mathsf{Inv}(\Z)$.
\end{corollary}

\chapter{The infinitesimal nerve and its realization}%
\label{ch:inf-nerve-and-realization}

The main thrust of Chapters \ref{ch:differential_bundles}, \ref{ch:involution-algebroids}, and \ref{chap:weil-nerve} has been that the tangent categories framework allows for Lie algebroids to be regarded as tangent functors
\[
    \wone \to \mathsf{SMan} 
\] which satisfy certain universality conditions. 
This chapter, which is more experimental than the previous four chapters and represents work still in progress, puts Lie algebroids into the framework of \emph{enriched functorial semantics}. This new perspective on algebroids uses Garner's enriched perspective on tangent categories (\cite{Garner2018}) and the enriched theories paradigm from \cite{Bourke2019}. The functorial-semantics presentation of the Lie functor will generalize the Cartan--Lie theorem (that the category of Lie algebras is a coreflective subcategory of Lie groups) into a statement within the general theory of functorial semantics. 

The goal is to show that the infinitesimal approximation of a groupoid, as discussed in Example \ref{ex:lie-algebroids}, may be constructed as a \emph{nerve}, just like Kan's original simplical approximation of a topological space. The nerve of a functor $K:\a \to \C$ approximates objects and morphisms in $\C$ by $\a$-presheaves, so it sends an object in $\C$ to the $\a$-presheaf
\[
    N_K: \a \to \C; C \mapsto \C(K-, C). 
\]
Thus there will be an infinitesimal object,
\[
    \partial: \wone^{op} \to \mathsf{Gpd}(\w)   
\] where $\mathsf{Gpd}(\w)$ denotes groupoids in the category $\W$ of Weil spaces (formally defined in Section \ref{sec:tang-cats-enrichment}). The nerve of $\partial$ has a left adjoint, the \emph{Lie realization}, given by the left Kan extension (just as Kan's geometric approximation of a simplicial set is, in \cite{Kan1958}):
\begin{equation}\label{eq:lan-lie}
    \begin{tikzcd}
        {\wone^{op}} & {\mathsf{Gpd}(\w)} \\
        {[\wone,\w]}
        \arrow["\yon"', from=1-1, to=2-1]
        \arrow["\partial", from=1-1, to=1-2]
        \arrow["{Lan_\yon\partial}"', dashed, from=2-1, to=1-2]
    \end{tikzcd}
\end{equation}

\begin{restatable*}[The Lie Realization]{theorem}{lie}
    \label{thm:lie-realization}
    There is a tangent adjunction between the category of involution algebroids and groupoids in $\w$, where each functor preserves products and the base spaces.
\end{restatable*}
\[
\begin{tikzcd}
    {\mathsf{Gpd}(\w)} & {\mathsf{Inv}(\w)}
    \arrow[""{name=0, anchor=center, inner sep=0}, "{N_\partial}", curve={height=-12pt}, from=1-1, to=1-2]
    \arrow[""{name=1, anchor=center, inner sep=0}, "{|-|_\partial}", curve={height=-12pt}, from=1-2, to=1-1]
    \arrow["\dashv"{anchor=center, rotate=-90}, draw=none, from=0, to=1]
\end{tikzcd}\]

Note, however, that the left Kan extension in Equation \ref{eq:lan-lie} does not immediately give the desired adjunction of Theorem \ref{thm:lie-realization}, as there is no guarantee that the nerve functor $N_\partial$ lands in the category of algebroids. To prove this we must revisit the work in Section \ref{sec:weil-nerve} presenting involution algebroids as those functors $A:\wone \to \C$ for which each $A(V)$ is the prolongation of its underlying anchored bundle; this leads naturally to the formalism for enriched theories developed in \cite{Bourke2019}.

The presentation of algebroids as models of an enriched theory requires situating the categories of differential bundles, anchored bundles, and involution algebroids in $\w$ as full subcategories of $\w$-presheaf categories on small $\w$-categories. 
Section \ref{sec:tang-cats-enrichment} reviews the work in \cite{Garner2018} characterizing tangent categories as categories enriched in $\w = \mathsf{Mod}(\wone, \s)$ (the cofree tangent category on $\s$, by Observation \ref{obs:cofree-tangent-cat}).

Section \ref{sec:enriched-structures} reconfigures the content of Chapter \ref{ch:differential_bundles} using the enriched perspective, so that lifts are $\w$-functors from the $\w$-monoid $D + 1$, while differential bundles are a reflective subcategory of functors from its idempotent splitting. Similarly, the category of anchored bundles are a reflective subcategory of the category $[\wone^1,\C]$, where $\wone^1$ is the category of 1-truncated Weil algebras (Definition \ref{def:truncated-wone}).

Section \ref{sec:enriched-nerve-constructions} reviews the basic idea of an enriched nerve/approximation. A particularly important example is the linear approximation of a reflexive graph, the functor introduced in Example \ref{ex:prolongations}, which is the nerve of a functor
\[
    \partial: (\wone^1)^{op} \to \mathsf{Gph}(\w).
\]

Section \ref{sec:enriched-theories} applies the enriched theories framework introduced in \cite{Bourke2019}, where a \emph{dense} subcategory of a locally presentable $\a \hookrightarrow \C$ forms the ``arities'', and a bijective-on-objects functor $\a \to \th$ is the theory. The category of models is the pullback of (enriched) categories:
\[\begin{tikzcd}
	{\C^{\th}} & {[\th, \vv]} \\
	\C & {[\a^{op}, \vv]}
	\arrow[hook, from=2-1, to=2-2]
	\arrow[from=1-2, to=2-2]
	\arrow[from=1-1, to=2-1]
	\arrow[from=1-1, to=1-2]
\end{tikzcd}\]
The first step is to freely complete the $\w$-category $\wone^1$ of truncated Weil algebras (Definition \ref{def:truncated-wone}) so that its base anchored bundle has all prolongations; call this $\w$-category $\prol$. Then, in every tangent category $\C$, the category of anchored bundles with chosen prolongations (Definition \ref{def:monoidal-category}) in $\C$ embeds fully and faithfully into the functor category:
\[
    \mathsf{Anc}^\prol(\C) \hookrightarrow [\prol, \C]  
\]
(here, $\mathsf{Anc}^\prol(\C)$ denotes the category of anchored bundles with chosen prolongations).
In particular, the tangent bundle on $\N$ in $\wone$ determines a bijective-on-objects functor
\[
    \prol \to \wone 
\]
so that the category of involution algebroids in any tangent category $\C$ is the pullback in $\w$Cat:
\[\begin{tikzcd}
    {\mathsf{Inv}^*(\C)} & {[\wone,\C]} \\
    {\mathsf{Anc}^*(\C)} & {[{\prol}^{op},\C]}
    \arrow[hook, from=2-1, to=2-2]
    \arrow[from=1-2, to=2-2]
    \arrow[from=1-1, to=2-1]
    \arrow[hook, from=1-1, to=1-2]
    \arrow["\lrcorner"{anchor=center, pos=0.125}, draw=none, from=1-1, to=2-2]
\end{tikzcd}\]
This means that the category of involution algebroids in $\w$ is monadic over the category of anchored bundles in $\w$ using the monad-theory correspondence from \cite{Bourke2019}.

The final section looks at the category of $\w$-groupoids. Essentially, the free groupoid over the linear approximation of a graph will now give an infinitesimal object $\wone \to \mathsf{Gpd}(\w)$. The nerve of  $\partial: \wone^{op} \to \mathsf{Gpd}(\w)$---the \emph{infinitesimal approximation}---has a right adjoint via the \emph{realization} of the nerve functor from Definition \ref{def:realization}. This is used to prove the culminating Theorem \ref{thm:lie-realization}.

The individual pieces of categorical machinery used in this chapter are not new (the enriched perspective on tangent categories, enriched nerve constructions, enriched theories). However, all of the results dealing with the application of enriched nerve constructions and enriched theories to tangent categories is original work of the author.

\section{Tangent categories via enrichment}%
\label{sec:tang-cats-enrichment}

This section gives a quick introduction to Garner's enriched perspective on tangent categories. The enriched approach to tangent categories first appeared in \cite{Garner2018} and builds on the category perspective on tangent categories introduced in \cite{Leung2017}. Garner was able to exhibit some of the major results from synthetic differential geometry as pieces of enriched category theory; for example, the Yoneda lemma implies the existence of a well-adapted model of synthetic differential geometry.

The category of Weil spaces is the site of enrichment for tangent categories and is closely related to Dubuc's \emph{Weil topos} from his original work on models of synthetic differential geometry \cite{Dubuc1981}; a deeper study of this topos may be found in \cite{Bertram2014}. Recall that the category $\wone$ is the free tangent category over a single object. The category of Weil spaces is the \emph{cofree} tangent category over $\s$, which is the category of transverse-limit-preserving functors $\wone \to \s$ by Observation \ref{obs:cofree-tangent-cat}. Call this the category of \emph{Weil spaces}, and write it $\w$. Just as a simplicial set $S: \Delta \to \s$ is a gadget recording homotopical data, a Weil space records \emph{infinitesimal} data.

\begin{definition}%
    \label{def:weil-space}
    A \emph{Weil space} is a functor $\wone \to \s$ that preserves transverse limits (Definition \ref{def:transverse-limit}): that is, the $\ox$-closure of the set of limits
    \[
        \left\{
\begin{tikzcd}
    {T_{n+m}} & {T_m} \\
    {T_n} & \N
    \arrow[from=1-1, to=2-1]
    \arrow[from=1-2, to=2-2]
    \arrow[from=2-1, to=2-2]
    \arrow[from=1-1, to=1-2]
    \arrow["\lrcorner"{anchor=center, pos=0.125}, draw=none, from=1-1, to=2-2]
\end{tikzcd}
        ,
\begin{tikzcd}
    {T_2} & {T^2} \\
    \N & T
    \arrow[from=1-1, to=2-1]
    \arrow[from=1-2, to=2-2]
    \arrow["0"{description}, from=2-1, to=2-2]
    \arrow["\mu"{description}, from=1-1, to=1-2]
    \arrow["\lrcorner"{anchor=center, pos=0.125}, draw=none, from=1-1, to=2-2]
\end{tikzcd}
        ,
\begin{tikzcd}
    {T_n} & {T_n} \\
    {T_n} & {T_n}
    \arrow[Rightarrow, no head, from=1-1, to=1-2]
    \arrow[Rightarrow, no head, from=1-1, to=2-1]
    \arrow[Rightarrow, no head, from=1-2, to=2-2]
    \arrow[Rightarrow, no head, from=2-1, to=2-2]
\end{tikzcd}
        \right\}
    \]
    A morphism of Weil spaces is a natural transformation. Write the category of Weil spaces as $\w$.
\end{definition}
\begin{example}
    ~\begin{enumerate}[(i)]
        \item Every commutative monoid may be regarded as a Weil space canonically. Observe that for every $V \in \wone$ and commutative monoid $M$, one has the free $V$-module structure on $M$ given by $|V| \ox_{\mathsf{CMon}} M$ (here $|V|$ is the underlying commutative monoid of $V$). The commutative monoid $|V|$ is exactly $\N^{\mathsf{\dim V}}$, so that
        \[
            |V| \ox_{\mathsf{CMon}} M \cong \oplus^{\dim V} M.
        \]
        This agrees with the usual tangent structure on a category with biproducts.
        \item Following (i), any tangent category $\C$ that is \emph{concrete}---that is, admitting a faithful functor $U: \C \to \s$---will have a natural functor into Weil spaces (copresheaves on $\wone$). Every object $A$ will have an underlying Weil space $V \mapsto U_{\s}(T^V(A))$, and whenever $U$ preserves connected limits (such as the forgetful functor from commutative monoids to sets), each of the underlying copresheaves will be a Weil space.
        \item Consider a symmetric monoidal category with an infinitesimal object, which by Proposition \ref{prop:monoidal-functor-inf-obj} is a transverse-colimit-preserving symmetric monoidal functor $D:\wone \to \C$. Then for every object $X$, the nerve (Definition \ref{def:nerve-of-a-functor}) $N_D(X): \C(D-,X):\wone \to \s$ is a Weil space.
        \item For any pair of objects $A,B$ in a tangent category, $\C(A,T^{(-)}B):\wone \to \s$ is a Weil space by the continuity of $\C(B,-):\C \to \s$.
    \end{enumerate}
\end{example}
Unlike the category of simplicial sets, the category of Weil spaces is not a topos.
The category of Weil spaces does, however, inherit some nice properties from the topos of copresheaves on $\wone$ by applying results from \Cref{sec:enriched-nerve-constructions}, as it is \emph{locally presentable}. The basics of locally presentable categories can be found in the Appendix \ref{appendix}. Roughly speaking, a cocomplete category $\C$ has a subcategory of finitely presentable objects $\C_{fp}$, those $C$ so that
\[
    \C(C,-): \C \to \s  
\]
preserves \emph{filtered} colimits (i.e. those colimit diagrams that commute with all finite limits in $\s$). $\C$ is locally finitely presentable whenever every object is given by the coend
\[
    C \cong \int^{X \in \C_{fp}} \C(X, C) \cdot X
\]
where $S\cdot X$ is the (possibly infinite) product $X^{|S|}$ with $|S|$ the cardinality of the set $S$. This means that $\C = \mathsf{Lex}(\C_{fp}^{op}, \s)$, where $\mathsf{Lex}$ means the category of finite-limit-preserving functors. 

The third point of Corollary \ref{cor:properties-of-w} below, that $\w$ is locally finitely presentable as a cartesian monoidal category, means that the category $\w_{fp}$ is closed under products. This implies that $\w$ is locally presentable \emph{as a} $\w$-category (this ends up being an important technical condition whereby locally presentable $\w$-categories make sense). Whenever we discuss an arbitrary $\vv$-category, we will assume that $\vv$ is locally presentable as a monoidal category.

\begin{proposition}[\cite{Garner2018}]
    The category of Weil spaces is a cartesian-monoidal reflective subcategory of\, $[\wone,\s]$.
\end{proposition}
\begin{corollary}%
    \label{cor:properties-of-w}
    The category of Weil spaces is
    \begin{enumerate}[(i)]
        \item a cartesian closed category;
        \item a representable tangent category, where the infinitesimal object is given by the restricted Yoneda embedding $\yon: \wone^{op} \to \w$;
        \item Locally finitely presentable as a cartesian monoidal category.
    \end{enumerate}
\end{corollary}
The cofree tangent structure on $\w$ is given by precomposition, that is:
\[
    T^U.M.(V) = M.(U \ox V) = M.T^U.V.
\]
This tangent coincides with the representable tangent structure induced by the Yoneda embedding. The proof is an application of the Yoneda lemma. Observe that
\[
    [D,M](V) \cong [\wone,\s](D \x \yon(V), M) \cong [\wone, \s](\yon(WV), M) \cong M(WV)
\]
where $D \x \yon(V) = \yon(WV)$ follows because the tensor product in $\wone$ is cocartesian and the reflector is cartesian monoidal.
\begin{notation}
    The tangent category $\w$ is a representable tangent category, where $D = \yon W$ (Definition \ref{def:inf-object}). 
    We will write the Yoneda functor $\yon: \wone^{op} \to \w$ as $D(-)$, so it is closer to the usual notation used in representable tangent categories or synthetic differential geometry (a single $D$ may be used as shorthand for $D(W)$, $D(n)$ for $D(W_n)$, etc.).
    Note that
    \[
        D(V) = D(\ox^k W_{n(i)}) = \prod^K D(n_i)  
    \]
    and that $D(\ell) = \otimes, D(c) = (\pi_1, \pi_0)$, $D(+) = \delta$, and so on.
\end{notation}

At this point, we are ready to move to the enriched perspective on tangent categories. The basics of enriched category theory may be found in the Appendix \ref{appendix}, but one definition in particular is important to include here.
\begin{definition}\label{def:power}
    Let\, $\C$ be a $\vv$-category for $\vv$ a closed symmetric monoidal category. For $J \in \vv$, the \emph{power} by $J$ of an object $C \in \C$ is an object $C^J$ so that the following is an isomorphism:
    \[
        \forall D: \vv(J, \C(D,C)) \cong \C(D,C^J)
    \]
    whereas the \emph{copower} is given by
    \[
        \forall D: \vv(J, \C(C,D)) \cong \C(J \bullet C, D).
    \]
    Let $\j \hookleftarrow \vv$ be a full monoidal subcategory of $\vv$. A $\vv$-category\, $\C$ has \emph{coherently chosen powers} by $\j$ if there is a choice of $\j$-powers\, $(-)^J$ so that
    \[
        (C^J)^K = C^{J \ox K}.
    \]
    Likewise, \emph{coherently chosen copowers} are a choice of $\j$-copowers so that
    \[
        K \bullet (J \bullet C) = (K \ox J) \bullet C.
    \]
    The sub-2-categories of $\vv$-categories equipped with coherently chosen powers and copowers are $\vv\cat^\j$ and $\vv\cat_\j$, respectively.
\end{definition}

\cite{Wood1978} proved that the 2-category of actegories over a monoidal $\C$ is equivalent to the 2-category of $[\C,\s]$-enriched categories with powers by representable functors (Definition \ref{def:power}), using the monoidal structure on $[\C, \s]$ induced by Day convolution\footnote{
    The \emph{Day convolution} tensor product of presheaves $X,Y: \C^{op} \to \s$ is given by $X \widehat\ox Y := \mathsf{Lan}_{\ox}(X \boxtimes Y)$, where $(X \x Y)(V) =X(V) \x Y(V)$ (\cite{Day1970}).
    }. 
Moreover, \cite{Garner2018} showed that a monoidal reflective subcategory $\vv \hookrightarrow \hat \a$ exhibits the 2-category of $\vv$-categories as a reflective sub-2-category of $\a$-categories; this proves that tangent categories are equivalent to a particular class of enriched category.
\begin{proposition}[\cite{Garner2018}]
    A tangent category is exactly a $\w$-category with powers by representables.
\end{proposition}
\begin{proof}
    For every $A, B \in \C_0$, the Weil space is defined as
    \[
        \underline{\C}(A,B) := U \mapsto \C(A, T^{U}B).
    \]
    The functor $\C(A,-)$ is continuous and $T^{-}B$ is an infinitesimally linear functor, so this is a Weil space. The following diagram gives the composition map :
\[\begin{tikzcd}
	{\underline\C(B,C) \x \underline\C(A,B)} && {\underline\C(A,C)} \\
	{\C(B, T^{U}C) \x \C(A, T^{V}B)} && {\C(A, T^{VU}C)} \\
	& {\C(T^VB, T^{V}.T^{U}C) \x \C(A, T^{V}B)}
	\arrow[Rightarrow, no head, from=1-1, to=2-1]
	\arrow["{T^V \x id}"', from=2-1, to=3-2]
	\arrow["m"', from=3-2, to=2-3]
	\arrow[Rightarrow, no head, from=2-3, to=1-3]
	\arrow[dashed, from=1-1, to=1-3]
\end{tikzcd}\]
    Note that it is natural in $U$ and $V$.

    By the Yoneda lemma (as the internal hom in $\w$ is the internal hom of copresheaves on $\wone$), \[\C(A, T^VB) =  (U \mapsto \C(A, B)(V \ox U)) = [D(V), \C(A,B)]\]
    so this category has coherently chosen powers by representable functors.
\end{proof}

Now the original notions of (lax, strong, strict) tangent functors can be shown to correspond to power-preservation properties of $\w$-functors between $\w$-categories with coherently chosen powers:
\begin{theorem}[\cite{Garner2018}]\label{thm:2-cat-tangcat}
    We have the following equivalences of 2-categories:
    \begin{enumerate}[(i)]
        \item the 2-category of $\w$-categories with coherently chosen powers and $\mathsf{TangCat}_{\mathsf{Lax}}$;
        \item the 2-category of $\w$-categories with coherently chosen powers and power-preserving $\w$ functors  and $\mathsf{TangCat}_{\mathsf{Strong}}$;
        \item The 2-category of $\w$-categories with coherently chosen powers and chosen-power-preserving $\w$ functors and $\mathsf{TangCat}_{\mathsf{Strict}}$.
    \end{enumerate}
\end{theorem}
Note that for a lax tangent functor $(F,\alpha): \C \to \D$, the map
\[
    \alpha.X: F.T.X \to T.F.X
\] can be seen as the unique morphism induced by universality. Conversely, for a strong tangent functor, the natural isomorphism $\alpha$ is the isomorphism from \emph{a} power $F.T.X$ to the \emph{coherently chosen} power $T.F.X$. In contrast, a strict tangent functor preserves the coherent choice of powers.

The ability to work with $\w$-categories that do not have powers by representables allows for significant flexibility. It is useful to observe that there are $\w$-categories which are not tangent categories.
\begin{example}%
    \label{ex:w-cats}
    ~\begin{enumerate}[(i)]
        \item Every monoid $(M, m, e)$ in $\w$ gives rise to a one-object $\w$-category whose hom-object is $M$, composition is $m$, and unit is $e$.
        \item Given a tangent category\, $\C$, it is possible to take the full\, $\w$-category over some set of objects\, $\D$, even though $D$ may not be closed under iterated applications of the tangent functor.
        \item The dual of a\, $\w$-category is a\, $\w$ category, where
        \[
            \C^{op}(A,B) = \C(B,A): \wone \to \s.    
        \]
        Dually, in the case that\, $\C$ is a tangent category, $\C^{op}$ will have coherent \emph{copowers} by representables.
        \item If a cartesian category\, $\C$ has an infinitesimal object $D$ (Definition \ref{def:inf-object}), it has a natural $\w$-category structure (with coherent \emph{copowers} by representables)
        \[
            \C(A,B):\wone \to \s := \underline{\C}(A \x D(-), B).
        \]
        Using the dual tangent structure on $\C^{op}$ from Proposition \ref{prop:inf-object-tangent-structures}, the enrichment on $\C$ is exactly the enrichment found by regarding $\C$ as the dual $\w$-category of the tangent category\, $\C^{op}$.
    \end{enumerate}
\end{example}

As a final remark, note that the Yoneda lemma applies to $\w$-categories, so there is an embedding
\[
    \C \hookrightarrow [\C^{op},\s].
\]
The powers and copowers by representables are computed pointwise in a presheaf category, so they inherit the coherent choice. Thus the following holds:
\begin{corollary}[\cite{Garner2018}]
    Every tangent category embeds into a $\w$-cocomplete representable tangent category.
\end{corollary}
(In fact, showing that the Yoneda embedding applies to tangent categories is the main theorem of \cite{Garner2018}.)

\section{Differential and anchored bundles as enriched structures}%
\label{sec:enriched-structures}

This section gives an enriched-categorical reinterpretation of the work in Chapter \ref{ch:differential_bundles} regarding differential bundles and Chapters \ref{ch:involution-algebroids} and \ref{chap:weil-nerve} regarding anchored bundles.

\subsection*{Differential bundles as enriched structures}%
\label{sub:differential-bundles-as-enriched-structures}

As a first case study using the enriched perspective for tangent categories, consider differential bundles. Most of the work in \Cref{ch:differential_bundles} uses the intuition that differential bundles are some sort of tangent-categorical algebraic theory; this section will make that intuition concrete. Recall that a \emph{lift} (Definition \ref{def:lift}) is a map $\lambda:E \to TE$. Using the enriched perspective and treating $TE$ as a power (recall Definition \ref{def:power}), this gives the following correspondence:
\[
    \infer{\hat{\lambda}:D \to \C(E,E)}{\lambda: 1 \to \C(E, E^D)}
\]
The commutativity condition for $\hat{\lambda}$, then, is translated as follows:
\[
\begin{tikzcd}
	E & TE && {D \x D} & {\C(E,E) \x \C(E,E)} \\
	TE & {T^2E} && D & {\C(E,E)}
	\arrow["\odot", from=1-4, to=2-4]
	\arrow["{m_{E,E,E}}", from=1-5, to=2-5]
	\arrow["{\hat{\lambda} \x \hat{\lambda}}", from=1-4, to=1-5]
	\arrow["{\hat{\lambda}}"', from=2-4, to=2-5]
	\arrow["\lambda", from=1-1, to=1-2]
	\arrow["\lambda"', from=1-1, to=2-1]
	\arrow["{T.\lambda}", from=1-2, to=2-2]
	\arrow["\ell"', from=2-1, to=2-2]
\end{tikzcd}\]
That is, $\hat{\lambda}$ is a semigroup morphism $D \to \C(E,E)$. In any cartesian closed category with coproducts, a semigroup may be freely lifted to a monoid using the "exception monad" $(-) + 1$ from functional programming (see, for example, \cite{Seal2013}).
\begin{definition}%
    \label{def:lambda-monoid}
    Regard the following monoid as the one-object $\w$-category $\Lambda$:
    \[
        \infer{m: (D+1) \x (D+1) \to D+1}{D \x D + D + D + 1 \xrightarrow{(\iota_L\o m | \iota_L | \iota_L |\iota_R)} D + 1}
    \]
\end{definition}
Thus, a lift $\lambda$ is exactly a functor $\hat{\lambda}:\Lambda \to \C$. Now check that morphisms are tangent natural transformations. Note that the semigroup $D$ is commutative, so $\Lambda = \Lambda^{op}$; the choice of using $\Lambda^{op}$ in the next lemma is to be consistent with conventions used in \Cref{sec:enriched-nerve-constructions}.

\begin{lemma}%
    \label{lem:cat-of-lifts-iso}
    The category of lifts in a tangent category $\C$ is isomorphic to the category of $\w$-functors and $\w$-natural transformations\, $\Lambda^{op} \to \C$.
\end{lemma}
\begin{proof}
    Check that a $\w$-natural transformation is exactly a morphism of lifts $f: \lambda \to \lambda'$. Start with the $\w$-naturality square:
    \[
\begin{tikzcd}
    {D+1} & {\C(A,A)\x \C(A,B)} \\
    {\C(A,B)\x \C(B,B)} & {\C(A,B)}
    \arrow["{(\lambda, f \o !)}", from=1-1, to=1-2]
    \arrow["{(f \o !, l)}"', from=1-1, to=2-1]
    \arrow["{m_{ABB}}"', from=2-1, to=2-2]
    \arrow["{m_{AAB}}", from=1-2, to=2-2]
\end{tikzcd}\]
    Now, rewriting $D+1 \to \C(A,B)$ as a semigroup map $D \to \C(A,B)$, we have:
    \[
        \infer{
            \infer{
                1 \xrightarrow{Tf \o \lambda'}\C(A,TB)
            }{
                1 \xrightarrow{(\lambda',f)} \C(A,TA) \x \C(A,B) \xrightarrow{1 \x T} \C(A,TA) \x \C(TA,TB) \xrightarrow{m_{A,TA,TB}} \C(A,TB)
            }}{
            D \xrightarrow{(\lambda,f\o !)} \C(A,A) \x \C(A,B) \xrightarrow{m_{AAB}} \C(A,B)
        }
    \]
    Similarly, the other path is exactly $\lambda' \o f$. Thus, a $\w$-natural transformation is exactly a morphism of lifts. 
\end{proof}

It is a classical result in synthetic differential geometry that the object $D$ has only one point. In the case of $\w$, this follows from the Yoneda lemma (regarding $\w$ as a $\s$-category):
\[
    \w(1,D) = \wone(\N[x]/x^2, \N) = \{ !: \N[x]/x^2 \to \N \}
\]
Note that the natural idempotent $e:id \Rightarrow id$ from Proposition \ref{prop:idempotent-natural}, then, must be the point $0:1 \to D$. Note that this idempotent is an absorbing element of the monoid $D+1$, so for any $f:X \to D+1$, it follows that $m(f,0\o !) = m(0\o !, f) = 0\o !$.  A pre-differential bundle is exactly a lift with a chosen splitting of the natural idempotent $p\o\lambda$.
\begin{lemma}%
    \label{lem:splitting-of-idemp-is-lambdaplus}
    For every lift $\bar{\lambda}:\Lambda^{op} \to \C$, the natural idempotent $e = p \o \lambda$ is exactly
    \[1 \xrightarrow[]{0} D \xrightarrow[]{\iota_L} D+1 \to \C(E,E).\]
\end{lemma}
Now that the natural idempotent is understood as a map in $\Lambda$, that idempotent splits to give the theory of a pre-differential object.
\begin{definition}%
    \label{def:lambda-plus}
    The $\w$-category $\Lambda^+$ is given by the set of objects $\{0,1\}$ with hom-Weil-spaces. Specifically,
    \begin{itemize}
        \item The hom-spaces are $\Lambda^+(1,1) = D+1$, otherwise $\Lambda^+(i,j) = 1$.
        \item Composition (writing the original composition from $\Lambda$ as $m$) is given by
        \begin{gather*}
            m_{111}: (D+1) \x (D+1) \xrightarrow[]{m} (D+1) \\
            m_{101}: 1 \x 1 \xrightarrow[]{\iota_R} (D + 1) \\
            \text{otherwise: } m_{ijk} = ! 
        \end{gather*}  
    \end{itemize}
\end{definition}
Idempotent splittings are absolute (co)limits, and are preserved by all functors; as this is a limit completion, we have the following.
\begin{lemma}%
    \label{lem:lambda-plus-is-pdb}
    The category of pre-differential bundles is exactly the category of $\w$-functors $\Lambda^+\to \C$ (that is, $\C$-valued presheaves).
\end{lemma}
It is straightforward to exhibit the category of differential bundles as a reflective subcategory of pre-differential bundles in $[\Lambda^+, \C]$ (so long as $\C$ has equalizers).
\begin{proposition}%
    \label{prop:Lambda-is-refl-subcat}
    The category of differential bundles in a tangent category\, $\C$ with $T$-equalizers and $T$-pullbacks is a reflective subcategory of\, $[\Lambda^+, \C]$.
\end{proposition}
\begin{proof}
    The category of pre-differential bundles in $\C$ is isomorphic to $[\Lambda^+, \C]$. By Corollary \ref{cor:idemp-dbun}, the category of differential bundles is the category of algebras for an idempotent monad on the category of pre-differential bundles in $\C$. The reflector sends a pre-differential bundle to the $T$-equalizer: 
\[\begin{tikzcd}
    A & TE & TE
    \arrow["{e.E}", shift left=1, from=1-2, to=1-3]
    \arrow["{T.e}"', shift right=1, from=1-2, to=1-3]
    \arrow[dashed, from=1-1, to=1-2]
\end{tikzcd}\] This equalizer will always exist if $\C$ has equalizers, and pullbacks of the projection will exist if $\C$ has pullbacks, so the pullback is a differential bundle. This reflection gives a left-exact idempotent monad on $[\Lambda^+, \C]$ whose algebras are differential bundles.
\end{proof}
Now, in the case that $\C$ is a \emph{locally presentable} tangent category (such as $\w$), $[\Lambda^+, \C]$ is locally presentable and so is the reflective subcategory of differential bundles; thus the following holds.
\begin{corollary}%
    \label{cor:Lambda-dense}
    If\, $\C$ is locally presentable, then $\mathsf{DBun}(\C)$ is a locally presentable category.
\end{corollary}

\subsection*{Anchored bundles}

There are two ways to think about anchored bundles:
\begin{enumerate}[(i)]
    \item an anchored bundle is a differential bundle with an anchor $A \to TM$, or
    \item an anchored bundle is an involution algebroid without an involution.
\end{enumerate} 
These two perspectives can be unified by regarding $A$ as a cylinder for the weighted limit $TM$, so that the anchor is induced by the unique map $A.T.0 \to T.A.0$:
\[
\begin{tikzcd}
	{\Lambda^+} & \C
	\arrow[""{name=0, anchor=center, inner sep=0}, "A", curve={height=-12pt}, from=1-1, to=1-2]
	\arrow[""{name=1, anchor=center, inner sep=0}, "TM"', curve={height=12pt}, from=1-1, to=1-2]
	\arrow["\anc", shorten <=3pt, shorten >=3pt, Rightarrow, from=0, to=1]
\end{tikzcd}\]
That is, the syntactic category for anchored bundles is constructed as a full $\w$-category of $\wone$ that doesn't include the map $c$. This may be found by taking the full subcategory of $\wone$ whose objects are constructed out of $W_n, n\in \mathbb{N}$.
\begin{definition}%
    \label{def:truncated-wone}
    A Weil algebra has \emph{width} $k \in \N$ if it can be written
    \[
        V = \ox^{0 \le i < k} W_{n(i)},\; n(i) \in \N
    \]
    The category of $k$-truncated Weil algebras, written\, $\wone^k$, is the full $\w$-subcategory of\, $\wone$ of Weil algebras of width $k$ or less.

    The full subcategory whose objects are $\{ \N, W \}$ will be given the special notation $\wone^*$.
\end{definition}
Note that for each $V$ in $\wone^k$, the enrichment is given by
\[
    \wone^k(U,V) := (X \mapsto \underline{\wone}(U, X \ox V) ).
\]
The maps in $\wone^1$---that is, the full subcategory of $\wone$ whose objects are
\[
    \{ W_n | n \in \mathbb{N} \}
\]
---have the useful property that they may be written without the flip $c$. This makes $\wone^1$ a natural candidate for the syntactic category of anchored bundles. 
\begin{lemma}\label{lem:writing-maps-in-wone}
    Every morphism
    \[
        W_n \to V \in \wone
    \]
    may be written without $c$; that is, it is generated by the set of maps $\{p,+,0,\ell\}$ closed under tensor, composition, and maps induced by transverse limits.
\end{lemma}
\begin{proof}
    Every map $W_n \to V$ in $\wone$ is the finite sum
    \[
        v \mapsto \sum_{x} (A_x v)\bullet x
    \]
    where $x \in \mathsf{var}(V)$, and $A_x \in \N^n$ so that $A_n v$ is the ordinary dot product.
    Each term can be written without $c$, and the whole term is then constructed by adding each component using the appropriate $U.+.V$-symbols. \footnote{This may also be regarded as a consequence of the graphical notation for maps in $\wone$ in Table 1 on page 308 of \cite{Leung2017}.}
\end{proof}
 It is possible, then, to show that the category of anchored bundles in $\C$ is a full sub-tangent-category of functors $\wone^1 \to \C$; note that $W$ acts as a cone for $(-)^{D}$, so this induces a map
\[
    \anc: F(W) \to T.F(\N).
\]
\begin{proposition}
    \label{prop:nerve-anc-work}
    Every anchored differential bundle in $\C$ determines a functor $\wone^1 \to \C$; an anchored bundle morphism is exactly an enriched natural transformation.
\end{proposition}
\begin{proof}
    Start with an anchored bundle $(q:E \to M, \xi, \lambda, \anc)$; then for hom-objects with domain $\N$,
    \[
        \wone^1(\N,\N) = 1 \mapsto id_M, \hspace{0.5cm} \wone^1(\N,T) = 1 \mapsto \xi.
    \]
    For the hom-objects with domain $T$, the problem is slightly more difficult, as $\wone^1(T,\N)(T^V) = \wone(T, T^V)$ and $\wone^1(T, T)(T) = \wone(T, T.T^V)$. This part of the  proof amounts to constructing maps
    \[
        \wone(T, T^V) \to \C(E, T^V.M), \hspace{0.5cm}  \wone(T,T^V.T) \to \C(E, T^V.E)/
    \]
    The first mapping is straightforward: send $\theta$ to $\theta.M \o \anc$. For the second map, observe that the following diagrams commute:
    \[
\begin{tikzcd}
	TA & {T^2M} &[-2em] A & TM &[-2em] A & TM &[-2em] {A_2} & {T_2M} \\
	A & TM & M & M & M & M & A & TM
	\arrow["\anc"', from=2-1, to=2-2]
	\arrow["{T.\anc}", from=1-1, to=1-2]
	\arrow["\lambda"{description}, from=2-1, to=1-1]
	\arrow["\ell"{description}, from=2-2, to=1-2]
	\arrow["\xi"{description}, from=2-3, to=1-3]
	\arrow[Rightarrow, no head, from=2-3, to=2-4]
	\arrow["0"{description}, from=2-4, to=1-4]
	\arrow["\anc", from=1-3, to=1-4]
	\arrow["\anc", from=1-5, to=1-6]
	\arrow["\pi"', from=1-5, to=2-5]
	\arrow["p", from=1-6, to=2-6]
	\arrow[Rightarrow, no head, from=2-5, to=2-6]
	\arrow["{+_q}"{description}, from=1-7, to=2-7]
	\arrow["{\anc_2}", from=1-7, to=1-8]
	\arrow["{+.M}"{description}, from=1-8, to=2-8]
	\arrow["\anc"', from=2-7, to=2-8]
\end{tikzcd}
    \]
    The idea is to take an anchored bundle morphism $f$ and rewrite it as a string of compositions that does not include $c$, switching out every occurrence of
    \[
        T^V.\theta.M, \theta \in \{p,0,+,\ell\}
    \]
    and replacing it with $T^V.(\theta')$, where $\theta'$ is the corresponding map in $\{q,\xi,+_q,\lambda\}$. This induces a map \[\wone^1(W,W) \to \C(E,E) \in \w\] and the $\anc$ is exactly the unique $\alpha: F.W \to T.F$ induced by universality.

    For morphisms, the inclusion $(\Lambda^+)^{op} \to \wone$ ensures that any tangent natural transformation will be a linear morphism on the underlying differential bundle, and the tangent natural transformations coherences ensure that a tangent natural transformation will preserve the anchor. Conversely, an anchored bundle morphism will preserve each of the constructed morphisms $E \to T^V.E, E \to T^V.M$ (as it preserves each of $\{ q, +_q, \xi, \lambda\}$), giving a natural tangent transformation. Thus there is a faithful embedding $\mathsf{Anc}(\C) \hookrightarrow [\wone^1, \C]$.
\end{proof}
The converse, identifying those functors $\wone^1 \to \C$ that determine anchored bundles, is immediate.
\begin{corollary}%
    \label{cor:anchored-bundle-as-nerve}
    The category of anchored bundles comprises precisely the $\w$-functors and $\w$-natural transformations
    $
        A: \wone^* \to \C    
    $ (Definition \ref{def:truncated-wone})
    so that the precomposition $\Lambda^+ \to \wone^* \to \C$ determines a differential bundle.
    That is to say, the category of anchored bundles in $\C$ is the following pullback in $\w$Cat:
    \[\begin{tikzcd}
        {\mathsf{Anc}(\C)} & {[\wone^1, \C]} \\
        {\mathsf{DBun}(\C)} & {[\Lambda^+, \C]}
        \arrow[hook, from=1-1, to=1-2]
        \arrow[from=1-2, to=2-2]
        \arrow[hook, from=2-1, to=2-2]
        \arrow[from=1-1, to=2-1]
        \arrow["\lrcorner"{anchor=center, pos=0.125}, draw=none, from=1-1, to=2-2]
    \end{tikzcd}\]
\end{corollary}

\section{Enriched nerve constructions}%
\label{sec:enriched-nerve-constructions}

Nerve constructions present a powerful generalization of the Yoneda functor that sends an object $C \in \C$ to the representable presheaf $\C(-, C) \in[\C^{op},\vv]$ (for a general reference on nerve constructions and their realizations, see Chapter 3 of \cite{Loregian2015}). 
The Yoneda lemma states that this functor is an \emph{embedding} (that is, it is fully faithful), so no information is lost when embedding a category into its category of presheaves. 
Nerve constructions, then, move this towards an \emph{approximation} of the original category $\C$ by some subcategory $\D \hookrightarrow \C$, or more generally by some functor $K:\a \to \C$.
In Section \ref{sec:inf-nerve-of-a-gpd}, the ``infinitesimal'' approximation of a Lie groupoid as a Lie algebroid will be exhibited as approximation by a $\w$-functor $\partial: \wone^{op} \to \mathsf{Gpd}(\w)$.

We work with a $\vv$ that is locally presentable as a monoidal category, such as $\w, \s$, or the category of commutative monoids.
\begin{definition}%
    \label{def:nerve-of-a-functor}
    The \emph{nerve} of a $\vv$-functor $K: \a \to \C$ is the functor
    \[
        N_K: \C \to [\a^{op}, \vv]; \hspace{0.5cm} C \mapsto \C(K-, C)
    \]
    Any presheaf $A:\a \to \vv$ that is in the image of $N_K$ is a \emph{$K$-nerve}.
\end{definition}
\begin{remark}
    The ``$K$-nerve'' terminology seems to go back to Grothendieck/Segals's original intuition for the nerve construction of a category\footnote{Segal published the result, but seems to have credited the theorem to Grothendieck.} (\cite{Segal1974}) and appears in, for example, \cite{Berger2012} and \cite{Bourke2019}. However, the ``approximation'' of a category by a functor $K:\a \to \C$ originally used in topology seems to be a more intuitive description of the functor $N_K$.
\end{remark}

\begin{example}%
    \label{ex:nerve-functors}
    ~\begin{enumerate}[(i)]
        \item The nerve of the identity functor $\C = \C$ is the usual Yoneda embedding $\C \hookrightarrow[\C^{op}, \w]$.
        \item 
        The first example of a nerve construction in the mathematical literature is the simplicial approximation of a topological space by \cite{Kan1958}. Recall the original construction of the simplicial nerve of a topological space $X$, where $X_n = \mathsf{Top}(\Delta_n,X)$. This is exactly the nerve of the functor $\Delta \to \mathsf{Top}$ that sends $n$ to the $n$-simplex
        \[
            \left\{
                x \in \R^n | \sum x_i = 1
            \right\}.
        \]
        \item The following example figures into Segal's original nerve construction for a category, and is revisited in Section \ref{sec:enriched-theories}. Define a reflexive graph to be a presheaf over the full subcategory of $\mathsf{Cat}$ whose objects are the two preorders
        \[
            [0] = 0, \hspace{0.2cm} [1] =  0 < 1.    
        \]
        This is equivalent to the free category with two parallel arrows and a common retract,
        \[\begin{tikzcd}
            {[1]} & {[0]}
            \arrow["t"', shift right=2, from=1-2, to=1-1]
            \arrow["s", shift left=2, from=1-2, to=1-1]
            \arrow["e"{description}, from=1-1, to=1-2]
        \end{tikzcd}\]
        so that a graph in a category has an object-of-vertices $V$ and an object-of-edges $E$, along with source and target maps $s,t:E \to V$; morphisms of graphs are maps $(f_E,f_V)$ that commute with the source and target maps.

        By Corollary \ref{cor:dense-ff-result}, any full subcategory containing the representables will be dense.
        Define the category of \emph{paths}, $\mathsf{Pth}$, to be full subcategory of graphs generated by
        
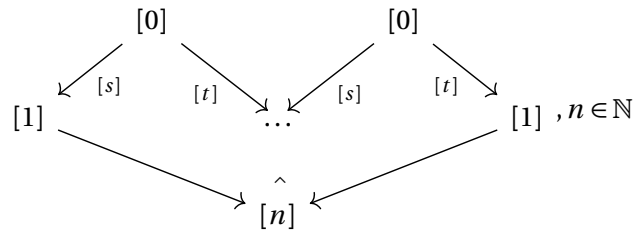
\[ 
\begin{tikzcd}
	& {[0]} && {[0]} \\
	{[1]} && \dots && {[1]} \\
	&& {[n]}
	\arrow["{[t]}"', from=1-2, to=2-3]
	\arrow["{[s]}", from=1-4, to=2-3]
	\arrow[from=2-1, to=3-3]
	\arrow[from=2-5, to=3-3]
	\arrow["{[s]}", from=1-2, to=2-1]
	\arrow["{[t]}"', from=1-4, to=2-5]
	\arrow["\lrcorner"{anchor=center, pos=0.125, rotate=135}, draw=none, from=3-3, to=2-3]
\end{tikzcd}, n \in \mathbb{N} \]
        so that $\mathsf{Gph}([n],G)$ picks out the set of paths of length $n$ in a graph $G$. We call this subcategory $P:\mathsf{Pth} \to \mathsf{Gph}$, and see that $N_P$ sends a graph to its $\mathsf{Pth}$-presheaf of composable paths:
        
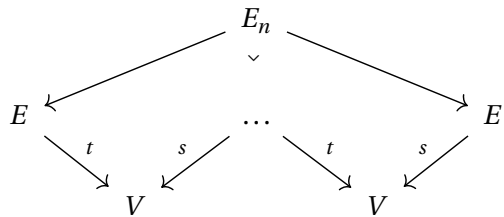
\[\begin{tikzcd}
            && {E_n} \\
            E && \dots && E \\
            & V & {} & V
            \arrow["t", from=2-1, to=3-2]
            \arrow["s"', from=2-3, to=3-2]
            \arrow["t", from=2-3, to=3-4]
            \arrow["s"', from=2-5, to=3-4]
            \arrow[from=1-3, to=2-1]
            \arrow[from=1-3, to=2-5]
            \arrow["\lrcorner"{anchor=center, pos=0.125, rotate=-45}, draw=none, from=1-3, to=2-3]
        \end{tikzcd}\]
        The pushout $[n]$ is precisely the graph
        \[
            0 \xrightarrow[]{} 1 \xrightarrow[]{}\dots \xrightarrow[]{} (n-1) \xrightarrow[]{} n    
        \]
        where the reflexive map $e$ at each vertex is suppressed.
    \end{enumerate}
\end{example}
Recall the functor sending reflexive graphs to anchored bundles from Example \ref{ex:prolongations}(iv).
This may be restated as a nerve construction.
\begin{proposition}%
    \label{prop:lin-approx-gph}
    There is a functor
    \[
        \partial: \wone^* \to \mathsf{Gph}(\w)  
    \]
    so that $N_\partial: \mathsf{Gpd}(\w) \to [\wone^*, \w]$ is the linear approximation of a reflexive graph as described in Example \ref{ex:prolongations} (iv).
\end{proposition}
\begin{proof}
    Let $s,t:C \to M, e:M \to C$ be a reflexive graph, and recall that the anchored bundle $C^\partial \to C$ is induced by the equalizer
    \[
\begin{tikzcd}
	{C^\partial} & {T.C} & {T.C}
	\arrow["{e.T_1}", shift left=1, from=1-2, to=1-3]
	\arrow["{T.e^s}"', shift right=1, from=1-2, to=1-3]
	\arrow[from=1-1, to=1-2]
\end{tikzcd}\]
    The category $\mathsf{Gph}$ is a representable tangent category, where $D$ is the discrete graph $D = D$. By the Yoneda lemma, we have that
    \[
        [\mathsf{Gph}, \w](\yon 1, G) = G_1,
        [\mathsf{Gph}, \w](\yon 0, G) = G_0,
        [\mathsf{Gph}, \w](D \x \yon(i), G) = T.G_i,
    \]
    so the limit diagram defining the infinitesimal approximation of a graph becomes
    \[\begin{tikzcd}
        {C^\partial} & {\mathsf{Gph}(D \x I, C)} & {\mathsf{Gph}(D \x I, C).}
        \arrow[from=1-1, to=1-2]
        \arrow["{(e \x I)^*}", shift left=1, from=1-2, to=1-3]
        \arrow["{(D \x e^-)^*}"', shift right=1, from=1-2, to=1-3]
    \end{tikzcd}\]
    Now observe that $\mathsf{Gph}(\w)^{op}$ with the dual tangent structure from Proposition \ref{prop:inf-object-tangent-structures} has a reflexive graph object \[s,t:1 \to I, !:I \to 1.\] Construct its linear approximation as in Example \ref{ex:prolongations}(iv) (e.g. take the coequalizer of $\w$-graphs):
    \[\begin{tikzcd}
        {D\x I} & {D\x I} & \partial.
        \arrow["{e \x I}", shift left=1, from=1-1, to=1-2]
        \arrow["{D \x e^s}"', shift right=1, from=1-1, to=1-2]
        \arrow[from=1-2, to=1-3]
    \end{tikzcd}\]
    By Proposition \ref{prop:nerve-anc-work}, this determines a functor
    \[
        \partial: \wone^1 \to \mathsf{Gph}(\w)^{op}. 
    \]
    By the continuity of the hom-functor, the nerve $N_\partial: \mathsf{Gph}(\w) \to [\wone^*, \w]$ lands in the category of anchored bundles. The continuity of the hom-functor ensures that this is indeed the linear approximation from Example \ref{ex:prolongations} (iv).
\end{proof}

The simplicial localization in \cite{Kan1958} has a left adjoint, the \emph{geometric realization}, that constructs a topological space using the data of a simplicial set. This realization may be constructed using a left Kan extension.
\begin{definition}%
    \label{def:realization}
    Let $K: \a \to \C$ be a $\vv$-functor for a cocomplete $\,\C$.
    The \emph{realization} of $K$ is the left Kan extension
    \[
\begin{tikzcd}
    A & {\widehat{A}} \\
    \C
    \arrow["K"', from=1-1, to=2-1]
    \arrow["\yon", from=1-1, to=1-2]
    \arrow["{|-| := Lan_K\yon}"'{pos=1}, from=2-1, to=1-2]
\end{tikzcd}
        \hspace{0.5cm}
        |C|_K := \int^A \C(KA, C) \bullet KA
    \]
    A \emph{nerve/realization context} is a $\vv$-functor $K:\a \to \C$ from a small $A$ to a cocomplete $\C$.   
\end{definition}
The adjunction between simplicial sets and topological spaces follows from general categorical machinery as a nerve/realization context:
\begin{lemma}
    For every nerve/realization context $K: \a \to \C$, the realization of $K$ is left adjoint to the nerve of $K$.
\end{lemma}

\section{Nervous monads and algebroids}%
\label{sec:enriched-theories}

Transposing the characterization of algebroids from Theorem \ref{thm:weil-nerve} to the enriched perspective introduces the challenge of finding an appropriate framework to describe algebroids as enriched structures. Kelly's theory of enriched sketches (see Chapter 6 of \cite{Kelly2005}), small $\vv$-categories equipped with a chosen class of limits cones, seems to be a natural candidate. However, when regarding a tangent functor as a $\w$-functor, the natural part 
\[
    \alpha: F.T \Rightarrow T.F 
\]
is the unique morphism induced by the universality of $T$ as a weighted limit, so it becomes unclear how to translate the condition that $\alpha$ is a cartesian natural transformation (Definition \ref{def:cart-nat}).

A clue for how to proceed may be found in \cite{Kapranov2007}, which proved that Lie algebroids are monadic over anchored bundles (when allowing for infinite-dimensional vector bundles). 
Recent work in \cite{Bourke2019} and \cite{Berger2012} has developed the appropriate notion of (enriched) theories that correspond to monads over general locally presentable categories. 
A critical insight is that for a filtered-colimit-preserving monad $\mathbb{T}$ on $\s$, the opposite category of the Lawvere theory $\th$ is precisely the full subcategory of the Kleisli category whose objects are $[n] = \coprod_n 1$, and the \emph{nerve} of the inclusion
\[
    \th^{op} \hookrightarrow \s^{\mathbb{T}}
\]
is fully faithful; that is, when the functor $K$ is dense.
Formally, a functor is \emph{dense} whenever its nerve behaves like the Yoneda embedding. The theory of dense functors is developed in Chapter 5 of \cite{Kelly2005}. We continue to assume that an arbitrary site of enrichment $\vv$ is locally presentable as a $\vv$-category.
\begin{definition}%
    \label{def:dense}
    A functor $K: \a \to \C$ is \emph{dense} whenever the nerve of $K$ is fully faithful, and a subcategory inclusion that is dense will be called a \emph{dense subcategory}.
\end{definition}

Dense functors are poorly behaved under composition, but there is a useful cancellativity result from Section 5.2 of \cite{Kelly2005}.
\begin{proposition}
    Consider a diagram of $\vv$-categories
    \[
\begin{tikzcd}
    X & Z \\
    Y
    \arrow["K", from=1-1, to=1-2]
    \arrow["F"', from=1-1, to=2-1]
    \arrow[""{name=0, anchor=center, inner sep=0}, "J"', from=2-1, to=1-2]
    \arrow["\alpha"', shorten >=2pt, Rightarrow, from=1-1, to=0]
\end{tikzcd}\]
    where $\alpha$ is a natural isomorphism and $K$ is dense.
    If this diagram exhibits $(\alpha, J)$ as the left Kan extension of $K$ along $F$, then $J$ is also dense.
\end{proposition}
\begin{corollary}%
    \label{cor:dense-ff-result}
    If $K$ is dense, $F$ is fully faithful, and  $J.F = K$, then $F$ is a dense subcategory.
\end{corollary}
Generally speaking, we will often refer to dense \emph{subcategories} rather than dense \emph{functors}. 
This is achieved by factoring the dense functor
\[
    K = \a \xrightarrow[]{K'} \mathsf{im}(K) \hookrightarrow \C
\]
where $\mathsf{im}(K)$ is the $\vv$-category whose objects are those of $\a$ and whose hom-objects are given by $\mathsf{im}(K)(A,B) = \C(KA, KB)$,  so that the inclusion of $\mathsf{im}(K)$ into $\C$ is fully faithful and therefore dense by Corollary \ref{cor:dense-ff-result}.

\begin{example}\label{ex:fin-card-def}
    ~\begin{enumerate}[(i)]
        \item The category of finite cardinals $\Sigma$ is the full subcategory of $\,\s$ whose objects are given by finite coproducts of the terminal object, $[n] = \coprod_n 1$. This is a skeleton of the category of finite sets, and a dense subcategory of $\,\s$ (see e.g. \cite{Bourke2019}).
        \item Recall that the category of $\vv$-presheaves on $\a$, $[\a^{op}, \vv]$ is the free colimit completion of $\a$ (see e.g. \cite{Kelly2005}). 
        If $\C$ is cocomplete, and $K:\a \to \C$ is dense, then the realization $|-|_K$ exhibits $\C$ as a reflective subcategory of $[\a^{op}, \vv]$, and is this a locally presentable category\footnote{In fact, an equivalent definition of a locally presentable category is as a cocomplete category with a dense subcategory.}. Conversely, if $\,\C$ is a reflective subcategory of $\,[\a^{op}, \vv]$ that contains the representable functors, then $\a$ is a dense subcategory of $\,\C$. 
        \item By Corollary \ref{cor:Lambda-dense}, the category of differential bundles in $\w$ is a reflective subcategory of $[{\Lambda^+}{op}, \vv]$, so $\yon: \Lambda^+ \hookrightarrow [{\Lambda^+}{op}, \vv]$ is a dense subcategory of differential bundles in $\w$ following the argument in the above example.
    \end{enumerate}
\end{example}

Nervous theories (\cite{Bourke2019}) are generalizations of Lawvere theories that extend to arbitrary locally finitely presentable $\vv$-categories. Recall that a classical Lawvere theory is a bijective-on-objects, product-preserving functor
\[
    t: \Sigma \to \th   
\]
where $\th$ is a cartesian category. Nervous theories replace $\Sigma$ with a dense subcategory of some locally presentable $\vv$-category, and the product preservation condition with conditions on the nerve of the theory map.
\begin{definition}%
    \label{def:nerve-theory}
    Let $K:\a \to \C$ be a dense $\vv$-subcategory of a locally finitely presentable $\,\C$. We call the replete image of $N_K$ in $[\a^{op}, \vv]$ the category of $K$-nerves. An \emph{$\a$-theory} is a bijective-on-objects $\vv$-functor $J: \a \to \th$, where each
    \[
        \th(J-,a): \a \to \vv
    \] is a $K$-nerve.
    The category of models for an $\a$-theory is the pullback in $\vv\mathsf{CAT}$:
    \[
\begin{tikzcd}
	{\C^\th} & {\widehat{\th}} \\
	\C & \widehat\a
	\arrow["{J^*}", from=1-2, to=2-2]
	\arrow["{N_K}", hook, from=2-1, to=2-2]
	\arrow[from=1-1, to=2-1]
	\arrow[from=1-1, to=1-2]
	\arrow["\lrcorner"{anchor=center, pos=0.125}, draw=none, from=1-1, to=2-2]
\end{tikzcd}\]
    (These are called \emph{concrete} models in \cite{Bourke2019}.)
\end{definition}
\begin{remark}
    The category of models for a theory is monadic. The core of the argument is due to Weber's \emph{nerve theorem}, found in \cite{Weber2007}, but an exposition on that result is beyond the scope of this thesis. 
\end{remark}
\begin{example}%
    \label{ex:theories}
    ~\begin{enumerate}[(i)]
        \item A functor $t:\Sigma \hookrightarrow \th$ is a $\Sigma$-theory if and only if $\th$ is a Lawvere theory, where  we use the fact that $\Sigma$ (Example \ref{ex:fin-card-def}) a  skeletal subcategory of finite sets. The nerve conditions in this case identify the models of the Lawvere theory, as the nerve of $\,\Sigma \hookrightarrow \s$ sends a set to the strict product-preserving functor $[n] \mapsto A^n$.
        \item As shown in \cite{Berger2012} and \cite{Bourke2019}, the original nerve construction from \cite{Segal1974} may be restated as saying that small categories arise as models of a $\mathsf{Pth}$-theory (Example \ref{ex:nerve-functors}). The set $\mathsf{Gph}([n], G)$ is the set of paths of through the graph $G$ that have $n$ non-identity elements, as
        \[
            G \ts{t}{s} G \cong G \ts{t}{id} M \ts{id}{s} G \cong G \ts{t}{s \o e \o s} G \ts{t \o e \o t}{s} G.     
        \]

        Next, recall that the category $\Delta$ may be regarded as follows:
        \begin{itemize}
            \item Objects: A strict order $[n] = 0 < 1 < \dots < n$, for $n \ge 0$, regarded as a category.
            \item Maps: Functors.
        \end{itemize}
        There is a bijective-on-objects functor from $\mathsf{Pth} \to \Delta$ that sends the graph $[n]$ to the pre-order $[n]$, as every graph homomorphism between paths will be order-preserving. Now observe that a model is precisely a simplicial set, where 
        \[
            X([0]) = M, 
            X([1]) = C,
            X([2]) = C \ts{t}{s} C,
            X([3]) = C \ts{t}{s} C \ts{t}{s} C.  
        \]
        Furthermore, the map
        
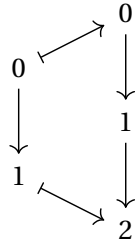
\[\begin{tikzcd}[row sep = tiny]
            & 0 \\
            0 \\
            & 1 \\
            1 \\
            & 2
            \arrow[maps to, from=2-1, to=1-2]
            \arrow[maps to, from=4-1, to=5-2]
            \arrow[from=2-1, to=4-1]
            \arrow[from=1-2, to=3-2]
            \arrow[from=3-2, to=5-2]
        \end{tikzcd}\]
        in $\Delta$ becomes a composition map: 
        \[ 
            \infer{C \ts{t}{s} C \to C}{X([2]) \to X([1])}
        \] 
        Associativity and unitality (for the section $e: M \to C$) follow by functoriality. Thus, a model of $\mathsf{Pth} \to \Delta$ is a small category, and a morphism is exactly a functor.
    \end{enumerate}
\end{example}
Corollary \ref{cor:anchored-bundle-as-nerve}, then, gives a monadicity result for $\w$-anchored bundles over $\w$-differential bundles.
\begin{proposition}
    $\w$-anchored bundles are models of the theory $(\Lambda^+)^{op} \to (\wone^*)^{op}$ (see Section \ref{sec:enriched-structures}).
\end{proposition}
\begin{proof}
    Note that the inclusion $\Lambda^+ \to \wone^1$ is a differential bundle, so it is a nerve. Then the diagram:
    \[\begin{tikzcd}
        {\mathsf{Anc}(\w)} & {[\wone^*,\w]} \\
        {\mathsf{DBun}(\w)} & {[\Lambda^+, \w]}
        \arrow[from=1-2, to=2-2]
        \arrow[hook, from=2-1, to=2-2]
        \arrow[from=1-1, to=2-1]
        \arrow[from=1-1, to=1-2]
        \arrow["\lrcorner"{anchor=center, pos=0.125}, draw=none, from=1-1, to=2-2]
    \end{tikzcd}\]
    exhibits $\w$-anchored bundles as models of a $(\Lambda)^{op}$-theory.
\end{proof}
As a corollary, we see that the category of anchored bundles is locally presentable.
\begin{corollary}
    The category of anchored bundles is locally presentable, and $(\wone^*)^{op} \hookrightarrow \mathsf{Anc}(\w)$ is dense.
\end{corollary}
In \cite{Bourke2019}, the authors identify exactly those monads on a locally presentable $\vv$-category that correspond to the models of a theory.
Recall the notation that the category of algebras for a monad is $\C^T$ and the category of free coalgebras is $\C_T$. For a dense subcategory $\a \hookrightarrow \C$, use $\a_T$ for the category of free algebras over objects in $\a$. Recall that the Lawvere theory \[K:\mathsf{FinSet} \to \th\] for a filtered-colimit-preserving monad $\mathbb{T}$ on $\s$ may be re-derived as the full subcategory of free algebras over finite sets, $K_T: \mathsf{FinSet}_{\mathbb{T}} \hookrightarrow \s_{\mathbb{T}} \hookrightarrow \s^{\mathbb{T}}$. Nervous monads abstract this property.
\begin{definition}
    Let $\C$ be a locally presentable $\vv$-category, with $K:\a \to \C$ a dense sub-$\vv$-category. 
    A $\vv$-monad $\mathbb{T}$ over $\C$ is \emph{$K$-nervous} if
    \begin{enumerate}
        \item the inclusion $K_T: \a_T \hookrightarrow \C^T$ is dense;
        \item the following diagram is a pullback in $\vv$CAT:
        \[\begin{tikzcd}
            {\C^T} & {[\a_T^{op}, \vv]} \\
            \C & {[\a^{op}, \vv]}
            \arrow[hook, from=2-1, to=2-2]
            \arrow[from=1-2, to=2-2]
            \arrow[from=1-1, to=2-1]
            \arrow[from=1-1, to=1-2]
        \end{tikzcd}\]
    \end{enumerate}
\end{definition}

\begin{theorem}[\cite{Bourke2019}]
    Let $K:\a \hookrightarrow \mathbb{C}$ be a dense sub-$\vv$-category of a cocomplete $\vv$-category $\,\C$.
    There is an equivalence of categories between $\a$-theories and $\a$-nervous monads on $\,\C$.
\end{theorem}

The first step in showing that algebroids are monadic over anchored bundles is to construct a dense subcategory of $\mathsf{Anc}(\w)$ that plays the role of the $V$-prolongations from Definition \ref{def:monoidal-category}. The enriched framework makes this straightforward: prolongations are weighted limits, with which we may freely complete $\wone^*$.

\begin{definition}
    \label{def:prol}
    Consider the category $\mathsf{Anc}(\w)$ as a representable tangent category, and take the dual tangent structure on $\mathsf{Anc}(\w)^{op}$.
    The category $\prol$ is defined as the full subcategory of $\mathsf{Anc}(\w)^{op}$ whose objects are generated by the prolongations of the Yoneda embedding $\yon:\wone^* \to \mathsf{Anc}(\w)$.
\end{definition}

For any small $\w$-category $\C$, the free completion of $\C$ is the opposite $\w$-category of the category of copresheaves on $\C$.\footnote{This is the dual statement of the classical theorem that the category of presheaves is the free cocompletion of a small category $\C$.} Thus, $\prol$ is the free completion of $\wone^1$ with prolongations. Therefore, a choice of prolongations on an anchored bundle $A$ determines a unique functor $\prol \to \C$ given by right Kan extension (see e.g. \cite{Kelly1982}). The continuity of 
\[
    \mathsf{Anc}(\w)(-, A): \mathsf{Anc}(\w)^{op} \to \w 
\]
ensures that $\mathsf{Anc}(\w)(L_V, A)$ is $A_V$. 

By Theorem \ref{thm:weil-nerve}, the category of involution algebroids in $\C$ is a full sub-$\w$-category of $[\wone, \C]$ that has a forgetful functor down to the category of anchored bundles. A consequence of Theorem \ref{thm:iso-of-cats-inv-emcs} is that a functor $\wone \to \C$ is an involution algebroid if and only if the precomposition \[\wone^* \hookrightarrow \wone \to \C\] restricts to an anchored bundle, with each $V$ sent to the $V$-prolongation of this anchored bundle. In other words, algebroids are models of the following enriched theorem (as in Definition \ref{def:nerve-theory}).

\begin{definition}
    \label{def:weil-theory}
    Define the $\prol^{op}$-theory of algebroids as the functor
    \[
        a: \prol \to \wone.
    \]
    This functor is bijective-on-objects by definition, and satisfies the nerve condition because the $V$-prolongation of the tangent bundle is $T^V$, so each presheaf
    \[
        \wone(a-, V):\prol \to \w   
    \] 
    is an $\prol$-nerve.
\end{definition}

The Weil nerve, then, translates to the following characterization of the category of involution algebroids in a tangent category $\C$. 
\begin{theorem}%
    \label{thm:pullback-in-cat-of-cats-inv-algd}
    The category of involution algebroids with chosen prolongations in a tangent category $\C$ is precisely the pullback in $\w\mathsf{CAT}$:
    \begin{equation}\label{eq:prol2}
        \begin{tikzcd}
    {Inv^*(\C)} & {[\wone,\C]} \\
    {Anc^*(\C)} & {[{\prol},\C]}
    \arrow[hook, from=2-1, to=2-2]
    \arrow[from=1-2, to=2-2]
    \arrow[from=1-1, to=2-1]
    \arrow[hook, from=1-1, to=1-2]
    \arrow["\lrcorner"{anchor=center, pos=0.125}, draw=none, from=1-1, to=2-2]
\end{tikzcd}

    \end{equation}
    Consequently, in $\w$ involution algebroids are monadic over anchored bundles.
\end{theorem}
\begin{proof}
    Recall the correspondence
    \[
        \infer{(\hat A, \alpha): \wone \to \C \in \mathsf{TangCat}_{lax}}
        {\bar A: \wone \to \C \in \w\cat}
    \]
    If
    \[
        \wone^* \to \prol \to \wone \to \C  
    \]
    determines an anchored bundle $(\pi:A \to M, \xi, \lambda, \anc)$ whose $V$-coprolongation is $\hat A.V$, so that $\hat A$ is the nerve of an involution algebroid by Corollary \ref{cor:the-prolongation-description}.
\end{proof}

Thus, we may characterize $\w$-involution algebroids as algebras for a $\prol$-nervous monad on the category of $\w$-anchored bundles.
\begin{corollary}
    The category of involution algebroids in $\w$ is equivalent to the category of algebras for the $\prol$-nervous monad on $\mathsf{Anc}(\w)$ generated by the theory
    \[
        a: \prol \to \wone
    \]
    from Definition \ref{def:weil-theory}, using Theorem 19 from \cite{Bourke2019}.
\end{corollary}

\begin{remark}
    The construction of involution algebroids as the category of models for a $\prol$-theory is remarkably similar to the original nerve construction for categories in \cite{Segal1974}, with the symmetric nerve construction for groupoids replacing $\Delta$ with the category of finite sets $\Sigma$ (see Example 44 (iv) of \cite{Bourke2019}). Each construction truncates the original category to two objects and builds a new category with a bijective set of objects by freely adding limits to the two-object truncation.
\end{remark}

\section{The infinitesimal approximation of a groupoid}%
\label{sec:inf-nerve-of-a-gpd}

This section contains the main theorem of the chapter, namely that there is an adjunction
\[\begin{tikzcd}
    {\mathsf{Gpd}(\w)} & {\mathsf{Inv}(\w)}
    \arrow[""{name=0, anchor=center, inner sep=0}, curve={height=-12pt}, from=1-1, to=1-2]
    \arrow[""{name=1, anchor=center, inner sep=0}, curve={height=-12pt}, from=1-2, to=1-1]
    \arrow["\dashv"{anchor=center, rotate=-90}, draw=none, from=0, to=1]
\end{tikzcd}\]
This adjunction follows by inducing a nerve/realization context on the category of groupoids in $\w$:
\[
    \partial: \wone^{op} \to \mathsf{Gpd}(\w).   
\]
A simplified version of this functor appeared in Proposition \ref{prop:lin-approx-gph}, the linear approximation of a reflexive graph. 
In fact, $\partial$ will be exactly the free groupoid over the graph $\partial$ from Proposition \ref{prop:lin-approx-gph} (if we had worked in the category of reflexive graphs equipped with an involution).
Recall that groupoids and Weil spaces  are both models of a sketch (recall Section \ref{sec:tang-cats-enrichment}). 
The symmetry of the theories gives two equivalent presentations of groupoids in $\w$:
\begin{definition}%
    \label{def:w-gpd}
    The tangent category of $\w$-groupoids is equivalently
    \begin{enumerate}[(i)]
        \item the tangent category of internal groupoids in $\w$, where the tangent structure is computed pointwise;
        \item the tangent category of transverse-limit-preserving functors $\wone \to \mathsf{Gpd}(\s)$ with the cofree tangent structure from Observation \ref{obs:cofree-tangent-cat}.
    \end{enumerate}
\end{definition}
It is important to note that the tangent structure on $\w$-groupoids is representable, as it is a cartesian closed category with an infinitesimal object given by $\wone \hookrightarrow \w \hookrightarrow \mathsf{Gpd}(\w)$ (a proof that $\mathsf{Gpd}(\C)$ is cartesian closed for a cartesian closed category $\C$ with sufficient limits a may be found in Section B2.3 of \cite{Johnstone2002}).

There are two classes of ``trivial'' $\w$-groupoids that will be useful.
\begin{example}
    ~\begin{enumerate}[(i)]
        \item Every small groupoid $G$ may be regarded as a groupoid in $\w$ as the constant functor $\wone \to \mathsf{Gpd}(\s)$ sending $V$ to the groupoid $G$ preserves transverse limits, we may then apply the symmetry of theories to find a groupoid in the category of Weil spaces. 
        \item Every object $M$ in a category has a cofree groupoid given by $s,t:M = M$, this is the \emph{discrete} groupoid whose only morphisms are the identity maps. Thus, every Weil space $M \in \w$ has a corresponding discrete groupoid (this will often be written $M = M$ to denote the discrete groupoid over $M$).
    \end{enumerate}
\end{example}

Using the cartesian closure of $\w$ and the full subcategories of trivial and discrete groupoids, we have the following:
\begin{observation}
    The category of $\w$-groupoids is both
    \begin{enumerate}[(i)]
        \item a $\w$-category, with powers by $\w$. The power of a $\w$-groupoid $\mathcal{G} = s,t:G \to M$ by a Weil space $E$ is given by the $\w$-groupoid
        \[
            [E, \mathcal{G}] = \{[E,s], [E, t]: [E,G] \to [E, M], \hspace{0.15cm} [E,e]: [E, M] \to [E, G]  \};
        \]
        \item a $\mathsf{Gpd}$-enriched category with powers by small groupoids. The power of a $\w$-groupoid $\mathcal{G}:\mathsf{Gpd} \to \w$ by a groupoid $\mathcal{H}$ is the $\w$-groupoid
        \[
            [\mathcal{H}, \mathcal{G}](V) = [\mathcal{H}, \mathcal{G}(V)]_{\mathsf{Gpd}}.
        \]
    \end{enumerate}
\end{observation}

The following two $\w$-groupoids form the basic building block for the main result.
\begin{example}
    ~\begin{enumerate}[(i)]
        \item Set $D := \yon W$ in $\w$. The discrete groupoid $D = D$ represents the tangent functor internally, and the copresheaf $\mathsf{Gpd}(D, -): \mathsf{Gpd}(\w) \to \w$ sends a groupoid $s,t:G \to M$ to $TM$.
        \item The arrow groupoid is the free groupoid generated by the graph:
        \[\bullet \to \bullet\]
        Write the trivial $\w$-groupoid on this groupoid as $I$. Note that the power by $I$ will send a groupoid to its ``arrow groupoid'' $\mathcal{G}^\to$, the groupoid whose objects are arrows in $\mathcal{G}$, with a map $u \to v$ being a commuting square.
        It follows that $\mathsf{Gpd}(I, \mathcal{G})$ is the space of arrows of the groupoid, $\mathsf{Gpd}(I \x I, \mathcal{G})$ the space of commuting squares, and so on.
    \end{enumerate}
\end{example}

The next proposition gives the necessary properties of $I$ to construct the Lie derivative.
\begin{lemma}%
    \label{lem:arrow-gpd-facts}
    ~\begin{enumerate}[(i)]
        \item For every groupoid $\mathcal{G} = s,t:G \to M$, there is an isomorphism
        \[
            G \ts{t}{t} G \ts{s}{t} G \cong \mathsf{Gpd}(\w)(I \x I, \mathcal{G});
        \]
        this corresponds to the unique filler for the diagram
        \[\begin{tikzcd}
            \bullet & \bullet \\
            \bullet & \bullet
            \arrow["u", from=2-1, to=1-1]
            \arrow["v", from=1-1, to=1-2]
            \arrow["w"', from=2-2, to=1-2]
            \arrow["{\exists! w^{-1} \o v \o u}"', dashed, from=2-1, to=2-2]
        \end{tikzcd}\]
        \item The arrow groupoid has a semigroup with a zero structure (like an infinitesimal object $D$), and the multiplication is the coequalizer
        \[\begin{tikzcd}
            & {1 \x I} \\
            {I \x I} && {I \x I} & I \\
            & {I \x 1}
            \arrow["{! \x I}", from=2-1, to=1-2]
            \arrow["{s \x I}", from=1-2, to=2-3]
            \arrow["{I \x !}"', from=2-1, to=3-2]
            \arrow["{I \x s}"', from=3-2, to=2-3]
            \arrow["m", from=2-3, to=2-4]
        \end{tikzcd}\]

    \end{enumerate}
\end{lemma}
The dual result of part $(ii)$ is somewhat more obvious to see, as the following fork is always an equalizer in $\w$ for a groupoid $G$:
\[\begin{tikzcd}
    G & {G^\square} & {G^\square}
    \arrow["{G^{e[s] \x I}}", shift left=2, from=1-2, to=1-3]
    \arrow["{G^{I \x e[s]}}"', shift right=1, from=1-2, to=1-3]
    \arrow["{G^{m}}", from=1-1, to=1-2]
\end{tikzcd}\]
(where we write $G$ as the object of arrows and $\mathsf{Gpd}(\w)(I\x I, G)$ as $G^\square$). 
Note that $G^{e[s] \x I}$ and $G^{I \x e[s]}$ correspond to the idempotents on $G^\square$:
\[\begin{tikzcd}
    \bullet & \bullet & \bullet & \bullet & \bullet & \bullet \\
    \bullet & \bullet & \bullet & \bullet & \bullet & \bullet
    \arrow[""{name=0, anchor=center, inner sep=0}, "q"', from=1-3, to=2-3]
    \arrow["u", from=1-3, to=1-4]
    \arrow[""{name=1, anchor=center, inner sep=0}, "v", from=1-4, to=2-4]
    \arrow["w"', from=2-3, to=2-4]
    \arrow["u", from=1-5, to=1-6]
    \arrow[""{name=2, anchor=center, inner sep=0}, Rightarrow, no head, from=1-5, to=2-5]
    \arrow[Rightarrow, no head, from=1-6, to=2-6]
    \arrow["u"', from=2-5, to=2-6]
    \arrow["q"', from=1-1, to=2-1]
    \arrow[""{name=3, anchor=center, inner sep=0}, "q", from=1-2, to=2-2]
    \arrow[Rightarrow, no head, from=1-1, to=1-2]
    \arrow[Rightarrow, no head, from=2-1, to=2-2]
    \arrow["{G^{I \x e[s]}}", shorten <=10pt, shorten >=10pt, Rightarrow, maps to, from=1, to=2]
    \arrow["{G^{e[s] \x I}}"', shorten <=10pt, shorten >=10pt, Rightarrow, maps to, from=0, to=3]
\end{tikzcd}\]
Also note that a commuting square is equalized by these two maps if and only if $u = q = id$, which forces $w = v$; these commuting squares are exactly the image of $G^m$.

The semigroup structure on $I$ represents the map sending
\[\begin{tikzcd}
    Y & X & Y \\
    X & X & X
    \arrow[""{name=0, anchor=center, inner sep=0}, "u"{description}, from=2-1, to=1-1]
    \arrow["u"{description}, from=1-2, to=1-3]
    \arrow["u"{description}, from=2-3, to=1-3]
    \arrow[""{name=1, anchor=center, inner sep=0}, Rightarrow, no head, from=1-2, to=2-2]
    \arrow[Rightarrow, no head, from=2-2, to=2-3]
    \arrow[shorten <=13pt, shorten >=10pt, Rightarrow, maps to, from=0, to=1]
\end{tikzcd}\]
Now look at the $\partial$ defined in Proposition \ref{prop:lin-approx-gph}, and take the same colimit in $\mathsf{Gpd}(\w)$. 
Groupoids are monadic over reflexive graphs with an involution, so this is essentially the free groupoid over that graph (as the free functor is a left adjoint and therefore preserves colimits).
\begin{definition}%
    \label{def:partial-gpd}
    Define the $\w$-groupoid $\partial$ to be the $\w$-coequalizer
    \[\begin{tikzcd}
        {D\x I} & {D\x I} & \partial
        \arrow["{e \x I}", shift left=1, from=1-1, to=1-2]
        \arrow["{D \x e[s]}"', shift right=1, from=1-1, to=1-2]
        \arrow[from=1-2, to=1-3]
    \end{tikzcd}\]
    where $e = 0 \o !, e[s] = i \o s$. Note that just as in Proposition \ref{prop:lin-approx-gph}, $\partial$ is an anchored bundle.
\end{definition}

The properties of the arrow groupoid make it possible to prove the following theorem.

\begin{theorem}%
    \label{thm:inf-nerve}
    The object $\partial$ determines a cartesian $\w$-functor
    \[
        \partial(-): \wone^{op} \to \mathsf{Gpd}(\w)    
    \]
    that is both an infinitesimal object in $\w$-groupoids and an involution algebroid in $\mathsf{Gpd}(\w)^{op}$.
\end{theorem}
\begin{proof}
    Write the pushout powers of $\xi:1 \to \partial$ as $\partial(n)$, and for any Weil algebra $V = W_{n(1)} \ox \dots \ox W_{n(k)}$, write $\partial(V) = \partial(n_1) \x \dots \x \partial(n_k)$. Composition for internal categories will be written in the diagramattic order, writing composition as an infix semicolon ``;'' to keep it distinct from composition of morphisms in $\w$. 
    
    The proof has two main steps:
    \begin{enumerate}
        \item For every Weil algebra $V$, there is an isomorphism
        \[
            \partial_V \cong \partial(V)    
        \]
        where $\partial_V$ is the $V$-prolongation of the anchored bundle in $\mathsf{Gpd}(\w)^{op}$.
        \item The universal lift for the anchored bundle, given as a coequalizer
        \[\begin{tikzcd}
            {\partial \x \partial } & {\partial\x \partial} & \partial
            \arrow["{e^\partial \x \partial}", shift left=2, from=1-1, to=1-2]
            \arrow["{e^\partial \x \partial }"', shift right=2, from=1-1, to=1-2]
            \arrow["\boxdot", from=1-2, to=1-3]
        \end{tikzcd}\]
        is a semigroup.
    \end{enumerate}
    From (1) and (2), we may infer that the object $\partial$ is an infinitesimal object. The semigroup map $\boxdot$ is commutative and has a zero by (2), and satisfies all of the couniversality axioms by (1). 
    \begin{enumerate}
        \item   The proof of this step follows by induction on the width (Definition \ref{def:truncated-wone}) of the Weil algebra $V$, where the cases $0, 1$ both hold by definition.
        In the case $n=2$, we need only check this holds for $\partial \x \partial \cong \partial_{WW}$:
        \[
\begin{tikzcd}
	& {D \x \partial} \\
	D && {L_2} & {\partial \x \partial} \\
	& \partial
	\arrow["{(id,0)}", from=2-1, to=1-2]
	\arrow["\delta"', from=2-1, to=3-2]
	\arrow[from=1-2, to=2-3]
	\arrow[from=3-2, to=2-3]
	\arrow["\lrcorner"{anchor=center, pos=0.125, rotate=-135}, draw=none, from=2-3, to=2-1]
	\arrow["{(\delta,id)}"{description}, curve={height=-6pt}, from=1-2, to=2-4]
	\arrow["{(id,0)}"{description}, curve={height=6pt}, from=3-2, to=2-4]
	\arrow[dashed, from=2-3, to=2-4]
\end{tikzcd}\]
        For any $G$, a map $X: \partial \x \partial \to G$ corresponds to a commuting square in $T^2G$ of the form
        \[
\begin{tikzcd}
	\bullet & \bullet \\
	\bullet & \bullet
	\arrow["{T.e \o u}", from=1-1, to=1-2]
	\arrow["v"{description}, from=1-2, to=2-2]
	\arrow["{e.T \o v}"{description}, from=1-1, to=2-1]
	\arrow["u"{description}, from=2-1, to=2-2]
\end{tikzcd}\]
        so that $T.e \o v = id$ and $e.T \o u = id$.
        It follows that $u = (e.T \o v)^{-1};(T.e \o u);v$ (Lemma \ref{lem:arrow-gpd-facts} (i)), and that precomposition with the uniquely induced map determines $(\bar{u},v):X \to G^{L(2)}$, where
        \[ T.0 \o \bar{u} = u.\]
        Observe that any pair $(\bar{u}, v):X \to  G^{L(2)}$ determines a square
        \[
\begin{tikzcd}
	\bullet & \bullet \\
	\bullet & \bullet
	\arrow["{T.0 \o \bar{u}}", from=1-1, to=1-2]
	\arrow["v"{description}, from=1-2, to=2-2]
	\arrow["{e.T \o v}"{description}, from=1-1, to=2-1]
	\arrow[dashed, from=2-1, to=2-2]
\end{tikzcd}
        \]
        where $T.e \o v = id$ and $T.e \o \bar{u} = id$.
        Note that $T.0 \o T.p \o T.0 \o \bar{u} = T.0 \o \bar{u}$, and that the bottom horn is $u = (e.T \o v)^{-1};(T.0 \o \bar{u});v$; now check
        \begin{align*}
            T.e \o u
                & = (T.e \o e.T \o v)^{-1};(T.e \o T.0 \o \bar{u});(T.e \o v) \\
                & = id; T.0 \o \bar{u}; T.e \o v = T.0 \o \bar{u}
        \end{align*}
        and
        \begin{align*}
            e.T \o u
                & = e.T \o (e.T \o v)^{-1};(T.0 \o \bar{u});v                 \\
                & = (e.T \o e.T \o v)^{-1};(e.T \o T.0 \o \bar{u});(e.T \o v) \\
                & = (e.T \o v)^{-1};id;(e.T \o v) = id
        \end{align*}
        thus determining a map $X \to G^{\partial \x \partial}$.
        The two maps are inverse to each other, giving an isomorphism.
    
        For the inductive case, look at the prolongations of anchored bundles and recall that
        \[
            A_{UV} \boxtimes_M A_Z  =
            A_{UV} \boxtimes_{\prol(U,A)} A_{UZ}
        \]
        Where $A_{UV}$ and $A_{UZ}$ are treated as spans over $A_U$, so that
        \begin{gather*}
            A_U \xleftarrow[]{id \boxtimes \pi^V} A_{UV} \xrightarrow[]{anc^U \boxtimes A_V} T^U.A_V\\
            A_U \xleftarrow[]{id \boxtimes \pi^Z} A_{UZ} \xrightarrow[]{anc^U \boxtimes A_Z} T^U.A_Z
        \end{gather*}
        and observe that their span composition is
        \[
\begin{tikzcd}
	&& {A_{UVZ}} \\
	& {A_{UV}} && {T^V.A_{UZ}} \\
	{A_U} && {T^V.A_U} && {T^Z.T^V.A_U}
	\arrow[from=2-2, to=3-1]
	\arrow["{\anc^U\boxtimes A_V}"{description}, from=2-2, to=3-3]
	\arrow[from=2-4, to=3-3]
	\arrow["{T^U.(\anc^U \boxtimes A_Z)}"{description}, from=2-4, to=3-5]
	\arrow[from=1-3, to=2-2]
	\arrow[from=1-3, to=2-4]
	\arrow["\lrcorner"{anchor=center, pos=0.125, rotate=-45}, draw=none, from=1-3, to=3-3]
\end{tikzcd}
        \]
        So it now suffices to prove that the diagram
        \[
\begin{tikzcd}
	{\partial_U \x D_V} & {\partial_U \x (D_V \x \partial_Z)} \\
	{\partial_U \x \partial_V} & {\partial_U \x (\partial_V \x \partial_Z)}
	\arrow[from=1-1, to=2-1]
	\arrow[from=1-1, to=1-2]
	\arrow[from=1-2, to=2-2]
	\arrow[from=2-1, to=2-2]
	\arrow["\lrcorner"{anchor=center, pos=0.125, rotate=180}, draw=none, from=2-2, to=1-1]
\end{tikzcd}\]
        is a pushout. But by the inductive hypothesis,
        \[
\begin{tikzcd}
	{ D_V} & {D_V \x \partial_Z} \\
	{\partial_V} & {\partial_V \x \partial_Z}
	\arrow[from=1-1, to=2-1]
	\arrow[from=1-1, to=1-2]
	\arrow[from=1-2, to=2-2]
	\arrow[from=2-1, to=2-2]
	\arrow["\lrcorner"{anchor=center, pos=0.125, rotate=180}, draw=none, from=2-2, to=1-1]
\end{tikzcd}\]
        so the result follows by the cocontinuity of $\partial_U \x (-)$, and $\mathsf{Gpd}(\w)$ is a cartesian closed category (so $X \x (-)$ is cocontinuous).
        \item Recall that the fork
        \[\begin{tikzcd}
            {I \x I} & {I \x I} & I
            \arrow["m", from=1-2, to=1-3]
            \arrow["{e[s] \x I}", shift left=1, from=1-1, to=1-2]
            \arrow["{I \x e[s]}"', shift right=1, from=1-1, to=1-2]
        \end{tikzcd}\]
        is a coequalizer. 
        Also note that $e^\partial$ is the coequalizer of $e, e[s]$. The commutativity of colimits then ensures that there is a multiplication map induced by the following diagram:
        \[\begin{tikzcd}
            {(D \x I)^2} & {(D \x I)^2} & {\partial^2} \\
            {(D \x I)^2} & {(D \x I)^2} & {\partial^2} \\
            {D\x I} & {D\x I} & \partial
            \arrow[shift left=1, from=1-1, to=1-2]
            \arrow[shift right=1, from=1-1, to=1-2]
            \arrow[from=1-2, to=1-3]
            \arrow[shift right=1, from=1-1, to=2-1]
            \arrow[shift left=1, from=1-1, to=2-1]
            \arrow[shift right=1, from=1-2, to=2-2]
            \arrow[shift left=1, from=1-2, to=2-2]
            \arrow[shift right=1, from=1-3, to=2-3]
            \arrow[shift left=1, from=1-3, to=2-3]
            \arrow[shift left=1, from=2-1, to=2-2]
            \arrow[shift right=1, from=2-1, to=2-2]
            \arrow[from=2-2, to=2-3]
            \arrow[dashed, from=2-3, to=3-3]
            \arrow[from=2-2, to=3-2]
            \arrow[from=3-2, to=3-3]
            \arrow[from=2-1, to=3-1]
            \arrow[shift right=1, from=3-1, to=3-2]
            \arrow[shift left=1, from=3-1, to=3-2]
        \end{tikzcd}\]
        Thus the multplication is associative, is commutative, and has a zero given by $1 \xrightarrow[]{0 \x s} D \x I$.
    \end{enumerate}
\end{proof}

Aan immediate corollary of Theorem \ref{thm:inf-nerve} is that the $\partial$ determines a nerve/ realization where the nerve factors through the category of involution algebroids.
This puts the Lie functor into a nerve/realization context (recall Definition \ref{def:realization}).
\begin{definition}%
    \label{def:lie-realization}
    The \emph{Lie realization} is the left Kan extension
    \[
        |-|_\partial = Lan_\partial: \mathsf{Inv}(\w) \to \mathsf{Gpd}(\w).  
    \]
\end{definition}

This functor is well behaved. First, ote that it preserved products:
\begin{lemma}
    The realization functor preserves products.
\end{lemma}
\begin{proof}
    Product preservation is a consequence of $|-|_\partial$ being the left Kan extension of a cartesian functor along a cartesian functor (\cite{Day1995}).
    Note that this implies $\int^v\partial(v) = 1$. 
\end{proof}
Next, we see that the realization of an involution algebroid has the same base space.

\begin{lemma}\label{lem:base-of-groupoid}
    The base space of the groupoid $\partial$ is $D^v$.
\end{lemma}
\begin{proof}
    When constructing the colimit of groupoids, the colimit's base space is the ordinary colimit for the diagram of the base spaces. The reflector from simplicial objects to groupoids preserves products, so it suffices to check that the base space of $\partial(n)$ is $D(n)$.

    The base space of $I$ is $1+1$, and the map $e^s_0$ is given by $\delta^- \o !$.
    Since $D$ is a discrete cubical object, its base space is $D \x 1$.
    Thus the coequalizer defining $\partial$ is
    \[
\begin{tikzcd}
    {D + D} && {D + D} & {\partial_0}
    \arrow["{0\o ! + 0 \o !}", shift left=1, from=1-1, to=1-3]
    \arrow["{id + 0 \o !}"', shift right=1, from=1-1, to=1-3]
    \arrow[from=1-3, to=1-4]
\end{tikzcd}.\]
    A map $\gamma: D + D \to M$ is a pair of maps $\gamma_0, \gamma_1: D \to M$.
    We can see that $\gamma_0$ and $\gamma_1$ agree at $0$ (they are both $\gamma_0(0)$), and $\gamma(0)$ is a constant tangent vector. It follows that $\gamma_1 \o (id | id) = \gamma$.
\end{proof}

Now recall the co-Yoneda lemma: for any $\vv$-presheaf $A: \C^{op} \to \vv$,
\[
    F(C) = \int^{C' \in \C} \C(C,C') \ox F(C').
\]
In particular, for an involution algebroid in $\w$ (or any $\w$-presheaf on $\wone^{op}$),
\[
    A(U) = \int^{V \in \wone} \wone^{op}(U,V) \x A(V) = \int^{V \in \wone} \wone(V,U) \x A(V).
\]
\begin{lemma}\label{lem:as-a-presheaf}
    For any involution algebroid $A$,
    \[ A(R) = \int^{V \in \wone} A(V) \x D^V. \]
\end{lemma}
\begin{proof}
    Use the tangent structure on $\wone$ to regard it as a $\w$-enriched category:
    \begin{gather*}
        D^V = \yon(V) = \wone(V,-) = (U \mapsto \wone(V, U))  \\= (U \mapsto \wone(V, U\ox R)) = \wone(V,R).
    \end{gather*}
    The following computation gives the result:
    \[
        A(R) = \int^{V} \wone(V,R) \x A(V) = \int^{V} D^V \x A(V).
    \]
\end{proof}
We can now see that the base space of an involution algebroid is isomorphic to the base space of its realization. That is to say, the realization sends an involution algebroid over a Weil space $M$ to a groupoid over the Weil space $M$ (up to isomorphism).

\begin{proposition}\label{prop:lie-int-first-part}
    Let $A$ be an involution algebroid in $\w$. Then $|A|([0]) = A(R)$.
\end{proposition}
\begin{proof}
    Use the Yoneda lemma, and the fact that $\yon[0] = 1$ is a small projective so that $\mathsf{Gpd}(1,-)$ is $\w$-cocontinuous. Now apply Lemma \ref{lem:base-of-groupoid} and Lemma \ref{lem:as-a-presheaf}:
    \begin{align*}
        |A|([0]) & = \mathsf{Gpd}(1, |A|)                                                   \\
                 & = \mathsf{Gpd}\left(1, \int^{v \in \wone} A(v)\bullet \partial v \right) \\
                 & = \int^{v} A(v)\x \mathsf{Gpd}(1,  \partial v )                          \\
                 & = \int^{v}  A(v)\x D^v = A(R).
    \end{align*}
\end{proof}

Thus, as a final result, we have achieved an adjunction between the category of involution algebroids and groupoids in $\w$ that is product-preserving and stable over the base spaces.
\pagebreak 
\lie
\[
\begin{tikzcd}
    {\mathsf{Gpd}(\w)} & {\mathsf{Inv}(\w)}
    \arrow[""{name=0, anchor=center, inner sep=0}, "{N_\partial}", curve={height=-12pt}, from=1-1, to=1-2]
    \arrow[""{name=1, anchor=center, inner sep=0}, "{|-|_\partial}", curve={height=-12pt}, from=1-2, to=1-1]
    \arrow["\dashv"{anchor=center, rotate=-90}, draw=none, from=0, to=1]
\end{tikzcd}\]

\begin{remark}
    Just as the introduction to this thesis begins with the work of Charles Ehresmann, we should take a moment to see how the Lie realization relates to his original research into sketch theory. Rather than sketches, we use their closely related cousins, \emph{essentially algebraic theories}, which in this case are small, finitely complete $\w$-categories. We may regard $\mathsf{Gpd}(\w)$ and $\mathsf{Inv}(\w)$ as $\w$-functor categories
    \[
        \mathsf{Lex}(\th_{\mathsf{Gpd}}, \w), \hspace{0.5cm}\mathsf{Inv}(\th_{\mathsf{Gpd}}, \w),
    \]
    respectively (where $\mathsf{Lex}$ means finite-limit-preserving $\w$-functors). The functor $\partial:\wone^{op} \to \mathsf{Gpd}(\w)$ induces a left-exact $\w$-functor
    \[
        \hat{\partial}: \th_{\mathsf{Inv}} \to \th_{\mathsf{Gpd}}.
    \]
    This means the functor from Lie groupoids to Lie algebroids is induced by a morphisms of essentially algebraic theories, and may thus be presented as a morphism of sketches.  
\end{remark}


\chapter{Conclusions and future work}

\section{Conclusions}

We now take stock of what this thesis has accomplished.

\paragraph{Enriched essentially algebraic presentations of geometric structures}

This thesis has been concerned with categories of vector bundles and Lie algebroids in smooth manifolds.
While these are not models of a limit sketch in the category of smooth manifolds, classical results such as the Serre--Swann theorem or Vaintrob's presentation of Lie algebroids indicate that these categories do have some algebraic description.
By combining the results in Chapters \ref{ch:differential_bundles} and \ref{ch:involution-algebroids}, we can see that vector bundles and Lie algebroids are characterized by tangent categorical gadgets, differential bundles (Theorem \ref{iso-of-cats-dbun-sman}) and involution algebroids (Theorem \ref{thm:iso-of-cats-Lie}).
By combining these observations with the enriched perspective on tangent categories, the results in Section \ref{sec:enriched-structures} exhibit vector bundles and Lie algebroids as models of \emph{enriched sketches} (Proposition \ref{prop:Lambda-is-refl-subcat} and Theorem \ref{thm:pullback-in-cat-of-cats-inv-algd}), as can also be found in Chapter 6 of \cite{Kelly2005}.
Note that from this perspective, the construction of the Weil nerve of an involution algebroid is mostly important as an intermediate step.

\paragraph{New tangent structures for Lie algebroids and groupoids}

It is well known that the categories of Lie algebroids and groupoids have tangent structures induced by post-composition with the tangent functor, yielding the tangent algebroid and tangent groupoid, respectively.
The key result in Chapter \ref{chap:weil-nerve} exhibiting the category of Lie algebroids as transverse-limit-preserving cartesian-tangent functors $\wone \to \mathsf{SMan}$, then, gives the category of Lie algebroids a second tangent structure that corresponds to Martinez's \emph{prolongation} Lie algebroid. 
Similarly, the construction of the infinitesimal nerve of a local groupoid restricts to a new tangent structure on the category of Lie groupoids; in this case, the base space is the Lie algebroid of the groupoid $G$, and the total space is the Lie algebroid of the arrow groupoid $G^2$. 

While this may seem like a piece of formal category theory, it is closely related to symmetry reduction in classical mechanics using Lie theory. This goes back to Poincar\'{e}'s celebrated note \cite{poincare1901}, which introduced the \emph{Euler-Poincare} formulation of Lagrangian mechanics on a manifold equipped with a Lie group action, where the tangent bundle is replaced with a Lie algebra action (see \cite{marle2013} for a modern exposition). Poincar\'{e}'s formalism has been extended to Lie groupoids and Lie algebroids in \cite{Weinstein1996} and \cite{Martinez2001}, and the thesis \cite{fusca2018} investigates fluid mechanics through this lens. This work may be interpreted as using the novel tangent structures on Lie groupoids and Lie algebroids given in Chapters \ref{chap:weil-nerve} and \ref{ch:inf-nerve-and-realization}, and warrants further investigation.

\paragraph{Functorial semantics of Lie theory}
The work in Section \ref{sec:inf-nerve-of-a-gpd} puts the Lie derivative on a new formal grounding:
\begin{enumerate}[(i)]
    \item In the enriching category $\w$, the Lie derivative functor now arises as a nerve/realization context. This guarantees the existence of a left adjoint---the \emph{realization}---that constructs a Lie groupoid from an algebroid. The realization preserves products and is stable over the base space.
    \item It is only a small extension of the work in Section \ref{sec:inf-nerve-of-a-gpd} to show that the Lie derivative of a groupoid in a tangent category may generally be regarded as a weighted limit, where $A_V = \{\partial(V), G\}$ for each Weil prolongation of the involution algebroid of the groupoid.
\end{enumerate}
These appear to be new results, although in a nerve/realization approach they had previously appeared in Lie theory in the form of Sullivan's construction (\cite{Sullivan1977}), which takes the differential graded algebra of a Lie algebroid and constructs a simplicial set:
\[
    \mathsf{DGA}^{op}(\Omega(\Delta_n), A)
\]
where $\Omega(\Delta_n)$ is the de Rham cohomology of the $n$-simplex in cartesian space.

\section{Future work}

This section outlines various lines of research that were either cut from the thesis while writing due to time and/or space constraints, or new lines of research that the thesis-writing process has motivated but which have not yet been pursued.

\paragraph{Enriched sketches and Mackenzie theory}

Following \cite{Voronov2012}, Mackenzie theory refers to the body of research developed by Kirill Mackenzie and his collaborators into Lie theory, particularly structures like double Lie algebroids (\cite{Mackenzie1992}), VB-Lie algebroids (\cite{Bursztyn2016}), double Lie groupoids and so on.
Intuitively, a double Lie algebroid is a Lie algebroid in the category of Lie algebroids, while a VB-Lie algebroid is a vector bundle in the category of Lie algebroids.
These structures have an intuitive relationship with tensor products of sketches; that is, for limit sketches $A,B$ there is a chain of isomorphisms of categories:
\[
    \mathsf{Mod}(A, \mathsf{Mod}(B,\s)) \cong \mathsf{Mod}(A \ox B, \s) \cong \mathsf{Mod}(B, \mathsf{Mod}(A, \s)).
\]
We have already demonstrated that the Lie algebroids and vector bundles are $\w$-sketchable, so the natural next step is to revisit Mackenzie theory via this lens. Certain results such as the symmetry of partial Lie derivatives from \cite{Mackenzie1992} (given a double Lie groupoid, the order in which the Lie functor is applied doesn't matter) should be immediate.

\paragraph{A tangent categorical formalization of mechanics}

Several papers in synthetic differential geometry have translated aspects of Lagrangian and Hamiltonian mechanics into the synthetic setting (\cite{Bunge1984,Nishimura1997a}), and to a degree this has made it difficult to build enthusiasm for a tangent categorical presentation of mechanics. These approaches generally emphasize the ability to construct function spaces, or the use of a topos-theoretic internal language. The novel tangent structures for Lie algebroids and groupoids presented here, however, provide a new line of inquiry: to develop a unified framework for Lagrangian or Hamiltonian mechanics in a tangent category that agrees with the algebroid and groupoid approaches to mechanics.

\appendix
\chapter*{Appendix: Background on locally presentable $\vv$-categories}\label{appendix}


\addcontentsline{toc}{chapter}{Appendix}
\chaptermark{Appendix}
\markboth{Appendix}{Appendix}

Synthetic differential geometry uses a topos-theoretic framework, such as can be found in the final chapters of \cite{Lavendhomme1996}.
The internal language of a topos is a key tool, and models of synthetic differential geometry are constructed using sheaf constructions.
Tangent categories have an analagous toolbox, drawing from \emph{enriched locally presentable categories}. We begin with locally presentable categories; basic material on these may be found in \cite{Adamek1994}, and on enriched categories in \cite{Kelly2005}.
The following class of colimits is foundational in the theory of locally presentable categories.
\begin{definition}\label{def:filtered-colimit}
	A category is \emph{filtered} whenever any finite diagram in $\C$ has a cocone. A colimit whose diagram category is filtered is called a \emph{filtered colimit}. Note that a category $\d$ is filtered if and only if every colimit diagram $\d \to \s$ commutes with every finite limit in $\s$.
\end{definition}
\begin{example}
    ~\begin{enumerate}[(i)]
        \item Any finite category $\c$ with a terminal object is filtered: for any diagram $D:\d \to \c$, the natural transformation $!:id \Rightarrow K_1$ is a cocone via whiskering.\[
\begin{tikzcd}
    \d & \c & {} & \c
    \arrow["D", from=1-1, to=1-2]
    \arrow[""{name=0, anchor=center, inner sep=0}, curve={height=-12pt}, Rightarrow, no head, from=1-2, to=1-4]
    \arrow[""{name=1, anchor=center, inner sep=0}, "{K_1}"', curve={height=12pt}, from=1-2, to=1-4]
    \arrow["{!}"', shorten <=3pt, shorten >=3pt, Rightarrow, from=0, to=1]
\end{tikzcd}\]
        \item Any category with finite colimits is filtered, because each diagram has a colimit (and therefore a cocone). 
    \end{enumerate}
\end{example}
\begin{definition}%
    \label{def:lfp}
	An object in a cocomplete $\,\C$ is \emph{finitely presentable} whenever $\C(C,-)$ preserves filtered colimits. A cocomplete category $\,\C$ is \emph{locally finitely presentable} whenever it has a small subcategory $\,\C_{fp}$ of finitely presentable objects so that every object in $\C$ is given by the coend
	\[
		C \cong \int^{X \in \C_{fp}} \C(X, C) \cdot X.
	\]
\end{definition}
\begin{example}
	In the category $\s$, the subcategory of finitely presentable objects is precisely the category of finite sets. 
	Every set is a filtered colimit of finite sets, so that $\s$ is locally finitely presentable. 
\end{example}
Intuitively, this means that every object in $\C$ is generated by gluing together finitely presentable objects in a canonical way; hence the name locally finitely presentable categories. 

The subcategory of finitely presentable objects in a cocomplete category is itself finitely cocomplete.
This follows from the fact that $\C(-,X):\C^{op} \to \s$ is a continuous functor (it sends colimits in $\C$ to limits in $\s$) and also that filtered colimits commute with finite limits in $\s$.  
Thus, for any filtered colimit of representables,
\[
    \C(\mathsf{col}D, \mathsf{lim}F) \cong \mathsf{lim}_X \mathsf{lim}_Y  \C(D(X),F(Y)) \cong \mathsf{lim}_Y \mathsf{lim}_X  \C(D(X),F(Y)). 
\]
There are several ways to characterize locally finitely presentable categories; these may be found in \cite{Adamek1994}.
\begin{proposition}%
    \label{prop:lfp-defs}
    For a category $\,\C$ the following are equivalent:
    \begin{enumerate}[(i)]
        \item $\C$ is a locally finitely presentable category.
        \item For a small cocomplete category $\c$, $\C = \mathsf{Lex}(\c, \s)$ (where $\mathsf{Lex}$ means finite-limit-preserving) and $\c = \C^{op}_{fp}$.
        \item $\C$ is the category of models for some \emph{limit sketch} (e.g. a small category $\c$ equipped with a class of cones $\prol$, where a model is a functor into $\s$ sending cones to limit diagrams).
        \item $\C$ is a reflective subcategory of a presheaf topos for some small category $\c$ (this holds tautologically for presheaf topoi).
        \item $\C$ is a full subcategory of a locally presentable category that is closed under limits (this relies on a large cardinal axiom known as \emph{Vopenka's principle}).
    \end{enumerate}
\end{proposition}

One of the remarkable aspects of the theory of locally finitely presentable categories is that its results translate over to enrichment in a monoidal category $\vv$ without any significant changes.
\begin{definition}%
\label{def:v-cat}
    Let $\vv$ be a monoidal category. 
    A $\vv$-graph $\C$ is given by a (large) set of objects $\C_0$, and a map
    \[
        \C: \C_0 \x \C_0 \to \vv. 	
    \]
    Given a collection of maps
    \[
        m_{ABC}: \C(B,C) \x \C(A,B) \to \C(A,C), \hspace{0.25cm}
        j_A: I \to \C(A,A)	
    \]
    we say $\,\C$ is a $\vv$-category whenever the following associativity and unitality diagrams commute:
\[\begin{tikzcd}[column sep = tiny]
	& {\C(B, D) \ox \C(A,B)} \\
	{(\C(C, D) \ox \C(B, C)) \ox \C(A,B)} && {\C(A,D)} \\
	{\C(C, D) \ox (\C(B, C) \ox \C(A,B))} & {} & {\C(C, D) \ox \C(A,C)}
	\arrow["m", from=3-3, to=2-3]
	\arrow["{m \ox id}", from=2-1, to=1-2]
	\arrow["m", from=1-2, to=2-3]
	\arrow["\alpha"', from=2-1, to=3-1]
	\arrow["{id \ox m}"', from=3-1, to=3-3]
\end{tikzcd}\]
\[\begin{tikzcd}
	{I \ox \C(B,C)} & {\C(B,B)\ox\C(B,C)} \\
	{\C(B,C)} & {\C(B,C)} \\
	{\C(B, C) \ox I} & {\C(B,C)\ox\C(C,C)}
	\arrow["\lambda", from=2-1, to=1-1]
	\arrow["{j \ox id}", from=1-1, to=1-2]
	\arrow["m", from=1-2, to=2-2]
	\arrow["\rho"', from=2-1, to=3-1]
	\arrow["{id \ox j}"', from=3-1, to=3-2]
	\arrow["m"', from=3-2, to=2-2]
	\arrow[Rightarrow, no head, from=2-1, to=2-2]
\end{tikzcd}\]

\end{definition}
Examples of enrichment abound throughout category theory.
\begin{example}
    ~\begin{enumerate}[(i)]
        \item Any category with biproducts is enriched in the category of commutative monoids.
        \item A symmetric monoidal closed category is self-enriched using the internal hom $[-,-]$.
        \item A 2-category is a category enriched in $\mathsf{Cat}$. 
        \item The unit $\vv$ category is the 1-object $\vv$ category whose hom-object is $I$.
        \item When $\vv$ is symmetric monoidal, we can construct the opposite $\vv$-category of $\c$, where $\C^{op}(A,B) = \C(B, A)$.
        The new composition is defined by
        \[m^{op}: \C^{op}(B,C) \ox \C^{op}(A,B) \xrightarrow{\sigma} \C(B, A) \ox  \C(C,B) \xrightarrow{m} \C(C,A) = \C^{op}(A,C) \]
    \end{enumerate}
\end{example}
    
There is a 2-category of $\vv$-categories for any $\vv$.
\begin{definition}%
\label{def:vcat}
	A $\vv$-functor $F: \C \to \D$ is an assignment on objects $F_0: \C_0 \to \D_0$ and a collection of maps $F_{A,B}: \C(A,B) \to \D(F_0A,F_0B)$ so that the following diagrams are satisfied:
    \[
\begin{tikzcd}
    {\C(B,C)\ox \C(A,B) } & {\D(F_0B, F_0C) \ox \D(F_0A, F_0B)} &[-2em] I & {\C(A,A)} \\
    {\C(A, C)} & {\D(F_0A, F_0C)} && {\D(F_0A,F_0A)}
    \arrow["{m^C}"', from=1-1, to=2-1]
    \arrow["{F \ox F}", from=1-1, to=1-2]
    \arrow["{m^D}", from=1-2, to=2-2]
    \arrow["F"', from=2-1, to=2-2]
    \arrow["F", from=1-4, to=2-4]
    \arrow["{j^C_{A}}", from=1-3, to=1-4]
    \arrow["{j^D_{F_0A}}"', from=1-3, to=2-4]
\end{tikzcd}\]
	A $\vv$-natural transformation $F \Rightarrow G$ is a collection of maps $\gamma_{A}: I \to \D(F_0A,G_0A)$ so that the following diagram commutes: \[
\begin{tikzcd}
    {\D(F_0B,F_0C)} & {\D(F_0B,F_0C)\ox I} & {\D(F_0B,F_0C)\ox \D(G_0B,F_0B)} \\
    {\C(B,C)} && {\D(G_0B,F_0C)} \\
    {\D(G_0B,G_0C)} & {I \ox \D(G_0B,G_0C)} & {\D(G_0B, F_0B) \ox \D(G_0B,G_0C)}
    \arrow["G"{description}, from=2-1, to=1-1]
    \arrow["F"{description}, from=2-1, to=3-1]
    \arrow["\rho", from=1-1, to=1-2]
    \arrow["{id\ox \gamma.B}", from=1-2, to=1-3]
    \arrow["m", from=1-3, to=2-3]
    \arrow["\lambda"', from=3-1, to=3-2]
    \arrow["{\gamma.C \ox id}"', from=3-2, to=3-3]
    \arrow["m"', from=3-3, to=2-3]
\end{tikzcd}\]
    $\vv\cat$ is the 2-category of $\vv$-categories, $\vv$-functors and $\vv$-natural transformations.
\end{definition}
	
Enriched categories have a richer notion of (co)limit than ordinary categories.
A limit for a functor is defined as a universal cone; that is, a right Kan extension along $!$:
\[
\begin{tikzcd}
	& 1 \\
	\d && \C
	\arrow["{!}", from=2-1, to=1-2]
	\arrow["{K_A}", from=1-2, to=2-3]
	\arrow[""{name=0, anchor=center, inner sep=0}, "F"', from=2-1, to=2-3]
	\arrow["\beta"', shorten <=5pt, shorten >=3pt, Rightarrow, from=1-2, to=0]
\end{tikzcd}\]
so that the universal property
\[
    \forall C:
    \hspace{0.15cm}
    \C(C, \lim F) \cong [\C,\s](K_C, F-)
\]
is satisfied.
Dually, a colimit is a universal cocone, and so a left Kan extension:
\[
\begin{tikzcd}
	\d && \C \\
	& 1
	\arrow[""{name=0, anchor=center, inner sep=0}, "F", from=1-1, to=1-3]
	\arrow["{!}"', from=1-1, to=2-2]
	\arrow["{K_A}"', from=2-2, to=1-3]
	\arrow["\beta", shorten <=3pt, shorten >=5pt, Rightarrow, from=0, to=2-2]
\end{tikzcd}\]
where the couniversal property
\[
    \forall C:
    \hspace{0.15cm}
    [\C^{op},\s](F-, K_C) \cong \C(\mathsf{colim}F, C) 
\]
is satisfied.
This does not translate over to $\vv$-categories, for a few reasons. 
An obvious one is that $\vv$ need not be cartesian monoidal, so there is no reason a $\vv$-category should have a functor $\C \to 1$.
The notion of a \emph{weighted} colimit is more appropriate:
\begin{definition}
	Let $\d$ be a $\vv$-category, and call $W: \d \to \vv$ the weight and $D: \d \to \C$ the diagram.
	The limit of $D$ weighted by $W$, or the weighted limit $\{ W, \d \}$, is an object satisfying the following universal property:
	\[
        \forall C: \hspace{0.15cm} \vv(W, [\C,\vv](F-, K_C)) \cong \C({F,W}, C).
	\]
	Dually, given a diagram $P: \d^{op} \to \C$, the colimit of $P$ weighted by $W$, $W*P$, is an object in $\C$ satisfying the universal property
	\[
        \forall C: \hspace{0.15cm}  \vv(W, [\C,\vv](K_C, F-)) \cong \C(C, F*W ).
	\]
\end{definition}
\begin{example}
	~\begin{enumerate}[(i)]
		\item A conical (co)limit is weighted by the constant functor into the unit $I$, $K_I: \d \to w$.
		This coincides with the definition of ordinary (co)limits.
		\item The power of $C$ by an object $V \in \vv$ is an object $C^V$ satisfying the universal property that
		\[
			\C(B, C^V) \cong \vv(V, \C(B,C)).
		\]
		Dually, the copower by $V$ is given by 
		\[
			\C(B\bullet V, C) \cong \vv(V, \C(B, C)). 
		\]
		\item In ordinary categories, the power of an object $C$ by $V$ is given by the (possibly infinitary) $V$-indexed coproduct:
		\[
			C^V \cong \coprod_{v\in V} C	
		\]
		whereas the copower is the possibly infinitary $V$-index product of $C$:
		\[
			V\bullet C \cong \prod_{v \in V} C.
		\]
		\item In a symmetric monoidal closed category $\vv$, the power by $\vv$ is given by $[V,-]$ and the copower by $V \ox \_$.
	\end{enumerate}
\end{example}

Following Kelly's work, we observe that every finite weighted limit is composed of powers and conical limits.
\begin{proposition}
	Every weighted (co)limit is the composition of conical (co)limits and (co)powers.
\end{proposition}
\begin{corollary}
	Weighted limits are exactly conical/ordinary limits in ordinary categories.
\end{corollary}

The previous definitions relating to locally presentable categories transfer over mostly unchanged.
\begin{definition}%
    \label{def:v-lfp}
	~\begin{enumerate}[(i)]
		\item A monoidal closed $\vv$ is \emph{presentable as a closed category} if $\vv_{fp}$ is a monoidal subcategory.
		\item Given $\vv$ as in (i), a weight $W: \d \to \vv$ is \emph{filtered} whenever $W\star (-): \widehat\d \to \vv$ preserves finite limits.
		\item An object in a cocomplete $\vv$ category is \emph{finitely presentable} when $\C(C,-)$ preserves filtered weighted colimits.
		\item A cocomplete $\vv$-category $\c$ is \emph{locally finitely presentable} whenever every object $C$ is the coend 
			\[C \cong \int^{X \in \C_{fp}}\C(X,C)\cdot X. \]
	\end{enumerate}
\end{definition}
\begin{proposition}[\cite{Freyd1972, Kelly2005}]%
    \label{prop:v-lfp}
	All of the previous results about locally presentable categories lift up to locally presentable $\vv$-categories, provided that $\vv$ is presentable as a closed category.
\end{proposition}

\backmatter

\printpagenotes

\bibliography{References}




\end{document}